\input amstex
\documentstyle{amsppt}
\magnification=1200
\baselineskip=15pt plus 2pt

 \pagewidth{6.0 true in}
\pageheight{9.0 true in}

\hoffset 0.2in 
\voffset 0.5in

\parindent 0pt
\parskip 14pt

\def\bls{$\blacksquare$}
\def\mlra{\mathop{\longrightarrow}\limits}
\def\lra{\longrightarrow}
\def\Tr{\,\text{Tr}\,}
\def\Ad{\,\text{Ad}\,}
\def\ad{\,\text{ad}\,}
\def\Td{\,\text{Td}\,}
\def\ad{\,\text{ad}\,}
\def\tr{\,\text{tr}\,}
\def\ch{\,\text{ch}\,}
\def\END{\,\text{End}\,}
\def\Ker{\,\text{Ker}\,}
\def\GL{\,\text{GL}}
\def\odd{\,\text{odd}\,}
\def\hom{\,\text{Hom}\,}
\def\Ind{\,\text{Ind}\,}
\def\Spin{\,\text{Spin}\,}

\def\ev{\,\text{ev}\,}
\def\odd{\,\text{odd}\,}

\def\Id{\,\text{Id}\,}

\def\sech{\,\text{sech}\,}
\def\rank{\,\text{rank}\,}

\def\tanh{\,\text{tanh}\,}
\def\coth{\,\text{coth}\,}
\def\sinh{\,\text{sinh}\,}
\def\rsp{\,\text{span}\,}
\def\vol{\,\text{vol}\,}
\def\bs{\backslash}
\def\n{\noindent}
\def\wt{\widetilde}
\def\cl#1{\Cal{#1}}
\def\ov#1{\overline{#1}}
\def\lno{\leqalignno}
\def\ind{\,\text{Ind}\,}
\def\bx{$\blacksquare$}
\def\exam#1#2{\dimen1=2.25em
                      \noindent\hangindent=\dimen1\hangafter1
                      \hbox to\dimen1{#1\hfil~~}#2} 
\def\exama#1#2#3{\dimen1=2.25em \dimen2=4.5em
                      \noindent\hangindent=\dimen2\hangafter1
		      \hbox to\dimen1{#1\hfil~~}\hbox to\dimen1{#2\hfil~~}#3} 
\def\examb#1#2{\dimen1=6.25em
                      \noindent\hangindent=\dimen1\hangafter1
                      \hbox to\dimen1{#1\hfil~~}#2} 
\def\list#1#2{\noindent\hangindent 3em\hangafter1\hbox to 3em{\hfil#1\quad}#2}

\def\wh{\widehat}
\overfullrule=0pt

\rightheadtext{Holomorphic torsion and zeta functions}
\leftheadtext{H. Moscovici \& R. J. Stanton}

\topmatter
\title Holomorphic torsion and geometric zeta functions
for certain Hermitian locally symmetric manifolds
\endtitle
\author Henri Moscovici$^1$
and Robert J. Stanton$^2$
\endauthor

\address The Ohio State University, Columbus, Ohio 43210
\endaddress
\thanks {Supported in part by the NSF
award DMS-024548 and DMS-16005, IH\'ES and MPIM Bonn}
\endthanks
\thanks {Supported in part by the NSF, IRMA at Universit\'e de Strasbourg, IH\'ES and MPIM Bonn}
\endthanks
\endtopmatter

\n{\bf \S 0~~Introduction}
\medskip

Holomorphic torsion, also called $\bar\partial$-torsion, 
 is a positive 
real number  $\tau (X, \Bbb V)$ defined by Ray and Singer [R-S; II] for
 a compact complex Hermitian manifold
 $X$ endowed with a Hermitian complex vector bundle $\Bbb V$,
in terms of zeta-regularized determinants of
Hodge-Laplace operators on $(0, q)$-forms on $X$ with values in $\Bbb V$. 
 It was shown
by Quillen [Q] in the case of complex curves, and by Bismut, Gillet and
Soul\'e  [BGS; I, II, III] in general, that the $\bar\partial$-torsion provides
``the'' natural Hermitian metric on the determinant line bundle for direct images
over complex algebraic manifolds. Given the perceived pattern of formal
analogies between the Arakelov theory in arithmetic geometry and dynamical systems
(cf. e.g. [Den]), 
it appears desirable to have a geometric-dynamical description of the $\bar\partial$-torsion.
Such an interpretation was obtained by D. Fried [F2] in the case of
complex hyperbolic manifolds. These are compact quotients 
$X=\Gamma \backslash {\frak H}^n_{\Bbb C}$ of
complex hyperbolic spaces modulo
discrete subgroups of holomorphic isometries $\Gamma \subset SU(n,1)$.  
In this paper we extend the geometric-dynamical description of the $\bar\partial$-torsion
to the case of compact Hermitian locally symmetric manifolds
$X=\Gamma \backslash G/K$ 
whose connected group $G$ of isometries for
the universal cover has only one class of cuspidal maximal parabolic subgroups 
or, equivalently, whose Satake diagram has no double lines.
 
To formulate our main result, let us fix such a compact complex manifold
 $X$; thus, its universal cover 
 is a Hermitian globally symmetric
space $\widetilde{X} \simeq G/K$, with
 $G$ semisimple non-compact, $K$ a maximal compact
subgroup, and $\Gamma$ a discrete, co-compact, torsion-free subgroup of $G$.  
We endow $X$ with the locally symmetric metric $g$ induced by the 
standard Hermitian metric on  $\widetilde{X}$, and also
fix a holomorphic, Hermitian, homogeneous vector bundle  $\Bbb V$
over $X$, as well as a flat bundle $\Bbb F$ associated to a unitary representation 
$\varphi: \pi_1(X)= \Gamma \to U(F)$. To these data we add a crucial new feature, namely the choice of a
standard compactification of $\widetilde{X}$, more specifically the closure 
$\overline{X} = \overline{G\cdot K_\Bbb CP_-}$ 
of the Borel embedding of $\widetilde{X}$ into $G_\Bbb C/K_\Bbb CP_-$.
Recall that the periodic set of the geodesic flow is a disjoint union of
connected components $\coprod\limits_{[\gamma ]}X_\gamma$, parametrized
by the conjugacy classes in $\Gamma \bs\{e\}$. The components $X_\gamma$
are themselves locally symmetric manifolds 
$X_\gamma\cong\Gamma_\gamma\backslash G_\gamma /K_\gamma$.
All geodesics in $X_\gamma$ have common
length $\ell_\gamma$, and we let $\mu_\gamma$ denote the generic 
multiplicity. Although there are Hermitian groups with two maximal cuspidal parabolics,  
we restrict to those $G$ which have only one, and build a dynamical zeta function supported by
the set ${\Cal E}_1(\Gamma)$ consisting of the nontrivial conjugacy classes $[\gamma]$ of $\Gamma$ such that $\gamma$ is conjugated to the maximal cuspidal parabolic corresponding to the 
largest boundary $G$-orbit $\Cal O_\kappa$ in $\partial \overline{X}=\overline{X}-\widetilde{X}$. 
The zeta function $Z_{\Bbb V ,\Bbb F}$ is defined by the formula
$$  Z_{\Bbb V ,\Bbb F}(z) =  \exp \left( - \sum_{[\gamma] \in {\Cal E}_1(\Gamma)}
  \Tr \varphi (\gamma) \, \chi_\gamma (\Bbb V)  \,
{e^{-z\ell_\gamma}\over \mu_\gamma} \right),  \qquad \text{where} $$
$$ \chi_\gamma (\Bbb V) = {1 \over \det(I-P_{\Cal U} (\gamma)) 
\det(I- P_{\Cal V_-} (\gamma))} \left\{ {\Td T(\widehat{X}_\gamma)
\cup \ch [\Bbb V_{\gamma ,\kappa}](\gamma) 
\over\ch [\wedge_{-1}{\Cal N}_+(\widehat{X}_\gamma)^*](\gamma )}\right\} [\widehat{X}_\gamma ]\, .
$$
$\widehat{X}_\gamma = \Gamma_\gamma\backslash \widetilde{X}_\gamma /\tilde{\Phi}$ 
is the orbifold obtained from the quotient of $\widetilde{X_\gamma}$ by the geodesic flow
$\tilde{\Phi}$; $\widetilde{X_\gamma} /\tilde{\Phi}$ is viewed as embedded in the boundary orbit
$\Cal O_\kappa$, and $\Bbb V_{\gamma ,\kappa}$ denotes the canonical extension of $\Bbb V$
to the boundary. As such,
$\widehat{X}_\gamma$ gets equipped with geometric structures and bundles implicit in the above
formula (which are explained in detail in \S2). Suffices to say that ${\Cal N}_+(\widehat{X}_\gamma)$
is inherited from the normal bundle relative to $\Cal O_\kappa$, while 
$ P_{{\Cal U}\oplus{\Cal V}}$ stands for a boundary version of the Poincar\'e map.

The zeta function thus defined can be shown to be holomorphic in an appropriate domain.
However, in order to make sure that it has meromorphic 
continuation to $\Bbb C$
we appeal to an ad hoc Ansatz (described in \S 5). 
The main result (Theorem 9.1) asserts then that
the leading term of the Laurent expansion of $Z_{\Bbb V ,\Bbb F}$ 
captures the holomorphic torsion $\tau(X,\Bbb V, \Bbb F)$ as part of its leading coefficient, 
while its order at $z = 0$ involves the derived Euler characteristic of the Dolbeault complex
$\left(\Omega^{0, *}(X, \Bbb V \otimes \Bbb F), \bar\partial\right)$.
Instead of giving the precise statement, 
which would require extra notation, we formulate  here a relative version (viz. Corollary 9.2):
given two flat bundles $\Bbb F_1$ and $\Bbb F_2$ with $\rank \Bbb F_1=\rank \Bbb F_2$ (and
assuming that the holomorphic bundle $\Bbb V$ satisfies a mild condition), then
$$ {Z_{\Bbb V ,\Bbb F_1} \over Z_{\Bbb V ,\Bbb F_2}} \sim
\biggl({\tau(X,\Bbb V,\Bbb F_1)\over \tau(X,\Bbb V,\Bbb F_2)}\biggr)^2 z^{\nu_1 - \nu_2}, \qquad
\text{around} \quad z= 0,
$$
where \qquad $\nu_1 - \nu_2 =  \sum (-1)^q q \left(h^{0,q}_1 - h^{0,q}_2\right)$, \,  with \,
$h^{0,q}_i := h^{0,q}_i (X,\Bbb V \otimes \Bbb F_i), \, i=1,2$.

 It is worth mentioning that torsion (in its standard algebraic sense) as well as the
 derived Euler characteristic already appear, at least conjecturally, in a similar posture in
the context of arithmetic zeta functions (cf. e.g. [Den], [Licht], [Mor]).

After some Lie-theoretical groundwork in \S1, we described the geometric zeta function
in full detail in \S 2 and  then work out its representation-theoretic expression in \S 3. 
The analysis necessary to develop its properties,
which is based on the Selberg trace formula applied to the heat kernel of
the Hodge-Laplace operators, is carried out in \S\S 4--8 under two basic assumptions, viz. that there is only one class of cuspidal parabolic subgroup, and the validity of the technical {\it Ansatz} (explained in
\S 5). Evaluating explicitly the orbital integrals involved 
in this approach is usually a formidable task, and especially so
in the higher rank case. However,
in this instance, the underlying geometric nature of the integral 
operator introduces enough cancellations to allow a tractable 
calculation. More precisely, these simplifications arise from the 
``matching'' of orbital integrals stemming from the Ansatz with those of the heat kernel, in
conjunction with the Lie-algebraic results established in \S 1. 
The material developed in the previous sections is used in \S9 to express the geometric
zeta function $Z_{\Bbb V ,\Bbb F}$  in terms of the intermediate zeta functions constructed
by means of the Ansatz, which in turn allows to achieve the proof of the main result. Finally,
\S 10 contains some comments of a conjectural nature and a proportionality result for the
``anti-derivative'' of the holomorphic torsion (Proposition 10.1).
  
This work was essentially completed 20 years ago but for the determination of the validity
of the {\it Ansatz}. Lacking a satisfactory understanding of its scope, we postponed its
formal submission for publication.
The efforts of our junior colleague, Jan Frahm, to clarify the problem and to resolve several cases 
 of it are presented in an Appendix. We are extremely grateful to him for this. From his work one can see that for $\Bbb R$-rank one groups our construction works 
with completely general vector bundle coefficients, for some rank two groups with specific vector bundle coefficients, and also that there are groups which do not satisfy the {\it Ansatz}.

%%%
The postponement of the present work for publication 
had the unfortunate effect of delaying the rectification of a gap in the
justification of the convergence of certain orbital integral expansions 
in the paper [M-S;II]. 
Instead of the flawed presentation therein, one should use exactly the same reasoning
as in \S 5 below, which essentially reproduces Fried's method. In the case of trivial
coefficients a different correct justification was supplied by Shu Shen [Sh]. His proof of Fried's
conjecture is based on Bismut's orbital integral formula [Bi]. 
It would be interesting to investigate if Bismut's 
method of hypoelliptic Laplacians gives rise to a canonical class of zeta function 
encoding the holomorphic torsion with coefficients, independent of the choice of 
compactification for the globally symmetric space  $\widetilde{X}= G/K$, and
of the validity of the {\it Ansatz}.

%\newpage

\n{\bf \S 1~~Lie-theoretical preparations}
\medskip

Let $\wt X$ be a $2n$ dimensional globally symmetric
space, Hermitian and of non-compact type.  The connected
group of isometries, $G$, is a semisimple Lie group \footnote {On p. 88 [Sa] $G$ is identified as the "algebraic adjoint group". Whereas on p. 114 for classical groups $G$ is allowed to be the usual matrix group. For compatibility with [Sa] and the Appendix we allow $G$ to be the usual matrix group. Hence there may be elements of $G$ that act trivially on $\wt X$.}. If
we fix a basepoint in $\wt X$ and let $K$ be the isotropy
subgroup at this point, then $K$ is a maximal compact
subgroup of $G$ and $\wt X$ is diffeomorphic with $G/K$. 
Let the Lie algebra of $G$ (resp. $K$) be denoted by $\frak g$ (resp. 
$\frak k$).  The tangent space to $\wt X$ at $eK$ can
be identified with a subspace $\frak p\subseteq \frak g$ so that 
$\frak g=\frak k\oplus \frak p$ is a Cartan decomposition.  A complex
structure as well as a Hermitian metric on $G/K$ arise
naturally from $\frak g$.  Indeed, there is an element
$H_0$ from the center of $\frak k$ such that $ad\,H_0$ on $\frak p$
defines an almost complex structure which we assume agrees with that on $\wt X$. The 
usual metric on $G/K$ is 
obtained simply from the restriction of the Killing form to $\frak p$ (then 
translated by $G$ around $G/K$). We shall also need metrics obtained by 
scaling the usual metric. The effect of scaling in the formulas will 
be 
indicated later, for now we assume that $\wt X$ is isometric to $G/K$ with the usual 
metric.

One has the decomposition $\frak p_{\Bbb C}=
\frak p_+\oplus \frak p_-$, with $\frak p_{\pm}$ the $\pm i$ eigenspaces of 
$ad\,H_0$. As $H_0$ is central in 
$\frak k$, each of $\frak p_{\pm}$ is invariant by $Ad\,K$.  We shall 
denote by $\Lambda^p_+$ (resp. $\Lambda^q_-$) the representation of $K$ on 
$\Lambda^p\frak p_+$ (resp. $\Lambda^q\frak p_-$). As $\frak p_+$ is $K$-equivariantly isomorphic to $(\frak p_-)^\ast$, we have $\Lambda^p_+$ is $K$-equivariantly isomorphic to $(\Lambda^p_-)^\ast$. 

For later purposes let 
us fix now a Cartan subalgebra $\frak t\subseteq \frak k$ containing $H_0$.  
The usual accompanying baggage of structure shall be with respect to 
$\frak t$ and used freely. In particular, we will use any order on $\Delta (\frak 
g)$ (= $\Delta(\frak g_{\Bbb C}, \frak t_{\Bbb C})$) 
such that $\frak p_+$ is the span of the positive non-compact root vectors. As $H_0$ 
is central in $\frak k$ this does not select an order on $\Delta (\frak k)$  (= $\Delta(\frak k_{\Bbb C}, \frak t_{\Bbb C})$) .

We fix $(\tau,V)$ an irreducible, unitary representation of $K$ subject to the 
following restriction:  if $\lambda$ is the highest 
weight of $V$ relative to any choice of order on the compact roots, then we require: 
$$\eqalign{&\text{(i)\quad $\lambda$ is $G$-integral}\cr
&\text{(ii)\quad $\lambda+\rho$ is $G$-regular.}\cr}$$  

\n These
restrictions will be needed for our proof of Theorem 1.12.  That they
do not seriously limit geometric applications is pointed out at the end of 
\S 1.

If $(\pi,H_{\pi})$ is an admissible, quasi-simple representation of $G$ then
the space of smooth vectors, $H^{\infty}_{\pi}$, is a module for $\frak g$ 
(and $\frak g_{\Bbb C}$) with the property that $K$ isotypic spaces have 
finite multiplicity.  Usually $(\pi,H_{\pi})$ will be unitary, but when
it is finite-dimensional we assume it has been endowed with an
admissible inner product.
We recall the $\ov{\partial}$ complex associated to 
the datum $(\pi,V,\frak p_-)$.  Let $\{Z_i\}$ be any orthonormal basis of 
$\frak p_+$ and let $e(\cdot)$ denote exterior multiplication on 
$\Lambda^{\ast}\frak p_{\Bbb C}$.  The operator defined by  
$$\lno{\ov{\partial}_{\pi}=\sum\pi(\ov{Z_i})\otimes e(Z_i)\otimes
I&&(1.1)\cr}$$
is a linear map 
$$\ov{\partial}_{\pi}:H^{\infty}_{\pi}\otimes\Lambda^q\frak p_+\otimes V
\to H^{\infty}_{\pi}\otimes \Lambda^{q+1} {\frak p_+}\otimes V.$$
It can easily be seen to satisfy $(\ov{\partial}_{\pi})^2=0$ and to commute 
with the action of $K$ via $\pi\otimes\Lambda^{\bullet}_+\otimes \tau$.  Thus 
it restricts to a map on the spaces of $K$-invariant vectors, and hence one 
obtains the $\ov{\partial}$ complex of finite dimensional vector spaces 
associated to $(\pi,V,\frak p_-)$: 
$$0\to [H^{\infty}_{\pi}\otimes \Lambda^0\frak p_+\otimes
V]^K\mathop{\longrightarrow}\limits^{\ov{\partial}_{\pi}}[H^{\infty}_{\pi}
\otimes\Lambda^1\frak p_+\otimes V]^K\mathop{\longrightarrow}\cdots
\mathop{\longrightarrow}\limits^{\ov{\partial}_{\pi}}[H^{\infty}_{\pi}
\otimes\Lambda^n\frak p_+\otimes V]^K\to 0.$$
The homology groups of this complex are denoted by $\{H^{0,q}(\frak g,K;\pi
\otimes \tau)\}$.

Relative to the Killing form on $\frak p$ and the (admissible) inner product on 
$H_{\pi}$ the formal adjoint, $\ov{\partial}^{\ast}_{\pi}$, is given by 
$$\ov{\partial}^{\ast}_{\pi}=\sum^n_{i=1}\pi(\ov{Z}_i)^{\ast}\otimes e(
Z_i)^{\ast}\otimes I,$$
where $\pi(\ov{Z}_i)^{\ast}=-\pi(Z_i)$ and $e(Z_i)^{\ast}=i(\ov{Z}_i)$ with
$i(\cdot)$ denoting interior multiplication.  The formal Laplacian 
$$\square^{0,q}_{\pi}=\ov{\partial}_{\pi}\ov{\partial}^{\ast}_{\pi}
+\ov{\partial}^{\ast}_{\pi}\ov{\partial}_{\pi}$$
acts on $H^{\infty}_{\pi}\otimes\Lambda^q\frak p_+\otimes V$ and restricts 
to an operator on the $K$-invariant vectors.  A useful formula for these 
operators in terms of Casimir operators, $\Omega$, is given in 

\proclaim{Lemma 1.1}

\exama{}{(i)}{On $H^{\infty}_{\pi}\otimes\Lambda^q\frak p_+\otimes V$,
$$\square^{0,q}_{\pi}={1\over 2}\left\{-\pi(\Omega_G)\otimes I\otimes
I+(\pi\otimes\Lambda^q_+)(H_{2\rho_n})\otimes
I+(\pi\otimes\Lambda^q_+)(\Omega_K)\otimes I\right\}.$$}

\exama{}{(ii)}{On $[H^{\infty}_{\pi}\otimes \Lambda^q\frak p_+\otimes V]^K$,
$$\square^{0,q}_{\pi}=-{1\over 2}\pi(\Omega_G)\otimes I\otimes I+{1\over
2}\langle \lambda^{\ast}+2\rho,\lambda^{\ast}\rangle I\otimes I\otimes I,$$
where $\lambda^{\ast}$ is any highest weight of $V^{\ast}$.}
\endproclaim

\demo{Proof}   (i)\quad The proof in [O--O] of Proposition 5.1 carries over 
to $\square^{0,q}_{\pi}$ mutatis mutandis.

~\hskip .2in (ii)\quad On invariant vectors 
$$(\pi\otimes\Lambda^q_+)(H_{2\rho_n})\otimes
I=-I\otimes\tau(H_{2\rho_n})=-\lambda(H_{2\rho_n})I$$
and 
$$(\pi\otimes \Lambda^q_+)(\Omega_K)\otimes I=I\otimes\tau(\Omega_K)=(\Vert
\lambda+\rho_c\Vert^2-\Vert\rho_c\Vert^2)I.$$
Now $H_{2\rho_n}$ is central in $\frak k$ so
$-\lambda(H_{2\rho_n})=\langle\lambda^{\ast},2\rho_n\rangle$ and $\Vert
\lambda+\rho_c\Vert^2=\Vert \lambda^{\ast}+\rho_c\Vert^2$.  From these and the 
independence of $\langle \lambda^{\ast}+2\rho,\lambda^{\ast}\rangle$ on the choice 
of 
order for the compact roots, the
result follows.\hfill \bx

\remark{Remark}   Since $G/K$ is K\"ahler one can show that
$2\square^{0,q}_{\pi}=\Delta^q_{\pi}$, the usual Laplacian of the $d$-complex.  
Statement (ii) is then also obtainable from Kuga's Lemma.
\endremark

For any complex of finite dimensional vector spaces and linear maps
$$0\lra E_0\mlra^{\partial_0} E_1\mlra^{\partial_1} E_2\lra\cdots\lra
E_r\mlra^{\partial_r} 0$$
one has the Poincar\'e polynomial    
$$P(t)=\sum^r_{i=0}\dim E_i\text{ } t^i$$
and the Euler number 
$$e(\{E_i,\partial_i\})=\sum (-1)^i\dim E_i=P(-1).$$
As the vector spaces are finite dimensional and they form a complex
$(\partial_{i+1}\partial_i=0)$, one also has 
$$\lno{e(\{E_i,\partial_i\})=\sum(-1)^i\dim H_i,&&(1.2)\cr}$$
where $H_i$ is the $i^{\text{th}}$ homology group of the complex.

One can define the derived Euler number $e'(\{E_i,\partial_i\})$ by 
$$e'(\{E_i,\partial_i\})=\sum (-1)^ii\dim E_i=P'(-1).$$
We impose this definition on the $\ov{\partial}_{\pi}$ complex to obtain 

\definition{Definition}  Set $e'(\pi,V,\frak p_-)=\sum(-1)^qq\dim
[H^{\infty}_{\pi}\otimes\Lambda^q\frak p_+\otimes V]^K$ and call $e'(\pi,V,
\frak p_-)$ a spectral derived Euler number.
\enddefinition

To evaluate these derived Euler numbers one can use various methods,
adjusted to the type of the representation $(\pi,H_{\pi})$. In [K] K\"{o}hler 
evaluates them for finite-dimensional representations $(\pi,H_{\pi})$ with $G$ highest root $\Lambda_\pi$ dominant;
however, his parametrization of those with 
$e'(\pi,V,\frak p_-)\neq 0$
is not suitable for our purpose, so we shall give an alternative
parametrization. For $(\pi,H_{\pi})$ a discrete series representation,
one can determine those satisfying $e'(\pi,V,\frak p_-)\not= 0 $ by 
decomposing $\Lambda^q\frak p_+\otimes V$ and using Blattner's formula, similar to 
the 
proof of II.5.3 in [B-W]. However, for us it will suffice to observe that only 
finitely many (unitary equivalence classes of) discrete series have $e'(\pi,V,\frak 
p_-)\not= 0 .$ By contrast, the parabolically induced representations play a more 
important role in our context, and treating them
will require some additional structure theory.

We take $\frak a\subseteq \frak p$ a maximal abelian subalgebra 
constructed, for example, from a maximal set of strongly orthogonal
non-compact roots.  Let $\frak q$ be a parabolic subalgebra with Levi
decomposition $\frak q=\frak m_{1 \frak q}\oplus \frak n_{\frak q}$ and 
$\frak m_{1\frak q}=\frak m_{\frak q}\oplus\frak a_{\frak q},  \frak a_{\frak 
q}\subseteq \frak a$.  Let $Q$ be the normalizer of $\frak q$ 
in $G$ and let $Q=M_QA_QN_Q$ be the corresponding Langlands decomposition.  
We recall the construction of principal series representations.  Let $(\xi,
W_{\xi})$ be a unitary representation of $M_Q$ and $e^{\nu}$ a character of 
$A_Q$.  Form $\pi_{\xi,\nu}=\ind^G_Q\xi\otimes e^{\nu}\otimes 1$, acting via 
the left regular representation on 
$$H_{\xi,\nu}=\{f:G\to W_{\xi}\vert f(gman)=e^{-(\nu+\rho_Q)\log
a}\xi(m)^{-1}f(g)\},$$
with $\Vert f\Vert^2=\int_K\Vert f(k)\Vert^2_{\xi}dk$.  For reasons arising
from the disconnectness of $M_Q$ we shall use the subgroup commonly denoted
$M^+_Q$, $M^+_Q=\{m\in M_Q\vert \text{ Ad}(m):\frak M_{\frak q}\to\frak M_{\frak q}$ 
is inner$\}$.  One knows that $M^+_Q=M^0_QF, F$ central in $M^+_Q$ and generated by 
elements of order $2$ in $\exp i\frak a_{\frak q}$ (e.g. [Sc] p. 67).  
Henceforth $\pi_{\xi,\nu}$ shall refer to $Ind^G_{M^+_QA_QN_Q}\xi\otimes 
e^{\nu}\otimes 1$. The proof of the following result is from [M-S;II].

\proclaim{Lemma 1.2}   Let $Y\in \frak p_+$ be a non-zero vector fixed by 
$K\cap M^+_Q$ and set $\frak p^Y_+=(\Bbb CY)^{\perp}\cap \frak p_+$.  Then 
$$e'(\pi_{\xi,\nu},V,\frak p_-)=\sum^{n-1}_{\ell=0}(-1)^{\ell+1}\dim
[W_{\xi}\otimes \Lambda^{\ell}\frak p^Y_+\otimes V]^{K\cap M^+_Q}.$$
\endproclaim

\demo{Proof}  From Frobenius reciprocity one obtains 
$$e'(\pi_{\xi,\nu},V,\frak p_-)=\sum (-1)^qq\dim[W_{\xi}\otimes\Lambda^q
\frak p_+\otimes V]^{K\cap M^+_Q}.$$
Since $\frak p_+=\frak p^Y_+\oplus \Bbb CY$ as $K\cap M^+_Q$ modules, we 
have in the Grothendieck ring of $K\cap M^+_Q$
$$\eqalign{\sum^n_{q=0}(-1)^qq\Lambda^q\frak p_+&=\sum^n_{q=0}(-1)^qq[
\Lambda^q\frak p^Y_+\oplus
\Lambda^{q-1}\frak p^Y_+]\cr
&=\sum^{n-1}_{q=0}(-1)^qq\Lambda^q\frak p^Y_+\oplus
\sum^{n-1}_{q=0}(-1)^{q+1}(q+1)\Lambda^q\frak p^Y_+\cr
&=\sum^{n-1}_{q=0}(-1)^{q+1}\Lambda^q\frak p^Y_+ \, .\cr}$$
If one tensors this with $W_{\xi}\otimes V$ and takes $K\cap M^+_Q$
invariants, one obtains the result. ~~~\hfill\bx
\enddemo

\proclaim{Lemma 1.3}  If $\frak q$ is a parabolic subalgebra with $\dim 
\frak a_{\frak q}\ge 2$ then $e'(\pi_{\xi,\nu},V,\frak p_-)=0$. 
\endproclaim

\demo{Proof}  We show there is a two dimensional subspace of $\frak p_+$ on 
which $K\cap M^+_Q$ acts trivially.  Now $\frak a_{\frak q}$ is contained in 
$\frak a$ and is spanned by $X_{\alpha}+X_{-\alpha}$, $\alpha\in S$ a set of
strongly orthogonal non-compact roots.  $M^+_Q$ centralizes 
$\frak a_{\frak q}$ and $K$ stabilizes $\frak p_{\pm}$.  Hence $K\cap 
M^+_Q$ must fix $X_{\pm\alpha}$.  If $\dim\frak a_{\frak q}\ge 2$ then
$\sum\limits_{\alpha\in S}\oplus \Bbb RX_{+\alpha}$ furnishes the desired
subspace of $\frak p_+$.  If $\dim\frak a_{\frak q}\ge 2$ then we may assume 
some $X_{+\alpha}$, $\alpha\in S$, is in $\frak p^Y_+$.  But then $I\otimes
[e(X_{+\alpha})+i(X_{+\alpha})]\otimes I$ is a $K\cap M^+_Q$ intertwining
operator between $\Lambda^{ev}\frak p^Y_+$ and $\Lambda^{\odd}\frak p^Y_+$, 
and non-trivial since $[e(X_{+\alpha})+i(X_{+\alpha})]^2=\pm\Vert
X_{+\alpha}\Vert^2\not= 0$.  Hence $\Lambda^{ev}\frak p^Y_+$ and
$\Lambda^{\odd}\frak p^Y_+$ are equivalent $K\cap M^+_Q$ modules and the 
result follows from Lemma 1.2.\hfill\bx
\enddemo

Consequently for parabolically induced representations, it suffices to restrict our 
attention to maximal, proper, parabolic
subgroups, that is $\dim\frak a_{\frak q}=1$.  For these there is a refined
structure theory due to a large extent to Koranyi and Wolf, and for which 
an excellent presentation of most of what we shall need can be found in 
Chapter III in [Sa], to which we enthusiastically refer the reader.  We 
shall summarize below the salient facts for our purposes.

To each such $\frak q$ there is a homomorphism 
$$\kappa:\frak s\frak l(2,\Bbb R)\to \frak g$$
so that $\frak a_{\frak q}=\Bbb R X_{\kappa}$ where 
$$X_{\kappa}=\kappa\pmatrix 1&0\\ 0&-1\endpmatrix,$$
and such $\kappa$ intertwines the almost complex structures induced by $ad
\pmatrix 0&1/2\\ -1/2&0\endpmatrix$ and $ad\,H_0$.  If we set 
$$H_{\kappa}=\kappa\pmatrix 0&1\\ -1&0\endpmatrix~~\quad\text{and}~~\quad 
c_\kappa =\exp{\pi i \over 4}\kappa\pmatrix 0 &1\\ 1 &0\endpmatrix\,,$$
then $H_{\kappa}$ is in $\frak k$, and $C_\kappa =\Ad c_\kappa$ is an
automorphism of $\frak g_\Bbb C$, the $\kappa$-Cayley transform,
$$\lno{C_{\kappa}(H_{\kappa})=-iX_{\kappa}\quad \text{and}\quad
C_{\kappa}(X_{\kappa})=-iH_{\kappa}.&&(1.3)\cr}$$

To a maximal set of strongly orthogonal non-compact roots one can 
associate a maximal set of mutually commuting homomorphisms $\kappa_j$'s.  
For such a set, $\frak a = \rsp\{ X_{\kappa_j}\text{'s}\}$ is a maximal 
abelian subalgebra of $\frak p$.  Then we may arrange a choice of representative, 
$\frak q_{\kappa_j} ,$ of each maximal, 
proper parabolic corresponding to the so-called canonical homomorphisms 
$\kappa_j$.  If $\dim\frak a =r$, there are $r$ canonical $\kappa_j$'s.  
We will use $\kappa$ to denote a typical one of these. To avoid even more burdensome 
notation we will use subscripts involving $\kappa$ only when it seems useful.

We fix a canonical $\kappa$ and let $\frak q_{\kappa} = \frak m_{q_\kappa}^1 \oplus 
\frak n_{q_\kappa}$ be the corresponding maximal parabolic subalgebra.

The tds $\kappa(\frak s\frak l(2,\Bbb R))$ acts on $\frak g$ giving the
decomposition $\frak g=\frak g^{[0]}\oplus \frak g^{[1]}\oplus 
\frak g^{[2]}$ according to the $\frak s\frak l(2,\Bbb R$) isotypic 
components, here $0,1,2$ refer to the highest weight of the isotypic 
component.  These subspaces are invariant under the Cartan involution as 
well as the complex structure on $\frak p$, thus (after complexifying) we 
have
$$\lno{\frak p_{\pm}=\frak p^{[0]}_{\pm}\oplus \frak p^{[1]}_{\pm}\oplus
\frak p^{[2]}_{\pm}.&&(1.4)\cr}$$

The algebra $\frak k\cap \frak m_{\frak q_{\kappa}}$ is equal to $\frak k^{[0]}$, 
that is, it centralizes $\kappa(\frak s\frak l(2,\Bbb R))$. Thus $(K\cap 
M_{Q_\kappa})^0$ (which is $K\cap M^0_{Q_\kappa}$) respects the decomposition (1.4).  
The 
subgroup $F_{\kappa}$ is contained in $K\cap\exp i\frak a_{\frak q_{\kappa}}$, in 
fact 
is generated by $f_{\kappa}=\exp\pi iX_{\kappa}.$ Since $\frak g^{[i]}$ is a 
$\kappa(\frak s\frak l(2,\Bbb R))$ isotypic component, $Ad\,F_{\kappa}$ 
preserves each $\frak g^{[i]}$ and, since $H_0$ is central in $\frak k$, $\frak 
p^{[i]}$.  Moreover, it is easy to see that on $\frak g^{[0]}$ and $\frak g^{[2]}$, 
$Ad\exp\pi iX_{\kappa}=I$, while on $\frak g^{[1]}$, $Ad\exp\pi iX_{\kappa}=-I$.  
Thus 
each $\frak p^{[i]}_{\pm}$ is a module for $K\cap M^+_{Q_\kappa}$.  With the use of 
additional structure theory each of these modules will be identified with 
ones having more intrinsic geometric nature.

We make a choice of $\frak q_\kappa$ so that in terms of $\text{ad }X_\kappa$ the 
nilradical $\frak n_{\frak q_\kappa} = \frak v \oplus \frak u$ where $\frak v$ 
(resp. $\frak 
u$) is the +1 (resp. +2) eigenspace of $\text{ad }X_\kappa$. There is also a 
decomposition of the algebra $\frak m_{\frak q_{\kappa}}$ into
$\frak m^{(1)}_{\frak q_{\kappa}}\oplus \frak m^{(2)}_{\frak q_{\kappa}}$ where 
$$\lno{\frak m^{(1)}_{\frak q_{\kappa}}=\frak l_2\oplus
\frak g^{(1)}_{\kappa},&&(1.5)\cr}$$
$\frak l_2$ a compact ideal in $\frak m_{\frak q_{\kappa}}$ and 
$\frak g^{(1)}_{\kappa}$ of Hermitian non-compact type having 
$$\lno{H^{(1)}_0=H_0-{1\over 2}H_{\kappa}&&(1.6)\cr}$$
defining the almost complex structure.  Let $M^{(i)}_{Q_\kappa}$ be the connected 
subgroup with Lie algebra $\frak m^{(i)}_{\frak q_{\kappa}}$.  Then 
$M^0_{Q_\kappa}=M^{(1)}_{Q_\kappa}M^{(2)}_{Q_\kappa}$ and 
$M^+_{Q_\kappa}=M^{(1)}_{Q_\kappa}M^{(2)}_{Q_\kappa}F_{\kappa}$.

For the restriction of the Cartan
decomposition one gets 
$\frak g^{(1)}_{\kappa}=\frak k^{(1)}\oplus \frak p^{(1)}$,
and $\frak p^{(1)}_{\Bbb C}=\frak p^{(1)}_+\oplus \frak p^{(1)}_-$, with 
$\frak p^{(1)}_{\pm}$ simultaneous
$\pm i$ eigenspaces of $ad\,H^{(1)}_0$ and $ad\,H_0$.  Moreover we have the 
identification
$\frak p^{[0]}_{\pm}=\frak p^{(1)}_{\pm}$.

In order to identify $\frak p^{[2]}_{\pm}$ one considers $\frak m^{(2)}_{\frak q_{\kappa}}$, 
the 
other summand 
of $\frak m_{\frak q_{\kappa}}$,  with Cartan 
decomposition 
$\frak k\cap \frak m^{(2)}_{\frak q_\kappa}\oplus \frak p^{(2)}$, together with the 
subalgebra of $\frak k$, so that in the notation of [Sa]
$$\lno{\frak k^{\ast}_{\kappa}=\frak k\cap \frak m^{(2)}_{\frak q_{\kappa}}\oplus
\frak k^{[2]}.&&(1.7)\cr}$$
For our purposes the useful facts about $\frak k^{\ast}_{\kappa}$ are:
$$\lno{&H_{\kappa}\in \text{Center}\,(\frak k^{\ast}_{\kappa}),&(1.8)\cr
&C^{-1}_{\kappa}:(\frak k^{\ast}_{\kappa})_{\Bbb C}\cong (
\frak m^{(2)}_{\frak q_{\kappa}}\oplus \frak a_{\frak q_{\kappa}})_{\Bbb 
C}, \text{with }C^{-1}_{\kappa}\vert_{\frak k\cap \frak m_{\frak q_{\kappa}}} = Id.&(1.9)\cr}$$
In other words, $\frak k^{\ast}_{\kappa}$ and $\frak m^{(2)}_{\frak q_{\kappa}}
\oplus\frak a_{\frak q_{\kappa}}$ are $\Bbb R$-forms of isomorphic complex algebras.

Now $\frak p^{[2]}_{+} = C_{\kappa}(\frak u_{\Bbb C})$ ([Sa] III 2.7), and $\frak 
k\cap \frak m_{\frak q_{\kappa}}$ ($\subseteq \frak g^{[0]}$) commutes with 
$C_{\kappa}$. So as $\frak k\cap \frak m_{\frak q_{\kappa}}$ module, $\frak 
p^{[2]}_{+}\simeq \frak u_{\Bbb C}.$ Furthermore, on $\frak u_{\Bbb C}$ the algebra 
$\frak k\cap \frak m_{\frak q_{\kappa}}^{(1)}$ acts trivially, while $\frak m_{\frak 
q_{\kappa}}^{(2)}\oplus \frak a_{\frak q_{\kappa}}$ acts faithfully via the 
restriction of Ad ([Sa] Th. III.2.3). By means of the isomorphism of vector spaces  
$\frak u_{\Bbb C} \cong \frak p^{(2)}_{\Bbb C}\oplus\frak a_{\frak q_{\kappa},\Bbb 
C}$ 
([Sa] p.98) we transfer the action. Hence as $\frak k\cap \frak m_{\frak 
q_{\kappa}}$-modules we have $\frak u_{\Bbb C}\simeq \frak p^{(2)}_{\Bbb 
C}\oplus\frak 
a_{\frak q_{\kappa},\Bbb C}$ with $\frak a_{\frak q_{\kappa},\Bbb C}$ an invariant 
subspace.

To identify $\frak p^{[1]}_{\pm}$, as well as later use, we consider 
$$L_{\kappa}=\text{Centralizer}_{K}(H_{\kappa})=\text{Centralizer}_{K}(H_0^{(1)}),$$
with Lie algebra $\frak l_{\kappa}$ the sum of the ideals
$$\lno{&\frak l_{\kappa}=\frak k\cap \frak m^{(1)}_{\frak q_{\kappa}}\oplus
\frak k^{\ast}_{\kappa}.&(1.9)'\cr}$$
In fact $H_{\kappa}$ defines an almost complex structure on $K/L_{\kappa}$ 
and we take $\frak n^c_{\kappa}$ (resp. $\ov{\frak n}_{\kappa}^c$) to be the $+i$ 
(resp. $-i$) eigenspace in $\frak k_{\Bbb C}$ of $H_{\kappa}$. With this 
choice of $\frak n^c_{\kappa}$ one has from [Sa] 3.7, and the table on p.101,
$$\lno{C_\kappa\circ\theta\circ C_\kappa^{-1}(\frak p^{[1]}_+) =
\ov{\frak n}^c_{\kappa},&&(1.10)\cr}$$
as $K\cap M^+_{Q_\kappa}$-modules. 

For convenience we summarize these various identifications in the next
statement.

\proclaim{Lemma 1.4}   As $K\cap M^+_{Q_\kappa}$ modules we have 
$$\hbox{(i)}~~~\frak p^{[0]}_{\pm}=\frak p^{(1)}_{\pm}\qquad \hbox{(ii)}~~~
\frak p^{[1]}_+\simeq \ov{\frak n}_{\kappa}^c\qquad \hbox{(iii)}~~~
\frak p^{[2]}_{\pm}\simeq
\frak p^{(2)}_{\Bbb C}\oplus\frak a_{\frak q_{\kappa},\Bbb C}.$$
\endproclaim

\proclaim{Corollary 1.5}   As $K\cap M^+_{Q_\kappa}$ virtual module 
$$\sum^n_{q=0}(-1)^qq\Lambda^q\frak p_+\simeq\sum^{\dim
\frak n^c_{\kappa}}_{\ell=0}(-1)^{\ell+1}(\Lambda^{\ev}\frak p^{(1)}_+-
\Lambda^{\odd}\frak p^{(1)}_+)\otimes
(\Lambda^{\ev}\frak p^{(2)}_{\Bbb C}-\Lambda^{\odd}\frak p^{(2)}_{\Bbb C})
\otimes\Lambda^{\ell}\ov{\frak n}^c_{\kappa}\,.$$
\endproclaim

\demo{Proof}  The result follows from Lemma 1.4 and the proof of Lemma 1.2 choosing 
for $Y$ the projection of
$X_{\kappa}$ into $\frak p_{+}$, namely $X_{\kappa}+ \ov{X}_{\kappa}$.
\hfill\bx
\enddemo

\remark{Remark}  The difference between this virtual module for $K\cap 
M^+_{Q_\kappa}$ versus $K\cap M^0_{Q_\kappa}$ is slight.  Let $\varepsilon$ be the 
non-trivial character of $F_{\kappa}$, $\varepsilon(f_{\kappa})=-1$.  
Then using (1.10) and the earlier observation about $Ad\exp\pi iX_{\kappa}$ we see 
that $F_{\kappa}$ acts on $\Lambda^{\ell}\ov{\frak n}_{\kappa}^c$  by 
$\varepsilon^{\ell}$ and on the other factors above 
trivially.\endremark

To obtain a further simplification we shall need a theorem of Kostant, the
statement of which requires a subset of the Weyl group of $K$.  Associated to 
$H_{\kappa}$ is a parabolic subalgebra
$\frak q^c_{\kappa}\subseteq \frak k_{\Bbb C}$, namely $\frak
q^c_{\kappa}=\frak l_{\kappa,\Bbb C}\oplus\frak n^c_{\kappa}$. Recall that an 
arbitrary choice of order on $\Delta (\frak k)$ was made. Since $\frak t\subseteq 
\frak l_{\kappa}$ we may choose $\Delta^+(
\frak l_{\kappa })$ compatibly with the positive $K$-roots.  Set 
$$\lno{W_{\kappa}=\{w\in W(\frak k_{\Bbb C})\vert
w^{-1}\alpha>0,~~\alpha\in\Delta^+(\frak l_{\kappa})\}.&&(1.11)\cr}$$

Here, $K$ is reductive not semisimple, nevertheless we still have the 
familiar statement:

\proclaim{Theorem} (Kostant)\quad Let $V$ be an irreducible unitary
representation of $K$ with highest weight $\lambda$.  Then as $L_{\kappa}$
virtual modules 
$$ \wedge_{-1}\ov{\frak n}^c_{\kappa}\otimes V := \sum (-1)^{\ell}\Lambda^{\ell}\ov{\frak n}^c_{\kappa}\otimes
V=\sum_{W_{\kappa}}(-1)^{\ell(w)}W_{w(\lambda+\rho_c)-\rho_c},$$
where $W_{\mu}$ is the irreducible $L_{\kappa}$ module with highest weight
$\mu$.
\endproclaim

For convenience we denote the highest weight $w(\lambda+\rho_c)-\rho_c$ by $\lambda_w$, and let
$\lambda^\vee_w$ be the highest weight of the contragredient representation of $L_{\kappa}$. 
As $\frak l_{\kappa}=\frak k\cap \frak m^{(1)}_{\frak q_{\kappa}}\oplus
\frak k^{\ast}_{\kappa}$ a sum of ideals, we may, using obvious notation, express 
$$W_{\lambda_w}=W_{\lambda^{(1)}_w}\otimes W_{\tilde \lambda^{(2)}_w},$$
where $W_{\lambda^{(1)}_w}$ is an irreducible module for $\frak k\cap 
\frak m^{(1)}_{\frak q_{\kappa}}$ with highest weight ${\lambda^{(1)}_w}$, the 
restriction of $\lambda_w$ to $\frak t\cap\frak k\cap\frak m_{\frak 
q_{\kappa}}^{(1)}$, and $W_{\tilde\lambda^{(2)}_w}$ is one for 
$\frak k^{\ast}_{\kappa}$ with $\tilde \lambda^{(2)}_w$ defined similarly.  Composition with the 
isomorphism (cf. (1.9)) $C_{\kappa}: (\frak m^{(2)}_{\frak q_{\kappa}}\oplus \frak a_{\frak q_{\kappa}})_{\Bbb C} \cong (\frak k^{\ast}_{\kappa})_{\Bbb C}$  makes $W_{\tilde\lambda^{(2)}_w}$ a module for 
$\frak m^{(2)}_{\frak q_{\kappa},\Bbb C}\oplus\frak a_{\frak q_{\kappa},\Bbb C}$ and 
fixes an order on the roots of $\frak m^{(2)}_{\frak q_{\kappa},\Bbb C}$ for
which  $W_{\tilde\lambda^{(2)}_w}$ has highest weight $\lambda^{(2)}_w$, the 
restriction of $\lambda_w$ to $\frak t\cap\frak k\cap\frak m_{\frak 
q_{\kappa}}^{(2)}$. As $W_{\lambda_w}$ is irreducible for $\frak l_\kappa$ and $H_\kappa\in \frak l_\kappa$ is central, it acts in $W_{\lambda_w}$ by the scalar $\lambda_w (H_{\kappa})$. By means of the isomorphism $C_\kappa$ then $X_{\kappa}$ can act 
by the scalar
$$\lno{\lambda_w \circ C_\kappa (X_{\kappa}).&&(1.12)\cr}$$

However if in ($1.9)' \,\,\frak l_{\kappa}=\frak k\cap \frak m_{\frak q_{\kappa}}\oplus
\Bbb{ R} H_{\kappa}$ then by (1.9) and (1.3) $ C_\kappa^{-1}: (\frak k\cap \frak m_{\frak q_{\kappa}}\oplus
\Bbb {R} X_{\kappa})_{\Bbb C} \cong (\frak l_{\kappa})_{\Bbb C}$ and so $X_{\kappa}$ can then act by the scalar 
$$\lambda_w \circ C_\kappa^{-1} (X_{\kappa})  .$$
Since $F_{\kappa}$ is generated by elements in $\exp
i\frak a_{\frak q_{\kappa}}$, $f_{\kappa}$ acts only on the second factor and  
the module for $M^{(2)}_{Q_\kappa}F_{\kappa}$ is determined. As module for $K\cap 
M^{(1)}_{Q_\kappa} \times M^{(2)}_{Q_\kappa}F_{\kappa}$ we shall denote this by 
$$W_{\lambda_w}=W_{\lambda^{(1)}_w}\otimes W_{\lambda^{(2)}_w}.$$ 
The character through which $\exp \frak a_{\frak q_{\kappa},\Bbb C}$ acts
ultimately will be related to principal series data.   

For similar 
reasons an irreducible unitary representation $(\xi,W_{\xi})$ of $M^+_{Q_\kappa}$ 
will 
be written
$(\xi^{(1)},W^{(1)}_{\xi})\otimes (\xi^{(2)},W^{(2)}_{\xi})$, here 
$\xi^{(2)}$ is a representation of $M^{(2)}_{Q_\kappa}F_{\kappa}$. As $W_{\lambda^{(1)}_w}$ is 
an 
irreducible module for $K\cap M^{(1)}_{Q_\kappa}$ we have  
$e(\xi^{(1)},W_{\lambda^{(1)}_w},\frak p^{(1)}_-)$ the Euler number 
for the $\ov{\partial}$ complex associated to this datum for $M^{(1)}_{Q_\kappa}$.  
In an entirely similar way, one has for $M^{(2)}_{Q_\kappa}$ and the datum
$(\xi^{(2)},W_{\lambda^{(2)}_w},\frak p^{(2)}_{\Bbb C})$ the $d$ complex 
and the Euler number $e(\xi^{(2)},W_{\lambda^{(2)}_w},\frak p^{(2)}_{\Bbb C}
)$.  Of course it is crucial here that $W_{\lambda^{(2)}_w}$ is a representation of 
$M^{(2)}_{Q_\kappa}$, not just $K\cap M^{(2)}_{Q_\kappa}$.  

\proclaim{Proposition 1.7}   Let $Q_{\kappa}$ be a maximal, proper, parabolic 
subgroup 
of $G$
and let\hfil\break
$\pi_{\xi,\nu}=\ind^G_{M^+_{Q_\kappa}A_{Q_\kappa}N_{Q_\kappa}}(\xi\otimes 
e^{\nu}\otimes
1)$.  Then 
$$e'(\pi_{\xi,\nu},V,\frak p_-)=\sum_{W_{\kappa}}(-1)^{\ell(w)+1}e(\xi^{(1
)},W_{\lambda^{(1)}_w},\frak p^{(1)}_-)e(\xi^{(2)},W_{\lambda^{(2)}_w},
\frak p^{(2)}_{\Bbb C}).$$
\endproclaim

\demo{Proof}   One uses in turn Lemma 1.2, Corollary 1.5, Kostant's Theorem 
and the previous discussion.\hfill \bx
\enddemo

The non-vanishing of $e'(\pi_{\xi,\nu},V,\frak p_-)$ imposes strong 
restrictions on the infinitesimal character of $\xi$.  To derive these 
restrictions we shall treat $\xi^{(1)}$ and $\xi^{(2)}$ separately.

Recall that $M^{(1)}_{Q_\kappa}$ is of Hermitian type, so for $M^{(1)}_{Q_\kappa}$
there are $\rho^{(1)}_n=\rho(\frak p^{(1)}_+)$, \hfil\break  
$\rho^{(1)}_c=\rho(\frak k\cap
\frak m^{(1)}_{\frak q_{\kappa}})$ and $\rho^{(1)}=\rho^{(1)}_n+\rho^{(1)}_c$. Here 
$\rho^{(1)}_c$ is determined by the order on $\Delta (\frak l_{\kappa })$ 
which 
was previously specified. Let 
$w^{(1)}$ denote the opposition element in $W(K\cap M^{(1)}_{Q_\kappa})$.

\proclaim{Lemma 1.8}   If $e(\xi^{(1)},W_{\lambda^{(1)}_w},\frak p^{(1)}_-)
\not= 0$ then 
$$\xi^{(1)}(\Omega_{M^{(1)}_{Q_\kappa}})=\Vert
\lambda^{(1)\vee}_w+\rho^{(1)}\Vert^2-\Vert\rho^{(1)}\Vert^2$$
and the infinitesimal character of $\xi^{(1)},\Lambda_{\xi^{(1)}}$, is given by
$$\align\Lambda_{\xi^{(1)}}&=-w^{(1)}(\lambda^{(1)}_w+\rho^{(1)})\\
&= \lambda ^{(1)\vee}_w + \rho_{c}^{(1)} - \rho_{n}^{(1)}\,.
\endalign$$
\endproclaim

\demo{Proof}   This is essentially Theorem 5.1 in [De-W]. Indeed, as observed before 
$$\eqalign{e(\xi^{(1)},W_{\lambda^{(1)}_w},\frak p^{(1)}_-)&=\sum(-1)^{\ell}
\dim[\xi^{(1)}\otimes\Lambda^\ell \frak p^{(1)}_+\otimes
W_{\lambda^{(1)}_w}]^{K\cap M^{(1)}_{Q_\kappa}}\cr
&=\sum (-1)^{\ell}H^{0,\ell}(\frak m^{(1)}_{\frak q_{\kappa}},K\cap
M^{(1)}_{Q_\kappa};\xi^{(1)}\otimes W_{\lambda^{(1)}_w}).\cr}$$

\n If $e(\xi^{(1)},W_{\lambda^{(1)}_w},\frak p^{(1)}_-)\not= 0$ then one of 
the $H^{0,\ell}$ is non-zero, so by Lemma 1.1
$$\xi^{(1)}(\Omega_{M^{(1)}_{Q_\kappa}})=\langle
\lambda^{(1)\vee}_w+2\rho^{(1)},\lambda^{(1)\vee}_w\rangle.$$
If we denote the usual half-spin representations of $K\cap M^{(1)}_{Q_\kappa}$ 
arising from $\frak p^{(1)}_{\Bbb C}$ by $\sigma^{(1)}_{\pm}$ and $\dim
\frak p^{(1)}_+$ by $q$, then in the representation ring of $K\cap M^{(1
)}_{Q_\kappa}$
$$\Lambda^{\ev}\frak p^{(1)}_+-\Lambda^{\odd}\frak p^{(1)}_+=(-1)^q(
\sigma^{(1)}_+ -\sigma^{(1)}_-)\otimes \xi_{\rho^{(1)}_n}.$$
But if $e(\xi^{(1)},W_{\lambda^{(1)}_w},\frak p^{(1)}_-)\not=0,$ then
$$\lno{e(\xi^{(1)},W_{\lambda^{(1)}_w},\frak p^{(1)}_-)=(-1)^q\dim[\xi^{(
1)}\otimes(\sigma^{(1)}_+-\sigma^{(1)}_-)\otimes \xi_{\rho^{(1)}_n}\otimes
W_{\lambda^{(1)}_{w}}]^{K\cap M^{(1)}_{Q_\kappa}}&&(1.13)\cr}$$
so for the full spin representation $L$, 
$$0\not=\dim [\xi^{(1)}\otimes L\otimes \xi_{\rho^{(1)}_n}\otimes
W_{\lambda^{(1)}_w}]^{K\cap M^{(1)}_{Q_\kappa}},$$
or 
$$\hom_{K\cap
M_{Q_\kappa}^{(1)}}[W^{\ast}_{\lambda^{(1)}_w}\otimes\ov{\xi}_{\rho^{(1)}_n}\otimes
L,\xi^{(1)}]\not= 0.$$
DeGeorge-Wallach's result then says that if $\lambda^{(1)\vee}_w-\rho^{(1
)}_n+\rho^{(1)}_c =-w^{(1)}(\lambda^{(1)}_w+\rho^{(1)})$ is $\frak m^{(1
)}_{\frak q_{\kappa}}$ integral and regular, then the infinitesimal character of 
$\xi^{(1)}$ is $-w^{(1)}(\lambda^{(1)}_w+ \rho^{(1)})$.  Thus the following 
Lemma will complete the proof and justify the restrictions on $\lambda$ 
originally imposed.\hfill\bx
\enddemo 

\proclaim{Lemma 1.9}   If $\lambda$ is $K$ dominant, $\lambda +\rho$ is 
$G$ integral, and $\lambda+\rho$ is $G$ regular, then the same is true for 
$\lambda^{(1)}_w$ relative to $M^{(1)}_{Q_\kappa}$.
\endproclaim

\demo{Proof}   A useful observation is that $\frak m^{(1)}_{\frak q_{\kappa}}$ roots 
are in fact $\frak g$ roots, something not true about $\frak m^{(2
)}_{\frak q_{\kappa}}$.  To see this it suffices to notice from (1.10) and 
Lemma 1.4 (i) that $\frak m^{(1)}_{\frak q_{\kappa}}\subseteq \frak g^{[0]}$ and 
that $\frak t\subseteq \frak g^{[0]}\oplus \Bbb RH_{\kappa}$.  Also 
${\frak k}^{\ast}_{\kappa}$ is an ideal in $\frak l_{\kappa}$.  
The significance of this is that to see $\lambda^{(1)}_w$ satisfies the 
conditions we may use $\lambda_w$ and do calculations with $\frak g$-roots.

As $\lambda_w$ is $\frak l_{\kappa}$ dominant and $\frak k\cap 
\frak m^{(1)}_{\frak q_{\kappa}}$ is an ideal in $\frak l_{\kappa}$,  
$\lambda^{(1)}_w$ 
is $\frak k\cap \frak m^{(1)}_{\frak q_{\kappa}}$ dominant.

Now to show $\frak m^{(1)}_{\frak q_{\kappa}}$ regular and integral, we consider 
$\alpha$ a root of $\frak m^{(1)}_{\frak q_{\kappa}}$.
$$\eqalign{{2\langle\alpha ,\lambda^{(1)}_w+\rho^{(1)}\rangle \over \langle
\alpha ,\alpha\rangle} &={2\langle\alpha ,w(\lambda +\rho_c)-\rho_c+\rho^{(
1)}\rangle \over \langle\alpha ,\alpha\rangle}\cr
&= {2\langle\alpha ,w(\lambda +\rho )-\rho +\rho^{(1)}\rangle \over \langle
\alpha ,\alpha\rangle}\cr
&= {2\langle w^{-1}\alpha ,\lambda +\rho\rangle \over\langle w^{-1}\alpha ,
w^{-1}\alpha\rangle} - {2 \langle\alpha ,\rho -\rho^{(1)}\rangle \over 
\langle\alpha ,\alpha\rangle}\,.\cr}$$
Since $w^{-1}\alpha$ is a root for $\frak g$ and $\lambda +\rho$ is $G$ 
regular and integral, the first term is a non-zero integer.  On $\frak m^{(
1)}_{\frak q_{\kappa}}$, $\rho =\rho^{(1)}$ so the second term is zero.\hfill\bls
\enddemo

\remark{Remark}  Recall from (1.5) that $\frak m^{(1)}_{\frak q_{\kappa}}=\frak l_2
\oplus\frak g^{(1)}_{\kappa}$ and need not be of non-compact type.  But 
since $\frak l_2$ is an ideal in $\frak m^{(1)}_{\frak q_{\kappa}}$ one easily sees
that Lemma 1.8 is valid in this generality.
\endremark

\remark{Remark}  The group $M^{(2)}_{Q_\kappa}$ can be handled somewhat analogously 
to deduce that $\xi^{(2)}$ is a fundamental series representation.  But at 
this time we have no application of this complication so we have imposed the 
assumption that $\frak q_{\kappa}$ is cuspidal.  Let 
$\xi^{(2)}=\xi^{(2)}_0\otimes \chi$ as a representation of 
$M^{(2)}_{Q_\kappa}F_{\kappa}$ where $\xi^{(2)}_0 = \chi$ on $M^{(2)}_{Q_\kappa}\cap 
F_{\kappa},$ and let $w^{(2)}$ be the opposition element for $M^{(2)}_{Q_\kappa}$.
\endremark

\proclaim{Lemma 1.10}   If $e(\xi^{(2)},W_{\lambda^{(2)}_w},
\frak p^{(2)}_{\Bbb C})\not= 0$ then 
$$\xi_0^{(2)}(\Omega_{M^{(2)}_{Q_\kappa}})=\Vert
\lambda^{(2)\vee}_w+\rho^{(2)}\Vert^2-\Vert\rho^{(2)}\Vert^2$$
and the infinitesimal character of $\xi_0^{(2)}$, $\Lambda_{\xi_0^{(2)}}$, is given 
by
$$\Lambda_{\xi_0^{(2)}}=\lambda^{(2)\vee}_w+\rho^{(2)} = - 
w^{(2)}(\lambda^{(2)}_w+\rho^{(2)}).$$

\endproclaim

\demo{Proof}   One observes that the non-vanishing of
$H^{\ell}(\frak m^{(2)}_{Q_\kappa},K\cap M^{(2)}_{Q_\kappa};\xi_0^{(2)}\otimes
W_{\lambda^{(2)}_w})$ allows one to invoke Proposition 2.4 in [B] (or  I.5.3 in 
[B-W] 
for disconnected $K$)
directly to obtain the result.\quad\hfill\bx
\enddemo

For $\xi_0^{(2)}$ an irreducible discrete series, from [B-W, p. 60] it follows that
$e(\xi^{(2)},W_{\lambda^{(2)}_w},\frak p^{(2)}_{\Bbb C})\not= 0$, 
then 
$$\lno{e(\xi^{(2)},W_{\lambda^{(2)}_w},\frak p^{(2)}_{\Bbb C})= (-1)^{{1\over 2}\text{dim} \frak p^{(2)}_{\Bbb C}}.
&&(1.13)' \cr}$$

We shall need to be more precise about the various lexicographic orders for a cuspidal $\frak q_\kappa$. Set 
$\frak t_\kappa = \frak t\cap\frak k\cap\frak m_{\frak q_{\kappa}}$ and 
$\frak h_\kappa =\frak t_\kappa\oplus\Bbb RX_\kappa$. Essentially the order 
$\Delta^+
(\frak g_{\Bbb C},\frak h_{\kappa ,\Bbb C})$ is chosen according to the positive 
eigenvalues of $X_\kappa ,$ and, since $C_\kappa^{-1}
\colon\frak h_{\kappa ,\Bbb C}\to\frak t_{\Bbb C}$, in such a way that $\Delta^+
(\frak g_{\Bbb C},\frak h_{\kappa ,\Bbb C})=C^{-1\ast}_\kappa\Delta^+(
\frak g_{\Bbb C},\frak t_{\Bbb C})$.
Then relative to $\frak t_{\Bbb C}$ we have $\Delta^+(\frak g_{\Bbb C},
\frak t_{\Bbb C})=\Delta^+_c\cup\Delta^+_n$, where $\frak p_+\leftrightarrow 
\Delta^+_n$, and $\Delta^+_c$ has ${\frak n}^c_\kappa\leftrightarrow$ positive 
compact roots, and 
some order has been chosen for $\frak l_{\kappa ,\Bbb C}$. Since $X_\kappa = 
C_\kappa (i H_\kappa)$, and $\ad H_\kappa $ on 
$\frak p_-$ has values $0$, $-i$, $-2i$, 
one can easily check, [Sa] \S 2 (2.1) and p. 143, that
$$C_\kappa (\frak k_{\Bbb C}\oplus\frak p_-)\cap\frak g \supset \frak k\cap\frak 
m_{\frak q_{\kappa}}\oplus\frak a_{\frak q_\kappa}\oplus\frak n_{\frak 
q_\kappa}\,.$$

Relative to the Killing form on $\frak g$ one has $H_\kappa$  
orthogonal to $\frak t_\kappa$, and similarly $X_\kappa$ orthogonal
to $\frak t_\kappa$.  We set (recall (1.3))
$$\widehat{X}_\kappa = {X_\kappa \over [-\Vert H_\kappa\Vert^2]^{1/2}}\,.$$
  
\proclaim{Lemma 1.11}
$$\eqalign {\text{(i)  } \quad\langle H_0,H_0\rangle &= -2 \text{dim } \frak p_{-} 
\cr
\text{(ii) }\quad \langle H_\kappa,H_\kappa\rangle &= -4 \text{dim } \frak 
p^{[1]}_+  -8 \text{dim 
} \frak p^{[2]}_{+}\cr
\text{(iii) } \quad \langle H_0,H_\kappa\rangle &= -2\text{dim } \frak p^{[1]}_+  
-4\text{dim } 
\frak p^{[2]}_{+}\cr
\text{(iv)  } \langle H^{(1)}_0,H^{(1)}_0\rangle &= -2 \text{dim } \frak p^{(1)}_-  
- \text{dim } 
\frak p^{[1]}_+\cr
\text{(v)   }\quad \langle H^{(1)}_0,H_\kappa\rangle &=0 \cr}$$
\endproclaim

\demo{Proof} For any $H_1$, 
$H_2$ in $\frak t$
$$\eqalign { \langle H_1,H_2\rangle &= tr (ad H_1 ad H_2) = \sum_{\alpha\in\Delta} 
\alpha 
(H_1)\alpha (H_2) \cr
&= \sum_{\Delta(\frak k)} \alpha (H_1) \alpha (H_2) + \sum_{\Delta (\frak p)} \alpha 
(H_1) \alpha (H_2).\cr}$$

(i) As $H_0$ is central in $\frak k$, for all $\alpha\in\Delta (\frak k)$, 
$\alpha (H_0) = 0$; while for $\alpha\in \Delta (\frak p)$, $\alpha (H_0) = \pm i$ 
according as $\alpha\in\Delta (\frak p^{\pm}).$

(ii) As $H_\kappa$ is central in $\frak l_\kappa$, for all $\alpha\in\Delta (\frak 
l_\kappa)$ $\alpha( H_\kappa) = 0$; while for $\alpha \in\Delta (\frak n^c_\kappa)$ 
$\alpha (H_\kappa) = i.$ From [Sa] one finds for $\alpha\in\Delta (\frak 
p^{[2]}_{\pm})$ that $ \alpha(H_\kappa) = \pm 2i,$ and since $\frak n^c_\kappa\cong\frak 
p^{[1]}_-$, if $\alpha\in\Delta (\frak p^{[1]}_{\pm}), \alpha (H_\kappa) = \pm i$. 
Then
$$\eqalign{\langle H_\kappa,H_\kappa\rangle &= -2 \text{dim } \frak n^c_\kappa -2 
\text{dim } 
\frak p^{[1]}_+  -8 \text{dim } \frak p^{[2]}_+\cr
&= -4 \text{dim } \frak p^{[1]}_+ - 8\text{dim } \frak p^{[2]}_+.\cr}$$

(iii) This follows from the proofs of (i) and (ii).

(iv) 
$$\eqalign {\langle H^{(1)}_0,H^{(1)}_0\rangle &= \langle H_0 - H_{\kappa}/2, H_0 - 
H_{\kappa}/2\rangle \cr
&= \langle H_0,H_0\rangle - \langle H_0,H_\kappa\rangle + {1\over 4} \langle 
H_\kappa,H_\kappa\rangle \cr
&= \langle H_0,H_0\rangle - \langle H_0,H_\kappa\rangle + \frac{1}{2}\langle 
H_0,H_\kappa\rangle \cr
&= \langle H_0,H_0\rangle - {1\over 2}\langle H_0,H_\kappa\rangle .\cr}$$
The result then follows from (i), (iii) and (1.4).

(v)
$$\align  \langle H^{(1)}_0,H_\kappa\rangle &= \langle H_0 -H_\kappa /2, 
H_\kappa\rangle\\
&= \langle H_0,H_\kappa\rangle - {1\over 2}\langle H_\kappa, H_\kappa\rangle = 0.
\endalign $$
\hfill\bx
\enddemo

\proclaim{Theorem 1.12}  Let $\frak q_{\kappa} = \frak m_{\frak q_{\kappa}}\oplus
\frak a_{\frak q_{\kappa}}\oplus \frak n_{\frak q_{\kappa}}$ be a cuspidal maximal 
parabolic 
subalgebra in $\frak g$.  Let $\xi$ be a discrete series representation of 
$M^+_{Q_\kappa}$.  If
$$e(\xi^{(1)},W_{\lambda^{(1)}_w},\frak p^{(1)}_-)e(\xi^{(2)},W_{\lambda^{(2
)}_w},\frak p^{(2)}_{\Bbb C})\not= 0$$
then
$$\lno{\Lambda_\xi &=  (w^{\ast}(\lambda^\ast +\rho -2\rho_n))\circ C_\kappa 
	+ (\lambda_w\big|_{\Bbb C H_\kappa})\circ C_\kappa +\rho_{Q_\kappa} 
&(1.14)\cr}$$
and
$$\lno{\Vert\Lambda_\xi\Vert^2 -\Vert\lambda^\ast +\rho\Vert^2 &= -
\{( \lambda_w\circ C_\kappa + \rho_{Q_\kappa})(\widehat{X}_\kappa )\}^2\, + 
4\langle\lambda,\rho_n\rangle\,. 
&(1.15)\cr}$$
\endproclaim

\demo{Proof}  From Lemma 1.8 and Lemma 1.10 we have
$$\Lambda_\xi = \Lambda_{\xi^{(1)}}+\Lambda_{\xi^{(2)}} = -w^{(1)}
(\lambda^{(1)}_w+\rho^{(1)})-w^{(2)}(\lambda^{(2)}_w+\rho^{(2)}).$$

As $\frak l_{\kappa}=\frak k\cap \frak m^{(1)}_{\frak q_{\kappa}}\oplus
\frak k^{\ast}_{\kappa}$ with these ideals, and given (1.9) with the resulting 
orders, 
the pair $(w^{(1)},w^{(2)})$ is the pull-back of $w^{L_\kappa}$, where 
$w^{L_\kappa}$ 
is the opposition element for $L_\kappa$, i.e.
$$\eqalign{\Lambda_\xi &=(-w^{L_\kappa}(\lambda_w))\circ C_\kappa\big|_{\frak 
t_\kappa}\,+(\rho^{(1)}_c - \rho^{(1)}_n +\rho^{(2)})\cr
&= (-w^{L_\kappa}(\lambda_w - \lambda_w\big|_{\Bbb C H_\kappa}))\circ 
C_\kappa\,+(\rho^{(1)}_c 
- \rho^{(1)}_n +\rho^{(2)}).\cr}$$
Now
$$\eqalign{-w^{L_\kappa}(\lambda_w) &= (-w^{L_\kappa})(w(\lambda +\rho_c)-\rho_c)\cr
&= (-w^{L_\kappa})(w(-w^K)(\lambda^\ast +\rho_c)-\rho_c)\cr
&= w^\ast (\lambda^\ast +\rho_c)-\rho_c+2\rho (\frak n^c_\kappa )\,.\cr}$$
It is easy to see that $w^\ast$ is in $W_\kappa$, so one could write
$-w^{L_\kappa}(\lambda_w) = \lambda^\ast_{w^\ast}+2\rho (\frak n^c_\kappa )\,.$
Substituting the above and using $\rho^{(2)}\circ C_\kappa = \rho (\frak 
k^\ast_\kappa)$ one obtains
$$\eqalign {\Lambda_\xi &= (w^\ast (\lambda^\ast +\rho ))\circ C_\kappa -\rho\circ 
C_\kappa
+2\rho (\frak n^c_\kappa )\circ C_\kappa +(\lambda_w\big|_{\Bbb C H_\kappa})\circ 
C_\kappa 
+(\rho^{(1)}_c - \rho^{(1)}_n +\rho^{(2)})\,\cr
&= (w^\ast (\lambda^\ast +\rho ))\circ C_\kappa -2\rho_n\circ C_\kappa 
+(\lambda_w\big|_{\Bbb C 
H_\kappa})\circ C_\kappa +(\rho_n - \rho^{(1)}_n + \rho (\frak n^c_\kappa ))\circ C_\kappa.\cr}$$
Recall that $\frak q_\kappa = \frak m_{\frak q_{\kappa}}^1 \oplus \frak n_{\frak 
q_{\kappa}}$, and that $\frak n_{\frak q_{\kappa}} = \frak v \oplus 
\frak u$. Writing ([Sa] p.100) $\frak v_{\Bbb C} = \frak v_+ \oplus \frak v_-$, and 
using ([Sa] 
p.101) $C_\kappa \frak v_+ = \frak p^{[1]}_+$, $C_\kappa\frak u_{\Bbb C} = \frak 
p^{[2]}_+$, 
and $C_\kappa \frak  v_- = \frak n^c_\kappa$, it follows that $(\rho_n - 
\rho^{(1)}_n + \rho 
(\frak n^c_\kappa ))\circ C_\kappa = \rho_{Q_\kappa}$. Thus 
$$\Lambda_\xi = (w^\ast (\lambda^\ast +\rho ))\circ C_\kappa -2\rho_n\circ C_\kappa 
+(\lambda_w\big|_{\Bbb C H_\kappa})\circ C_\kappa + \rho_{Q_\kappa}$$
and (1.14) follows.

For (1.15) one uses (1.14) and that $H_\kappa$ is orthogonal to $\frak t_\kappa$:
$$\eqalign{\Vert\Lambda_\xi\Vert^2 &= \Vert (w^\ast(\lambda^\ast +\rho 
-2\rho_n))\circ 
C_\kappa\Vert^2 +\Vert (\lambda_w\big|_{\Bbb C H_\kappa})\circ C_\kappa + 
\rho_{Q_\kappa}\Vert^2\cr
&\qquad +2\langle(w^\ast(\lambda^\ast +\rho -2\rho_n))\circ 
C_\kappa,(\lambda_w\big|_{\Bbb C 
H_\kappa})\circ C_\kappa + \rho_{Q_\kappa}\rangle\cr
&= \Vert\lambda^\ast +\rho\Vert^2 + 4\langle\lambda,\rho_n\rangle +  
\Vert(\lambda_w\big|_{\Bbb 
C 
H_\kappa})\circ C_\kappa + \rho_{Q_\kappa}\Vert^2\cr
&\qquad +2\langle(w^\ast(\lambda^\ast +\rho -2\rho_n))\circ 
C_\kappa,(\lambda_w\big|_{\Bbb C H_\kappa})\rangle\circ 
C_\kappa + \rho_{Q_\kappa}\rangle\cr
&= \Vert\lambda^\ast +\rho\Vert^2 + 4\langle\lambda,\rho_n\rangle +  
\Vert(\lambda_w\big|_{\Bbb 
C 
H_\kappa})\circ C_\kappa + \rho_{Q_\kappa}\Vert^2 \cr
&\qquad+2\langle(-w^{L_\kappa}(\lambda_w +\rho_c +\rho_n))\circ 
C_\kappa,(\lambda_w\big|_{\Bbb C H_\kappa})\circ 
C_\kappa + \rho_{Q_\kappa}\rangle\cr
&= \Vert\lambda^\ast +\rho\Vert^2 + 4\langle\lambda,\rho_n\rangle +  
\Vert(\lambda_w\big|_{\Bbb 
C 
H_\kappa})\circ C_\kappa + \rho_{Q_\kappa}\Vert^2\cr
&\qquad -2\langle(\lambda_w +\rho_c +\rho_n)\circ C_\kappa,(\lambda_w\big|_{\Bbb C 
H_\kappa})\circ 
C_\kappa + \rho_{Q_\kappa}\rangle\cr
&= \Vert\lambda^\ast +\rho\Vert^2 + 4\langle\lambda,\rho_n\rangle +  
\Vert(\lambda_w\big|_{\Bbb C 
H_\kappa})\circ C_\kappa + \rho_{Q_\kappa}\Vert^2 -2\Vert(\lambda_w\big|_{\Bbb C 
H_\kappa})\circ C_\kappa + \rho_{Q_\kappa}\Vert^2\cr}$$
where in the last line we used that $H_\kappa$ is orthogonal to $\frak t_\kappa$ and 
that $(\rho_c +\rho_n)\circ C_\kappa\big|_{\Bbb C H_\kappa} = 
\rho_{Q_\kappa}$.
\hfill\bx
\enddemo

\proclaim{Corollary 1.13}  Let $\xi$ be a discrete series representation of 
$M^+_{Q_\kappa}$ and $e^\nu$ a character of $A_{Q_\kappa}$.  Set $\pi_{\xi ,\nu}=
\Ind^G_{M^+_{Q_\kappa} A_{Q_\kappa}N_{Q_\kappa}}(\xi\otimes e^\nu\otimes 1)$.  If
$$e'(\pi_{\xi ,\nu},V,\frak p_-)\not= 0\,,$$
then for some $w\in W_\kappa$
$$\Lambda_\xi =  (w^{\ast}(\lambda^\ast +\rho))\circ C_\kappa -
2\rho_n\circ C_\kappa  + (\lambda_w\big|_{\Bbb C H_\kappa})\circ C_\kappa 
+\rho_{Q_\kappa} $$
and
$$\lno{\Vert\Lambda_\xi\Vert^2 -\Vert\lambda^\ast +\rho\Vert^2 &= -
\{ (\lambda_w\circ C_\kappa + \rho_{Q_\kappa})(\widehat{X}_\kappa )\}^2\, + 
4\langle\lambda,\rho_n\rangle\,. 
 &(1.16)\cr}$$
\endproclaim

\remark{Remarks} (1) If $\lambda = 0$, then $\Vert\Lambda_\xi\Vert^2 -\Vert\lambda^\ast 
+\rho\Vert^2 
= -(n-k)^2$ for some $k$, a fact proved in [M-S;II] Lemma 2.5 for non-equal rank groups $G$.

(2) The significance of (1.16) is that it will provide passage from the harmonic analysis 
(i.e., 
the 
infinitesimal characters of unitary representations) to geometry (i.e. the holonomy action of 
geodesics 
on the coefficient bundle). Said plainly, the transfer 
from spectral information to geodesic properties makes possible the construction of 
a 
geometric zeta function.

(3) The term $4\langle\lambda,\rho_n\rangle$ seems a little strange especially 
because using $\Vert\lambda^\ast +\rho\Vert^2 + 4\langle\lambda,\rho_n\rangle = 
\Vert\lambda +\rho\Vert^2$ would yield a cleaner formula. However, as 
$4\langle\lambda,\rho_n\rangle$ is independent of $w$ it will introduce only a 
scalar into the trace of the heat operators. Also it follows easily from Lemma 1.11 
that $4\langle\lambda,\rho_n\rangle = -i\lambda(H_0)$.

(4) Obviously
$$ \{ (\lambda_w\circ C_\kappa + \rho_{Q_\kappa})(\widehat{X}_\kappa )\}^2  = 
\{ -(\lambda_w\circ C_\kappa + \rho_{Q_\kappa})(\widehat{X}_\kappa )\}^2
= \{ (-\lambda_w\circ C_\kappa - \rho_{Q_\kappa})(\widehat{X}_\kappa )\}^2
$$
and since (cf. (1.3)) $C^{-1}_\kappa (X_\kappa) = - C_\kappa (X_\kappa) $ we have
$$\lno{-\lambda_w\circ C_\kappa (\widehat{X}_\kappa ) &=  \lambda_w\circ C_\kappa^{-1}(\widehat{X}_\kappa ),&(1.17)\cr
 \{ (\lambda_w\circ C_\kappa + \rho_{Q_\kappa})(\widehat{X}_\kappa )\}^2  &= \{ (\lambda_w\circ C_\kappa^{-1} - \rho_{Q_\kappa})(\widehat{X}_\kappa )\}^2.\cr}$$
%Now $iH_\kappa\in \frak k_{\Bbb C}$ acts on $V$, as $\frak k_{\Bbb C}$ is part of the 
%Lie 
%algebra of the isotropy at $eK_{\Bbb C}P_{-}.$ As recalled in (1.9)
%$C_{\kappa}:(\frak m^{(2)}_{\frak q_{\kappa}}\oplus \frak a_{\frak q_{\kappa}})_{\Bbb 
%C}\cong (\frak k^{\ast}_{\kappa})_{\Bbb C},$
%and as in (1.12) the 
%transferred 
%action of $X_\kappa$ on $V$ is then via $-iH_\kappa$. Thus 
%$(\lambda_w\circ C_\kappa) (X_\kappa )$ is the action of $X_\kappa$ on 
%$W_{\lambda_w}$. This will be used in later sections.

(5) As observed earlier (circa Lemma 1.11) from [Sa] \S 2 (2.1) and p. 143
$$C_\kappa (\frak k_{\Bbb C}\oplus\frak p_-)\cap\frak g \supset \frak k\cap\frak 
m_{\frak q_{\kappa}}\oplus\frak a_{\frak q_\kappa}\oplus\frak n_{\frak 
q_\kappa}\,,$$
and from (1.9)  $C_{\kappa}: (\frak m^{(2)}_{\frak q_{\kappa}}\oplus \frak a_{\frak q_{\kappa}})_{\Bbb C} \cong (\frak k^{\ast}_{\kappa})_{\Bbb C}$ so if $\frak p^{(2)}_{\frak q_\kappa} = \{0\}$, then
$C_{\kappa}: (\frak k\cap \frak m^{(2)}_{\frak q_{\kappa}}\oplus \frak a_{\frak q_{\kappa}})_{\Bbb C} \cong (\frak k^{\ast}_{\kappa})_{\Bbb C}.$

\endremark

Before moving on to the analysis and topology sections of the paper we want 
to point out that the conditions imposed on $\lambda$ at the beginning of 
\S 1 do not exclude the case of holomorphic $p$-torsion as originally 
defined by Ray and Singer [R-S; II].  Indeed, since $\frak p_-\cong
\frak p_+^\ast$, it is enough to consider $\Lambda^p\frak p_+$ and 
take any irreducible $K$ submodule, say $V_{\lambda}$, and show that 
$\lambda$ is of the required form.  Now any highest weight of $\Lambda^p
\frak p_+$ is of the form $\lambda=\alpha_{k_1}+\cdots +\alpha_{k_p}$ where 
$\alpha_{k_i}$ are roots of $\frak p_+$.  This is clearly $G$-integral.  
Notice that if $\alpha$ is a positive root of $K$
$$\eqalign{\langle\alpha,\alpha_{k_1}+\cdots+\alpha_{k_p}+\rho\rangle &=
\langle\alpha,\alpha_{k_1}+\cdots+\alpha_{k_p}+\rho_c+\rho_n\rangle\cr
&=\langle\alpha,\alpha_{k_1}+\cdots +\alpha_{k_p}+\rho_c\rangle>0.\cr}$$
So it is enough to show $\lambda+\rho$ is regular, and from the above we 
need consider only non-compact roots, and here for $\frak g$ simple and of Hermitian 
type, any non-compact root $\alpha=\pm (\alpha_0+\alpha_c)$ where $\alpha_0$ 
is a simple non-compact root and $\alpha_c$ a sum of compact roots (or zero).  
Then, say $\alpha>0$, 
$$\eqalign{\langle\alpha,\lambda+\rho\rangle&=\langle\alpha_0+\alpha_c,
\alpha_{k_1}+\cdots
+\alpha_{k_p}+\rho\rangle\cr
&=\langle\alpha,\alpha_{k_1}+\cdots +\alpha_{k_p}\rangle+{1\over
2}\langle\alpha_c,2\rho_c\rangle+\langle\alpha_0,\rho\rangle.\cr}$$
Now the first is non-negative because $\alpha,\alpha_{k_i}$ are non-compact
roots and the sum is never (Hermitian) a root.  The second is non-negative
always and positive if $\alpha_c\not= 0$.  The third,
$\langle\alpha_0,\rho\rangle$, is positive because $\alpha_0$ is simple and, well, 
$\rho$ is $\rho$.

The last business for this section is to consider a finite-dimensional 
irreducible representation $(\pi,H_{\pi})$. A parametrization of those 
with $e'(\pi,V,\frak p_-)\not= 0 $ follows from [K]; however, a little structure 
theory gives a parametrization that suggests some
proportionality between $L^2$-torsion and torsion on the compact dual.
Indeed, it appears that the proportionality principle is a consequence 
of a more basic relationship. We begin with the following observation and return to this in \S 10.

\proclaim{Lemma 1.14} Let $ch$ denote the virtual character of an
element of the representation ring $R(K)$. Then
$$ ch(\sum(-1)^qq\Lambda^q_+)= ch(\sum(-1)^q\Lambda^q_+)({n \over 2}
+ {1 \over 2}\sum_{\Delta^+_n} coth{\alpha \over 2}).$$
\endproclaim

\demo{Proof} This follows from the already mentioned fact that
$$ \sum(-1)^q\Lambda^q_+ = (-1)^n \xi_{\rho_n} \otimes (\sigma^+ - 
	\sigma^-) $$
and that the almost complex structure determines a number
operator on $ \Lambda^* \frak p_+ $. Thus,
$$ \eqalign{ch(\sum(-1)^qq\Lambda^q_+)  
	 &=-iH_0 \cdot ch(\sum(-1)^q\Lambda^q_+) 
\cr &=-i (-1)^n [\rho_n (H_0) \xi_{\rho_n}  ch((\sigma^+ - 
	\sigma^-) + \xi_{\rho_n} H_0 \cdot \prod_{\Delta^+_n}
	(e^{\alpha \over 2} - e^{-{\alpha \over 2}})]
\cr &= -i[i{n \over 2} ch(\sum(-1)^q\Lambda^q_+) +
	{i \over 2} ch(\sum(-1)^q\Lambda^q_+) \cdot 
	\sum_{\Delta^+_n} coth{\alpha \over 2}]
\cr &	= ch(\sum(-1)^q\Lambda^q_+) \cdot ({n \over 2}
	+ {1 \over 2}\sum_{\Delta^+_n} coth{\alpha \over 2}) \,. \cr}$$
\hfill\bx
\enddemo
	
Since $W_K = W(\frak k_{\Bbb C})$ permutes $\Delta^+_n$, the function
${n \over 2}+ {1 \over 2}\sum_{\Delta^+_n} coth{\alpha \over 2}$ has
a Fourier series in characters of $K$. The elements of 
$\hat K$ which occur we shall identify in terms of holomorphically induced representations from 
the
$L_{\kappa}$'s. For $(\xi,\nu) \in {\hat L}_{\kappa}$ where $e^\nu$ is a character
of the circle generated by $H_{\kappa}$ and $\xi$ is compatible with $e^\nu$ we let 
$\pi_{\xi, \nu}\in 
{\hat K}$ be the representation given by the theorem of Bott-Kostant.

For notational convenience we suppose $G$ is simple. From [Kn, p. 479-480]
or [Sc] one finds that there are at most two $W_K$-orbits in $\Delta^+_n$.
If $\alpha_{\kappa_1}$ is the unique simple non-compact root, it defines 
one orbit; the other, if at all, is defined by $\alpha_{\kappa_2}:=
\alpha_{\kappa_1}+\delta$, where $\delta$ denotes a sum of compact roots.
As the notation suggests, $\alpha_{\kappa_i}$ is associated to a tds
$\kappa_i$.

\proclaim{Lemma 1.15}
$${n \over 2}+{1 \over 2}\sum_{\Delta^+_n}\coth{\alpha \over 2} = 
\sum_{\kappa_i}\sum_{k\ge 
0}\sum_{W_{\kappa_i}}\varepsilon (w)ch(\pi_{0_w,-k\alpha_{\kappa_i}})\,.$$
\endproclaim

\demo{Proof}  As before, the notation $\lambda_w$ denotes $w(\lambda +
\rho_c)-\rho_c$ and is the highest weight of the representation of 
$L_\kappa$.  Also it is standard that $W_K =W(\frak l_\kappa )W_\kappa$ 
and since $L_\kappa$ centralizes $H_\kappa$, the $W_\kappa$ orbit of 
$\alpha_\kappa$ is $\{ w^{-1}\cdot\alpha_\kappa\mid w\in W_\kappa\}$.  Now
$$\eqalignno{{n \over 2}+{1 \over 2}\sum_{\Delta^+_n}\coth{\alpha \over 2}
&=\sum_{\Delta^+_n}{1 \over 1-e^{-\alpha}}\cr
&= \sum_{\kappa_i}\sum_{W_{\kappa_i}}{1 \over 1-e^{-w^{-1}_i
\alpha_{\kappa_i}}}\,.\cr}$$
Fix a $\kappa_i$ and denote it by $\kappa$.  The following calculation 
involves only elementary manipulations.
$$\eqalign{\left(\sum_{W_\kappa}{1 \over 1-e^{-w^{-1}\alpha_\kappa}}
\right)\left(\sum_{W_K}\varepsilon (\sigma )e^{\sigma\rho_c}\right) &= 
\sum_{k\ge 0}\sum_{W_\kappa} e^{-kw^{-1}\alpha_\kappa}\sum_{W_K}\varepsilon 
(\sigma )e^{\sigma\rho_c}\cr
&=\sum_{k\ge 0}{1 \over \vert W(\frak l_\kappa )\vert}\sum_{W_K}e^{-ks^{-1}
\alpha_\kappa}\sum_{W_K}\varepsilon (\sigma )e^{\sigma\rho_c}\cr
&= \sum_{k\ge 0}{1 \over \vert W(\frak l_\kappa )\vert}\sum_{\sigma\in W_K}
\varepsilon (\sigma )\sum_{t\in W_K}e^{\sigma (-kt^{-1}\alpha_\kappa +
\rho_c)}\cr
&=\sum_{k\ge 0}\sum_{\sigma\in W_K}\varepsilon (\sigma )\sum_{w\in W_\kappa}
e^{\sigma (-kw^{-1}\alpha_\kappa +\rho_c)}\cr
&=\sum_{k\ge 0}\sum_{w\in W_\kappa}\varepsilon (w)\sum_{\sigma\in W_K}
\varepsilon (\sigma )e^{\sigma (-k\alpha_\kappa +w\rho_c)}\,.\cr}$$
Next notice that $-k\alpha_\kappa +w\rho_c$ is $\Delta^+(\frak l_\kappa )$ 
dominant and regular. Indeed, for $\beta\in\Delta^+(\frak l_\kappa )$
$$\langle -k\alpha_\kappa +w\rho_c,\beta\rangle = \langle w\rho_c,\beta\rangle 
=\langle\rho_c,w^{-1}\beta\rangle >0$$
since $L_\kappa$ centralizes $H_\kappa$ and $w\in W_\kappa$.  It follows 
that $-k\alpha_\kappa +w\rho_c -\rho_c$ corresponds to the representation, 
unfortunately denoted $0_{w, -k\alpha_\kappa}$.

Repeating the argument for the other $\kappa$'s and using the Weyl character formula 
concludes the proof.\hfill\bx
\enddemo

\remark{Remark}  One can easily determine from the structure of $K/{L_\kappa}$ 
at which threshold the weight $-k\alpha_\kappa +\sigma\rho_c$ is $\Delta^+_c$ 
singular 
or not dominant.  In the latter case the necessary $W_K$ element to establish 
dominance is easily determined.  In this way, in specific cases such as rank 
$1$, the Fourier series in Lemma 1.15 becomes simpler.

\proclaim{Proposition 1.16}  Let $(\pi ,H_\pi )$ be a finite dimensional 
irreducible representation of $G$.  If
$$e'(\pi ,V, \frak p_-)\not= 0\,,$$
then for some $\sigma\in W(\frak g_{\Bbb C}, \frak t_{\Bbb C})$ and $w^\ast\in 
W_{\kappa_i}$ and $k\ge 1$
$$\sigma (\Lambda_\pi +\rho )=w^\ast (\lambda^\ast +\rho )-2\rho_n+k
\alpha_{\kappa_i}\,.$$
\endproclaim

\demo{Proof}
$$\eqalign{e'(\pi ,V, \frak p_-) &= \sum (-1)^q q\dim [H_\pi\otimes\Lambda^q \frak 
p_+
\otimes V]^K\cr
&=\int_K ch \pi ch\left(\sum (-1)^q q\Lambda^q_+\right) ch\tau\cr
&=\int_K ch\pi ch\left(\sum (-1)^q\Lambda^q_+\right)\left({n \over 2} +
{1 \over 2}\sum_{\Delta^+_n}\coth {\alpha \over 2}\right) ch\tau\cr
&=(-1)^n\int_K ch\pi ch\left(\sum (-1)^q\Lambda^q_-\right) ch 2\rho_n\left(
{n \over 2}+{1 \over 2}\sum_{\Delta^+_n}\coth {\alpha \over 2}\right)ch \tau
\,.\cr}$$
Now we apply Kostant's Theorem twice, obtaining (with $W^{\frak p_+}=W(
\frak g_{\Bbb C}, \frak t_{\Bbb C})/W(\frak k_{\Bbb C}, \frak t_{\Bbb C})$)
$$ch\pi ch\left(\sum (-1)^q\Lambda^q_-\right)=\sum_{W^{\frak p_+}}(-1)^{\ell (
\sigma )}ch V_{\sigma\pi}$$
and
$$V_\lambda\otimes\sum (-1)^\ell \Lambda^\ell\overline{\frak n}^c_\kappa =
\sum (-1)^{\ell (w)} W_{\lambda_w}\,,$$
and then using the Weyl formulas to get
$$\sum_{\kappa_i}{(-1)^{n+m_i} \over \vert W_K\vert}\int_T\sum_{W^{\frak p_+}}
(-1)^{\ell (\sigma )} ch V_{\sigma\pi}\Delta_K ch 2\rho_n\sum_{W_{\kappa_i}}
{1 \over {1-e^{-s^{-1}\alpha_{\kappa_i}}}} ch \rho (\frak n^c_\kappa )
\sum_{W_{\kappa_i}}(-1)^{\ell (w)}\Delta_{L_\kappa}ch W_{\lambda_w}dt\,.$$
Since after expanding $(1-e^{-s^{-1}\alpha_{\kappa_i}})^{-1}$ the integrand will 
consist of 
products of exponentials, in order for $e'(\pi ,V, \frak p_-)$ to be 
non-zero we need
$$\eqalign{0 &= \sigma (\Lambda_\pi +\rho )-\rho +\rho_c+2\rho_n+\rho 
(\frak n^c_\kappa )+\lambda_w+\rho (\ell_\kappa )-ks^{-1}\alpha_{\kappa_i}\cr
&= \sigma (\Lambda_\pi +\rho )+\lambda_w+\rho -ks^{-1}\alpha_{\kappa_i}\cr
&= \sigma (\Lambda_\pi +\rho )+w(\lambda +\rho_c)+\rho_n-ks^{-1}
\alpha_{\kappa_i}\,.\cr}$$
Then
$$\eqalign{s\sigma (\Lambda_\pi +\rho )&=-sw(\lambda +\rho_c)-\rho_n+k
\alpha_{\kappa_i}\cr
&= -sw(-w^K)(\lambda^\ast +\rho_c)-\rho_n+k\alpha_{\kappa_i}\cr}$$
and since $w^{L_\kappa }$ fixes $\alpha_{\kappa_i}$,
$$(w^{L_\kappa}s\sigma )(\Lambda_\pi +\rho )=(-w^{L_\kappa})(sw)(-w^K)(\lambda^\ast 
+\rho_c)
-\rho_n+k\alpha_{\kappa_i}\,.$$
Relabeling gives the result.\hfill\bx
\enddemo

\remark{Remark}  A comparison of Theorem 1.12 and Proposition 1.16 shows that 
the unitary principal series and the finite dimensional representations that 
have non-zero derived Euler number are in the same parametrized family, with 
the unitary character $e^{\nu}$ on $\exp tX_\kappa$ being replaced with the 
countable collection of characters on $\exp tH_\kappa$. This, together with some 
analysis, eventually leads to a sharper result than the usual one relating the $L^2$-torsion to holomorphic torsion 
on 
the compact 
dual. The argument will be sketched in \S 9.
\endremark

%\vfil\eject

\n{\bf \S 2~~Geometric Zeta Function}
\medskip

Let $\Gamma$ be a discrete, co-compact, torsion-free subgroup of $G$.  The 
quotient of $\widetilde{X} \simeq G/K$ by the action of $\Gamma$, $X:=\Gamma
\backslash \widetilde{X}$, is a compact, complex, Hermitian symmetric manifold having 
$\widetilde{X}$ as its universal cover. For the unitary 
representation $(\tau ,V)$ of $K$ on the complex vector space $V$ we let 
$\widetilde{\Bbb V}$ denote the associated 
holomorphic, Hermitian, homogeneous vector bundle over $\widetilde{X}$, 
and denote by $\Bbb V$ the quotient by $\Gamma$ over $X$.  Without loss of generality, we can assume that $\tau$ is 
irreducible, in which case we shall denote by $\lambda$ a
highest weight, as before, and by $V_\lambda$ its representation space.

The zeta function we will define in this section encodes geometric data from 
the Hermitian metric and the coefficient bundle restricted to the set of periodic geodesics of period $1$
for the canonical metric. The latter is a disjoint union of locally symmetric spaces
$\coprod\limits_{[\gamma ]\in[\Gamma \bs\{e\}]}X_\gamma$, each corresponding to a nontrivial
 conjugacy class in $\Gamma$. 
From Selberg we know that $\gamma$ is semisimple, so that its centralizer, $G_\gamma$, is reductive. 
Without loss of generality, we may assume that $\gamma$ has been conjugated 
to a standard Cartan subgroup with split factor contained in the $\exp\frak a$ defined in \S 1. 
Then   $\gamma =\gamma_I \, \exp\ell_\gamma Y$, with $Y \in \frak a$ 
to be defined below (see \S 2.1) and
$X_\gamma\cong\Gamma_\gamma\backslash G_\gamma /K_\gamma$ ([D-K-V], p.67). 
Each geodesic in $X_\gamma$ has the 
same length $\ell_\gamma$ and we let $\mu_\gamma$ denote the generic multiplicity. 

The simply connected cover of $X_\gamma$ is the set  
$\widetilde{X}_\gamma\cong G_\gamma /K_\gamma$ of $\gamma$-periodic points of the geodesic flow 
$\tilde\Phi$
acting on the tangent unit sphere bundle $S\widetilde{X}$. On the other hand, we
denote by $\widehat{X}_\gamma =X_\gamma/\Phi$ the orbit space 
under the (oriented) circle action of the restriction to $X_\gamma$ of the geodesic flow $\Phi$ on $SX$.     
 As in [M-S; I] and  [M-S; II] the zeta function will be supported only on the conjugacy classes  $[\gamma]$ such that $\gamma$ belongs to a cuspidal maximal parabolic 
subgroup. Significantly though, unlike the zeta functions in those papers, the present one will
also involve data from a compactification of the Hermitian symmetric space
$\widetilde{X}$. To make its construction explicit will require a digression into more detailed structure theory.
With a view to future applications, in the subsection that follows we present the structural details in more
generality for all conjugacy classes, regardless of the parabolic with which they are associated.

\n{\bf  2.1.~~} We shall use one of the standard compactifications of $\widetilde{X}$, namely the closure 
$\overline{X} = \overline{G\cdot K_\Bbb CP_-}$ 
of the Borel embedding of $\widetilde{X}$ into $G_\Bbb C/K_\Bbb CP_-$, or equivalently the closure  $\overline{X}_b$ of the Harish-Chandra embedding of $\widetilde{X}$ as a bounded domain in $\frak p_+$. It is known that these closures are equivariantly diffeomorphic, e.g. [\'O-S] Lemma 1.7.

We recall from [Sa] (p. 140 - 150) that $\partial\overline{X}=\overline{X}-\widetilde{X}$ is 
stratified, with each strata an orbit of $G$, and that the following 
are in $1-1$ correspondence: $\{G$-orbits $\Cal O_\kappa$ in $\partial
\overline{X}=\overline{X}-\widetilde{X}\}$, $\{$conjugacy classes of standard maximal parabolic 
subalgebras$\}$, and $\{\kappa$-Cayley transforms $C_\kappa =\Ad c_\kappa$, 
$c_\kappa\in G_\Bbb C$, associated to (classes of) canonical $\kappa\}$.

The $G$-orbit $\Cal O_\kappa = G\cdot c_\kappa K_\Bbb CP_- \cong G/E_\kappa\subset\partial\overline{X}$ comes with additional geometry (we refer
 the reader also to [Wo, p. 299] or [\'O-S Thm.1.10] for more details concerning that geometry).  
 Namely,  $\Cal O_\kappa = G/E_\kappa $ fibers over a compact flag manifold
$ G/Q_\kappa$, $Q_\kappa$ the maximal parabolic subgroup corresponding to $\kappa$, 
with typical fiber a Hermitian symmetric space  of noncompact type 
$Q_\kappa/E_\kappa \cong M_{Q_{\kappa}}^{(1)}
/K\cap M^{(1)}_{Q_{\kappa}}$, 
$$\lno{M_{Q_\kappa}^{(1)}/K\cap M^{(1)}_{Q_\kappa}\cong Q_\kappa/E_\kappa
 \to \Cal O_\kappa  {\buildrel \pi_\kappa \over \longrightarrow} \,G/Q_{\kappa}\cong K/K\cap M_{Q_\kappa}.&& (2.1) \cr}$$  
Denote by $T^{v}(\Cal O_\kappa)$ the subbundle of $T(\Cal O_\kappa)$ tangent to the fibers.

There is yet another relevant $K$-equivariant fibration of the space of holomorphic arc components, here with base a compact Hermitian symmetric space and typical fiber a Riemannian symmetric space
$$\lno{L_{\kappa}/K\cap M_{Q_\kappa} \to G/Q_{\kappa}\cong K/K \cap M_{Q_\kappa} 
{\buildrel \rho_\kappa \over \longrightarrow} \, K/L_{\kappa}.&& (2.2)\cr}$$

Thus for the structure of a typical orbit in the boundary one has 

$$\matrix &&L_{\kappa}/K\cap M_{Q_\kappa}\\
&&\downarrow\\
 M_{Q_\kappa}^{(1)}/K\cap M^{(1)}_{Q_\kappa}\cong Q_\kappa/E_\kappa&\to\quad \Cal O_\kappa \quad{\buildrel \pi_\kappa \over \longrightarrow} & G/Q_{\kappa}\cong K/K \cap M_{Q_\kappa} \\
&&\downarrow\\
&&  K/L_{\kappa}.\endmatrix $$

In \S1 we chose $\frak t \subset \frak k$ a toroidal Cartan subgroup. We denote by $\Psi = \{ \psi_1,\dots, \psi_r\}$ the maximal set of strongly orthogonal noncompact positive roots of $\frak t$ and denote by $\frak t^-\subset \frak t$  the span of the corresponding co-roots $\{ih_{\psi_i}\in \frak t\}$. Then $\frak a \subset \frak p$, an Iwasawa subalgebra, was defined to be the span of $\{X_i = e_{\psi_i} + e_{-\psi_i} \}$. There is a Cayley transform corresponding to $\Psi$ such that $ \frak h_{\Psi} := C_{\Psi} (\frak t_{\Bbb C})\cap \frak g= \frak t^+ \oplus \frak a$ is a maximal $\Bbb R$-split Cartan subalgebra and $C_\Psi(i\frak t^-) = \frak a$.   [Wo] p. 284 (4.5) shows for J a subset of $\Psi$ (N.B. Wolf uses $\Gamma\subset\Psi$ which would cause confusion here) using a partial Cayley transform, $C_J$, similarly one can obtain a representative of every conjugacy class of standard Cartan subalgebra, and moreover (using the Weyl group action) the class depends only on the number of elements chosen. Of course, $J = \emptyset$ gives $\frak t$, so we may assume $J \ne \emptyset$. We take $J = \{ \psi_1,\dots, \psi_k \mid 1\le\ k<r\}$ and denote the Cartan subalgebra by $\frak h_J = \frak t^+ \oplus \frak t^{-}_J \oplus \frak a_J$. Then $\frak a_J $ is spanned by $C_J(\{h_{\psi_1},\dots ,h_{\psi_k}\})$ and $\frak t^-_J$ is spanned by 
$\{ih_{\psi_{k+1}},\dots, ih_{\psi_r}\}$.  The corresponding Cartan subgroup is $H_J = Z_G(\frak h_J)= T_J \exp\frak a_J$, with $T_J/T_J^0$ generated by $\exp i\frak a_J \cap K$. 
The $\kappa$-Cayley transform associated to a boundary orbit can be identified with a $C_{J}$, $J \ne \emptyset$, i.e. $\Cal O_\kappa \cong G\cdot c_{J} K_\Bbb CP_-$. We shall use $c_J$ or $c_\kappa$ 
interchangeably according as the Cartan or the orbit is emphasized.
Similarly a representative of the maximal parabolic subalgebra associated to $\Cal O_\kappa$ can be obtained, namely in $\frak a_J$ the element $X_J = X_1+ \dots + X_k$ is a grading operator for $\frak q_\kappa$. Capitalizing on the Moore/Harish-Chandra description of the $\frak a$ restricted roots, Wolf ([Wo] p. 294) gave an explicit description of the root vectors in the various graded subspaces. A summary of his result follows. If $\delta \in \Delta(\frak g_{\Bbb C}, \frak t_{\Bbb C})$ let $\rho(\delta)$ be the restriction  of $\delta$ to $\frak t^-$, and set $J^c = \Psi \setminus J$. The eigenspaces of $\text{ad} \sum_{j\in J} 
h_{\psi_j}$ are the root vectors for the roots in
$$\eqalign{ E_{\pm 2} (J) & = \{\delta \in \Delta\vert \rho(\delta) = \pm \frac{1}{2} (\psi + \psi'), \psi, \psi' \in J\}\cr
E_{\pm 1 }(J) &= (1) \{ \delta \in \Delta\vert \rho(\delta) = \pm \frac{1}{2}(\psi +\psi'), \psi\in J, \psi' \in J^c\}\cr
&=(2)  \{\delta \in \Delta\vert  \rho(\delta) = \pm \frac{1}{2}(\psi -\psi'), \psi\in J, \psi' \in J^c\}\cr
&= (3)  \{\delta \in \Delta\vert  \rho(\delta) = \pm \frac{1}{2}\psi, \psi\in J\}\cr
E_0(J) &= (1)  \{ \delta \in \Delta\vert \rho(\delta) = \pm \frac{1}{2}(\psi \pm\psi'), \psi, \psi' \in J^c\}\cr
&= (2)  \{\delta \in \Delta\vert  \rho(\delta) = \pm \frac{1}{2}\psi, \psi \in J^c\}\cr
&=(3)  \{ \delta \in \Delta\vert \rho(\delta) = 0\}\cr
&=(4)  \{ \delta \in \Delta\vert \rho(\delta) = \pm \frac{1}{2}(\psi -\psi'), \psi,\psi'\in J\}\cr
&=(5)\,  \frak t_\Bbb C.\cr}
$$
The corresponding eigenspaces of $\text{ ad}(X_J)$ are obtained by applying $C_J$ to the root spaces corresponding to the above, i.e. $\frak g_l(J)_{\Bbb C} = \{ C_J(X_\delta) \vert \delta \in E_l(J)\}$.

As a consequence of this we have the following observation.
\proclaim{Lemma 2.1}
Let $Z\in \frak h_J$ with $Z= X+\sum_{i\in J} a_iX_i,a_i >0$, $X\in  \frak t^+ \oplus \frak t^{-}_J $. Then $\frak g_Z\subset \frak g_0(J)$.
\endproclaim
\demo{Proof}
As $Z$ is semisimple $\frak g_Z$ is reductive. Let $\frak g_Z = \frak c_Z \oplus \frak g^{(1)}_Z$ with $\frak c_Z$ the center of $\frak g_Z$ and $\frak g^{(1)}_Z$ semisimple. We complexify to $\frak g_{Z,\Bbb C} = \frak c_{Z,\Bbb C} \oplus  \frak g^{(1)}_{Z,\Bbb C}$. We may assume that $\frak h_J = \frak c_Z \oplus \frak h^{(1)}_J$ with $\frak h^{(1)}_J$ a fundamental Cartan in $\frak g^{(1)}_Z$ (see [H-C;InvInt] or [D-K-V] p.53). Let $\frak h^{(1)}_J = \frak t^{(1)}_Z \oplus  \frak a^{(1)}_J$ with $\frak a^{(1)}_J$ the $\Bbb R$-split part and $\frak a^{(1)}_J\subset \frak a_J$.  Then $\Delta(\frak g_{Z,\Bbb C}^{(1)}, \frak h_{J,\Bbb C}^{(1)})$ has no real roots.  The imaginary roots vanish on $ \frak a_{J,\Bbb C}^{(1)}$ thus on $\frak a_J$, so on $X_J$. $X_J$ is in the span of $\frak c_Z \oplus  \frak a_J^{(1)}$ so the restriction of the complex roots of $\Delta(\frak g_{Z,\Bbb C}^{(1)}, \frak h_{J,\Bbb C}^{(1)})$ to $\frak a_J^{(1)}$ are eigenvalues of $X_J$.  A glance at the table above shows that eigenvalues from $E_{\pm 2} (J) $ and $E_{\pm 1} (J) $ would give non-zero values on $Z$ since $a_i >0$ and all the restrictions, $\rho(\delta)$, to $\frak t^-$ involve at least one $\psi\in J$. Also, all the possible eigenvalues in $E_{0} (J)$, (1) to (4), can possibly occur depending on whether some $a_j = a_k$. \bx\enddemo

We return to $\Gamma$. 
As any nontrivial $\gamma \in \Gamma$ is semisimple, we may assume that $\gamma$ has been 
conjugated 
to a standard Cartan subgroup, say $H_J = Z_G(\frak h_J)= T_J \exp\frak a_J$ (later usually denoted $A_IA_{\Bbb R}$). As $\Gamma $ is torsion free, $H_J \ne T$, i.e. $J\ne \emptyset$. Furthermore, we may assume  that the Cartan subalgebra $\frak h_J$ is aligned, i.e. the Cartan subalgebra is fundamental in $\frak g_\gamma$ := $Z_\frak g (\gamma) = \frak k_\gamma \oplus \frak p_\gamma$. Then 
$\gamma =\gamma_I$ $\exp\ell_\gamma Y$ with $\gamma_I\in Z_K(Y) \cap Z(K_\gamma )$, $l_\gamma >0$, $Y$ of norm one in $\frak p_\gamma$ and 
$$Y= \sum_{j\in J} a_jX_j = -  C_J \left( \sum_{j\in J} a_jh_{\psi_j}\right)  ,  \qquad    a_j >0,
$$
as follows from $\frak h_J$ is aligned and the action of the Weyl group. Denote by $\gamma_I^0$ the factor of $\gamma_I$ in $T_J^0 $, then from Lemma 2.1 $\gamma_I^0 \in K\cap M_{Q_\kappa}$.  As $T_J/T_J^0$ is generated by $\exp i\frak a_J \cap K$ and $X_J\in \frak a_J$ thus $\gamma_I \in K\cap M_{Q_\kappa}. $

By sending each point in $\widetilde{X}_\gamma \cong G_\gamma/K_\gamma \subset \widetilde{X}$ to the point at infinity of 
the geodesic ray to which it belongs (see [Sa] p. 141), we get a $G_\gamma$-equivariant map  
$\widetilde{\nu}_{\gamma }$ from 
$\widetilde{X}_\gamma$ to the $G$-orbit $\Cal O_\kappa \cong G/c_\kappa K_{\Bbb C}P_-$ contained in
$\partial\overline{X}$.  
This map descends to the $G_\gamma$-equivariant  map 
  $$  \nu_{\gamma}\colon\widetilde{X}_\gamma /{\tilde{\Phi}}\cong G_\gamma /K_\gamma A
\hookrightarrow  \Cal O_\kappa \, , \qquad 
 \nu_{\gamma}(g K_\gamma A)=g c_\kappa K_{\Bbb C}P_- \, , \quad \forall g \in G_\gamma\,  ,
$$
 where $A =\exp (\Bbb R Y)$ corresponds to the geodesic flow $\tilde{\Phi}$. Then
 $\nu_{\gamma}(\widetilde{X}_\gamma /\tilde{\Phi})$
inherits a fibration  with typical fiber contained in 
$Q_\kappa/E_\kappa \cong M_{Q_\kappa}^{(1)}/K\cap M^{(1)}_{Q_\kappa}$ and base in $G/Q_{\kappa}$.

The familiar $SL(2,\Bbb C)$ computation gives the action of $\exp \frak a$ in the Harish-Chandra realization. For $t\in \Bbb R^{r}$ and
$\exp \sum_{j=1}^{r}t_jX_j$  we have
$$
\exp \sum_{j=1}^{r}t_jX_j\cdot \sum_{i=1}^{r}\xi_i  e_{\psi_i}
=\sum_{i=1}^{r} \frac{\cosh (t_i )\xi_{i}+\sinh (t_i)}{\sinh (t_i )\xi_i+\cosh (t_i)}\, e_{\psi_i}\, .
$$
In particular,
$\exp \sum_{j=1}^{r}t_jX_j\cdot 0=\sum_{i=1}^{r} \tanh (t_i) e_{\psi_i}\, .$ The lift to $G/K$ of the geodesic corresponding to $\gamma$ can be taken to be $ \exp t Y\cdot 0 = \exp t \sum_{j\in J} a_jX_j\cdot 0 = \sum_{i\in J} \tanh (t a_i) e_{\psi_i}\,$, and, since $a_i>0$,  as $t\to +\infty$ this is $ \sum_{i\in J} e_{\psi_i}\,$. The other geodesics in $G_\gamma/K_\gamma$ correspond to $\gamma$-periodic points in $S(\tilde X)$ and for $gK_\gamma\in G_\gamma/K_\gamma$ give curves $g\exp tY $. Thus $\nu_{\gamma}(\widetilde{X}_\gamma /\tilde{\Phi})$ is the $G_\gamma$ orbit of $ \sum_{i\in J} e_{\psi_i}\,$. 

Summarizing and using from Lemma 2.1 that $\frak g_\gamma\subset \frak g_0(J)\subset \frak q_\kappa$, as well as the determination in [Sa] Prop. 8.5 of the isotropy we obtain
\proclaim{Corollary 2.2} 
$\nu_{\gamma}(\widetilde{X}_\gamma /\tilde{\Phi}) = G_\gamma \cdot \sum_{i\in J} e_{\psi_i}\, {\buildrel \pi_\kappa \over \longrightarrow}  \,\,eQ_\kappa \in G/Q_\kappa$. For $\gamma \in \Gamma$ with $\gamma \ne e$ conjugate to an element in $H_{\{\psi_1\}}$ the map $  \nu_{\gamma}\colon\widetilde{X}_\gamma /{\tilde{\Phi}}\cong G_\gamma /K_\gamma A
\hookrightarrow  \Cal O_\kappa $ is an embedding.
\endproclaim

\definition{Definition/Notation} We denote by ${\Cal E}_1(\Gamma )$ the non-trivial $\Gamma$ conjugacy classes that can be conjugated to $H_{\{\psi_1\}}$. 
\enddefinition

\remark{Remark} The parabolic subalgebra $\frak q_{\{\psi_1\}}$ is maximal cuspidal, and the corresponding orbit ${\Cal O}_{\{\psi_1\}}$ is the maximal dimensional boundary orbit. It is for the ${\Cal E}_1(\Gamma )$ conjugacy classes that the geometric computations in \S 3 will be done. In contrast to [M-S; I] and [M-S; II], there are Hermitian groups $G$ with two maximal cuspidal parabolics however our geometric computations will be only for  ${\Cal E}_1(\Gamma )$.
\endremark

% In [Wo p. 296] it is shown that the corresponding maximal parabolic subalgebra $\frak q_\kappa$ has grading operator $\Ad c_{\Psi - J} (h_{\psi_{k+1}} + \cdots h_{\psi_r})$ where $i(h_{\psi_{k+1}} %+ \cdots h_{\psi_r})\in \frak t_J^-$, and that its orthocomplement in $ \frak t_J^-$ is in $\frak k \cap \frak m_{\frak {q}_\kappa}$. Thus $\Ad c_{\Psi - J} (\frak t^+ \oplus \frak t^{-}_J ) \subset \frak k %\cap \frak m_{\frak {q}_\kappa}\oplus \frak a^\Bbb C_\kappa.$ Denote by $\gamma_I^0$ the factor of $\gamma_I$ in $T_J^0 $, then $c_{\Psi - J}(\gamma_I^0)c_{\Psi - J}^{-1}  \in K\cap %M_{Q_\kappa}A_\kappa ^\Bbb C$.  But  $\frak a_J \subset \frak m^{(1)}_{\frak {q}_\kappa}$, so it follows that $T_J/T_J^0$ has generators in $K\cap M_{Q_\kappa}$, thus
% $\gamma_I^J := c_{\Psi - J}(\gamma_I)c_{\Psi - J}^{-1} \in K\cap M_{Q_\kappa}A_\kappa^\Bbb C. $

The restriction to $\nu_{\gamma}(\widetilde{X}_\gamma /\tilde{\Phi})$ of various vector bundles will be needed to define geometric data on $X_\gamma$. From the exact sequence of $E_\kappa$ modules 
$$0\to \frak e_\kappa \to \frak q_\kappa \to \frak g\to 0,$$ 
one obtains the exact sequence of tangent bundles to  \quad
$Q_\kappa/E_\kappa \to \Cal O_\kappa \cong G/E_\kappa {\buildrel \pi_\kappa \over \longrightarrow}\,  G/Q_\kappa,$
 $$T(Q_\kappa/E_\kappa) \to T(\Cal O_\kappa)\cong
 T(G/E_\kappa) \, {\buildrel \pi_{\kappa \ast} \over \longrightarrow}\, T(G/Q_\kappa).$$
Each can be associated to a representation of $E_\kappa$ and is respectively $Q_\kappa$ or $G$ homogeneous.

% we now proceed to construct its normal bundle in $\Cal O_\kappa$

\n{\bf  2.2.~~} From now on we enforce throughout the rest of the paper the requirement that  
$[\gamma ]\in {\Cal E}_1(\Gamma)$. Turning our attention to $\widetilde{X}_\gamma /\tilde{\Phi}$,  
we freely identify it with its embedded image 
$\nu_{\gamma}(\widetilde{X}_\gamma /\tilde{\Phi}) \subset \Cal O_\kappa$ and proceed to define an 
appropriate analogue of a normal bundle for $\widetilde{X}_\gamma /\tilde{\Phi}$ relative to the 
boundary orbit $\Cal O_\kappa$, as well as a boundary version of a Poincar\'e map.

\proclaim{Proposition 2.3} Assuming $[\gamma] \in {\Cal E}_1(\Gamma)$ and aligned,
the embedded orbit $\nu_{\gamma}(\widetilde{X}_\gamma /\tilde{\Phi})$ inherits
a canonical Hermitian symmetric structure.
\endproclaim
\demo{Proof}  
Under the above hypotheses, $G_\gamma \subset M_{Q_{\kappa}} \subset Q_\kappa$ and 
 $K_\gamma A$  is contained in the stabilizer, $E_\kappa = c_\kappa(K_\Bbb CP_-)c_\kappa^{-1}\cap G$, 
of $c_\kappa K_\Bbb CP_-$ in $G$. It follows that
$\pi_\kappa (\nu_{\gamma}(\widetilde{X}_\gamma /\tilde{\Phi})) = eQ_\kappa$ hence
$\pi_{\kappa \ast} T(\widetilde{X}_\gamma /\tilde{\Phi}) =0$. Therefore 
$T(\widetilde{X}_\gamma /\tilde{\Phi})$
is a subbundle of  $T^{v}(\Cal O_\kappa)\vert\nu_{\gamma}(\widetilde{X}_\gamma /\tilde{\Phi})$,
where $T^{v}(\Cal O_\kappa): = \Ker \pi_{\kappa \ast}$.
Furthermore, noting that the restrictions of all
 the above tangent bundles to $\nu_{\gamma}(\widetilde{X}_\gamma /\tilde{\Phi})$
are $G_\gamma$-homogeneous bundles associated to a semisimple
representation of $K_\gamma A$, we equip them with Hermitian inner products on which $K_\gamma$
acts unitarily. In particular the orthocomplement of $T(\widetilde{X}_\gamma /\tilde{\Phi}) $ in
 $T^{v}(\Cal O_\kappa)\vert \nu_{\gamma}(\widetilde{X}_\gamma /\tilde{\Phi})$ is well-defined and we
 call it ${\Cal N}(\widetilde{X}_\gamma /\tilde{\Phi})$. 
 The latter inherits an almost complex structure from the Hermitian symmetric space 
 $M_{Q_{\kappa}}^{(1)}/K\cap M^{(1)}_{Q_{\kappa}}$ via the isomorphism
 $T^{v}(\Cal O_\kappa)\cong T(M_{Q_{\kappa}}^{(1)}/K\cap M^{(1)}_{Q_{\kappa}})$  (cf. fibration (2.1)). Explicitly, the almost complex structure on  $M_{Q_{\kappa}}^{(1)}/K\cap M^{(1)}_{Q_{\kappa}}$  is given by the operator $H_0^{(1)}$ of eq. (1.6) which
 splits $\frak p^{(1)}_\Bbb C$ into  $\frak p^{(1)}_+ \oplus \frak p^{(1)}_-$, the $\pm i$-eigenspaces of
 $\ad H_0^{(1)}$. Since  $[Y, H_0^{(1)}]=0$,  $\ad H_0^{(1)}$ commutes with $\ad \gamma$ and therefore
 preserves $\frak p^{(1)}_\gamma := 
 \frak p^{(1)}\cap \frak p_\gamma$ and its complement orthocomplement 
$\frak p^{(1)}_{\gamma, \perp}$ in $\frak p^{(1)}$, inducing complex structures 
(after complexification) on both $\frak p^{(1)}_\gamma$ and $\frak p^{(1)}_{\gamma, \perp}$. 
In particular, endowed with this structure
$\nu_{\gamma}(\widetilde{X}_\gamma /\tilde{\Phi})$ becomes itself a Hermitian symmetric space.
\bx\enddemo

The above argument also shows that 
the complexified normal bundle ${\Cal N}_\Bbb C(\widetilde{X}_\gamma /\tilde{\Phi})$
relative to  $M_{Q_{\kappa}}^{(1)}/K\cap M^{(1)}_{Q_{\kappa}}$ also inherits an almost complex structure.
We denote by
${\Cal N}_\pm(\widetilde{X}_\gamma /\tilde{\Phi})$ the subbundles of  
${\Cal N}_\Bbb C(\widetilde{X}_\gamma /\tilde{\Phi})$
corresponding to the $\pm i$-eigenspaces of $\frak p_{\gamma, \perp}^{(1)} \otimes \Bbb C$.
 
Next, recalling the fibration (2.2), we decompose the tangent space
$T(G/Q_\kappa)\vert_{eQ_\kappa} \cong T(K/K \cap M_{Q_\kappa})\vert _{eK \cap M_{Q_\kappa}}$
into the subspace $U=$ Ker\,$ \rho_{\kappa, \ast}$\footnote{which is $1$-dimensional in this case}
and its orthocomplement $V$. The pulled-back bundles 
${\Cal U}(\widetilde{X}_\gamma /\tilde{\Phi}):=\pi_\kappa^\ast (U)$ and
${\Cal V}(\widetilde{X}_\gamma /\tilde{\Phi}):=\pi_\kappa^\ast  (V)$ are of course trivial.
However the flow $\tilde\Phi$ acts nontrivially on them. 
The subspace $V$ comes equipped with a complex structure  inherited
from the Hermitian symmetric structure of the base of the fibration (2.2), cf. [Sa, p.100]. 
This complex structure
passes to ${\Cal V}(\widetilde{X}_\gamma /\tilde{\Phi})$ and we let 
 ${\Cal V_\pm}(\widetilde{X}_\gamma /\tilde{\Phi})$ be the corresponding $\pm$-subbundles of
 ${\Cal V}_\Bbb C (\widetilde{X}_\gamma /\tilde{\Phi})$.

We denote the monodromy representations (induced by $\Ad \gamma$) on the
bundles $\Cal N_\pm(\widetilde{X}_\gamma /\tilde{\Phi})$, $\Cal U(\widetilde{X}_\gamma /\tilde{\Phi})$ and 
$\Cal V_\pm(\widetilde{X}_\gamma /\tilde{\Phi})$ respectively by $ N_\pm (\gamma)$,
$P_{\Cal U} (\gamma)$ and  $P_{\Cal V_\pm} (\gamma)$. The last two taken together can be regarded as an 
avatar of the Poincar\'e map at the ``boundary at infinity''.

The above bundles, being $G_\gamma$-homogeneous, descend to the orbifold 
$\widehat{X}_\gamma\cong\Gamma_\gamma\backslash\widetilde{X}_\gamma /A$ where they shall be denoted, respectively, by ${\Cal N} (\widehat{X}_\gamma )$,
 ${\Cal U}(\widehat{X}_\gamma)$ and ${\Cal V}(\widehat{X}_\gamma)$.  

In addition, we shall need the localization to $\widehat{X}_\gamma$ of the coefficient 
bundle $\Bbb V$ over $X$.  Its lift $\widetilde{\Bbb V}$ to $\widetilde{X} = 
G/K$ has a natural holomorphic extension $\overline{\Bbb V}$ to the 
compactification $\overline{X}$ in $G_{\Bbb C}/K_{\Bbb C}P_-$ \, 
associated to the representation $\tau:  K_{\Bbb C}P_- \to \GL(V_\lambda)$  obtained by
extending the original representation $\tau$ trivially on $P_-$. Next 
$\overline{\Bbb V}$ restricts to a homogeneous vector bundle $\overline{\Bbb V}_\kappa$ over 
$\Cal O_\kappa \cong G/E_\kappa$ with action of the isotropy $E_\kappa$ on the fiber given by $\tau(C^{-1}_\kappa(E_\kappa))$. 
This in turn restricts to a bundle $\overline{\Bbb V}_{ \gamma, \kappa}$ over 
$\nu_{\gamma}(\widetilde{X}_\gamma /\tilde{\Phi})$. It descends to give 
a vector bundle ${\Bbb V}_{\gamma, \kappa}$ over the orbifold 
$\widehat{X}_\gamma \cong \Gamma_\gamma\backslash \widetilde{X}_\gamma /\tilde{\Phi}$. 
As in the case of the aforementioned trivial bundles, while fixing each point of $\widehat{X}_\gamma$, 
$\gamma$ is acting nontrivially on the fibers of ${\Bbb V}_{\gamma, \kappa}$. 
 
%A concomitant effect of the embedding of $\widetilde{X}_\gamma/\tilde\Phi$ into the boundary is the compactification of the monodromy action of the flow on the above vector bundles. Indeed, since $\gamma =\gamma_I \gamma_R \in {\Cal E}_1(\Gamma)$, the split factor  $\gamma_R:= \exp(\ell_\gamma Y)$ belongs to the standard Cartan  ${\frak a}_J = {\Bbb R} Y$,  with $Y= -a_1 C_J h_{\psi_1}$  and $J = \{\psi_1\}$ in the notation of \S {\bf  2.1}, and $A= \exp  {\Bbb R} Y$ becomes compactified by the Cayley transform (cf.  \S 1). Let $A_\kappa:= \exp  {\Bbb R} C_\kappa(Y)$, and  $\gamma_\kappa :=\gamma_I  \exp(\ell_\gamma C_\kappa(Y))$. Noting that $\gamma_R$ acts trivially on $T(\widetilde{X}_\gamma /\tilde{\Phi})$ and on ${\Cal N}(\widetilde{X}_\gamma /\tilde{\Phi})$, we let $\gamma_\kappa$ act on those bundles in the same way as $\gamma$. 

We denote by  $\langle \gamma \rangle$ the topological cyclic group generated by $\gamma$,
and by  $K_{\langle \gamma \rangle}^0 (\widehat{X}_\gamma)$ the associated equivariant  
 $K$-theory group formed with vector bundles endowed with semisimple $\langle \gamma \rangle$-actions. 
All the vector bundles mentioned above define elements in this equivariant $K$-theory group.
Since $\langle \gamma \rangle$ acts trivially on $\widetilde{X}_\gamma/\tilde\Phi$, $K_{\langle \gamma \rangle}^0 (\widehat{X}_\gamma)  \cong  K^0 (\widehat{X}_\gamma) \otimes R (\langle \gamma \rangle)$, where $R (\langle \gamma \rangle)\cong \Bbb Z[\Bbb C^\ast]$ 
is the ring of semisimple representations of $\langle \gamma \rangle$.  
The evaluation homomorphism 
 $K_{\langle \gamma \rangle}^0 (\widehat{X}_\gamma) \to H^*(\widehat{X}_\gamma; \Bbb C)$ associated to $g \in \langle \gamma \rangle$ is defined as in [A-S: III, \S 2] by
$$\lno{&  \xi =\Sigma u_i\otimes \chi_i \longmapsto \ch\xi (g)=\Sigma \chi_i(g) \ch u_i \, .
&(2.3)\cr}$$
 
Finally, after adding to the above geometric data a flat Hermitian bundle ${\Bbb F}$ associated to
 a finite-dimensional unitary representation $\varphi: \Gamma\to U (F)$, 
we  make the following definition of a Weil-type zeta function:
$$\lno{& \log Z_{\Bbb V ,\Bbb F}(z) =  - \sum_{[\gamma] \in {\Cal E}_1(\Gamma)}
  \Tr \varphi (\gamma) \, \chi_\gamma (\Bbb V)  \,
{e^{-z\ell_\gamma}\over \mu_\gamma}\, , &(2.4)\cr}
$$
where
 $$\lno{\qquad \chi_\gamma (\Bbb V) &= {1 \over \det(I-P_{\Cal U} (\gamma)) 
\det(I- P_{\Cal V_-} (\gamma))}
 \left\{ {\Td T(\widehat{X}_\gamma)
\cup \ch [\Bbb V_{\gamma ,\kappa}](\gamma) 
\over\ch [\wedge_{-1}{\Cal N}_+(\widehat{X}_\gamma)^*](\gamma )}\right\} [\widehat{X}_\gamma ]\,. 
&(2.5)\cr}$$
The above expression makes sense for the following reasons.
First, each factor in the denominator is invertible, since the (diagonalizable) operators involved
have only eigenvalues $\not =1$. Actually,  
$ \det(I-P_{\Cal U} (\gamma)) = 1-e^{-2\ell_\gamma}$ \,  and \,
 $\det(I-P_{\Cal V_-} (\gamma)) = \det(I-\Ad \gamma^{-1}\mid \frak v_-)$. Also recall that the Lie algebra of the isotropy  in $G$ at $c_\kappa K_\Bbb CP_-$ contains the nilradical $\frak u \oplus \frak v$, therefore the opposite Lie algebras are tangent there. By a small abuse of notation we use the notation $\frak u \oplus \frak v$ for tangent vectors but change the action of elements of the Cartan to their inverse - hence the $\gamma^{-1}$.
Secondly,  $\widehat{X}_\gamma$ is
a ``good'' orbifold, having a finite cover which is a manifold, and w.l.o.g. can also be assumed 
orientable, so that the Todd class $\Td T(\widehat{X}_\gamma)$ is well-defined and so is
the fundamental class $[\widehat{X}_\gamma ]$(cf. M-S;I, \S 5], [M-S;II, \S 3]).

 \bigskip

\n{\bf \S 3~~Explicit formula of Zeta}
\medskip 
 
The technical lemmas provided in [M-S;II, \S 3], 
under the similar assumption $[\gamma] \in{\Cal E}_1(\Gamma )$,
remain valid with the proviso that the connected 
component $G^0_\gamma$ employed in {\it loc. cit.} has to be replaced by 
the Harish-Chandra class subgroup $G^+_\gamma$, $G^0_\gamma\subseteq 
G^+_\gamma\subseteq G_\gamma$. Modulo dealing with the non-orientable case as 
indicated in [M-S;I, \S 5], we are assuming from now on
that $\widehat{X}_\gamma$ is orientable and proceed 
to express $ \chi_\gamma (\Bbb V)$ in group-theoretical terms.

We begin by rewriting the formula (2.5) in the equivalent way  
$$\lno{&\qquad\quad \chi_\gamma (\Bbb V) = \det(I-P_{\Cal U \oplus  \Cal V} (\gamma))^{-1}
\left\{ {\Td T(\widehat{X}_\gamma)\cup \ch \Bbb V_{\gamma ,\kappa}(\gamma)
\cup \ch \wedge_{-1}{\Cal V}_+(\widehat{X}_\gamma)(\gamma)
 \over \ch \wedge_{-1}({\Cal N_+}(\widehat{X}_\gamma)^*(\gamma) }\right\} [\widehat{X}_\gamma ].
&(3.1)\cr}$$
Indeed, since ${\Cal V}_+$ is a trivial bundle $\ch \wedge_{-1}{\Cal V}_+(\widehat{X}_\gamma)(\gamma)
= \det(I- P_{\Cal V_+} (\gamma))$.
 Noting that  $\det(I- P_{\Cal V} (\gamma)) = \det(I- P_{\Cal V_\Bbb C} (\gamma)) =  
 \det(I- P_{\Cal V_+} (\gamma)) \det(I- P_{\Cal V_-} (\gamma))$, one recovers (2.5).

Recall that the typical fiber of ${\Cal N_+}(\widehat{X}_\gamma )$ is $\frak p^{(1)}_{\gamma ,\perp +}$. 
Now  $\frak p^{(1)}\cong \frak m_{\frak q_\kappa}^{(1)}/\frak k\cap \frak m_{\frak q_\kappa}^{(1)}$,
 and in making this identification we regard $\frak p^{(1)}_\gamma$ as a 
$K_\gamma A$-module as follows.  First, $K^{(1)}_\gamma=K_\gamma
\cap M^{(1)}_Q\subset K\cap M^{(1)}_Q$ acts on $\frak p^{(1)}$ via the 
adjoint representation $\Ad$.  In turn 
 $K_\gamma =K^{(1)}_\gamma \cdot K^{(2)}_\gamma$ 
(almost direct product with central intersection). Although $\frak p^{(2)} = 0$, the group
 $K^{(2)}_\gamma$ may not be trivial. However it acts trivially on  $\frak p^{(1)}$, and so does
 $A$.  
 
Letting $S_{\gamma ,\perp}^{(1)}$ denote the graded spin module of $\Spin
\frak p^{(1)}_{\gamma ,\perp}$ regarded as a ${\tilde K}_\gamma$-module, 
one has  
$$ \wedge_{-1} \frak p^{(1)}_{ \gamma ,\perp -} = (-1)^{q_{\gamma, \perp}}(S^{(1)}_{\gamma,\perp+}
- S^{(1)}_{\gamma ,\perp -})\otimes \overline{\Bbb C}_{\rho_{ \gamma, \perp, n}^{(1)}},
$$
where $q_{\gamma, \perp} = \dim \frak p^{(1)}_{ \gamma ,\perp -}$ \, and \,
$\rho_{ \gamma, \perp, n}^{(1)} =\rho(\frak p^{(1)}_{ \gamma ,\perp +})$.
Therefore
$$\ch \wedge_{-1}{\Cal N_+}(\widehat{X}_\gamma )^*  = (-1)^{q_{\gamma, \perp}}
\ch (\Bbb S_{\gamma ,\perp}^{(1)}\,_+-\Bbb S_{\gamma ,\perp}^{(1)}\,_-)
\otimes {\Bbb L}^*_{\rho_{ \gamma, \perp, n}^{(1)}}\,,$$
where $\Bbb S_{\gamma ,\perp}^{(1)}$ stands for the vector bundle over 
$\widehat{X}_\gamma$ associated to $S^{(1)}_{\gamma ,\perp}$ and 
$\Bbb L_{\rho_{ \gamma, \perp, n}^{(1)} }$ for the line bundle associated to the 
character $\xi_{\rho_{\gamma, \perp, n}^{(1)} }$. 
Accordingly, formula (3.1) becomes
$$\lno{\chi_\gamma (\Bbb V) &= (-1)^{q_{\gamma, \perp}} \det (I-\Ad \gamma^{-1}\mid
{\frak v}\oplus{\frak u})^{-1}\cr
&\quad \left\{{\Td T(\widehat{X}_\gamma ) \cup\ch [\Bbb V_{\gamma ,\kappa}
\otimes\wedge_{-1}{\Cal V}_+(\widehat{X}_\gamma )] (\gamma)\over
\ch [(\Bbb S_{\gamma ,\perp, +}^{(1)}-\Bbb S_{\gamma ,\perp, -}^{(1)}) \otimes
{\Bbb L}^*_{\rho_{ \gamma, \perp, n}^{(1)} }](\gamma)} 
\right\} [\widehat{X}_\gamma ]\,. &(3.2)\cr}$$ 

Now the Todd class $\Td T(\widehat{X}_\gamma )$, 
can also be expressed in terms of characters of spin bundles. Indeed,
$$\Td T(\widehat{X}_\gamma ) = 
\widehat{A}(T(\widehat{X}_\gamma))\cup \ch{\Bbb L}_{\rho_{ \gamma, n}^{(1)} }\,,$$
where $\widehat{A}$ denotes the stable Hirzebruch $\widehat{A}$-class, and 
${\Bbb L}_{\rho_{ \gamma, n}^{(1)} }$ corresponds to 
$\rho_{ \gamma, n}^{(1)} =\rho(\frak p^{(1)}_{ \gamma})$, and in turn,
$$\widehat{A}(T(\widehat{X}_\gamma )) = {e(T\widehat{X}_\gamma )\over
\ch (\Bbb S^{(1)}_{\gamma ,+} -\Bbb S^{(1)}_{\gamma , -})}\,,$$
where $ e(T\widehat{X}_\gamma )$ is the Euler class and
$\Bbb S^{(1)}_\gamma$ is the spin bundle associated to $\Spin
\frak p^{(1)}_{\gamma}$. Thus
$$\Td T(\widehat{X}_\gamma ) = { e(T\widehat{X}_\gamma ) \cup \ch{\Bbb L}_{\rho_{\gamma, n}^{(1)} }
\over \ch (\Bbb S^{(1)}_{\gamma ,+} -\Bbb S^{(1)}_{\gamma , -})} =
 { e(T\widehat{X}_\gamma ) \over \ch [(\Bbb S^{(1)}_{\gamma ,+} -\Bbb S^{(1)}_{\gamma , -})\otimes {\Bbb L}^*_{\rho_{\gamma, n}^{(1)}}]} \,.$$
Since $e^\alpha(\gamma) = 1$ for any $\alpha \in P_{\gamma}= \Delta^+ (\frak g_{\gamma)}$\, ,
the above expression
does not change by evaluation at $\gamma$. By inserting it
in (3.2) one obtains:
$$\lno{\chi_\gamma (\Bbb V) &= (-1)^{q_{\gamma, \perp}}\det (I- \Ad \gamma^{-1}\mid
{\frak v}\oplus{\frak u})^{-1}\cr
&\quad \left\{{e(T\widehat{X}_\gamma )\cup\ch [\Bbb V_{\gamma ,\kappa}
\otimes\wedge_{-1}{\Cal V}_+(\widehat{X}_\gamma )\otimes{\Bbb L}_{\rho_{n}^{(1)} }](\gamma )
 \over \ch [(\Bbb S^{(1)}_+-\Bbb S^{(1)}_-)](\gamma )}
\right\} [\widehat{X}_\gamma ]\,, &(3.3)\cr}$$
where $\Bbb S^{(1)}$ is the bundle associated to the spin module of  $\Spin\frak p^{(1)}$.
% and ${\Bbb L}_{\rho_n^{(1)} }$ is the line bundle associated to $\rho_n^{(1)} = \rho (\frak p^{(1)})$.

%%%%%%%%%%%%%%%%%%%%%%%%
 Since for $[\gamma] \in{\Cal E}_1(\Gamma)$ we have $\frak l_{\kappa} = \frak k\cap \frak m_{Q_{\kappa} }\oplus \Bbb R H_\kappa$ (see \S 1, eq. $(1.9)'$, (1.9)) and $C_\kappa (\gamma)=C_\kappa (\gamma_I \gamma_R)=\gamma_IC_\kappa(\gamma_R) \in L_\kappa$ 
 and $C_\kappa \theta (V_+) =  \overline{\frak n}_\kappa^c$ as $ L_\kappa$-modules, using 
 Kostant's Theorem in \S 1 applied to the $L_\kappa$-module $V_\lambda \otimes \wedge_{-1}V_+$,
 one obtains the decomposition
$$V_{\lambda}\otimes\wedge_{-1}{V}_+
  = \sum_{w\in W_\kappa}(-1)^{\ell (w)}W_{\lambda_w} \, , \qquad  \text{where} \quad
\lambda_w= w(\lambda + \rho_c) - \rho_c . $$
 By Corollary 2.2, $K_\gamma \subset K\cap M^+_{Q_\kappa}$.
Each $W_{\lambda_w}$ occuring in the sum is irreducible for $ \frak k\cap \frak m_{Q_{\kappa}} \oplus \Bbb R H_\kappa$, on which $H_\kappa$ acts as $\lambda_w(H_\kappa)\Id$ and (cf. (1.12))   $X_\kappa$ acts by 
$\lambda_w\left(C_{\kappa}(X_\kappa)\right) \Id$ (cf. (1.3) and  (1.12)). Consequently,
the virtual bundle in the numerator of (3.3) decomposes as 
$$\lno{\Bbb V_{\gamma ,\kappa}\otimes\wedge_{-1}{\Cal V}_+(\widehat{X}_\gamma)
  &= \sum_{w\in W_\kappa}(-1)^{\ell (w)}  \Bbb W_{\lambda_w} \,. &(3.4)\cr}$$
Denoting now
$$\chi_{\gamma,w} ({\Bbb V}) = \left\{ {(-1)^{q_{\gamma, \perp}}e(T\widehat{X}_\gamma )\cup\ch \Bbb W_{\lambda_w+\rho_{n}^{(1)}} 
(\gamma) 
\over \ch (\Bbb S^{(1)}_+-\Bbb S^{(1)}_-)(\gamma )}\right\} [\widehat{X}_\gamma]\, ,$$
and 
substituting the decomposition (3.4) into the formula (3.3) one obtains
$$\lno{\chi_\gamma (\Bbb V)&=(\det (I-\Ad \gamma^{-1})\mid
{\frak v}\oplus{\frak u})^{-1}
\sum_{w\in W_\kappa}(-1)^{\ell  (w)} \chi_{\gamma,w} ({\Bbb V})\,.&(3.5)\cr}$$ 
For the record, observing that $\rho_{n}^{(1)}=\rho_{M, n}$ and with
the abbreviated notation $M$ for (the connected component of) $M^+_{Q_\kappa}$, 
the expression of $\chi_{\gamma,w} ({\Bbb V})$ becomes
 $$\lno{\chi_{\gamma,w} ({\Bbb V}) &= \left\{ {(-1)^{q_{\gamma, \perp}}e(T\widehat{X}_\gamma )\cup
 \ch \Bbb W_{\lambda_w +\rho_{M, n}}(\gamma )
\over \ch (\Bbb S^{(1)}_+-\Bbb S^{(1)}_-)(\gamma )}\right\}
 [\widehat{X}_\gamma ]\, .&(3.6)\cr}$$

In order to express the above formula in group-theoretical terms we shall appeal to the known formalism 
explained in [A-S III, \S 2] and also used in [H-P, \S 2]. Essentially it amounts to expressing
the equivariant characteristic classes of $H$-vector bundles over a closed manifold  $X$, where
$H$ is a compact Lie group, 
by means of the natural functor from principal $H$-bundles to rational cohomology of the base. 
If $P$ is such a principal bundle, $V$ is an $H$-module and $\Bbb V=P\times_H V$ is the associated
vector bundle, then for any element $\gamma$ in the center $Z(H)$ of $H$,
$$ \lno{\ch {\Bbb V}(\gamma) &= \ch V (\gamma)(P) \,, &(3.7)\cr} $$
where the left hand side is defined as in (2.3).

This formalism extends as well to the present case where
 $H=K_\gamma A$ and $X= \widehat{X}_\gamma$ is an orbifold. The reason is that, on the one hand
 $ \widehat{X}_\gamma$ admits finite normal coverings which are closed locally symmetric manifolds 
 (cf. [M-S I, Lemma 5.7], see also the diagram (3.9) below), and on the other hand the Weyl character formula extends to
 groups of Harish-Chandra class (see [Va], p. 454), to which $K_\gamma A$ belongs.
In particular, applying (3.7) in the extended context yields the equality
$${\ch \Bbb W_{\lambda_w+\rho_{M, n}} (\gamma ) \over  \ch (\Bbb S^{(1)}_+-\Bbb S^{(1)}_-) (\gamma )} = 
\left({\ch W_{\lambda_w +\rho_{M, n}}(\gamma)  \over  \ch (S^{(1)}_+ - S^{(1)}_-)(\gamma )}\right) (P^{(\gamma )})\,,$$
where $P^{(\gamma )}=\Gamma_\gamma\backslash G_\gamma$ is viewed as a 
principal  $K_\gamma A$-bundle over $\widehat{X}_\gamma$. Using then the 
 extended Weyl character formula, together with the known expression of the character in the
denominator, one obtains 
$${\ch \Bbb W_{\lambda_w+\rho_{M, n}} (\gamma ) \over  \ch (\Bbb S^{(1)}_+-\Bbb S^{(1)}_-) (\gamma )} = 
{e^{\lambda_w\left(C_{\kappa}(l_\gamma Y)\right)}
\sum\limits_{s\in W(K\cap M^+_{Q_\kappa})} \varepsilon (s)e^{s(\lambda_w +\rho_{M})}(
\gamma_I)e^{s(\lambda_w +\rho_{M})} 
\over\prod\limits_{\alpha\in 
P(\frak{m}_{Q_\kappa})}(e^{\alpha /2}(\gamma_I)e^{\alpha /2}-e^{-\alpha /2}(
\gamma_I)e^{-\alpha /2})}\,.$$
From eq. (3.6) it follows that 
$$  \chi_{\gamma,w} ({\Bbb V}) = (-1)^{q_{\gamma, \perp}}e^{\lambda_w\left(C_{\kappa}(l_\gamma Y)\right)}
\widehat{L}(\gamma ;\lambda_w ) \, ,$$
 where  
$$\lno{\qquad \widehat{L} (\gamma ;\lambda_w ) &= \left\{e(T\widehat{X}_\gamma )\cup
 {\sum\limits_{s\in W(K\cap M^+_{Q_\kappa})}\varepsilon (s)e^{s(\lambda_w +
\rho_{M})}(\gamma_I)e^{s(\lambda_w +\rho_{M})} \over \prod\limits_{\alpha\in 
P(\frak{m}_{Q_\kappa})}(e^{\alpha /2}(\gamma_I)e^{\alpha /2}-e^{-\alpha /2}(
\gamma_I)e^{-\alpha /2})}(P^{(\gamma )})\right\}[\widehat{X}_\gamma] .
&(3.8)\cr}$$
The latter is the precise counterpart of the Lefschetz number in [H-P, (2.16)]. Indeed, notice that  $W(K\cap M^+_{Q_\kappa}))$ is the $W_G$ in (2.16).
By making use of the results in [M-S; I, \S 5], [M-S; II, \S 3] the
computation of $\widehat{L} (\gamma ;\mu )$ can be reduced to that performed in [H-P].

As mentioned above,
 $X_\gamma$ admits a finite normal cover $^0X_\gamma = ^0\Gamma_\gamma\backslash G_\gamma /K_\gamma$, constructed in  [M-S;I, Lemmas 5.4 - 5.6],
on which the lift $\Phi'$ of the flow $\Phi$ acts by a free $S^1$-action (see [M-S;I,  Lemma 5.7]). The
quotient space $^0X'_\gamma = ^0X_\gamma /S^1$ is of the form
$^0X'_\gamma =^0\Gamma'_\gamma \backslash G'_\gamma /K'_\gamma$, where $G'_\gamma$ is
the semisimple factor of $G_\gamma$ and $K'_\gamma = G'_\gamma \cap K_\gamma$. 
The commutativity of the diagram
$$\lno{&\matrix ^0X_\gamma &\longrightarrow &^0X'_\gamma = ^0X_\gamma /S^1\\
\Big\downarrow &&\Big\downarrow ~~~~~~~~~\\
X_\gamma &\longrightarrow &\widehat{X}_\gamma = X_\gamma /S^1\endmatrix 
&(3.9)\cr}$$
ensures the proportionality  of the corresponding characteristic numbers.  In particular, the 
calculations of the Euler numbers in  [M-S;II, \S 3] imply \footnote {We use $\vol$ to denote Harish-Chandra's definition as in [D-K-V p. 44] and, if needed, $\vol_0$ for Riemannian as in loc. cit.}
$${\chi (\widehat{X}_\gamma ) \over \chi (^0X'_\gamma )} = {\vol (
\Gamma_\gamma \backslash G_\gamma ) \over \vol (^0\Gamma '_\gamma 
\backslash G'_\gamma )\, \lambda_{\gamma^*}} ;$$
here $ \lambda_{\gamma^*}$ 
is the period of the geodesics corresponding to the conjugacy class
of $\gamma$ (see [M-S;I Lemma 5.4]).

 Similarly, if we set
$$^0L'(\gamma ;\lambda_w ) = \left\{ e(T\, ^0X'_\gamma)\cup
{ \sum\limits_{s\in W(K\cap M^+_{Q_\kappa})}\varepsilon (s) e^{s(\lambda_w +
\rho_{M})}(\gamma_I)e^{s(\lambda_w +\rho_{M})} \over \prod\limits_{\alpha\in 
P(\frak{m}_{Q_\kappa})}(e^{\alpha /2}(\gamma_I)e^{\alpha /2}-e^{-\alpha /2}(
\gamma_I)e^{-\alpha /2}))}(P^{\prime (\gamma )})\right\} [^0X'_\gamma ]\,, 
$$
where $P^{\prime (\gamma )} = \, ^0\Gamma'_\gamma \backslash G'_\gamma $,  one has
$$\lno{\widehat{L}(\gamma ;\lambda_w ) &= \, ^0L'(\gamma ;\lambda_w ){\vol (\Gamma_\gamma
\backslash G_\gamma ) \over \vol (^0\Gamma '_\gamma\backslash G'_\gamma )
\lambda_{\gamma^*}}\,. &(3.10)\cr}$$
Now by a computation completely parallel to that in [H-P, \S 2], which in turn is based on
that of W. Schmid [Sc, \S 5], and is applied here to the group 
$M^0_{Q_\kappa} = M^{(1)}_{Q_\kappa} M^{(2)}_{Q_\kappa}$, one obtains
 $$\lno{\qquad\qquad^0L'(\gamma ;\lambda_w ) &=  [G_\gamma:G_\gamma^0]^{-1}v(^0\Gamma 
'_\gamma\backslash 
G'_\gamma )(-1)^{\# P_{\gamma,n} + \# P_{I,n} }\vert 
{W(\frak g_\gamma,\frak h_{\psi_1})}\vert^{-1}\prod_{\alpha\in P_{\gamma_I}}
\langle\rho_{\gamma_I},\alpha\rangle^{-1} &(3.11)\cr
&\quad\cdot {\sum\limits_{s\in W(K\cap M^+_{Q_\kappa}) /W_\gamma}\varepsilon (s)e^{s
(\lambda_w +\rho_M) -\rho_M}(\gamma_I)\prod\limits_{\alpha\in P_\gamma}
\langle s(\lambda_w +\rho_M),\alpha\rangle \over \prod\limits_{\alpha\in 
P(\frak{m}_{Q_\kappa})\backslash P_\gamma}(1-e^{-\alpha}(\gamma_I))}\,,\cr}$$
where $v(^0\Gamma '_\gamma\backslash G'_\gamma )$ denotes the volume 
normalized according to [H-P], $W_\gamma$ is $W(G^0_\gamma, A^0_I)$, and $W^{\Bbb C}_{\gamma}$ in [H-P] is the Weyl group for $(\frak g_{\gamma,\Bbb C}, \frak h_{\{\psi_1\},\Bbb C})$, i.e. $ W(\frak g_\gamma,\frak h_{\{\psi_1\}})$.

\proclaim{Proposition 3.1}  For any $[\gamma]\in{\Cal E}_1(\Gamma )$, one has
$$\eqalign{\chi_\gamma (\Bbb V) &=\det (I-\Ad\gamma^{-1}\mid{\frak v})^{-1} \det
(I-\Ad\gamma^{-1}\mid{\frak u})^{-1}\sum_{w\in  W_\kappa}(-1 )^{\ell
(w)+1}e^{\lambda_w\left(C_{\kappa}(l_\gamma Y)\right)}\cr
&\cdot [G_\gamma:G_\gamma^0]^{-1}\vol 
(\Gamma_\gamma\backslash G_\gamma )\cdot\lambda_{\gamma^*}^{-1}(2\pi )^{-\# P_{\gamma}}
2^{-\nu_{\gamma}/2}v(K'_\gamma )v(A_{I})^{-1} \cr
&\cdot  \sum_{s\in W(K\cap M^+_{Q_\kappa})/W_\gamma}\varepsilon (s)
e^{s(\lambda_w+\rho_{M})- \rho_M}(\gamma_I){\prod\limits_{\alpha\in P_\gamma}\langle s(\lambda_w
+\rho_{M}), \alpha\rangle\over \prod\limits_{\alpha\in P(\frak{m}_{Q_\kappa})\backslash 
P_\gamma}(1-e^{-\alpha}(\gamma_I))}\,,\cr
\chi_\gamma (\Bbb V)&= \sum_{w\in  W_\kappa}(-1 )^{\ell
(w)+1}{\chi_{\gamma,w} ({\Bbb V})\over \det (I-\Ad\gamma^{-1}\mid{\frak v}) \det
(I-\Ad\gamma^{-1}\mid{\frak u})} .\cr}\tag 3.12$$
 \endproclaim

\demo{Proof}  Notice that the sign in (3.11), $(-1)^{\#P_{\gamma,n} + \#P_{I,n} }$, equals $(-1)^{q_{\gamma, \perp}} $ since $G_\gamma\subset M_{Q_\kappa}$ and so cancels the original sign in (3.3).   Then the rest follows from the successive identities (3.3) -- (3.11), after reconciling the 
different volume normalizations as in [M-S;I, p. 659] and [M-S;II, p. 203]. For convenience we summarize these. Using $v'_S$, $v_{H-P}$ and $\vol$ for the respective measures, and the notation of [H-C; I \S 37] $\nu_\gamma=\dim G_\gamma/K_\gamma$--rank $G_\gamma/K_\gamma$, $v(K'_\gamma ),v(A_{I})$ we have:
$$ \eqalign{\quad v_{H-P}(\cdot)  = &\, v'_S(\cdot)\, (2\pi)^{-\#P_\gamma} \vert W_\gamma^{\Bbb C}\vert \,\,\wt{\omega}_\gamma(\rho_\gamma) \,\quad \text{ from [M-S; I] (5.7),}\cr
\quad v'_S(\cdot) = &\, \vol(\cdot) \, 2^{-\nu_{\gamma}/2}v(K'_\gamma )v(A_{I})^{-1}\quad \,\text{ from [M-S; II] p 203}.\cr}$$
 \hfill\bls
\enddemo

\remark{Remark - sign} The computation in this section is modeled on those by [Sc] and [H-P] for an equirank group G, whereas we will use the result for the Levi subgroup, $M_{Q_\kappa}$, of a cuspidal maximal parabolic subgroup. To reflect the split rank of the parabolic and to be consistent with the analysis results to be derived from Proposition 1.7, in the topological Proposition 3.1 we have changed $\ell(w)$ to $\ell(w)+1$.
\endremark

\bigskip

%\vfil\eject

\n{\bf \S 4~~Orbital Integral Expansion for heat kernels}
\medskip

We begin the analysis necessary to develop the properties of the geometric zeta function (2.4). 
In 
this section we use the procedure pioneered by Selberg to study the 
trace of an integral operator on a locally symmetric space through its 
expansion into orbital integrals.  The evaluation of orbital integrals in 
explicit terms, especially in the higher rank case, is usually a formidable 
task.  In this instance, the underlying geometric nature of the integral 
operator introduces enough cancellations to lead to a tractable 
calculation. Indeed the main result of this section, Proposition 4.4, is a consequence of some remarkable symmetry properties of the Fourier series of this kernel.

Continuing from \S 2, $\Gamma$ is a discrete, torsion-free co-compact 
subgroup of $G$.  The quotient of $\wt X$ by the action of $\Gamma$, $X=
\Gamma\bs\wt X$, is a compact complex manifold with Hermitian metric, having 
$\wt X$ as the simply connected covering space. Recall also that $\tau:K\to U(V_{\lambda})$ is an irreducible, unitary representation
of $K$ ($\lambda$ the highest weight) on a complex vector space.  Associated 
to $(\tau,V_{\lambda})$ are holomorphic, Hermitian vector bundles 
$\wt{\Bbb V}\to \wt X$ and $\Bbb V\to X$.  In addition to $(\tau,V_{\lambda}
)$ we have taken $\varphi:\Gamma\to U (F)$ a unitary representation on a finite
dimensional complex vector space.  Associated to $(\varphi,F)$ is a 
Hermitian bundle $\Bbb F$ with a flat connection.  

The sections of $\Bbb F$ 
and the Kodaira Laplacian can be easily identified in this locally 
homogeneous setting.  The group $\Gamma$ acts on $F\otimes C^{\infty}(G)$ 
via $\varphi\otimes L$, $L$ the left regular representation.  The
$\Gamma$-invariants, $(F\otimes C^{\infty}(G))^{\Gamma}$, can be identified
with 
$$C^{\infty}(\Gamma\bs G;\varphi)=\{f:G\to F\vert
f~~\text{is}~~C^{\infty}~~\text{and}~~f(\gamma g)=\varphi(\gamma)f(g)\}.$$
A similar statement holds for the $L^2$ sections.  The group $G$ acts
unitarily on $L^2(\Gamma\bs G;\varphi)$ via the right regular representation
$R_{\Gamma,\varphi}$.  As in the first section, one has an operator between the 
spaces 
$C^{\infty}(\Gamma\bs G;\varphi)\otimes\Lambda^q
\frak p_+\otimes V_{\lambda}$ given by 
$$\ov{\partial}=\sum R_{\Gamma,\varphi}(\ov{Z}_i)\otimes e(Z_i)\otimes I$$
that commutes with the action $R_{\Gamma,\varphi}\otimes \Lambda^q_+\otimes\tau$ 
of $K$ to give the 
$\ov{\partial}$ complex (see, for example, [0-0]).  Indeed the $K$ invariant 
vectors, $[C^{\infty}(\Gamma\bs G;\varphi)\otimes \Lambda^q\frak p_+\otimes
V_{\lambda}]^K$, can be identified with smooth $(0,q)$ forms.  The
Cauchy-Riemann operator $\ov{\partial}$ is a closed, densely defined 
operator.  The formal Laplacian $\square^{0,q}$ is defined as usual and 
its restriction to smooth forms of type $(0,q)$ gives an elliptic operator 
with a unique self-adjoint extension to $L^2$ forms.

The heat operator $e^{-t\square^{0,q}}$ has Schwartz kernel $h^q_t(\cdot,
\cdot)$ obtained by the familiar procedure of
averaging over $\Gamma$ a heat kernel on $\wt X$.  For this purpose let $\wt{\square}^{0,q}$ be the
operator in \S 1 acting on $[C^{\infty}(G)\otimes
\Lambda^q\frak p_+\otimes V_{\lambda}]^K$.  The Schwartz kernel 
$\tilde h^q_t$ of the heat operator $e^{-t\wt{\square}^{0,q}}$ is known 
to be in
$$[\cl S(G)\otimes \END (\Lambda^q\frak p_+\otimes V_{\lambda})]^{K\times 
K}\,,$$
where $\cl S(G)=\bigcap\limits_{p>0}\cl C^p(G)$, $\cl C^p(G)$ the
Harish-Chandra space of $p$-integrable functions, is the nec plus ultra
Schwartz space.  Then the heat kernel $h^q_t(\cdot,\cdot)$ for
$e^{-t\square^{0,q}}$ is given, in this identification, by 
$$\lno{h^q_t(x,y)=\sum_{\gamma\in\Gamma}\varphi(\gamma)\otimes 
\tilde h^q_t(y^{-1}\gamma x)\,.&&(4.1)\cr}$$

We define the holomorphic torsion theta function by 
$$\lno{\theta_{\lambda,\varphi}(t)=\sum^n_{q=0}(-1)^qq\Tr
e^{-t\square^{0,q}}.&&(4.2)\cr}$$
In order to evaluate $\theta_{\lambda,\varphi}(t)$ using Selberg's technique, we 
set 
$$\lno{\tilde k^{\lambda}_t=\sum^n_{q=0}(-1)^qq\tr \tilde h^q_t\,,
&&(4.3)\cr}$$

\n $\tr$ denoting a local trace.  As $\tilde k^{\lambda}_t$ is in $\cl S
(G)$, it is admissible in the sense that one has an absolutely convergent 
orbital integral expansion 
$$\lno{\theta_{\lambda,\varphi}(t)=\sum_{\{\gamma\}}\Tr \varphi(\gamma)\text{vol}
(\Gamma_{\gamma}\bs G_{\gamma})\cl O_{\tilde k^{\lambda}_t}(\gamma)\,.
&&(4.4)\cr}$$
As $\Gamma$ is torsion free, there are no non-trivial elliptic conjugacy classes.  
Also, 
as $\Gamma$ is co-compact, all elements are semisimple.  For semisimple 
elements one knows from Harish-Chandra that $f\mapsto \cl O_f(\gamma)$ is 
an invariant, tempered distribution.  As such, the inversion of the 
invariant Fourier transform will involve only $\Tr \pi(f)$.

Now $\tilde h^q_t$ is in $[\cl S(G)\otimes \END (\Lambda^q\frak p_+\otimes
V_{\lambda})]^{K\times K}$, so if $(\pi,H_{\pi})$ is an irreducible
unitary representation of $G$, $\pi(\tilde h^q_t)$ is defined, trace-class, 
and in $[\END(H^{\infty}_{\pi})\otimes \END(\Lambda^q\frak p_+\otimes
V_{\lambda})]^{K\times K}$.  In fact one easily sees that it is in
$\END [(H^{\infty}_{\pi}\otimes \Lambda^q\frak p_+\otimes V_{\lambda})^K]$
where it agrees with $e^{-t\wt{\square}^{0,q}_{\pi}}$.  Then $\Tr
\pi(\tr\tilde h^q_t)=\Tr e^{-t\wt{\square}^{0,q}_{\pi}}$, in particular
$$\Tr\pi(\tilde k^{\lambda}_t)=\sum^n_{q=0}(-1)^qq\Tr\pi(\tr\tilde h^q_t)
=\sum^n_{q=0}(-1)^qq\Tr e^{-t\wt{\square}^{0,q}_{\pi}}.$$
We are now in a position to use the results of \S 1.

\proclaim{Proposition 4.1}   Let $Q=M_QA_QN_Q$ be a parabolic subgroup of 
$G$ and\hfil\break 
$\pi_{\xi,\nu}=\ind^G_{M^+_QA_QN_Q}(\xi\otimes e^{\nu}\otimes 1)$.

\exam{(i)}{If $\dim \frak a_{\frak q}\ge 2$ then $\Tr\pi_{\xi,\nu}(
\tilde k^{\lambda}_t)=0$.}

\exam{(ii)}{If $\dim \frak a_{\frak q}=1$ then }
$$\Tr\pi_{\xi,\nu}(\tilde k^{\lambda}_t)=e^{{t\over
2}[\pi_{\xi,\nu}(\Omega_G)-\langle\lambda^\ast ,\lambda^\ast +2\rho\rangle]}
\sum_{W_{\kappa}}(-1)^{\ell(w)+1}e(\xi^{(1
)},W_{\lambda^{(1)}_w},\frak p^{(1)}_-)e(\xi^{(2)},W_{\lambda^{(2)}_w},
\frak p^{(2)}_{\Bbb C}).\tag4.5$$

\exam{(iii)}{If $Q=G$ and $\pi_{\xi}\in\cl E_2(G)$ then 
$$\Tr \pi_{\xi}(\tilde k^{\lambda}_t)=e^{{t\over 2}(\pi_{\xi}(\Omega_G)\,-\,\langle
 \lambda^\ast,\lambda^\ast+2\rho\rangle} e'(\xi,V_{\lambda},\frak p_-).$$}
\endproclaim

\demo{Proof}  (i) and (ii) follow from Lemma 1.1, Lemma 1.3 and Proposition
1.7.  \hfill\bx
\enddemo

\remark{Remark}  Since the orbital integrals are tempered distributions we will consider cuspidal maximal parabolic subgroups. An examination of the simple $\frak g$ of Hermitian type shows 
that only $\frak s\frak p(n,\Bbb R)$ and $\frak s\frak o(2n+1,2)$, the non-simply laced 
Dynkin diagrams, have more than one class of cuspidal maximal
parabolic subgroup -- they have two.  In this case, Proposition 4.1 computes the invariant Fourier transform of all principal series 
induced from either cuspidal parabolic subgroup.  \endremark

Before actually evaluating the orbital integrals we must comment on the
normalization of Haar measures, etc.  We try to be consistent with that 
used by Harish-Chandra, as in say [H-C; I] and [H-C; S].  In [M-S; I] we 
gave a summary of these conventions so we shall omit repeating it here.  
Fix a $\theta$-stable parabolic subalgebra $\frak q$  with the split component $\frak a$. We set $ \frak m^1 = \frak q \cap \theta \frak q$. Then $\frak m^1 = Z_{\frak g} (\frak a)$ and let $\frak m$ be the orthogonal complement of $\frak a \subset \frak m^1$. Let $\frak a\subset \frak h\subset \frak q$ be a Cartan subalgebra. For any choice of compatible orders on $\Delta_{\frak m}$ and $\Delta (\frak g)$ let 
 $$\eqalign{\Delta_+ (h) &= \vert \text{Det}(I - Ad(h^{-1}) |_{\frak g/\frak m^1} \vert ^\frac{1}{2}\cr
 '\Delta_I(h) &= \prod\limits_{\alpha\in \Delta_{\frak m}^+} [1 - \xi_{-\alpha} (h)] = \prod\limits_{\alpha\in \Delta_{\frak m}^+} [1 - \xi_{\alpha} (h^{-1})].\cr}$$

From [H-C;InvEig p.493, p.484] one finds $$\vert \text{Det}(I - Ad(h^{-1}) |_{\frak g/\frak m^1} \vert  = \xi_{2\rho_{Q}}(h)\prod\limits_{\alpha\in \Delta_{\Bbb R}^+\cup \Delta_{\Bbb C}^+} [1 - \xi_{\alpha} (h^{-1})]^2.$$

Of course, the imaginary roots $\Delta_{\frak m}$ can come from roots of both non-compact factors of ${\frak m}$ as well as those of any compact factor.
 
 As is customary we shall work with the invariant integrals 
$'F^H_{\tilde k^{\lambda}_t}$, which are related to the orbital integrals 
by 
$$\lno{\cl O_{\tilde k^{\lambda}_t}(h)={'F^H_{\tilde
k^{\lambda}_t}(h)\over\Delta_+(h)'\Delta_I(h)},\quad h\in H'.&&(4.6)\cr}$$
The basic result we use to evaluate $\cl O_{\tilde k^{\lambda}_t}$ is due
to Harish-Chandra, although [Hb] could also be used.  To state this result
and because we shall use it in several proofs we recall the relevant
notation. 

Let $B\subseteq M^1$ be a Cartan subgroup, say $B=A_IA$, with $A_I$
fundamental in $M$.  As the dimension of $\frak a$ is one, 
either $B$ is connected or $B=A^0_IAF$, $F$ as before, and since $\frak q$ is cuspidal $A^0_I$ is a torus.  Let $B^{\ast}$ 
denote the set of unitary characters of $B$.  A character $a^{\ast}_0\in 
A^{0\ast}_I$ is called regular if $\prod\limits_{\alpha\in
\Delta^+_{\frak m}}(\log a^{\ast}_0,\alpha)\not= 0$, then to the $W(M^0,A^0_I)=W(M^+, A_I)$ orbit of 
$a^{\ast}_0$ is associated a discrete series representation of $M^0$.  For convenience we will denote by $a^{\ast}_0$ a character with $\log a^{\ast}_0$ dominant.
For $\chi$ a character of $F$, we set $a^{\ast}=(a^{\ast}_0,\chi)$ and we 
let $\xi(a^{\ast})$ denote the discrete series representation of $M^+$ 
with global character $\Theta_{\xi(a^\ast)}$.  As $\frak m$ need not be simple it is likely that $\xi(a^\ast)$ is the tensor product of a discrete series on $M^{(1)}$ with an irreducible finite dimensional on the compact factors of $M^+$. For any $\tilde a^{\ast}=(
\tilde a^{\ast}_0,\chi)$, $\tilde a^{\ast}_0\in W(M^+,A_I) 
a^{\ast}_0$, let $\Theta_{\tilde a^{\ast}}$ be the
invariant eigendistribution associated by Harish-Chandra to $\tilde a^{\ast}$. One has $\Theta_{\tilde a^{\ast}}=(-1)^l\Theta_{\xi(
a^{\ast})}$ for appropriate $l$.  There is yet another action of $W(G,B)$ on $B^{\ast}$  used in [H-C; S] (actually in [H-C; S] he uses the notation $W(G/B)$). Following [Va] p.454, for $s\in W(G,B)$ and $b^\ast \in B^{\ast}$ we denote the usual action by $s b^\ast$ and the new action by $s\cdot b^\ast$ which is defined by $s\cdot b^\ast = s\cdot (a^\ast,\nu) =(s\cdot a^\ast,s\nu)$ where $s\cdot a^\ast_0(a) = a^\ast_0(s^{-1}a)\xi_{s\rho_{\frak m}-\rho_{\frak m}}(a)$ for $s\in  W(M^+,A_I)$. Let $\pi_I$ denote the product of all positive imaginary roots of $(\frak g,\frak b)$, i.e. the positive roots of ($\frak m, \frak a_I$). For $s\in W(G,B)$ define $\varepsilon_I(s)\in \{\pm 1\}$ by $s\cdot \pi_I = \varepsilon_I(s) \pi_I$, in particular for $s\in W(M^+,A_I), \,\varepsilon_I(s) = \varepsilon(s).$ 
With this notation for $\tilde a^{\ast} = s  a^\ast$, and $q = {1\over 2} \text{ dim }M^+/K\cap M^+ = \text{ dim } \frak p^{(1)}_+$ we have $\Theta_{\tilde a^{\ast}} = \Theta_{sa^\ast} = (-1)^q \varepsilon_I(s)\Theta_{\xi(a^\ast)}$.
As before we let $\pi_{\xi,\nu}$ denote 
the induced representation of $G$ and $\Theta_{\pi_{\xi,\nu}}$ the 
character.  Following the notation in [H-C; S] we set $\hat f_B(b^{\ast})
=\hat f_B(a^{\ast},\nu)=\Theta_{\pi_{\xi,\nu}}(f)$ and extend $\hat f_B$ 
to $B^{\ast}$ as a $W(G,B)$ skew function, i.e. $\hat f_B(s\cdot b^\ast) = \varepsilon_I(s) \hat f_B(b^\ast)$. 

For $a^{\ast}_0\in A^{0\ast}_I$ singular, the situation is more complicated 
but analogous notation will be used.  Let $W(a^{\ast}_0)\subseteq
W(M^+,A_I)$ be the isotropy subgroup and let $\Theta_{a^{\ast}_0}$ be the
invariant eigendistribution associated by Harish-Chandra to $a^{\ast}_0$.  
One has ([Hb-S] Cor. 2.8)
$$\Theta_{a^{\ast}_0}=\vert
W(a^{\ast}_0)\vert^{-1}(-1)^q\sum_{W(a^{\ast}_0)}\varepsilon(w)
\Theta^w_{a^{\ast}_0}$$
where $(-1)^q\varepsilon(w)\Theta^w_{a^{\ast}_0}$ is the character of an irreducible unitary principal series
representation of $M^0$.  
As before we set $a^{\ast}=(a^{\ast}_0,\chi)$ and let 
($\xi^w(a^{\ast}), H(w,\xi)$) be the representation of $M^+$ with 
character $\Theta^w_{a^{\ast}}=(-1)^q\varepsilon(w)\Theta^w_{a^{\ast}_0}\otimes \chi$.  Then we 
set $\hat f_B(a^{\ast},\nu)=\vert 
W(a^{\ast}_0)\vert^{-1}\sum\Theta_{\pi_{\xi^w,\nu}}(f)$ and 
extend $\hat f_B$ to the $W(G,B)$ orbit of $b^{\ast}=(a^{\ast},\nu)$ by skew 
symmetry.  We shall let $\cl E_2(M^+)$ denote the set of $W(M^+,A_I)$
orbits of regular $a^{\ast}$ and $\cl E^s_2(M^+)$ the $W(M^+,A_I)$ 
orbits of singular $a^{\ast}$.

\proclaim{Theorem}  (Harish-Chandra)\quad Let $H$ be a Cartan subgroup,
$h\in H'$ and $f\in \cl C^2(G)$.  Assume  for every Cartan subgroup $A \not = H$ that $\text{if}\quad A\succ H$
$$\hat f_A\equiv 0.$$
Then 
$$\lno{'F^H_f(h)=\int_{H^{\ast}}[W(G,H)]^{-1}\sum_{s\in
W(G,H)}\varepsilon_I(s)\langle s\cdot h^{\ast},h\rangle\hat
f_H(h^{\ast})dh^{\ast}.&&(4.7)\cr}$$ \endproclaim

\remark {Remark~~}
In the case of two classes of cuspidal maximal parabolic subgroups one can show in this case that ${\frak w}(\frak a_1,\frak a_2) = \emptyset$, thus the Theorem can, in principle, be used to compute the orbital integrals for regular elements in either Cartan. However the Lemmas to follow will be proved only for the parabolic corresponding to $H_{\{\psi_1\}}$; the technical details needed for the other Levi are not clear to us at this time. So we assume henceforth not only
that $\frak g$ is simple but also that 
$G$ has only one class of cuspidal maximal parabolic subgroup, i.e. the Dynkin diagram 
is simply-laced. In these cases we have observed already that $M^{(2)}_Q$ is compact. Although using this 
observation from the classification would simplify the notation, we prefer to carry 
it along so that we can refer to the computations on a later occasion when we do the 
case of more than one cuspidal maximal parabolic.
\endremark

We will use the Theorem for $H=B = A_IA \, (i.e. \, H_{\{\psi_1\}}) $, $Q = Q_\kappa = Q_{\{\psi_1\}} $,  and $f = \tilde k^{\lambda}_t$. From (4.7) it is clear that we will need to sum certain Fourier series on $A_I^\ast$.  Since in this case $M^{(2)}$ is compact  ($B \sim H_{\{\psi_1\}} $ so $\frak p^{(2)} = \{0\}$), thus for each $w\in W_\kappa$ the representations $\xi^{(2)}$ and $\lambda^{(2)}_w$ are irreducible for  $K\cap M^{(2)}$. So if $e(\xi^{(1
)},W_{\lambda^{(1)}_w},\frak p^{(1)}_-)e(\xi^{(2)},W_{\lambda^{(2)}_w},\frak p^{(2)}_{\Bbb C})) \not = 0$ we have from $(1.13)'$ that $e(\xi^{(2)},W_{\lambda^{(2)}_w},\frak p^{(2)}_{\Bbb C})  = 1 $ and $\xi^{(2)} = {\lambda^{(2)}_w}^\ast$.

We will need Lemma 4.2 from [M-S; I] (coincidentally Lemma 4.2 herein) which we now recall. Let $\Cal W$ be any finite dimensional virtual representation of $K\cap M^+$ and suppose that $(\sigma^{(1)}_+ -\sigma^{(1)}_-)$ divides it. We denote this by $\Cal W\in (\sigma^{(1)}_+ -\sigma^{(1)}_-)$.

\proclaim{Lemma 4.2}   Let $\Cal W$ be a finite dimensional virtual representation of $K\cap M^+$ and suppose that  $\Cal W\in (\sigma^{(1)}_+ -\sigma^{(1)}_-)$.  Then on ${A_I}'$ 
$$\eqalign{& { '\ov{\Delta}_{M_c} \over '\Delta_{M^{(1)}_n}}      \chi_{\ov {\Cal W}}=\sum_{\xi\in
\cl E_2(M^+)}\text {dim}[H(\xi)\otimes \Cal W]^{K\cap M^+} \,'\ov{\Delta_I}\ov{\Theta_\xi}\cr
&\quad +
\sum_{\xi\in\cl E^S_2(M^+)}\vert 
W(\xi)\vert^{-1}\sum_{W(\xi)}dim[H(w,\xi)\otimes \Cal W]^{K\cap M^+}\,'\ov{\Delta}_{M^{(1)}}\,\ov{\Theta}_{ \xi}^w
\,.\cr}$$
\endproclaim

For convenience of the reader we recall two facts needed in the proof of the following proposition.

\proclaim{Lemma 4.3}  Let $B=A_IA$ be a split rank one 
Cartan subgroup and $Q$ the corresponding parabolic subgroup. Then if $e(\xi^{(1
)},W_{\lambda^{(1)}_w},\frak p^{(1)}_-)e(\xi^{(2)},W_{\lambda^{(2)}_w},\frak p^{(2)}_{\Bbb C})) \not = 0$ 
$$\eqalign{e(\xi^{(1)},W_{\lambda^{(1)}_w},\frak p^{(1)}_-)= &(-1)^q\dim[\xi^{(1)}\otimes(\sigma^{(1)}_+-\sigma^{(1)}_-)\otimes \xi_{\rho^{(1)}_n}\otimes
W_{\lambda^{(1)}_{w}}]^{K\cap M^{(1)}},\cr
e(\xi^{(2)},W_{\lambda^{(2)}_w},\frak p^{(2)}_{\Bbb C})  = &1, \text {  and  } \xi^{(2)} = ({\lambda^{(2)}_w})^\ast,\, i.e. \dim [{\xi^{(2)}}\otimes W_{\lambda^{(2)}_w}]^{K\cap M^{(2)}} = 1. \cr}$$
\endproclaim

\demo{Proof} These assertions are (1.13) and $(1.13)'$ from \S 1.
\enddemo

\proclaim{Proposition 4.4}  Let $B=A_IA$ be a split rank one 
Cartan subgroup and $h\in B'$.  Then 
$$\lno{\cl O_{\tilde k^{\lambda}_t}(h)={e^{-\Vert\log h_{\Bbb R}\Vert^2/2t}
\over (2\pi t)^{1/2}}{e^{2t\langle\lambda,\rho_n\rangle}\over
\Delta_+(h)'\Delta^{(1)}_{I,n}(h_I)}\sum_{W_{\kappa}}(-1)^{\ell(w)+1}e^{-{t
\over 2}\hat c^2_{\lambda_w}}{\chi}_{W_{\lambda_w}}(h_I),&&(4.8)\cr}$$
where 
$$\eqalign{\hat c^2_{\lambda_w}&=\{(\lambda_w\circ C_{\kappa} + \rho_Q)(\wh X_{\kappa})\}^2\cr
&=\{(\lambda_w\circ C_\kappa^{-1}-\rho_Q)(
\wh X_{\kappa})\}^2\cr}.$$

\endproclaim

\demo{Proof}   The holomorphic torsion $\theta$-function $\tilde
k^{\lambda}_t$ is in $\cl S(G)\subseteq\cl C^2(G)$.  Also $(
\widehat{\tilde k^{\lambda}_t})_{\tilde A}\equiv 0$ for ${\tilde A}\succ B$, ${\tilde A}\not= B$, as 
follows from Prop. 4.1.  Then $'F^B_{\tilde k^{\lambda}_t}(h)$ can be evaluated 
from (4.7).  

For $a^{\ast}=(a^{\ast}_0,\chi)$ regular  and $\log a^{\ast}_0$  dominant we have from (4.5)
$$\eqalign{(\widehat{\tilde k^{\lambda}_t})_B(b^{\ast})&=(\widehat{\tilde
k^{\lambda}_t})_B(a^{\ast},\nu)=\Tr\pi_{\xi,\nu}(\tilde k^{\lambda}_t)\cr
&=e^{{t\over
2}[\pi_{\xi,\nu}(\Omega)-\langle\lambda^\ast ,\lambda^\ast +2\rho\rangle]}
\sum_{W_{\kappa}}(-1)^{\ell(w)+1}e(\xi^{(1)},W_{\lambda^{(1)}_w},
\frak p^{(1)}_-)e(\xi^{(2)},W_{\lambda_w^{(2)}},\frak p^{(2)}_c).\cr}$$
  
Since 
$$\pi_{\xi,\nu}(\Omega)=\Vert\Lambda_{\xi}\Vert^2-\Vert \nu 
\Vert^2-\Vert\rho\Vert^2,$$
and setting 
$$\eqalign{\hat c_{\lambda_w}&=(\lambda_w\circ C_\kappa+\rho_Q)(
\wh X_{\kappa})\cr 
-\hat c_{\lambda_w}&=(\lambda_w\circ C_\kappa^{-1}-\rho_Q)(
\wh X_{\kappa})\cr} \tag4.9$$
we obtain from (1.16) and the previous Lemma ($\xi^{(2)} = (\lambda^{(2)}_w)^\ast$)
$$\eqalign{(\widehat{\tilde k^{\lambda}_t})_B(b^{\ast})=e&^{-{t\over
2}\Vert\nu\Vert^2}e^{2t\langle\lambda,\rho_n\rangle}\sum_{W_{\kappa}}(-1)^{\ell(w)+1}e
^{-{t\over2}\hat c^2_{\lambda_w}}\cr
&\dim[\xi^{(1)}\otimes (-1)^q(\sigma^{(1)}_+-\sigma^{(1)}_-)\otimes \xi_{\rho^{(1)}_n}\otimes
W_{\lambda^{(1)}_{w}}]^{K\cap M^{(1)}}\text{  dim}[{\xi^{(2)}}\otimes W_{\lambda^{(2)}_w}]^{K\cap M^{(2)}}.\cr}$$

\n  For $a^{\ast}$ singular one obtains 
$$(\widehat{\tilde k^{\lambda}_t})_B(b^{\ast}) =(\widehat{\tilde
k^{\lambda}_t})_B(a^{\ast},\nu) =\vert 
W(a^{\ast}_0)\vert^{-1}\sum\Tr\Theta_{\pi_{\xi^w,\nu}}(\tilde k^{\lambda}_t).$$
 
Now in [Hb-S] the representations $\xi^w$ are shown to come from irreducible unitary representations induced off parabolic subgroups of $M^0$ of positive split rank. Using double induction and Proposition 4.1 (i) we get $\Tr\Theta_{\pi_{\xi^w,\nu}}(\tilde k^{\lambda}_t) = 0.$ So it suffices to use regular $b^\ast$.

We recall some facts about $W(G,B)$. From [Wa] Prop. 1.4.2.1 we see that representatives for $W(G,B)$ may be taken from $K$. Then  $s\in W(G,B)$ preserves both $\frak a_I\subset \frak l_\kappa\subset \frak k$ and $\frak a$,  and since $\text{ dim } \frak a = 1$, $s\vert_{\frak a}=\pm I$.   If $s\vert_{\frak a} =I $ then $s\in K\cap M^+$, i.e. in $W(M^+,A_I)$. If $s\vert_{\frak a} =-I$  i.e. $s(X_\kappa) = -X_\kappa$, then, similar to the proof of [Sa] III.8.3 ( as $K$ is connected, $s (H_0)=H_0$),  since $[H_0,X_\kappa ] = -Y_\kappa$ so $s(Y_\kappa) = -Y_\kappa.$ Since $[X_\kappa, Y_\kappa] = H_\kappa$ we get $s(H_\kappa) = H_\kappa,$ i.e. $s\in L_\kappa$.

%from [Hi] p. 251 we see that such $s\vert_{\frak a_I}\in W(G,A_I)$ but also represented by $K$. Since $A_I\text { exp } {\Bbb R H_\kappa}$ is a maximal torus in $K$, then $s\in W(\frak k).$ 
%an examination of the proof of [Sa] III.8.3 ( as $K$ is connected, $s H_0=H_0$) shows that $sH_{\kappa}= -H_{\kappa}$. 

\n Any $\sigma$ in $W(M^+,A_I)$ satisfies $\sigma=I$ on $\frak a^{\ast}$, 
so skew-symmetric means
$$\eqalign{(\widehat{\tilde k^{\lambda}_t})_B((\sigma a^{\ast}_I,i\nu))&=\Theta_{\pi_{(\sigma a^\ast_I,\nu)}}(\widehat {(\tilde k^{\lambda}_t})) \cr
&=(-1)^q \varepsilon(\sigma)\Theta_{\pi_{(\xi (a^\ast_I),\nu})}(\widehat{(\tilde k^{\lambda}_t})) \cr
 &=(-1)^q \varepsilon (\sigma)(\widehat{\tilde k^{\lambda}_t})_B((a^{\ast}_I,i\nu))\,.
\cr}$$

We substitute the expression for $(\tilde k^{\lambda}_t)_B$ into (4.7), and carry out the calculation: 
$$\eqalign{'F^B_{\tilde k^{\lambda}_t}(h)=&\sum_{A^{\ast}_I}
\int_{\frak a^{\ast}} [W(G,B)]^{-1}\sum_{s\in W(G,B)}\varepsilon_I(s)\langle s\cdot b^{\ast},h\rangle
(\widehat{\tilde k^{\lambda}_t})_B(a^{\ast}_I, i\nu)d\nu\cr
= & \sum_{\cl E_2(M^+)}\sum_{W(M^+,A_I)}\int_{\frak a^{\ast}}[W(G,B)]^{-1}\sum_{s\in W(G,B)}
\varepsilon_I(s)\langle s\cdot (\sigma a^{\ast}_I, i\nu),h\rangle(\widehat{\tilde k^{\lambda}_t})_B((\sigma a^{\ast}_I,i\nu))\cr
= & \sum_{\cl E_2(M^+)}\sum_{W(M^+,A_I)}(-1)^q\varepsilon(\sigma)\int_{\frak a^{\ast}}[W(G,B)]^{-1}\sum_{s\in W(G,B)}
\varepsilon_I(s)\langle s\cdot (\sigma a^{\ast}_I, i\nu),h\rangle e^{-{t\over2}\Vert\nu\Vert^2}d\nu\cr
& \quad\quad e^{2t\langle\lambda,\rho_n\rangle}\sum_{W_{\kappa}}(-1)^{\ell(w)+1}e^{-{t\over2}\hat c^2_{\lambda_w}} \cr
&\quad  \dim[\xi^{(1)}\otimes (-1)^q(\sigma^{(1)}_+-\sigma^{(1)}_-)\otimes \xi_{\rho^{(1)}_n}\otimes W_{\lambda^{(1)}_{w}}]^{K\cap M^{(1)}}
\text{  dim}[{\xi^{(2)}}\otimes W_{\lambda^{(2)}_w}]^{K\cap M^{(2)}}\cr
=& \sum_{\cl E_2(M^+)}\sum_{W(M^+,A_I)}(-1)^q\varepsilon(\sigma)[W(G,B)]^{-1}\sum_{s\in W(G,B)}
\varepsilon_I(s)\langle s\cdot (\sigma a^{\ast}_I),h_I \rangle\cr
& \quad\quad \int_{\frak a^{\ast}} e^{is\nu(\log h_{\Bbb R})} e^{-{t\over2}\Vert\nu\Vert^2}d\nu \quad e^{2t\langle\lambda,\rho_n\rangle}\sum_{W_{\kappa}}(-1)^{\ell(w)+1}e^{-{t\over2}\hat c^2_{\lambda_w}} \cr
&\quad  \dim[\xi^{(1)}\otimes (-1)^q(\sigma^{(1)}_+-\sigma^{(1)}_-)\otimes \xi_{\rho^{(1)}_n}\otimes W_{\lambda^{(1)}_{w}}]^{K\cap M^{(1)}}
\text{  dim}[{\xi^{(2)}}\otimes W_{\lambda^{(2)}_w}]^{K\cap M^{(2)}}.\cr}$$

%$$\eqalign{F^B_{\tilde k^{\lambda}_t}(h)= &{e^{-\Vert\log 
%h_{\Bbb R}\Vert^2/2t}\over(2\pi t)^{1/2}} \sum_{\cl
%E_2(M^+)}\sum_{W(M^+,A_I)}\varepsilon(\sigma)[W(G,B)]^{-1}\sum_{W(G,B)}
%\varepsilon_I(s)\langle
%s\cdot \sigma\cdot\xi(a^{\ast}_I),h_I\rangle\cr
%&\quad\quad e^{2t\langle\lambda,\rho_n\rangle}\sum_{W_{\kappa}}(-1)^{\ell(w)+1}e^{-{t\over
%2}\hat c^2_{\lambda_w}}\cr
%&\dim[\xi^{(1)}\otimes (-1)^q(\sigma^{(1)}_+-\sigma^{(1)}_-)\otimes \xi_{\rho^{(1)}_n}\otimes
%W_{\lambda^{(1)}_{w}}]^{K\cap M^{(1)}}\text{  dim}[{\xi^{(2)}}^\ast\otimes W_{\lambda^{(2)}_w}]^{K\cap M^{(2)}}\cr}$$

\n We move the $s$ action onto $h_I$.  Then using [H-C I] Lemma 27.2 and the transformation properties of discrete series distributions the sum over $W(M^+,A_I)$ 
rearranges to 
$$\eqalign{&\sum_{W(M^+,A_I)}\varepsilon(\sigma)\langle \sigma
a^{\ast}_I,s^{-1}h_I\rangle\xi_{s\cdot\rho_M-\rho_M}(h_I)\cr
&= \xi_{s\cdot \rho_M-\rho_M}(h_I)'\Delta_I(s^{-1}h_I)\Theta_{\xi(a_I^\ast)}(s^{-1}h_I)
\,,\cr}$$
giving 
$$\eqalign{'F^B_{\tilde k^{\lambda}_t}(h)=&  \sum_{\cl E_2(M^+)}\ [W(G,B)]^{-1}\sum_{s\in W(G,B)}
\varepsilon_I(s)\xi_{s\cdot \rho_M-\rho_M}(h_I)'\Delta_I(s^{-1}h_I)\Theta_{\xi(a_I^\ast)}(s^{-1}h_I)\cr
& \quad\quad \int_{\frak a^{\ast}} e^{is\nu(\log h_{\Bbb R})} e^{-{t\over2}\Vert\nu\Vert^2}d\nu \quad e^{2t\langle\lambda,\rho_n\rangle}\sum_{W_{\kappa}}(-1)^{\ell(w)+1}e^{-{t\over2}\hat c^2_{\lambda_w}} \cr
&\quad  \dim[\xi^{(1)}\otimes (-1)^q(\sigma^{(1)}_+-\sigma^{(1)}_-)\otimes \xi_{\rho^{(1)}_n}\otimes W_{\lambda^{(1)}_{w}}]^{K\cap M^{(1)}}
\text{  dim}[{\xi^{(2)}}\otimes W_{\lambda^{(2)}_w}]^{K\cap M^{(2)}}.\cr}$$

Since $s\vert_{\frak a} =\pm I$ the Fourier transform of the Gaussian computes to ${e^{-\Vert\log 
h_{\Bbb R}\Vert^2/2t}\over(2\pi t)^{1/2}}.$ Then
$$\eqalign{'F^B_{\tilde k^{\lambda}_t}(h) = & {e^{-\Vert\log h_{\Bbb R}\Vert^2/2t}\over(2\pi t)^{1/2}} e^{2t\langle\lambda,\rho_n\rangle}\sum_{W_{\kappa}}(-1)^{\ell(w)+1}e^{-{t\over2}\hat c^2_{\lambda_w}} \cr
&[W(G,B)]^{-1}\sum_{s\in W(G,B)}
\varepsilon_I(s)\xi_{s\cdot \rho_M-\rho_M}(h_I)\sum_{\cl E_2(M^+)}\, '\Delta_I(s^{-1}h_I)\Theta_{\xi(a_I^\ast)}(s^{-1}h_I)\cr
\cr
&\quad  \dim[\xi^{(1)}\otimes (-1)^q(\sigma^{(1)}_+-\sigma^{(1)}_-)\otimes \xi_{\rho^{(1)}_n}\otimes W_{\lambda^{(1)}_{w}}]^{K\cap M^{(1)}}
\text{  dim}[{\xi^{(2)}}\otimes W_{\lambda^{(2)}_w}]^{K\cap M^{(2)}}.\cr}$$

 To finish the evaluation we use Lemma 4.2 with $\Cal W = (-1)^q(\sigma^{(1)}_+ -\sigma^{(1)}_-)\otimes \xi_{\rho^{(1)}_n}\otimes W_{\lambda^{(1)}_{w}}\otimes W_{\lambda^{(2)}_{w}}$
and the well known fact 

$$\ov{ '\Delta_{M^{(1)}_n}}(h) = \prod\limits_{\alpha\in \Delta_{\frak m^{(1)}_n}^+} [1 - \xi_{\alpha} (h)] = (-1)^q \{\Tr{\sigma^{(1)}_+}(h) -  \Tr{\sigma^{(1)}_-}(h) \} \xi_{\rho^{(1)}_n}(h).$$
 
 The summations over $\cl E_2(M^+)\cong \cl E_2(M^{(1)})\otimes( K\cap M^{(2)})\hat{} $ give the following expression
 $$\eqalign{'F^B_{\tilde k^{\lambda}_t}(h)&= { e^{-\Vert\log h_{\Bbb R}\Vert^2/2t}\over(2\pi t)^{1/2}}e^{2t\langle\lambda,\rho_n\rangle}  \sum_{W_{\kappa}}(-1)^{\ell(w)+1}e^{-{t\over
2}\hat c^2_{\lambda_w}} \cr
&{1\over[W(G,B)]}\sum_{W(G,B)}
\varepsilon_I(s)\xi_{s\cdot \rho_M-\rho_M}(h_I) {\chi}_{W_{\lambda_w}}(s^{-1}h_I)'\Delta_{M^{(1)}_{c}}(s^{-1} h_I)'\Delta_{K\cap M^{(2)}}(s^{-1}h_I).    \cr}$$

\n We write 
$$'\Delta_{M^{(1)}_{c}}\,'\Delta_{K\cap M^{(2)}}={'\Delta_{M}\over
'\Delta_{M^{(1)}_{n}}}={'\Delta_{M}\over\vert
'\Delta_{M^{(1)}_{n}}\vert^2}  \ov{\Delta}_{M^{(1)}_{n}}.$$
From the invariance of $\vert '\Delta_{M^{(1)}_{n}}\vert^2$ by $s$, 
together with Lemma 27.1 in [H-C; I] we get 
$$\eqalign{'F^B_{\tilde k^{\lambda}_t}(h)&={e^{-\Vert\log 
h_{\Bbb R}\Vert^2/2t}\over (2\pi t)^{1/2}}{'\Delta_{M}(h_I)\over\vert
'\Delta_{M^{(1)}_{n}}(h_I)\vert^2}e^{2t\langle\lambda,\rho_n\rangle}\sum_{W_{\kappa}
}
(-1)^{\ell(w)+1}
e^{-{t \over 2}\hat c^2_{\lambda_w}}\cr
&\quad {1\over\vert
W(G,B)\vert}\sum_{W(G,B)}{\chi}_{W_{\lambda_w}}(s^{-1}
h_I)'\ov{\Delta}_{M^{(1)}_{n}}(s^{-1} h_I).\cr}$$

\n If $s\vert_{\frak a} = I$ then $s\in W(M^+,A_I)=W(K\cap M^+,A_I)$,
so $\chi_{_{W_{\lambda_w}}}(s^{-1}
h_I)=\chi_{_{W_{\lambda_w}}}(h_I)$ because $\chi_{_{W_{\lambda_w}}}$
is a class function on $K\cap M^+$. Similarly $'\Delta_{M^{(1)}_{n}}$ as the class
function $(\chi_{\sigma^{(1)}_+}-\chi_{\sigma^{(1)}_-})\chi_{\rho^{(1)}_n}$ 
is invariant under $W(K\cap M^+, A_I)$.  
 
 If $s\vert_{\frak a}= -I$ then $s\in L_\kappa$  so again we have $\chi_{W_{\lambda_w}}(s^{-1}
h_I)'\ov{\Delta}_{M^{(1)}_{n}}(s^{-1} h_I)= \chi_{W_{\lambda_w}}(
h_I)'\ov{\Delta}_{M^{(1)}_{n}}( h_I).$

Thus for the invariant integral we have
$$\lno{\qquad\quad'F^B_{\tilde k^{\lambda}_t}(h)={e^{-\Vert\log h_{\Bbb R}\Vert^2/2t}
\over(2\pi t)^{1/2}}{'\Delta_{M}(h_I)\over '\Delta_{M^{(1)}_{n}}(h_I)}
e^{2t\langle\lambda,\rho_n\rangle}\sum_{W_\kappa}(-1)^{\ell(w)+1}e^{-{t\over 
2}\hat c^2_{\lambda_w}}{\chi}_{W_{\lambda_w}}
(h_I).&&(4.10)\cr}$$

We recall that in the case at hand $\frak l_\kappa = \frak k\cap\frak m \oplus \Bbb R H_\kappa$, then $W(\frak 
l_\kappa) = W(K\cap M^+,A_I)$. Also, using $'\Delta_{M}(h_I) =\, '\Delta_{M^{(1)}_{n}}(h_I) '\Delta_{K\cap M}(h_I)$, and $\Delta_{K\cap M} = \xi_{\rho (\frak l_\kappa)} \,'\Delta_{K\cap M}$,  we remark for future use that an equivalent form of the above identity is
$$\lno{\qquad\qquad'F^B_{\tilde k^{\lambda}_t}(h)={e^{-\Vert\log h_{\Bbb R}\Vert^2/2t}
\over(2\pi 
t)^{1/2}}e^{2t\langle\lambda,\rho_n\rangle}\sum_{W_\kappa}(-1)^{\ell(w)+1}e^{-{t\over 
2}\hat c^2_{\lambda_w}}\sum_{W(\frak 
l_\kappa)}\epsilon(s){\xi}_{s(\lambda_w+\rho(\frak 
l_\kappa)) - \rho(\frak l_\kappa)}(h_I).&&(4.11)\cr}$$
The formula for $\cl O_{\tilde k^{\lambda}_t}(h)$ then follows from (4.10)
and (4.6).\hfill\bx

Suppose $h\in B$ is singular but $h\not= e$.  There is a well-established
procedure from [H-C; DSII p. 34] to obtain $\cl O_{\tilde k^{\lambda}_t}(h)$ from (4.10) or (4.11).  If one
multiplies (4.11) by $\xi_{\rho,I}$ to obtain 
$F^B_{\tilde k^{\lambda}_t}$ and applies the differential operator 
$\wt{\omega}_h$ to $F^B_{\tilde k^{\lambda}_t}$, one 
gets 

$$\cl O_{\tilde k^{\lambda}_t}(h)={N^{-1}_hc^{-1}_0\over\xi_{\rho}(h)
\prod\limits_{\alpha\in P^c_h}[1-\xi_{-\alpha}(h)]}(\wt{\omega}_h
F^B_{\tilde k^{\lambda}_t})(h),$$ where $N_h=[G_h:G^0_h]$ and as determined in [H-C I] p.203
$$ c_0=(-1)^{q_h}(2\pi)^{q_h}2^{\nu_h/2}\prod_{\alpha\in P_{h_I,c}}
\langle\rho_{h_I,c},\alpha\rangle\vert
W(G^0_h,A^0_I)\vert$$
with $q_h=\# P_{h,n}$, $\nu_h=\dim G_h/K_h$--rank $G_h/K_h$. From [H-C I] p. 203 
$$\prod_{\alpha\in P_{h_I,c}}\langle\rho_{h_I,c},\alpha\rangle = (2\pi)^{\# P_{h,c}} v(A_{I })v(K'_\gamma)^{-1},$$
then substituting we have
$$c_h :=  c_0=(-1)^{q_h}(2\pi)^{\# P_{h ,n}+\# P_{h,c}}2^{\nu_h/2}  v(A_{I })v(K'_h)^{-1}\vert W(G^0_h,A^0_I)\vert .$$
An explicit value of $c_0$ was computed in [Lang] but involves the Euclidean volume of $G^0_h/K^0_h$ and different normalizations of measures.  Before doing the computation of the differentiation we make some observations. Now $\rho = \rho(M) + \rho_Q = \rho(M^{(1)}_{n}) + \rho(\frak {l_\kappa}) +\rho_Q$, while $\rho_Q = \rho(\frak u) +\rho(\frak v)$. Now $\frak u$ is a real root space while on $\frak v$ there is a complex structure so, for the order being used, the roots occur in complex conjugate pairs. Thus $\xi_{\rho_Q}(h) = \xi_{\rho_Q}(h_{\Bbb R})$. Also if $s\in W(\frak {l_\kappa)} $ then $s\xi_{\rho(M^{(1)}_{n})}=s\chi_{\rho(M^{(1)}_{n})}=\xi_{\rho(M^{(1)}_{n})}$. Since $B$ is $\Bbb R$--rank one and $\Gamma$ is torsion-free,
$\xi_{\alpha}(h)\not= 1$ for $\alpha\in\Delta_{\Bbb R}$ and
$\alpha\in\Delta_{\Bbb C}$, provided $h\not= e$. 
Then $P_h=\{\alpha>0\vert\xi_{\alpha}(h)=1\}\subseteq\Delta^+_{M}=
\Delta^+_I$, and so $\wt{\omega}_h$ is independent of the $h_{\Bbb R}$ 
variable.

\proclaim{Lemma 4.5}   Let $h\in B$, $h\not= e$.  Then 
$$\eqalign{\cl O_{\tilde k^{\lambda}_t}(h)&={e^{-\Vert\log
h_R\Vert^2/2t}\over (2\pi
t)^{1/2}}N_h^{-1}c_h^{-1}{e^{2t\langle\lambda,\rho_n\rangle}\over 
 \xi_{\rho_{Q}}(h)\prod\limits_{\alpha\in
P^c_h}[1-\xi_{-\alpha}(h)]}\sum_{W_{\kappa}}(-1)^{\ell(w)+1}e^{-{t\over
2}\hat c^2_{\lambda_w}}\cr
&\quad\vert
W(G^0_h,A^0_I)\vert \sum_{W(K\cap M^+, A_I)/W(G^0_h,A^0_I)}\epsilon(s)\wt{\omega}_{h}(s(\lambda_w + 
\rho_M)){\xi}_{s(\lambda_w+\rho_M) - \rho_M}(h_I).\cr}$$
\endproclaim

\demo{Proof} To obtain the result it suffices to evaluate $\wt{\omega}_{h}$ on $F^B_{\tilde k^{\lambda}_t}$. We start from (4.11) and multiply by $\xi_{\rho,I}$; use $\xi_{\rho,I}=\xi_{\rho_M}= \xi_{\rho_{M^{(1)}_{n}}} \cdot \xi_{\rho(\frak l_\kappa)}$ to simplify the summand;. then differentiate the sum over $W(\frak l_\kappa) $, observe, as in [Lang] p.115, the sum drops to one over the quotient of $W(\frak l_\kappa) ( = W(K\cap M^+, A_I))$  by $W(G^0_h, A^0_I)$;  use $\xi_\rho = \xi_{\rho_M} \cdot\xi_{\rho_Q}$ from the denominator to complete the summand.
\remark{Remark} The calculation of this orbital integral was facilitated by two facts. 
First, the invariant Fourier transform of  $\tilde k^{\lambda}_t$ factors into a 
product, one for the Levi and the other for the split part thus reducing the calculation to 
the evaluation of a Fourier series and an integral. Second, Theorem 1.12 makes 
possible 
a clean and useful closed form of the Fourier series. This should be compared to the geometric term in the zeta function. There the effect of the geodesic on the bundles factors into that of $\gamma_I$ which is in the corresponding Levi, and that of the geodesic flow.
\endremark

To obtain the orbital integral expansion of
$\theta_{\lambda,\varphi}(t)$ we have as our final task to evaluate 
the contribution of the identity element,
i.e. to determine $\tilde k^{\lambda}_t(e)$.  For this we use
Harish-Chandra's inversion formula and will use $d_{\pi_{\omega}}$ to denote (in his 
normalization) the formal degree 
of the discrete series representation $\pi_{\omega}$.   

We recall that the form of the 
Plancherel density (see e.g. [Hb]) for an $\Bbb R$-rank one Cartan 
subgroup is 
$$\lno{\qquad\qquad{d\mu\over d\nu}(a^{\ast}_I,\nu)=c_A\prod\limits_{\alpha\in \Delta^+_I
\cup\Delta^+_{\Bbb C}}\langle \log
a^{\ast}_I+i\nu,\alpha\rangle\prod\limits_{\alpha\in\Delta^+_{\Bbb R}}
\langle i\nu,\alpha\rangle\left[{\cosh \pi\langle
\nu,\alpha^{\vee} \rangle-a^{\ast}_I(\gamma_{\alpha})\over\sinh
\pi\langle\nu,\alpha^{\vee}\rangle}\right],&&(4.12)\cr}$$
where $ \alpha^{\vee}= {{2\alpha} \over {\langle \alpha, \alpha \rangle}}$ and $c_A$ will be made precise in the proof of Proposition 6.6.
With $a^{\ast}_I = (a^{\ast}_0, \chi)$, according as $\chi(f_{\alpha})=\pm 1$ the term in brackets is
$\tanh {\pi\over 2}\langle \nu, \alpha^{\vee}\rangle$ or $\coth {\pi\over
2}\langle\nu, \alpha^{\vee}\rangle$.  One knows that in the 
cases at hand there is a unique positive real root, say having root vector in $\frak 
p^{[2]}_+$ the image under $\kappa$ of the one in $\frak s \frak l (2,\Bbb R)$ (recall that $\frak u_{\Bbb C}\simeq \frak p^{(2)}_{\Bbb 
C}\oplus\frak a_{\frak q_{\kappa},\Bbb C}$ and here $\frak p^{(2)}_{\Bbb C}$ is zero).  
Denote 
it $\alpha_{\Bbb R}\in
\Delta^+_{\Bbb R}$ and let $H_{\alpha_{\Bbb R}}$ satisfy $\nu(H_{\alpha_{\Bbb R}})
=\langle \nu,\alpha_{\Bbb R}\rangle $. Then $X_\kappa = H_{\alpha_{\Bbb R}} 
/\langle\alpha_{\Bbb R} ,
\alpha_{\Bbb R}\rangle$ and $f_{\alpha_{\Bbb R}} =\exp \pi iX_\kappa$.

Since $\vert \Delta^+_{\Bbb R}\vert=1$, we may identify
$ \frak a_{\frak q}^{\ast}$ with $\Bbb R$ so that $\nu\in\frak a_{\frak q}^{\ast}$
corresponds to $ {\nu\over 2} \alpha_{\Bbb R}$, $\nu\in\Bbb R$. Clearly an imaginary root 
$\alpha$ 
has $\langle\nu,\alpha\rangle = 0$, while for the real root we have $\langle\alpha_{\Bbb 
R},\log
a^{\ast}_I+i\nu\rangle = \langle \alpha_{\Bbb R}, i\nu\rangle$. Thus the polynomial factor of 
the 
Plancherel density can be written as $\varpi \left( \log a^{\ast}_I+i\nu  \right)$. Since the 
complex 
roots occur in conjugate pairs, notice that this polynomial factor is $\nu$ times a polynomial 
in 
$\nu^2$. 

Also, for any $w\in W_\kappa$ in order that $e(\xi^{(1)},W_{\lambda^{(1)}_w},\frak 
p^{(1)}_-)e(\xi^{(2)},W_{\lambda^{(2)}_w}
\frak p^{(2)}_{\Bbb C})\ne 0,$  we see 
from Corollary 1.13 that this can occur only for a unique discrete series 
representation of $M^+_Q$, say having infinitesimal character $\Lambda_{\xi (w)}$ with parameter, say, $a^{\ast}_I = (a^{\ast}_0, \chi)$.  Thus for each $w\in W_\kappa$ there is at most one possibility for $\chi(f_{\alpha})$, i.e. appearance of  $\tanh {\pi\over 2}\langle \nu, \alpha^{\vee}\rangle$ or $\coth {\pi\over
2}\langle\nu, \alpha^{\vee}\rangle$. As 
this integer will occur frequently in formulas we will abbreviate it with 
the notation  
$$e(\xi,\lambda_w) := e(\xi^{(1)}_w,W_{\lambda^{(1)}_w},\frak p^{(1)}_-) 
e(\xi^{(2)}_w,W_{\lambda^{(2)}_w},\frak p^{(2)}_{\Bbb C}).$$ 
\proclaim{Lemma 4.6}
$$\eqalign{\tilde k^{\lambda}_t(e)&=\sum_{\omega\in\cl
E_2(G)}e^{{t\over 2}(\omega(\Omega_G)\,-\,\langle 
\lambda^\ast,\lambda^\ast+2\rho\rangle)}d_{\pi_{\omega}} 
e'(\omega,V_{\lambda},\frak p_-)  \cr
&+ e^{2t\langle\lambda,\rho_n\rangle}\sum_{W_{\kappa}} 
(-1)^{\ell(w)+1}
\int_{\frak
a^{\ast}}e(\xi,\lambda_w)e^{-{t\over 2}\hat c^2_{\lambda_w}}
e^{-{t\over 2}{\Vert \nu \Vert^2}} {d\mu\over d\nu}\left( w^{\ast}(\lambda^{\ast}
+\rho - 2\rho_n)\vert_{\frak t_{\kappa}}, \nu \right) d\nu. \cr} $$
\endproclaim
\demo{Proof}   The inversion formula gives 
$$\tilde k^{\lambda}_t(e)=\sum_{\cl E_2(G)}d_{\pi_{\omega}} \Theta_{\omega}
(\tilde k^{\lambda}_t)+\sum_{A^{\ast'}_I}
\int_{\frak a^{\ast}_{\frak q}}\Theta_{\pi(\xi,\nu)}(\tilde k^{\lambda}_t)d\mu_I(\xi,
\nu).$$
\n In (4.5) (iii) the discrete series terms are found.  Then from 
Proposition 4.1, Lemmas 4.2, 4.3 and Kostant's Theorem we get 
$$\lno{&\qquad\qquad\qquad\qquad\qquad \sum_{A^{\ast'}_I}
\int_{\frak a^{\ast}}\Theta_{\pi(\xi,\nu}(\tilde k^{\lambda}_t)d\mu_I(\xi,\nu) =
&\cr
&\sum_{A^{\ast'}_I}\int_{\frak a^{\ast}}e^{-{t\over
2}\Vert\nu{\Vert}^2}\sum_{W_{\kappa}}(-1)^{\ell(w)+1}e^{{t\over 
2}(\Vert\Lambda_{\xi}
\Vert^2-\Vert\lambda^{\ast}+\rho\Vert^2)}e(\xi,\lambda_w) 
 d\mu_I(\xi,\nu). \cr}$$
From Theorem 1.12 and (4.9), but omitting the evaluation of the 
integral, we obtain the above
$$\lno{& 
=e^{2t\langle\lambda,\rho_n\rangle}\sum_{W_{\kappa}} (-1)^{\ell(w)+1}
\sum_{A^{\ast'}_I}\int_{\frak
a^{\ast}}e(\xi,\lambda_w)e^{-{t\over 2}\hat c^2_{\lambda_w}}
e^{-{t\over 2}{\Vert \nu \Vert^2}} {d\mu \over d\nu}\left( w^{\ast}(\lambda^{\ast}
+\rho - 2\rho_n)\vert_{\frak t_{\kappa}}, \nu \right) d\nu.&\cr}$$
\hfill\bx\enddemo

\remark{Remark} Historically, in the groundbreaking paper [R-S; II],  Ray and Singer used the Selberg Trace Formula but they observed that duality in low dimensions reduces the computation to heat kernels on sections ([R-S; II] (4.3)) where there is no discrete series contribution to the identity term. Also, as they rather compared the log of the quotient of holomorphic torsion for two characters ([R-S; II] (4.7)  i.e. here a $\varphi_1,\varphi_2$) for which the identity contribution would cancel as is seen in (5.11). The discrete series contribution for the expression similar to $\tilde k^{\lambda}_t(e)$ appears to have been overlooked in [F2] although this has minor influence on his main result.  
\endremark

Collecting the various terms we get the orbital integral expansion of the
holomorphic torsion $\theta$-function.

\proclaim{Theorem 4.7}   Let $X$ be a compact, locally symmetric Hermitian
manifold and $\Bbb F$ a flat Hermitian bundle associated to the unitary
representation $(\varphi,F)$ of $\pi_1(X)$.  Let $(\tau,V_{\lambda})$ be an
irreducible, unitary representation of $K$ with $\lambda +\rho$ $G$-integral 
and regular and $\Bbb V$ the associated holomorphic Hermitian vector 
bundle.  Set 
$$\theta_{\lambda,\varphi}(t)=\sum^n_{q=0}(-1)^qq\Tr e^{-t\square^{0,q}}.$$

Then for $G$ with only one class of cuspidal parabolic subgroup
$$\eqalign{&\theta_{\lambda,\varphi}(t)=\sum_{[\gamma]\in\cl
E_1(\Gamma)}\Tr\varphi(\gamma)\vol(\Gamma_{\gamma}\bs
G_{\gamma}){e^{-\Vert\log\gamma_{\Bbb R}\Vert^2/2t}\over (2\pi
t)^{1/2}}~e^{2t\langle\lambda,\rho_n\rangle}{N^{-1}_{\gamma}c^{-1}_{\gamma} \over 
 \xi_{\rho_{Q}}(\gamma)\prod\limits_{\alpha\in P^c_{\gamma}}[1-\xi_{-\alpha}(\gamma)]} \vert W(G^0_\gamma,A^0_I)\vert \cr
&\times\sum_{W_{\kappa}}(-1)^{\ell(w)+1}e^{-{t\over 
2}\hat c^2_{\lambda_w}}\sum_{W(K\cap M^+, A_I)/W(G^0_h,A^0_I)}\epsilon(s)\wt{\omega}_{\gamma}(s(\lambda_w+\rho_M)){\xi}_{s(\lambda_w+\rho_M) - \rho_M}(\gamma_I)\cr  
&\quad\quad\quad+\text{ dim F vol X }\sum_{\omega\in\cl
E_2(G)}e^{{t\over 2}(\omega(\Omega_G)\,-\,\langle 
\lambda^\ast,\lambda^\ast+2\rho\rangle)}d_{\pi_{\omega}} 
e'(\omega,V_{\lambda},\frak p_-)  \cr
&+\text{ dim F vol X } e^{2t\langle\lambda,\rho_n\rangle}\sum_{W_{\kappa}} 
(-1)^{\ell(w)+1}
\sum_{A^{\ast'}_I}\int_{\frak
a^{\ast}}e(\xi,\lambda_w)e^{-{t\over 2}\hat c^2_{\lambda_w}}
e^{-{t\over 2}{\Vert \nu \Vert^2}} {d\mu\over d\nu}\left( w^{\ast}(\lambda^{\ast}
+\rho - 2\rho_n)\vert_{\frak t_{\kappa}}, \nu \right) d\nu. \cr} $$

\endproclaim

%\vfill\eject

\n {\bf \S 5~~Heat operators: Ansatz and Expansion}
\medskip

Two analytical technicalities present themselves as obstacles to the construction of 
geometric zeta functions in the generality of this 
paper. The first technicality is the commonplace one of using admissible 
functions for convergence of the orbital integral expansion. This we handle by using 
heat kernels. The second is a consequence of the scalars {$\hat c_{\lambda_w}^2$} which appear 
geometrically through the action of the geodesic flow on the bundles constructed in \S 2. Their 
effect is to prevent a uniform meromorphic 
continuation of the intermediate (Selberg-like) zeta functions. Hence the need to treat these 
functions individually. These complications were not present in 
[M-S;I] in which all the $\hat c_{\lambda_w}$ happened to be zero. The powerful global analysis approach of Bismut, [Bi;eta], gives an independent corroboration of the results in [M-S;I]. But in [M-S;II] not all $\hat c_{\lambda_w}$ are zero.
The first approach to handle them, to our 
knowledge, was introduced by Fried [F1], who used certain combinations of heat kernels on 
bundles to handle orbital integral expansions related to closed and co-closed forms on hyperbolic space. While 
aware that we could generalize this to the setting considered
in [M-S;II], instead we chose there to introduce an alternative approach. However, the 
presentation of it is flawed, so we shall use Fried's technique, in the process simultaneously 
justifying the results in [M-S;II] and herein. Recently this technique was used by Shen [Sh] to give a correct justification of the results in [M-S;II].

Proceeding with this section, we fix $Q$ a cuspidal maximal parabolic subgroup, 
say corresponding to the homomorphism $\kappa$ of $\frak s\frak l(2,\Bbb R)$, and, for 
convenience, set $K^M = K\cap M^+_Q$.  
We recall the decomposition (1.4) 
$\frak p_+=\frak p_+^{[0]}\oplus
\frak p_+^{[1]}\oplus \frak p_+^{[2]}$,  and set $\frak p_+^{[\ev]}=
\frak p_+^{[0]}\oplus \frak p_+^{[2]}$.  Then
as $K^M$-modules one has the equivalence 
$$\Lambda^q\frak p_+\simeq \sum \oplus\Lambda^j\frak p_+^{[ev]}\otimes 
\Lambda^{\ell}\frak p_+^{[1]}.$$

Set $e^+_\kappa=\kappa\pmatrix 0&1\\ 0&0\endpmatrix \in \frak p^{[2]}_+$ and 
$e^-_\kappa=\kappa\pmatrix 0&0\\ 1&0\endpmatrix
\in \frak p_-^{[2]}$.  We fix a $K^M$ complement to $\Bbb C
e^+_{\kappa}$ in $\frak p_+^{[\ev]}$
$$\frak p_+^{[\ev]}\simeq \tilde \frak p_+^{[\ev]}\oplus\Bbb C e^+_\kappa,$$
so that  
$$\Lambda^q\frak p_+\simeq \sum\oplus \Lambda^j(\tilde{\frak  p}^{[\ev]}_+
\oplus\Bbb Ce^+_\kappa)\otimes\Lambda^{\ell}\frak p_+^{[1]}.$$
As in Corollary 1.5 we have as $K^M$-modules
$$\sum(-1)^qq\Lambda^q\frak p_+\simeq \sum(-1)^{j+\ell +1}\Lambda^j
\tilde{\frak p}_+^{[\ev]}\otimes\Lambda^{\ell}\frak p_+^{[1]},$$

$$\sum^n_{q=0}(-1)^qq\Lambda^q\frak p_+\simeq\sum^{\dim
\frak n^c_{\kappa}}_{\ell=0}(-1)^{\ell+1}(\Lambda^{\ev}\frak p^{(1)}_+-
\Lambda^{\odd}\frak p^{(1)}_+)\otimes
(\Lambda^{\ev}\frak p^{(2)}_{\Bbb C}-\Lambda^{\odd}\frak p^{(2)}_{\Bbb C})
\otimes\Lambda^{\ell}\ov{\frak n}^c_{\kappa}\,,$$
and so as before for $\sum(-1)^qq\Lambda^q\frak p_+\otimes V_\lambda$ we get
$$\lno{\simeq 
\sum(-1)^{\ell(w)+1}(\Lambda^{\ev}\frak p^{(1)}_+-
\Lambda^{\odd}\frak p^{(1)}_+)\otimes
(\Lambda^{\ev}\frak p^{(2)}_{\Bbb C}-\Lambda^{\odd}\frak p^{(2)}_{\Bbb C})\otimes 
W_{\lambda_w}.&&(5.1)\cr}$$

In order to construct the {\it virtual heat kernels}, initially we appealed
to the following {\it Ansatz}: 

\remark{\it Ansatz} The restriction map $Res$ induces a surjection of Grothendieck 
groups
$$\lno{Res \colon K^0(K) \to K^0(K^M)\twoheadrightarrow 0\,.&&(5.2)\cr}$$
\endremark

In fact, this statement is too strong for our purposes and certainly false. A revised version suitable for our purposes would have
for each $j$ and $w$, for some integers and $K$-modules one has in $K^0(K^M)$,
$$\Lambda^j\tilde{\frak p}_+^{[\ev]}\otimes W_{\lambda_w}
= {\ Res}\sum_{\mu\in \wh{K}}m^{j,\lambda_w}_{\mu}V_{\mu}\,.
 $$
Even this is  stronger than needed as one can group all terms in the sum over {\it j}. However, there remains one more simplification, viz. as the scalars {$\hat c_{\lambda_w}^2$} are the source of the problem, one can group terms in the Weyl sum that have the same scalars {$\hat c_{\lambda_w}^2$}. The following is the form that is actually used, and the current status of the validity of this version of the {\it Ansatz} is contained in the Appendix due to Jan Frahm. 

$$\lno{
 (\sum_j(-1)^{j+1}\Lambda^j\tilde{\frak p}_+^{[\ev]})\otimes(\sum_{{w\in W_\kappa,\hat{c}_{\lambda_w}^2=c^2}}(-1)^{\ell(w)+1}W_{\lambda_w})\,={\ Res}\sum_{\mu\in \wh{K}}m^{\hat c_{\lambda_w}^2}_{\mu}V_{\mu}\,.
&&(5.3)\cr}$$
 
For $\gamma \in {\Cal E}_1(\Gamma)$ it follows from Lemma 1.4 that the corresponding $\tilde{\frak p}_+^{[\ev]} = \frak p^{(1)}_+$. This possible simplification will not be used.

Take the regular representation $(R,L^2(G))$ and for each irreducible $K$-module $V_{\mu}$ 
which 
occurs above we consider the operator
$ e^{{t\over 2}(\Omega_G -\langle\lambda^\ast ,\lambda^\ast +2\rho\rangle I)}$ 
acting on the Hilbert space 
$[L^2(G)\otimes V_{\mu}]^K \,.$ Recall that $\Vert\lambda^\ast +\rho\Vert^2 + 
4\langle\lambda,\rho_n\rangle = \Vert\lambda +\rho\Vert^2$. From Proposition 5.1 in [O-O] (or 
cf. 
Lemma 1.1) up to a scalar operator depending on $\mu$, $\Omega_G$ agrees on the Hilbert space 
with 
$-2\square^{0,0}=-\Delta^0$. Since $ e^{-{t\over 2} \langle\lambda^\ast,\lambda^\ast +2\rho\rangle I}$ is 
a 
scalar operator,  $ e^{{t\over 2}(\Omega_G -\langle\lambda^\ast,\lambda^\ast +2\rho\rangle I)}$ viewed as an 
operator on the space of sections of the associated
homogeneous bundle over the symmetric space $G/K$ is
$G$-equivariant and smoothing. As such, it must be of the form 
$R(\tilde h_{\mu,t})$. Moreover, since this operator differs 
by a scalar factor from the heat operator associated to the 
corresponding connection Laplacian, its invariant Schwartz kernel
$\tilde h_{\mu,t}$ belongs to 
$[\cl S(G)\otimes \END(V_{\mu})]^{K\times K}$.
We set $$\lno{f_{\mu,t} = \tr \tilde h_{\mu,t}\,.&&(5.4)\cr} $$

The effect of the integers $m^{\hat c_{\lambda_w}^2}_{\mu}$ is handled in the usual way. For positive 
integers 
we take a direct sum of bundles corresponding to $V_{\mu}$, while for negative integers we take 
a 
direct sum of bundles and use a $\Bbb Z_2$ grading to compute the trace. With this convention 
we then 
define the {\it virtual heat kernel} associated to $\hat c_{\lambda_w}^2$ as:
$$\lno{f^{\hat c_{\lambda_w}^2}_t:=\sum_{\mu}m^{\hat c_{\lambda_w}^2}_{\mu}f_
{\mu,t}\,.&&(5.5)\cr}$$
 
\remark {Remark~~}  Due to the use of the {\it Ansatz}, we shall systematically
attach  $\hat c_{\lambda_w}^2$ to the notation (often as superscript) 
 to objects for which we take the $\sum_{{w\in W_\kappa,\hat{c}_{\lambda_w}^2=c^2}}(-1)^{\ell(w)+1}$. For example, instead of
 $$ e(\xi,\lambda_w) := e(\xi^{(1)}_w,W_{\lambda^{(1)}_w},\frak p^{(1)}_-) e(\xi^{(2)}, W_{\lambda^{(2)}_w},\frak p_{\Bbb C}^{(2)})$$  we will write 
$$e(\xi,{\hat c_{\lambda_w}}^2) : = \sum\limits_{{w\in W_\kappa,\hat{c}_{\lambda_w}^2=c^2}}(-1)^{\ell(w)+1}e(\xi^{(1)},W_{\lambda^{(1)}_w},\frak p_-^{(1)})
e(\xi^{(2)}, W_{\lambda^{(2)}_w},\frak p_{\Bbb C}^{(2)}).$$
\endremark

To construct the zeta functions, in analogy with the above we consider the quasi-regular 
representation 
$(R_{\Gamma,\varphi},L^2(\Gamma\bs G;\varphi))$ and we define the 
operator, $H^{\hat c_{\lambda_w}^2}_{\varphi ,t}$, acting on 
$$\sum_{\mu}\oplus \vert m^{\hat c_{\lambda_w}^2}_{\mu}\vert [L^2(\Gamma\backslash G;\varphi)\otimes 
V_{\mu}]^K$$
as the operator with Schwartz kernel 
$$h^{\hat c_{\lambda_w}^2}_{\varphi ,t}(\dot x,\dot y ) = \sum_{\gamma\in \Gamma}\varphi (
\gamma)\otimes \tilde h^{\hat c_{\lambda_w}^2}_t(y^{-1}\gamma x),\qquad \dot x = \Gamma x,\dot 
y =\Gamma y\,,$$
where
$$\tilde h^{\hat c_{\lambda_w}^2}_t:=\sum_{\mu}\oplus \vert m^{\hat c_{\lambda_w}^2}_{\mu}\vert \tilde 
h_{\mu,t}.$$
Equivalently   
$$H^{\hat c_{\lambda_w}^2}_{\varphi ,t}  = R_{\Gamma ,\varphi}(\tilde h^{\hat c_{\lambda_w}^2}_t).$$  
Relative to the $\Bbb Z_2$ gradings one has a 
super-trace, denoted $\Tr_s$ and we set 
$$\theta^{\hat c_{\lambda_w}^2}(t)=\Tr_s (H^{\hat c_{\lambda_w}^2}_{\varphi ,t}).\tag5.7$$

The following lemma gathers the information available to us 
regarding the invariant Fourier transform of $f^{\hat c_{\lambda_w}^2}_t$, 
$$ \hat f^{\hat c_{\lambda_w}^2}_t(\pi)=\Tr\pi(f^{\hat c_{\lambda_w}^2}_t), \,  \pi\in\widehat G. $$

\proclaim{Lemma 5.1} (a) For any $\pi\in\widehat G$, one has
$$ \hat f_{\mu,t}(\pi)= 
e^{{t\over 2}(\chi_{\pi} (\Omega_G)  -\langle\lambda^\ast ,\lambda^\ast +2\rho\rangle)}
\dim [H_\pi \otimes V_\mu]^K,$$
and so
$$ \hat f_t^{\hat c_{\lambda_w}^2}(\pi)= 
e^{{t\over 2}(\chi_{\pi} (\Omega_G) -\langle\lambda^\ast ,\lambda^\ast +2\rho\rangle)}
 \sum_{\mu}m^{\hat c_{\lambda_w}^2}_{\mu}\dim [H_\pi \otimes V_\mu]^K.$$
(b) Let $A$ be the standard Cartan subgroup corresponding
to the fixed cuspidal maximal parabolic $Q$. If $\pi_{A,\xi,\nu}=
\Ind^G_{Q}(\xi\otimes e^{\nu}\otimes 1)$ with $e^{\nu}$ unitary then 
$$\eqalign{ 
& Tr \,\pi_{A,\xi,\nu}(f^{\hat c_{\lambda_w}^2}_t)=\hat f^{\hat c_{\lambda_w}^2}_t(\pi_{A,\xi,\nu})=e^{-{t\over 2}(\hat c^2_{\lambda_w} - 4\langle\lambda,\rho_n\rangle+\Vert\nu\Vert^2)}
e(\xi,{\hat c_{\lambda_w}}^2) .\cr}$$
(c) Let $B\not= A$ be a standard Cartan subgroup such that 
$B\succ A$ and let $Q_B$ denote the associated parabolic. If 
$\pi_{B,\xi,\nu}=\Ind^G_{Q_B}(\xi\otimes e^{\nu}\otimes 1)$ then
$$ \hat f^{\hat c_{\lambda_w}^2}_t(\pi_{B,\xi,\nu})=0.$$
\endproclaim

\demo{Proof} Statement $(a)$ follows as $\pi$ is in $\widehat G$. For $(b)$ the formula
follows from $(a),$ (1.16), (4.9) and Frobenius reciprocity.

To prove $(c)$, we may assume that $B\subseteq M_QA_{\Bbb R}$, 
say $B=H_IH_{\Bbb R}$ with $H_{\Bbb R}\supseteq A_{\Bbb R}$.  Let $Q_B=M_B
H_{\Bbb R}N_B$ with $M^+_B\subseteq M^+_Q$.  Take any $Y\in H_{\Bbb R}
\bs A_{\Bbb R}$ with $\Vert Y\Vert=1$.  Let $a^{\ast}=(\xi,\mu,\nu)\in 
B^{\ast}=H^{\ast}_I\times H^{\ast}_{\Bbb R}$.  By double induction 
$$\pi:=\Ind^G_{Q_B}(\xi\otimes e^{\mu}\otimes  e^{\nu}\otimes 1)=\Ind^G_Q(\Ind(
\xi\otimes e^{\mu})\otimes e^{\nu}\otimes 1).$$
Using this, and arguing as above, we get
$$\lno{\quad&\hat f^{\hat c_{\lambda_w}^2}_t(\pi) = \qquad e^{{t\over 2}(\chi_{\pi}(\Omega_G)-\langle 
\lambda^\ast
,\lambda^\ast 
+2\rho\rangle)} \cr
&\times  \sum_{{w\in W_\kappa,\hat{c}_{\lambda_w}^2=c^2}}(-1)^{\ell(w)+1}\sum (-1)^j\dim [\Ind(\xi\otimes e^{\mu})\otimes \Lambda^j
\tilde{\frak  p}^{[\ev]}_+\otimes W_{\lambda_w}]^{K^M} &\cr
&=e^{{t\over 2}(\chi_{\pi}(\Omega_G)-\langle \lambda^\ast ,\lambda^\ast +2\rho\rangle)}
\sum_{{w\in W_\kappa,\hat{c}_{\lambda_w}^2=c^2}}(-1)^{\ell(w)+1}\sum(-1)^j\dim [W_{\xi}\otimes \Lambda^j\tilde{\frak p}_+^{[\ev]}\otimes 
W_{\lambda_w}]^{K^{M_B}}\,. \cr}$$
As in the proof of Lemma 1.3, the $\frak p_+^{[\ev]}$ component of $\text{Y}$ 
will provide a non-trivial $K^{M_B}$ intertwining operator between 
$\Lambda^{\ev}\tilde{\frak  p}^{[\ev]}_+$ and $\Lambda^{\odd}\tilde 
{\frak p}^{[\ev]}_+$; hence $\hat f^{\hat c_{\lambda_w}^2}_t(\pi)=0$.\hfill\bls
\enddemo

\medskip
\proclaim{Lemma 5.2}    Assume that $G$ has only one class of cuspidal maximal 
parabolic subgroup. Let $h\in B$, $h\not= e$.  Then 
$$\eqalign{{\Cal O}_{f^{\hat c_{\lambda_w}^2}_t}(h)&={e^{-\Vert\log h_{\Bbb R}\Vert^2/2t}\over 
(2\pi t)^{1/2}} N^{-1}_hc^{-1}_h {e^{2t\langle\lambda,\rho_n\rangle}\over
 \xi_{\rho_{Q}}h)\prod\limits_{\alpha\in P^c_h}[1-\xi_{-\alpha}(h)]}e^{-{t\over 2}\hat c^2_{\lambda_w}}\sum_{{w\in W_\kappa,\hat{c}_{\lambda_w}^2=c^2}}(-1)^{\ell(w)+1} \cr
&\quad \vert W(G^0_h,A^0_I)\vert \sum_{W(K\cap M^+, A_I)/W(G^0_h,A^0_I)}\varepsilon(s)\wt{\omega}_h(s\cdot(\lambda_w + 
\rho_M)){\xi}_{s\cdot(\lambda_w + \rho_M})(h_I).\cr}$$
\endproclaim

\demo{Proof}   In light of Lemma 5.1 the invariant Fourier transform of 
$f^{\hat c_{\lambda_w}^2}_t$ 
is a 
product of factors as was that of $\tilde k^{\lambda}_t$. Since the Fourier coefficients in each of the terms in $\sum_{{w\in W_\kappa,\hat{c}_{\lambda_w}^2=c^2}}$ is of the form in Lemma 4.3, the proof of the result then 
follows the same 
lines as that used for Proposition 4.4 and Lemma 4.5. Hence we omit the 
details.\hfill\bx
\enddemo

The contribution of the identity also is obtained as in \S 4.  However first 
we introduce notation to express it more compactly.  Recall that for any irreducible 
representation $\pi$ of $G$, we have introduced an associated Euler number 
$$e(\pi,V,\frak p_-)=\sum(-1)^{\ell}
\dim[H^{\infty}_{\pi}\otimes\Lambda^{\ell}\frak p_+\otimes V]^K. $$ 
With $M_Q$ and ${\hat c_{\lambda_w}}^2$ fixed, and with (5.1), (5.3) in mind, we set for $\pi\in \hat
 G$
$$\lno{ e^{\hat c_{\lambda_w}^2}_{M_Q}(\pi,V_{\lambda}):= \sum_{\mu}
 m^{\hat c_{\lambda_w}^2}_{\mu}\dim [H_\pi \otimes V_\mu]^K \,.&&(5.9)\cr}$$
\remark{Ansatz Dependence} The expression in (5.9) is possibly dependent on the {\it Ansatz}. From Lemma 5.1 it is clear that this will be relevant only for $\pi$ a discrete series representation.
\endremark

\remark{Remark}For each $w\in W_\kappa$, $W_{\lambda^{(2)}_w}$ is an irreducible module for 
$\frak m^{(2)}_{\frak q,\Bbb C}\oplus\frak a_{\frak q,\Bbb C}$.  Since $F_\kappa$ is 
generated by $f_\kappa =\exp\pi iX_\kappa$, 
$\lambda^{(2)}_w(f_\kappa)$ acts by either $\pm I$ on $W_{\lambda^{(2)}_w}$.  
Since $f_\kappa$ 
acts trivially on $\frak p^{(2)}_{\Bbb C}$, for $e(\xi^{(1)},W_{\lambda^{(1)}_w}
,\frak p^{(1)}_+)e(\xi^{(2)},W_{\lambda^{(2)}_w},\frak p^{(2)}_{\Bbb C})\not= 
0$ one must have $a^\ast_I(\gamma_\kappa )=\lambda^{(2)}_w(f_\kappa )$.  
As remarked earlier, for each $w\in W_\kappa$ there is at most a unique discrete series 
representation of $M^+_Q$, say having infinitesimal character $\Lambda_{\xi (w)}$ with parameter, say, $a^{\ast}_I = (a^{\ast}_0, \chi)$ with $e(\xi^{(1)},W_{\lambda^{(1)}_w}
,\frak p^{(1)}_+)e(\xi^{(2)},W_{\lambda^{(2)}_w},\frak p^{(2)}_{\Bbb C})\not= 
0$.  Thus for each $w\in W_\kappa$  there is also at most one possibility for $\chi(f_{\alpha})$, i.e. appearance of  $\tanh {\pi\over 2}\langle \nu, \alpha^{\vee}\rangle$ or $\coth {\pi\over
2}\langle\nu, \alpha^{\vee}\rangle$. Unfortunately, it does happen that there are different $w$ with the same $\hat c_{\lambda_w}^2$ but different $\chi(f_{\alpha})$.
\endremark
\proclaim{Lemma 5.3}   Assume that $G$ has only one class of cuspidal maximal 
parabolic subgroups.  Then 
$$\lno{&\qquad\qquad\qquad f^{\hat c_{\lambda_w}^2}_t(e)=\sum_{\pi_{\omega}\in\cl E_2(G)}
e^{{t\over 2}(\pi_{\omega} (\Omega_G) -\langle\lambda^\ast ,\lambda^\ast +2\rho\rangle)}
d_{\pi_{\omega}} e^{\hat c_{\lambda_w}^2}_{M_Q}(\pi_\omega ,V_{\lambda})\,\, + &(5.10)\cr
&e^{2t\langle\lambda,\rho_n\rangle}\sum_{{w\in W_\kappa,\hat{c}_{\lambda_w}^2=c^2}}(-1)^{\ell(w)+1}\int_{\frak a^{\ast}_{\frak q}} e^{-{t\over 2}(\hat c^2_{\lambda_w} +\Vert\nu\Vert^2)}
 \sum_{A^{\ast'}_I}e(\xi,\lambda_w)
{d\mu \over d\nu}\left(w^{\ast}(\lambda^{\ast}+\rho - 2\rho_n)\vert_{\frak t_\kappa} + 
i\nu 
\right)\,d\nu\,.\cr}$$
\endproclaim

\demo{Proof}   As in the proof of Lemma 4.6,
$$f^{\hat c_{\lambda_w}^2}_t(e)=\sum_{\cl 
E_2(G)}d_{\pi_{\omega}}\Theta_{\omega}(f^{\hat c_{\lambda_w}^2}_t)+
\sum_{A^{\ast}_I}\int_{\frak a^{\ast}_{\frak 
q}}\Theta_{\xi,\nu}(f^{\hat c_{\lambda_w}^2}_t)d\mu(\xi,\nu)
\,.$$
Thus, the statement follows from Lemma 5.1 and the notation before Lemma 4.6. \hfill\bx
\enddemo 

As we have made the ongoing assumption that $\Gamma$ is torsion free, the only elliptic 
conjugacy 
class is that of the identity element. The following orbital integral expansion of 
$\theta^{\hat c_{\lambda_w}^2}(t)$ should be considered as the sum of two terms - the elliptic 
contribution 
and the semisimple contribution.

\proclaim{Lemma 5.4}  
$$\lno{\qquad\qquad &\theta^{\hat c_{\lambda_w}^2}(t)=\text{dim F vol X 
}f^{\hat c_{\lambda_w}^2}_t(e)\,+e^{2t\langle\lambda,\rho_n\rangle}e^{-{t\over 2}\hat c^2_{\lambda_w}}&(5.11)\cr
& \cdot \sum_{[\gamma]\in 
{\Cal E_1(\Gamma)}}\Tr\varphi(\gamma)\vol (\Gamma_\gamma \backslash G_\gamma ) {e^{-\Vert\log 
\gamma_{\Bbb R}\Vert^2/2t}\over (2\pi
t)^{1/2}}{N_{\gamma}^{-1}c^{-1}_{\gamma} \over \xi_{\rho_{Q}}(\gamma) \prod\limits_{\alpha\in
P^c_{\gamma}}[1-\xi_{-\alpha}(\gamma)]}\vert W(G^0_h,A^0_I)\vert \cr
&\cdot \sum_{{w\in W_\kappa,\hat{c}_{\lambda_w}^2= \,c^2}}(-1)^{\ell(w)+1}\sum_{W(K\cap M^+, A_I)/W(G^0_h,A^0_I)}\varepsilon(s)\wt{\omega}_{\gamma}(s(\lambda_w + 
\rho_M))\xi_{s(\lambda_w + \rho_M) - \rho_M}(\gamma_I).\cr}$$
\endproclaim

The intermediate zeta functions will require a finite part Laplace transform of 
$\theta^{\hat c_{\lambda_w}^2}(t)$. The behaviour of $\theta^{\hat c_{\lambda_w}^2}(t)$ for small (resp. large) time 
will 
be needed. We show first that the behavior of $\theta^{\hat c_{\lambda_w}^2}(t)$ as $t\to 0^+$ is 
determined 
by 
the elliptic contribution.

\proclaim{Lemma 5.5}   Let $C_{\Gamma}=\bigcup\limits_{x\in G}x^{-1}(\Gamma\bs
\{e\})x$.  Then there exists a neighborhood of $e$, $V\subseteq G$, such that 
$V\cap KC_{\Gamma}K=\emptyset$.
\endproclaim

\demo{Proof}   First let us notice that $KC_{\Gamma}K=C_{\Gamma}K$.  We now 
argue by contradiction.  Suppose there exists a sequence $\{x^{-1}_n\gamma_n
x_nk_n\}\to e$.  As $K$ is compact we may assume that $\{k_n\}\to k\in K$ and 
thus $\{x^{-1}_n\gamma_nx_n\}\to k^{-1}$.  

Since $\Gamma$ is co-compact, $G=\Gamma Q$ with $Q$ compact.  Then there are 
$q_n\in Q$ and $\delta_n\in \Gamma$ with $x_n=\delta_nq_n$ and $\{q^{-1}_n
\delta^{-1}_n\gamma_n\delta_nq_n\}\to k^{-1}$.  But again since $Q$ is compact 
we may assume that $\{q_n\}\to q\in Q$.  Hence 
$$\{\delta^{-1}_n\gamma_n\delta_n\}\to qk^{-1}q^{-1}\in qKq^{-1}.$$
But $\delta^{-1}_n\gamma_n\delta_n$ is in $\Gamma$, which is discrete, hence
$\{\delta^{-1}_n\gamma_n\delta_n\}$ is eventually constant.  Then for $n>>0$,
$\delta^{-1}_n\gamma_n\delta_n\in \Gamma\cap qKq^{-1}$.  However $\Gamma$ is 
also assumed torsion-free, hence $\Gamma\cap qKq^{-1}=\{e\}$.  Then 
$\delta^{-1}_n\gamma_n\delta_n=e$ or $\gamma_n=e$, a contradiction.\hfill\bx

\proclaim{Lemma 5.6}   There is a $c>0$ so that 
$$\sum_{\gamma\not= e}\vert \Tr\varphi(\gamma)\tr\tilde h^{\hat c_{\lambda_w}^2}_t(x^{-1}
\gamma x)\vert=O(e^{-c/t}),\quad 0<t<T.$$
\endproclaim

\demo{Proof}   Since $\varphi$ is unitary we may ignore the $\Tr\varphi
(\gamma)$. For small time, we want to use the uniform 
estimates in [D], valid for scalar heat kernels on manifolds admitting a 
proper discontinuous group of isometries with compact quotient. To this
end, we take on $G$ a left-invariant metric $d(\cdot,\cdot)$ and let 
$p_t(\cdot)$ be the heat kernel for the associated Laplacian $\Delta$ on 
scalars.  According to [D], the estimate for $0<t<T$ is 
$$\lno{p_t(x)\le ct^{-n/2}\exp\left(-{d^2(x,e)\over 4t}\right).&&\cr}$$

Since $\tilde h_{\mu ,t}(x)$ is obtained, as in [B-M, (2.6)], by integrating
$p_t(\cdot x \cdot)$ along $K \times K$ against certain vector valued 
functions, one has for $0<t<T$
$$\lno{\Vert\tilde h_{\mu ,t}(x)\Vert\le Ct^{-n/2}\exp\left(-{d(KxK,e)^2
\over 4t}\right)&&\cr}$$
where $d(KxK,e)$ denotes the distance of the compact $KxK$ to 
the identity element.

According to Lemma 5.5 there is a $c>0$ such that 
$$d(K x^{-1}\gamma xK,e)>c.$$
Then 
$$\sum_{\gamma\not= e}\Vert\tilde h_{\mu ,t}(x^{-1}\gamma x)\Vert\le {Ce^{-c^2
/4t}\over t^{n/2}}\sum_{\gamma\not= e}e^{-{d(Kx^{-1}\gamma xK,e)^2-c^2\over 4
t}}$$
and the latter sum is bounded independently of $x$.  The Lemma follows easily 
from this and the finite sum
$$\tilde h^{\hat c_{\lambda_w}^2}_t:=\sum_{\mu}\oplus \vert m^{\hat c_{\lambda_w}^2}_{\mu}\vert \,\tilde h_{\mu,t}.$$
\hfill\bx
\enddemo

\proclaim{Corollary 5.7} $$\theta^{\hat c_{\lambda_w}^2}(t)\,\sim \text{dim F vol X 
}f^{\hat c_{\lambda_w}^2}_t(e),~~~t\to 
0^+.$$
\endproclaim

\demo{Proof} 
$$\align &\vert \theta^{\hat c_{\lambda_w}^2}(t) - \text{dim F vol X }f^{\hat c_{\lambda_w}^2}_t(e)\vert\\
&= \vert \sum_{\gamma\in A^G\bs\{e\}} Tr \varphi (\gamma){\Cal O}_{f^{\hat c_{\lambda_w}^2}_t}(\gamma 
)vol(\Gamma_\gamma\bs G_\gamma)\vert\\
&= \vert \sum_{\gamma\in A^G\bs\{e\}} Tr \varphi (\gamma){\Cal O}_{f^{\hat c_{\lambda_w}^2}_t}(\gamma 
)vol(\Gamma_\gamma\bs G_\gamma)\vert\\
&\leq \sum_{\gamma\in A^G\bs\{e\}} \vert Tr \varphi (\gamma) \vert\underset \Gamma\bs G 
\to\int \vert f^{\hat c_{\lambda_w}^2}_t(x^{-1}\gamma x)\vert d\dot x\\
&\leq\sum_{\gamma\in A^G\bs\{e\}}  \underset \Gamma\bs G \to\int \vert 
f^{\hat c_{\lambda_w}^2}_t(x^{-1}\gamma 
x)\vert d\dot x\\
&\leq {Ce^{-c^2
/4t}\over t^{n/2}}\sum_{\gamma\not= e}\underset \Gamma\bs G \to\int 
e^{-{d(Kx^{-1}\gamma xK,e)^2-c^2\over 4t}}\\
&\leq Ce^{-c^2/4t}
\endalign$$
where we have used $\varphi$ is unitary, $\Gamma\bs G$ is compact, and the 
series converges.  \hfill\bx
\enddemo

%\vfill\eject

\n {\bf \S 6~~Analysis of the elliptic contribution}
\medskip

\medskip

It is necessary to estimate $f^{\hat c_{\lambda_w}^2}_t(e)$.  The next several technical 
Lemmas 
contain these estimates as well as the important evaluation of the $Pf$ Laplace 
transform of $f^{\hat c_{\lambda_w}^2}_t(e)$.  This latter result is the key to the functional 
equation of the intermediate zeta functions. For clarity of exposition we present the Lemmas 
for a 
real variable. Later we make the transition to express their results in terms of Lie algebra 
variables. For background on pseudofunctions and Pf Laplace transforms we suggest [La].

 \proclaim{Lemma 6.1}  Let $k$ be a non-negative integer.
$$\eqalign{&(i)\qquad \int^{\infty}_0e^{-t\nu^2}\nu^{2k}\nu\tanh {\pi\over 2}
\nu\,d\nu\sim {c\over t^{k+3/2}},~~~t\to\infty;\cr
&(ii)\qquad\int^{\infty}_0e^{-t\nu^2}\nu^{2k}\nu\tanh{\pi\over 2}\nu\,d\nu\sim
\sum^{\infty}_{j=0}{a^k_j\over t^{k+1-j}},~~~t\to 0^+.\cr}$$
\endproclaim

\demo{Proof}   For (i) observe that $\nu\tanh {\pi\over 2}\nu$ is analytic at 
$\nu=0,$ with $\nu\tanh {\pi\over 2}\nu\sim{\pi\over 2}\nu^2$.  Then (i) 
follows immediately from Watson's Lemma.
\enddemo

For (ii) we consider first the case $k=0$.  Set 
$$I_0=\int^{\infty}_0e^{-t\nu^2}\nu\tanh {\pi\over 2}\nu\,d\nu.$$
Then 
$$\leqalignno{I_0&=-{1\over 2t}\int^{\infty}_0{d\over d\nu}e^{-t\nu^2}\tanh
{\pi\over 2}\nu\,d\nu\cr
&={\pi\over 4t}\int^{\infty}_0e^{-t\nu^2}\sech^2{\pi\over 2}\nu\,d\nu\quad 
\text{(Integration~by~parts)}\cr
&={1\over t}\int^{\infty}_0{e^{-x^2/4t}\over (4\pi t)^{1/2}}{x\over \sinh x}
dx\quad \text{(Parseval's~Identity)} &(6.1)\cr}$$
We may now take the Taylor series of $x/\sinh x$ (in the variable $x^2)$ and 
use Watson's Lemma to get an asymptotic expansion as ${1\over t}\to +\infty$
$$I_0\sim\sum^{\infty}_{j=0}{a_j\over t^{1-j}},\quad t\to 0^+.$$

Next we consider the case of a monomial $\nu^{2k}$, i.e. 
$$I_k=\int^{\infty}_0e^{-t\nu^2}\nu^{2k}\nu\tanh {\pi\over 2}\nu\,d\nu.$$
Then 
$$\lno{I_k&=\left(-{d\over dt}\right)^k\int^{\infty}_0e^{-t\nu^2}\nu\tanh {\pi
\over 2}\nu\,d\nu\cr
&=\left(-{d\over dt}\right)^k{\pi\over 4t}\int^{\infty}_0{e^{-x^2/4t}\over (4
\pi t)^{1/2}}{x\over
\sinh x}\,dx&(6.2)\cr
&=\left(-{d\over dt}\right)^k\pi\int^{\infty}_0-{d\over dx}{e^{-x^2/4t}\over 
(4\pi t)^{1/2}}{dx\over\sinh x}\cr
&=\pi\int^{\infty}_0\left(-{d^2\over dx^2}\right)^k\left(-{d\over dx}{e^{-x^2
/4t}\over (4\pi t)^{1/2}}\right){dx\over \sinh x},\cr}$$

\n where we have used the heat equation on $\Bbb R^1$.  It is straightforward 
to show 
$${d^{2k+1}\over dx^{2k+1}}\left({e^{-x^2/4t}\over (4\pi
t)^{1/2}}\right)=\sum^k_{j=0}b_j{x^{2j+1}\over t^{k+j+1}}{e^{-x^2/4t}\over 
(4\pi t)^{1/2}}.$$
Then 
$$I_k=\pi\sum^k_{j=0}{b_j\over t^{k+j+1}}\int^{\infty}_0{e^{-x^2/4t}\over (4\pi
t)^{1/2}}x^{2j}{x\over \sinh x}\,dx.$$
As before for $I_0,$ one uses the Taylor expansion in $x$ and Watson's Lemma 
to get 
$${1\over t}\int^{\infty}_0{e^{-x^2/4t}\over (4\pi t)^{1/2}}x^{2j}{x\over 
\sinh x}dx\sim \sum^{\infty}_{\ell=0}{c^j_{\ell}\over t^{1-j-\ell}}.$$
Hence after grouping terms we obtain 
$$I_k\sim \sum^{\infty}_{\ell=0}{a^k_{\ell}\over t^{k+1-\ell}}.$$
(Notice that this can be seen formally by substituting the asymptotic expansion 
for $I_0$ into $I_k$ and differentiating in $t$.)\hfill\bx
\medskip

\proclaim{Lemma 6.2}   Let $k$ be a non-negative integer.
$$\eqalign{&(i)\qquad \int^{\infty}_0e^{-t\nu^2}\nu^{2k}\nu\coth {\pi\over 2}
\nu\,d\nu\sim {c\over t^{k+{1\over 2}}},~~~t\to +\infty;\cr
&(ii)\qquad  \int^{\infty}_0e^{-t\nu^2}\nu^{2k}\nu\coth {\pi\over
2}\nu\,d\nu\sim\sum^{\infty}_{j=0}{b^k_j\over t^{k+1-j}},~~~t\to 0^+.\cr}$$
\endproclaim

\demo{Proof}  The proof of (i) is identical to that used in the previous Lemma 
with the different exponent resulting from the constant in the Taylor series of 
$\nu\coth {\pi\over 2}\nu$.
\enddemo

The proof of (ii) begins by deriving a Parseval Identity.  We start with 
formula 3.986 (4) in [G-R]
$${\pi\over 2}\nu\coth {\pi\over 2}\nu=2\pi\int^{\infty}_0{\sin^2{\pi\over
2}\nu x\over\sinh^2\pi x}dx+1.$$

\n Then 
$$\lno{J_0&=\int^{\infty}_0e^{-t\nu^2}\nu\coth {\pi\over 2}\nu\,d\nu\cr
&=\int^{\infty}_0e^{-t\nu^2}\left[4\int^{\infty}_0{\sin^2{\pi\over 2}\nu x
\over \sinh^2\pi x}\,dx+{2\over \pi}\right]d\nu\cr
&=4\int^{\infty}_0\left[\int^{\infty}_0e^{-t\nu^2}\sin^2{\pi\over 2}\nu x\,d
\nu\right]{dx\over\sinh^2\pi x}+{2\over \pi}\int^{\infty}_0e^{-t\nu^2}d\nu\cr
&=2\int^{\infty}_0\left[\left({\pi\over 4t}\right)^{1/2}-\int^{\infty}_0e^{-t
\nu^2}\cos \pi \nu x\,d\nu\right]{dx\over \sinh^2\pi x}+\left({1\over \pi t}
\right)^{1/2}\cr
&=2\int^{\infty}_0\left[\left({\pi\over 4t}\right)^{1/2}-\left({\pi\over 4t}
\right)^{1/2}e^{-\pi^2 x^2/4t}\right]{dx\over \sinh^2\pi x}+\left({1\over \pi 
t}\right)^{1/2}\cr
&=2\int^{\infty}_0-{1\over \pi}{d\over dx}\coth \pi x\left[\left({\pi\over
4t}\right)^{1/2}-\left({\pi\over 4t}\right)^{1/2}e^{-\pi^2 x^2/4t}\right]dx+
\left({1\over \pi t}\right)^{1/2}\cr
&=-{2\over \pi}\left({\pi\over 4t}\right)^{1/2}+{\pi^{3/2}\over t}
\int^{\infty}_0x\coth \pi x{e^{-\pi^2x^2/4t}\over (4t)^{1/2}}dx+\left({1\over 
\pi t}\right)^{1/2}\cr
&={\pi^{3/2}\over t}\int^{\infty}_0x\coth \pi x{e^{-\pi^2x^2/4t}\over (4t)^{1/
2}}\,dx.\cr}$$
Hence we obtain the Parseval Identity
$$\lno{\int^{\infty}_0e^{-t\nu^2}\nu\coth{\pi\over 2}\nu\,d\nu={1\over t}
\int^{\infty}_0{e^{-x^2/4t}\over (4\pi t)^{1/2}}\,x\coth x dx.&&(6.3)\cr}$$
Now the result for $k=0$ follows directly from Watson's Lemma, while the 
computation for $k>0$ proceeds as before.\hfill\bx
\medskip

Since the Plancherel density is of the form
$\sum\limits^m_{j=0}a_j\nu^{2m-2j}\left\{{\nu\tanh{\pi\over 2}\nu\atop \nu
\coth{\pi\over2}\nu}\right\}$ and finitely many $e(\xi,\lambda_w)$ are  nonzero, putting together Lemma 6.1 and Lemma 6.2 with 
Lemma 5.3 we obtain 

\proclaim{Proposition 6.3} Assume that $G$ has only one class of cuspidal 
maximal parabolic subgroups.  Then 
$$\align (i) f^{\hat c_{\lambda_w}^2}_t(e)-\sum\limits_{\omega\in\cl
E_2(G)}e^{{t\over 2}(\chi_{\omega} (\Omega_G) - \langle\lambda^\ast,\lambda^\ast + 
2\rho\rangle)}d_{\pi_{\omega}} e^{\hat c_{\lambda_w}^2}_{M_Q}(\omega,V_{\lambda})&\sim e^{2t\langle\lambda,\rho_n\rangle}e^{-{t\over 2}\hat c^2_{\lambda_w}}
 \left\{\matrix 
 {c\over t^{3/2}}\\ 
 {c\over t^{1/2}} 
 \endmatrix \right\},\,t\to +\infty;\\
(ii) f^{\hat c_{\lambda_w}^2}_t(e)-\sum\limits_{\omega\in \cl
E_2(G)}e^{{t\over 2}(\chi_{\omega} (\Omega_G) - \langle\lambda^\ast,\lambda^\ast + 
2\rho\rangle)}d_{\pi_{\omega}} e^{\hat c_{\lambda_w}^2}_{M_Q}(\omega,V_{\lambda})&\sim \\
\sum\limits_{{w\in W_\kappa,\hat{c}_{\lambda_w}^2=c^2}}(-1)^{\ell(w)+1}&e(\xi,\lambda_w)\sum\limits^{\infty}_{j=0}{c^{w}_{j-(m+1)}\over t^{m+1-j}},\quad t\to 0^+;\\
(iii) f^{\hat c_{\lambda_w}^2}_t(e)-\sum\limits_{\omega\in \cl
E_2(G)}e^{{t\over 2}(\chi_{\omega} (\Omega_G) - \langle\lambda^\ast,\lambda^\ast + 
2\rho\rangle)}d_{\pi_{\omega}} e^{\hat c_{\lambda_w}^2}_{M_Q}(\omega,V_{\lambda})&\sim 
\sum\limits^{\infty}_{j=0}{c^{\hat c_{\lambda_w}^2}_{j-(m+1)}\over t^{m+1-j}},\quad t\to 0^+.\endalign$$
\endproclaim

\remark{Remark~~} In $(i)$ for the asymptotics we have grouped together the cases according to $\chi(\gamma_\alpha) = \pm 1$ and performed the sum to obtain one constant $c$. 

In (ii) we keep the dependence of the constants on $w$ as this will be needed later. In (5.10) there occurs $\sum\limits_{{w\in W_\kappa,\hat{c}_{\lambda_w}^2=c^2}}(-1)^{\ell(w)}e(\xi,\lambda_w)$. As remarked earlier, for each $w$ there is at most one $\xi$ for which this is nonzero. 

In $(iii)$ similarly we have grouped terms and performed the sum to obtain the coefficients $c^{\hat c_{\lambda_w}^2}_l$.  Notice that in (iii) the right hand side reflects the contribution of principal series representations. Using Lemma 5.1 and (5.3)  it follows that $c^{\hat c_{\lambda_w}^2}_l$ is independent of the {\it Ansatz}, as is the easily determined constant $c$. Thus the right hand side of $(i)$, $(ii)$ and $(iii)$ is independent of the {\it Ansatz}.

Finally, we caution the reader concerning a substitution of ${t\over 2}$ for $t$. The real variable integrals appear cleaner as presented, but in Lemma 5.3 the variable involves ${t\over 2}$.
\endremark
\medskip
In order to compute $\theta$-regularized determinants we shall need 
the $Pf$ Laplace transform of $f^{\hat c_{\lambda_w}^2}_t(e)$ for each ${\hat c_{\lambda_w}^2}$. Notice that due to 
Proposition 
6.3, 
$f^{\hat c_{\lambda_w}^2}_t(e)$ defines a pseudo-function, and so we may use the symbolic 
calculus of 
$Pf$ Laplace transforms.  Again the tanh and coth cases require very different 
formulas, so will be treated separately.

\proclaim{Lemma 6.4}  Assume that $\Re z^2>0$ and that $\Re z \notin (-\infty, 0]$.  
Then 
$$\eqalign{Pf\int^{\infty}_0e^{-tz^2}&\left[\int^{\infty}_0e^{-t\nu^2}\nu^{2k}
\nu\tanh{\pi\over 2}\nu\,d\nu\right]\,dt\cr
&=-(iz)^{2k}{\Gamma'\left({1+z\over 2}\right)\over\Gamma \left({1+z\over
2}\right)}+\sum^k_{j=0}c^k_jz^{2j}.\cr}$$
Consequently it has a meromorphic continuation to $\Bbb C$.
\endproclaim

\demo{Proof}   First notice that from Lemma 6.1 it follows that the expression 
in brackets is a pseudofunction so that the $Pf$ Laplace transform is defined 
and holomorphic for $\Re z^2>0$.  We will do the $k=0$ case first, and then 
use the symbolic calculus to treat the general case.  We begin with the 
Parseval Identity (6.1)
$$\eqalign{Pf\int^{\infty}_0&e^{-tz^2}\left[\int^{\infty}_0e^{-t\nu^2}\nu\tanh
{\pi\over 2}\nu\,d\nu\right]dt\cr
&=Pf\int^{\infty}_0e^{-tz^2}\left[{1\over t}\int^{\infty}_0{e^{-x^2/4t}\over 
(4\pi t)^{1/2}}{x\over \sinh x}\,dx\right]dt\,.\cr}$$

\n From the symbolic calculus this Laplace transform is up to a constant an 
integral of 
$$Pf\int^{\infty}_0e^{-tz^2}\left[\int^{\infty}_0{e^{-x^2/4t}\over(4\pi 
t)^{1/2}}{x\over \sinh x}\,dx\right]dt\,.$$
As $\Re z^2>0$, the integrand is not singular, so we may use the Fubini theorem to 
get 
$$\eqalign{&\int^{\infty}_0{x\over \sinh x}\int^{\infty}_0e^{-tz^2}{e^{-x^2/4
t}\over (4\pi t)^{1/2}}dt\,dx\cr
& =\int^{\infty}_0{x\over \sinh x}{e^{-zx}\over 2z}dx \quad  (\Re z \notin (-\infty, 
0])\cr
& ={1\over 4z}\sum_{n\ge 0}
{1\over \left(n+{z+1\over 2}\right)^2}\quad \text{(3.552.1 [G-R])}.\cr}$$

\n Now in terms of the usual notation for the Digamma function $\psi(w) := {d\over dw}
\log\Gamma(w)$ then
$$ \sum_{n\ge 0}{1\over (n+w)^2}={d^2\over dw^2}\log\Gamma(w)={d\over dw}
{\Gamma'(w)\over \Gamma (w)} = {d\over dw}\, \psi(w),$$
so with $z^2=p$ we obtain 
$$Pf\int^{\infty}_0e^{-tp}\int^{\infty}_0{e^{-x^2/4t}\over (4\pi t)^{1/2}}
{x\over \sinh x}dx\,dt={d\over dp}\left[{\Gamma'\left({p^{1/2}+1\over 2}
\right)\over \Gamma\left({p^{1/2}+1\over 2}\right)}\right], 
 (\Re(p^{1/2}) > 0).$$
The symbolic calculus then gives 
$$Pf\int^{\infty}_0e^{-tp}{1\over t}\int^{\infty}_0{e^{-x^2/4t}\over(4\pi 
t)^{1/2}}{x\over \sinh x}dx\,dt=-{\Gamma'\left({p^{1/2}+1\over 2}\right)\over 
\Gamma\left({p^{1/2}+1\over 2}\right)}+c .$$

\n To determine $c$, notice that 
$${\Gamma'\left({p^{1/2}+1\over 2}\right)\over \Gamma\left({p^{1/2}+1\over 2}
\right)}\sim \log
p^{1/2}-\log 2,\quad p\to +\infty.$$
Then from the symbolic calculus we have 
$$c=-{\gamma\over 2}-\log 2$$ 
where $\gamma$ is Euler's constant.  So we obtain, for $\Re z>0$ with 
$\Re z^2>0$, 
$$Pf\int^{\infty}_0e^{-tz^2}\left[\int^{\infty}_0e^{-t\nu^2}\nu\tanh {\pi\over
2}\nu\,d\nu\right]dt=\left(-{\gamma\over 2}-\log 2\right)-{\Gamma'\left({1+z
\over 2}\right)\over\Gamma\left({1+z\over 2}\right)}.$$
Since $\Gamma' \over \Gamma$ is meromorphic on $\Bbb C$, we get a
meromorphic continuation of the $Pf$ integral to $\Bbb C$.

Next we consider the case $k=1$.
$$\eqalign{Pf\int^{\infty}_0&e^{-tz^2}\left[\int^{\infty}_0e^{-t\nu^2}\nu^2\nu
\tanh {\pi\over 2}\nu\,d\nu\right]dt\cr
&=Pf\int^{\infty}_0e^{-tz^2}\left[-{d\over dt}\int^{\infty}_0e^{-t\nu^2}\nu
\tanh {\pi\over 2}\nu\,d\nu\right]\,dt\cr
&=Pf\int^{\infty}_0e^{-tz^2}\left[-{d\over dt}g(t)\right]\,dt.\cr}$$
But from Lemma 6.1 (ii), 
$$g(t)\sim {a_0\over t}+a_1+O(t),$$
so the symbolic calculus gives 
$$Pf\int^{\infty}_0e^{-tp}\left[-{dg\over 
dt}(t)\right]\,dt=-p\,Pf\int^{\infty}_0e^{-tp}g(t)
dt+a_1+a_0p.$$

\n Using the result already proved, $k=0$, we get 
$$\eqalign{Pf\,\int^{\infty}_0&e^{-tz^2}\left[\int^{\infty}_0e^{-t\nu^2}\nu^2
\nu\tanh {\pi\over 2}\nu\,d\nu\right]\,dt\cr
&=z^2{\Gamma'\left({1+z\over 2}\right)\over\Gamma\left({1+z\over 2}\right)}+
z^2\left(-{\gamma\over 2}-\log 2\right)+a_1+a_0z^2\cr
&=- (iz)^2{\Gamma'\left({1+z\over 2}\right)\over\Gamma\left({1+z\over
2}\right)}+\sum^1_{j=0}c^1_jz^{2j}.\cr}$$
The result for general $k$ then follows from induction.\hfill\bx
\medskip

\proclaim{Lemma 6.5}  Assume that $\Re z^2>0$ and that $\Re z \notin (-\infty, 0]$.  
Then 
$$\eqalign{Pf\int^{\infty}_0&e^{-tz^2}\left[\int^{\infty}_0e^{-t\nu^2}\nu^{2k}
\nu\coth {\pi\over 2}\nu \,d\nu\right]dt\cr
&=(iz)^{2k}\left[-{\Gamma'\left({z\over 2}+1\right)\over \Gamma\left({z\over 
2}+1\right)}+{1\over z}\right]+\sum^k_{j=0}d^k_jz^{2j},\cr}$$
in particular has a meromorphic continuation to $\Bbb C$.
\endproclaim

\demo{Proof}   As usual we will do the $k=0$ case first and then the general 
case with the symbolic calculus.  We begin with the Parseval Identity (6.3).  
$$\eqalign{Pf\int^{\infty}_0&e^{-tz^2}\left[\int^{\infty}_0e^{-t\nu^2}\nu\coth 
{\pi\over 2}\nu\,d\nu\right]\,dt\cr
&=Pf\int^{\infty}_0e^{-tz^2}\left[{1\over t}\int^{\infty}_0{e^{-x^2/4t}\over 
(4\pi t)^{1/2}}x\coth x\,dx\right]\,dt.\cr}$$
From the symbolic calculus this Laplace transform is, up to a constant, an 
integral of
$$Pf\int^{\infty}_0e^{-tz^2}\left[\int^{\infty}_0{e^{-x^2/4t}\over(4\pi t
)^{1/2}}x\coth  x\,dx\right]dt.$$

\n From Lemma 6.2 we see that the integrand is not singular in $t$ so we may 
use the Fubini Theorem to get 
$$\eqalign{&\int^{\infty}_0x\coth x\left[\int^{\infty}_0e^{-tz^2}{e^{-x^2/4t}
\over (4\pi t)^{1/2}} dt\right]\,dx\cr
&=\int^{\infty}_0x\coth x{e^{-zx}\over 2z}\,dx\hskip .3in (\Re z \notin (-\infty, 
0])\cr
&={1\over 2z}\left[{1\over 2}\sum_{n\ge 0}{1\over (n+{z\over 2})^2}-{1\over 
z^2}\right]\hskip .3in \text{(3.551.3 [G-R]).}\cr}$$

\n As in Lemma 6.4, with $z^2=p$ we get 
$$\eqalign{Pf\int^{\infty}_0e^{-tp}&\int^{\infty}_0{e^{-x^2/4t}\over(4\pi 
t)^{1/2}}x\coth x\,dx\,dt\cr
&=\left[{d\over dp}{\Gamma'\left({p^{1/2}\over 2}\right)\over 
\Gamma\left({p^{1/2}\over 2}\right)}+{d\over dp}{1\over p^{1/2}}\right].\cr}$$

\n The symbolic calculus and the recurrence relation for $\psi$ then gives 
$$\eqalign{Pf\int^{\infty}_0e^{-tp}{1\over t}&\int^{\infty}_0{e^{-x^2/4t}\over 
(4\pi t)^{1/2}}x\coth x\,dx\,dt\cr
&=\left[-{\Gamma'\left({p^{1/2}\over 2}\right)\over \Gamma
\left({p^{1/2}\over 2}\right)}-{1\over p^{1/2}}\right]+c\cr
&=\left[-{\Gamma'\left({p^{1/2}\over 2}\right)\over \Gamma\left({p^{1/2}\over 
2}\right)}-{2\over p^{1/2}}+{1\over p^{1/2}}\right]+c\cr
&=\left[-{\Gamma'\left({p^{1/2}\over 2}+1\right)\over \Gamma \left({p^{1/2}
\over 2}+1\right)}+{1\over p^{1/2}}\right]+c.\cr}$$

\n To determine $c$ we, as before, use Stirling's estimate 
$${\Gamma'\left(z\right)\over \Gamma\left(z\right)}\sim 
{-1\over 2z}+\log z,\quad z\to +\infty$$
and then the symbolic calculus to find $c=-{\gamma\over 2}-\log 2$.  We 
complete the proof for $k\ge 1$ as in the previous Lemma.\hfill\bx

\medskip

\remark{Remarks~~} In order to put the $Pf$ Laplace transform of $\text{dim F vol X 
}f^{\hat c_{\lambda_w}^2}_t(e)$ in 
manageable, 
geometric form, we make some observations and notational simplifications. It is useful to keep in mind (5.10). 

Let $\chi_a(X)$ denote the arithmetic 
genus of $X$, i.e.
$$\chi_a(X) = \sum^n_{q=0}(-1)^q\dim H^{0,q}(X).$$
We shall denote the integer $\chi_a(X)\dim F$ by $\chi_a(X,\Bbb F)$, the arithmetic genus 
of $X$ with coefficients in the flat bundle $\Bbb F$. 

For $\omega\in \cl E_2(G)$ with infinitesimal character $\Lambda_\omega$, notice that 
$$d(\Lambda_\omega) := {\epsilon (\Lambda_\omega)\varpi (\Lambda_\omega)\over \varpi 
(\rho)}$$
is integral, being, up to a sign,  the dimension of a finite dimensional representation of $G$.

In the discussion circa (4.12) we parametrized $\frak a^{\ast}_{\frak q}$ as ${\nu\over 2} 
\alpha_{\Bbb R}$ where $\nu\in \Bbb R$, and set $ \alpha^{\vee}:= {{2\alpha} \over {\langle \alpha, \alpha \rangle}}$. We recall that $\varpi(\log a^\ast_I +i\nu )$, the 
polynomial 
factor in the Plancherel density, is $\nu$ times a polynomial in $\nu^2$. 

Let $c_j$ (resp. $d_j$), depending on $w$, be the numbers obtained from summing the 
$c^k_j$ in Lemma 6.4 (resp. $d^k_j$ in Lemma 6.5) over the various powers $\nu^{2k}$ from the 
polynomial factor of the Plancherel density. For each $w$ we define a polynomial function on $\frak a_{\frak q,\Bbb C}^\ast$,  $\widetilde p_{\lambda_w} (\cdot)' $, by the expression ${(-2i)(-1)^{q-q_M}\over\varpi 
(\rho)}$ $\sum c_jz^{2j}$ (resp. $\sum d_jz^{2j}$) when $\lambda^w(\gamma_\alpha )$ is $1$ (resp. 
$-1$) and $ z\in \Bbb C$.  

Again the term $\sum_{{w\in W_\kappa,\hat{c}_{\lambda_w}^2=c^2}}(-1)^{\ell(w)}e(\xi,\lambda_w)$ introduces the, by now, familiar notational complications. As before we will have to group terms, but for clarity we first do the computations for one $w$.

We come to a subtle observation concerning the variable $z$ in the Laplace transform, and thus 
in 
the 
zeta function to be defined in the next section - it must belong to the same space as the variable $\nu$. 
Hence 
the variable $z$ of the zeta function is in $\frak a_{\frak q,\Bbb C}^\ast$. So we must choose an 
identification 
of $\Bbb C$ with $\frak a_{\frak q,\Bbb C}^\ast$. For consistency with the previous parametrization for 
$\nu$ 
we 
take `z'$\in\frak a^{\ast}_{\frak q,\Bbb C}$ as ${z\over 2}\alpha_{\Bbb R}, z\in\Bbb C$. But we caution that this 
leads 
to 
annoying appearances of `two' in several places. Also, with the parametrization of $\frak a^{\ast}_{\frak q}$ as ${\nu\over 2}\alpha_{\Bbb R},$ the measure on $\frak a^{\ast}_{\frak q}$ in these coordinates becomes $d\frak a^{\ast}_{\frak q} ={2d\nu\over\Vert\alpha_{\Bbb R}\Vert}$.

Ordinarily in $Pf$ Laplace transforms, changes of 
variables by scaling are not valid. However, in the proofs of Lemma 6.4
and Lemma 6.5 the starting point is a Laplace transform
of a non-singular function, followed by repeated application  of the symbolic calculus. With a 
little 
care (as will be done in the proof of the next Proposition) one can easily verify that the Lie 
theoretic formulation of Lemma 6.4 becomes
$$\align\qquad &Pf\int^{\infty}_0 e^{-t\Vert {z\over 2}\alpha_{\Bbb R}\Vert^2}
\int_{\frak a_{\frak q}^\ast}e^{-t\Vert \nu\Vert^2} 
\varpi \left( \Lambda_{\xi(w)} + i\nu \right)\, \tanh ({\pi \over 2} \langle\nu,\alpha_{\Bbb R}^{\vee}\rangle 
)\,d\nu\,dt \\
& = {(-8i)\over\Vert \alpha_{\Bbb R}\Vert}\, {\varpi \left( \Lambda_{\xi(w)} + i \left( {iz\alpha_{\Bbb R}
 \over 2} \right)\right) \over\langle\alpha^{\vee}_{\Bbb R},i\left({iz\alpha_{\Bbb R}\over 2}\right)\rangle} 
{\Gamma'\left({1\over 2}({\langle\alpha^{\vee}_{\Bbb R},{z\over 2}\alpha_{\Bbb 
R}\rangle}+1)\right)\over \Gamma \left(
{1\over 2}({\langle\alpha^{\vee}_{\Bbb R},{z\over 2}\alpha_{\Bbb R}\rangle}+1)\right)} 
+ 
 {(-1)^{q-q_M}\varpi 
(\rho)}\widetilde p_{\lambda_w}'({z\over 2}\alpha_{\Bbb R})\\ 
& = {(-8i)\over\Vert \alpha_{\Bbb R}\Vert}\, {\varpi \left( \Lambda_{\xi(w)} + i \left( {iz\alpha_{\Bbb R}
 \over 2} \right)\right) \over\langle\alpha^{\vee}_{\Bbb R},i\left({iz\alpha_{\Bbb R}\over 2}\right)\rangle} 
\psi\left({1\over 2}({\langle\alpha^{\vee}_{\Bbb R},{z\over 2}\alpha_{\Bbb 
R}\rangle}+1)\right) 
+ {(-1)^{q-q_M}\varpi 
(\rho)}\widetilde p_{\lambda_w}'({z\over 2}\alpha_{\Bbb R}),
\tag6.4\endalign $$
and similarly for Lemma 6.5. Notice that this is consistent with a change of variables 
in an absolutely convergent integral.

Recall ((4.9)) that $\hat c_{\lambda_w} =(\lambda_w\circ C_\kappa+\rho_{Q_{\kappa}})(\wh 
X_{\kappa})$, 
and that 
$X_\kappa = H_{\alpha_{\Bbb R}} /\langle\alpha_{\Bbb R} ,\alpha_{\Bbb R}\rangle$. It 
is necessary to realize $\hat c_{\lambda_w}$ also as an element of $\frak a_{\frak q}^\ast$. 
Recalling that $\wh 
X_{\kappa}$ is a unit vector it follows from the above that with 
$$c_{\lambda_w} :=(\lambda_w\circ C_\kappa+\rho_{Q_{\kappa}})(X_{\kappa})$$

that $\hat c_{\lambda_w}$ is 
replaced with 
$$c_{\lambda_w}\alpha_{\Bbb R} =(\lambda_w\circ C_\kappa+\rho_{Q_{\kappa}})( 
X_{\kappa})\alpha_{\Bbb R},\tag 6.5$$ so that $\hat c^2_{\lambda_w}$ becomes $ 
c^2_{\lambda_w}\Vert\alpha_{\Bbb 
R}\Vert^2$. 

Finally, from Lemma 1.1 we see that the use of the Kodaira Laplacian rather than the Hodge Laplacian introduces a numerical factor of ${1\over 2}$, whose effect is seen clearly in, say, (5.10). As is seen in (6.4) the formulas are rather satisfactory without the factor ${1\over 2}$. There appear to be two methods to reconcile this: incorporate a factor of ${1\over 2}$ in the definition of the Laplace transform or change the `t' parameter to `2t'. We have chosen the latter approach for what follows. Recall the factor $e^{2t\langle\lambda,\rho_n\rangle}$ appears independently of $w$. Also we have seen in \S 1 that $4\langle\lambda,\rho_n\rangle = -i\lambda(H_0)$, i.e.  $e^{2t\langle\lambda,\rho_n\rangle}= e^{{t\over 2}\langle-i\lambda(H_0)\rangle}$. With the change  `t' parameter to `2t' we obtain $e^{4t\langle\lambda,\rho_n\rangle}= e^{-it\langle\lambda(H_0)\rangle}$. While this displays a curious interplay between the choice of the
 complex structure and the coefficient bundle, the lack of dependence on $w$ makes it a good candidate for a uniform translation change of variable. We will highlight where this is done. Finally, for the convenience of the reader we recall the formula $\Vert\lambda^\ast +\rho\Vert^2 + 
4\langle\lambda,\rho_n\rangle = \Vert\lambda +\rho\Vert^2$, in particular $\omega (\Omega _G) - \langle\lambda^\ast,\lambda^\ast + 2\rho\rangle = \Vert \Lambda_\omega \Vert^2 - \Vert \lambda^\ast + \rho\Vert^2 = \Vert \Lambda_\omega \Vert^2 - \Vert \lambda + \rho\Vert^2 +4\langle\lambda,\rho_n\rangle.$
\endremark

\proclaim{Proposition 6.6}
$$\align \langle\alpha_{\Bbb R}^\vee,z\alpha_{\Bbb R}\rangle Pf\int^\infty_0 &e^{-t(\Vert 
{z\over 
2}\alpha_{\Bbb 
R}\Vert^2 
-\Vert c_{\lambda_w}\alpha_{\Bbb R}\Vert^2)} \text{dim F vol X 
}e^{-4t\langle\lambda,\rho_n\rangle}f^{\hat c_{\lambda_w}^2}_{2t}(e)dt\, - 
\\ 
&\chi_a (X,\Bbb F) \sum_{\omega\in 
\cl 
E_2(G)}{\langle\alpha_{\Bbb R}^\vee,z\alpha_{\Bbb R}\rangle 
e^{\hat c_{\lambda_w}^2}_{M_Q}(\omega 
,V_\lambda)d(\Lambda_\omega)\over 
{\Vert {z\over 2}\alpha_{\Bbb R}\Vert^2 
-\Vert c_{\lambda_w}\alpha_{\Bbb R}\Vert^2-\Vert 
\Lambda_{\omega}\Vert^2+\Vert\lambda
+ \rho\Vert^2}} =\tag6.6\\ 
  - \chi_a(X,\Bbb F)\sum_{{w\in W_\kappa,\hat{c}_{\lambda_w}^2=c^2}}(-1)^{\ell(w)+1}&e(\xi,\lambda_w) 4(-1)^{q-q_M}
{\varpi \left(\Lambda_{\xi (w)}+i \left( {iz\alpha_{\Bbb R}
 \over 2} \right)\right) \over\varpi (\rho)} 
\left\{\matrix {\Gamma '\left({1+\langle\alpha^{\vee}_{\Bbb R},{z\over 2}\alpha_{\Bbb 
R}\rangle 
\over 
2}\right) \over 
\Gamma \left({1+\langle\alpha^{\vee}_{\Bbb R},{z\over 2}\alpha_{\Bbb R}\rangle \over 
2}\right)}\\
{\Gamma '\left( 1+{\langle\alpha^{\vee}_{\Bbb R},{z\over 2}\alpha_{\Bbb R}\rangle 
\over  
2}\right) 
\over 
\Gamma \left(
1+{\langle\alpha^{\vee}_{\Bbb R},{z\over 2}\alpha_{\Bbb R}\rangle \over 2}\right)}-
{1 \over \langle\alpha^{\vee}_{\Bbb R},{z\over 2}\alpha_{\Bbb R}\rangle}\endmatrix 
\right\}\\
 - \chi_a(X,\Bbb F)&\sum_{{w\in W_\kappa,\hat{c}_{\lambda_w}^2=c^2}}(-1)^{\ell(w)+1}e(\xi,\lambda_w)
\,2\langle\alpha_{\Bbb R}^\vee,z\alpha_{\Bbb R}\rangle \widetilde 
p_{\lambda_w}'\left({z\over 2}\alpha_{\Bbb R} 
\right)
\endalign$$
for $\xi_w\longleftrightarrow\Lambda_{\xi (w)}$, and accordingly as 
$\lambda^w(\gamma_\alpha)=\pm 1$.  Moreover this is a meromorphic 
function on $\Bbb C\,\alpha_{\Bbb R}$ with simple poles and integral residues.
\endproclaim

\demo{Proof}  We begin with Lemma 5.3,
$$\eqalign{\dim F &\vol X f^{\hat c_{\lambda_w}^2}_{2t}(e) - \sum_{\omega\in\cl E_2 (G)}\dim F\vol X 
e^{t(\omega (\Omega _G) - \langle\lambda^\ast,\lambda^\ast + 2\rho\rangle)}d_{\pi_{\omega}}  
e^{\hat c_{\lambda_w}^2}_{M_Q}(\omega ,V_\lambda )\,\,=\cr
&\sum_{A^{\ast'}_I}\dim F\vol Xc_A\int_{\frak a^\ast_{\frak q}}e^{-t(\Vert c_{\lambda_w}\alpha_{\Bbb 
R}\Vert^2+
\Vert\nu\Vert^2)}e^{4t\langle\lambda,\rho_n\rangle}\quad\cdot\cr
&\sum_{{w\in W_\kappa,\hat{c}_{\lambda_w}^2=c^2}}(-1)^{\ell(w)}e(\xi,\lambda_w)
\varpi \left(\Lambda_{\xi (w)}+ i\nu \right)
\left\{\matrix \tanh
{\pi \over 2} \langle \nu, \alpha^{\vee}_{\Bbb R} \rangle  \\
\coth{\pi \over 2} \langle \nu, \alpha^{\vee}_{\Bbb R} \rangle
\endmatrix\right\}\,d\nu\, \,.\cr}$$
 Then multiplying the formula above by $e^{-4t\langle\lambda,\rho_n\rangle}$ gives
$$\eqalign{\dim F &\vol X e^{-4t\langle\lambda,\rho_n\rangle}f^{\hat c_{\lambda_w}^2}_{2t}(e) - \sum_{\omega\in\cl E_2 (G)}\dim F\vol X 
e^{-4t\langle\lambda,\rho_n\rangle}e^{t(\omega (\Omega _G) - \langle\lambda^\ast,\lambda^\ast + 2\rho\rangle)}d_{\pi_{\omega}}  
e^{\hat c_{\lambda_w}^2}_{M_Q}(\omega ,V_\lambda )\,\cr
= & \,\sum_{A^{\ast'}_I}\dim F\vol Xc_A\int_{\frak a^\ast_{\frak q}}e^{-t(\Vert c_{\lambda_w}\alpha_{\Bbb 
R}\Vert^2+\Vert\nu\Vert^2)}\quad\cdot\cr
&\sum_{{w\in W_\kappa,\hat{c}_{\lambda_w}^2=c^2}}(-1)^{\ell(w)+1}e(\xi,\lambda_w)
\varpi \left(\Lambda_{\xi (w)}+ i\nu \right)
\left\{\matrix \tanh
{\pi \over 2} \langle \nu, \alpha^{\vee}_{\Bbb R} \rangle  \\
\coth{\pi \over 2} \langle \nu, \alpha^{\vee}_{\Bbb R} \rangle
\endmatrix\right\}\,d\nu\, \,.\cr}$$

As we use Harish-Chandra's normalization of measures, we have ([H-C;I]) for 
$\omega\in\cl E_2(G)$ (and summing over $\cl E_2(G)$ not $T^\ast$)
$$d_{\pi_{\omega}} = c^{-1}_G\epsilon (b^\ast_\omega)\varpi 
(b^\ast_\omega)\vert W(G,T)\vert\,,$$
where $\varpi (b^\ast_\omega )=\varpi 
(\log b^\ast_\omega +\rho ) = \varpi (\Lambda_\omega),$ $c_G=\vert W(G,T)\vert (2\pi 
)^{q} 2^{{\nu 
\over 2}}\varpi_k(\rho_k)$, $q={1 \over 2}\dim G/K$, $\nu=\dim G/K-\rank G/K$.  The 
arithmetic genus is 
the index of Dirac with coefficients in the line bundle associated to 
$\xi_{\rho_n}$.  A calculation (similar, though with minor corrections, to one in [B-M] 
for 
$\Bbb R$-rank one) gives 
this index to be
$$\chi_a (X) = c^{-1}_G\varpi (\rho )\vert W(G,T)\vert\vol X\,.$$

Then the discrete series term becomes 
$$\chi_a(X,\Bbb F)\sum_{\omega\in\cl E_2 (G)} e^{\hat c_{\lambda_w}^2}_{M_Q}(\omega ,V_\lambda 
)d(\Lambda_\omega)
e^{-4t\langle\lambda,\rho_n\rangle}e^{t(\omega (\Omega _G) - \langle\lambda^\ast,\lambda^\ast + 2\rho\rangle)}\, .\tag6.7$$
The first term of (6.6) results from an elementary calculation. We note that, provided $-\Vert c_{\lambda_w}\alpha_{\Bbb R}\Vert^2-\Vert 
\Lambda_{\omega}\Vert^2+\Vert\lambda
+ \rho\Vert^2\ne 0$, the 
resulting function is meromorphic with simple poles and integer residues because 
$\chi_a(X,\Bbb F)$,  $d(\Lambda_\omega)$ and $ e^{\hat c_{\lambda_w}^2}_{M_Q}(\omega ,V_\lambda )$ are 
integers.

For the principal series contribution the first order of business is to identify 
$c_A$ in (4.12). 
 This can be 
done using Herb's evaluation of the Plancherel measure and by comparing it on cusp 
forms to 
Harish-Chandra's in order to reconcile the different choices of Haar 
measure.  So from Herb and for the cases at hand,
$$\eqalign{&f(e) = {(-1)}^q{1 \over (2\pi )^{\vert\Delta_+
\vert}}\sum_{b^\ast\in \widehat{T}}\Theta (b^\ast )(f)\varpi (\log b^\ast )\,\,+\cr
&{(-1)}^q{1\over (2\pi)^{\vert\Delta_+\vert}}~{\vert W(G,T)
\vert \over \vert W(G,H)\vert}\left({i \over 2}\right){2\over \Vert\alpha_{\Bbb 
R}\Vert}\sum_{b^\ast\in \widehat{H}_I}\int_{\frak h^\ast_{\frak q}}\Theta (b^\ast 
,\mu 
)(f)\varpi
(\log b^\ast +i\mu )\left\{\matrix \tanh{\pi \over 2}\mu_\alpha\\
\coth{\pi \over 2}\mu_\alpha\endmatrix\right\}\,d\mu\,,\cr}$$
where $\mu_\alpha ={2\langle\mu ,\alpha\rangle \over \langle\alpha ,\alpha
\rangle}$, $\alpha\in\Delta^+_{\Bbb R}$.  For cuspidal $f$, in [H-C, I] (p.181) we 
find
$$f(e)=c^{-1}_G\sum_{b^\ast\in T^\ast} {(-1)}^q\varpi (b^\ast )
(\Theta_{b^\ast},f)\,.$$
It follows that
$$dx_{Hb}=(2\pi )^{\vert\Delta_+\vert}c^{-1}_G dx_{H-C}\,,$$
and hence that
$$c_A = (-1)^q {\vert W(G,T)\vert \over \vert W(G,H)\vert}\left({i 
\over 2}\right)c^{-1}_G {2\over \Vert\alpha_{\Bbb R}\Vert}\,.$$
Taking into account that ${(-1)^{q_M}}\Theta (b^\ast ,\mu )$ is the character of a 
principal series representation and that only one discrete series has $e(\xi,\lambda_w)\ne 
0$, 
the contribution from the principal series becomes a sum $\sum_{{w\in W_\kappa,\hat{c}_{\lambda_w}^2=c^2}}(-1)^{\ell(w)+1}$ of terms 
$$\eqalign{{\chi_a(X,\Bbb F) \over \varpi (\rho )}{(-1)^{q-q_M} \over \vert 
W(G,H)\vert} 
&\left({i \over 2}\right){2\over \Vert\alpha_{\Bbb R}\Vert}
\int_{\frak a^\ast_{\frak q}}e^{-t(\Vert c_{\lambda_w}\alpha_{\Bbb 
R}\Vert^2+\Vert\nu\Vert^2)}e(\xi,\lambda_w)
 \varpi \left( \Lambda_{\xi(w)}+i\nu  \right)\left\{\matrix 
\tanh{\pi \over 2}\langle \nu, \alpha^{\vee}_{\Bbb R} \rangle \\ 
\coth{\pi \over 2}\langle \nu, \alpha^{\vee}_{\Bbb R} \rangle  
\endmatrix\right\}\,d\nu\,.\cr}$$
Rewriting the sum over $A^{\ast '}$ as one over $\cl E_2(M^+_Q)$ and, as noted before, the sum has at most one non-zero term associated to $\xi(w)$; recognizing ${\frak a^\ast_{\frak q}}$ in [Hb] represents $\hat H_{\Bbb R}$ with normalized Haar measure, we get (6.8) $\sum\limits_{{w\in W_\kappa,\hat{c}_{\lambda_w}^2=c^2}}(-1)^{\ell(w)+1}e(\xi,\lambda_w) $ 
$$\align\chi_a (X,\Bbb F) 
 {(-1)^{q-q_M}}\left({i \over 2}\right){2\over \Vert\alpha_{\Bbb 
R}\Vert}\int_{\frak 
a^\ast_{\frak q}} e^{-t(\Vert c_{\lambda_w}\alpha_{\Bbb 
R}\Vert^2+\Vert\nu\Vert^2)}{\varpi 
(\Lambda_{\xi(w)}+i\nu ) 
\over 
\varpi 
(\rho )}\left\{\matrix \tanh{\pi \over 2}\langle \nu, \alpha^{\vee}_{\Bbb R} \rangle 
\\ \coth{\pi \over 2}\langle \nu, \alpha^{\vee}_{\Bbb R} \rangle \endmatrix
\right\}\,d\nu\,.\endalign$$

 Now when we 
coordinatise $\frak a^\ast_{\frak q}$ by $\nu={\mu \over 2}\alpha_{\Bbb R}, \, 
\mu\in{\Bbb R}, \, \alpha_{\Bbb R}\in \Delta^+_{\Bbb R}$ then the integral above 
becomes
$$ {\Vert \alpha_{\Bbb R} \Vert\over 2 } \int^\infty_0
e^{-t (\Vert c_{\lambda_w}\alpha_{\Bbb R}\Vert^2)}e^{-t {\Vert \alpha_{\Bbb R} 
\Vert^2 
\over 
4} \mu^2}
{ {\varpi \left( \Lambda_{\xi(w)} + i {\mu \alpha_{\Bbb R}\over 2}\right)}
\over {\varpi(\rho)}}\left\{\matrix \tanh{\pi \over 2}\mu \\
\coth{\pi \over 2}\mu \endmatrix \right\}\,d\mu\,.$$
One can check that $(-1)^{q-q_M} i \varpi(\Lambda_\xi + i {\mu \alpha_{\Bbb R}\over 
2})$ 
is positive for $\mu >0$, as is needed by the Plancherel density. The right side of 
(6.6) now follows from Lemma 6.4 and Lemma 6.5. 

Some remarks on the constants would be 
helpful. For the third term, recall that the constant 
${(-2i)(-1)^{q-q_M}\over\varpi 
(\rho)}$ was put into the definition of $\widetilde p_{\lambda_w}'$, in particular 
the 
minus sign was added to have a minus sign for this term in (6.6); also the scaling 
factor ${4 \over \Vert\alpha_{\Bbb R}\Vert^2}$ is incorporated into the factor 
$\langle\alpha_{\Bbb R},{z\over 2}\alpha_{\Bbb R}\rangle$ producing the 
$\alpha^{\vee}$ 
and the 
aforementioned $2$ in $\widetilde p_{\lambda_w}'$. For the second term, the presence 
of the missing real root which resulted from the integration was added back into 
$\varpi$ when one multiplies by $\langle\alpha_{\Bbb R},{z\over 2}\alpha_{\Bbb 
R}\rangle$ times ${4 
\over \Vert\alpha_{\Bbb R}\Vert^2}$; the minus sign is clearly seen to result from 
(6.5).

The partial fraction decomposition of the derivative of $\log$ $\Gamma$ (i.e.  $\psi$) ([W-W], p. 
247) shows that its poles are simple, occur at negative integral values of the 
argument and are of constant sign. Thus there are poles possibly at 
$z = -2m-1 $ (resp. $z = -2m-2$). Consequently we consider $\Lambda_{\xi(w)} + 
\left({(2m+1)\over 
2}\right)\alpha_\Bbb R$, $m\ge 0$ (resp. 
$\Lambda_{\xi(w)} + m\alpha_\Bbb R, m\ge 1)$. We claim that this is $\Delta(\frak 
g_{\Bbb 
C},\frak 
h_{\Bbb C})$ integral. Now from Cor. 1.13
$$\Lambda_{\xi(w)} =  (w^{\ast}(\lambda^\ast +\rho -2\rho_n))\circ C_\kappa 
	+ (\lambda_w\big|_{\Bbb C H_\kappa})\circ C_\kappa +\rho_{Q_\kappa},$$
and from the standing hypothesis on $\lambda$ the first term is $G$-integral. The 
remaining terms are supported on $\frak a_\frak q$, and are an integral or 
half-integral 
multiple of $\alpha_{\Bbb R}$. Since the possible poles are similarly located it 
suffices to consider such multiples of $\alpha_{\Bbb R}$. But it follows easily from 
the 
structure of roots in the Hermitian case that 
${l\over 2} \alpha_{\Bbb R}$ is integral for $l$ an integer, and thus that
$$ {2(-1)^{q-q_M}}{\varpi 
\left(\Lambda_{\xi(w)}+i \left( {{iz\over 2} \alpha_{\Bbb R}}\right) \right) \over 
\varpi (\rho )}$$
is an even integer at each of these poles. Moreover, this is dominant for the given 
order 
on the roots. \hfill\bls
\enddemo

%\remark{~~Remark 6.7} Each term on the right of (6.6) is of the form 
%\chi_a(X,\Bbb 
%F)$ times a function depending only on the simply connected cover $\widetilde{X}$. 
%This is not surprising since $f^{{\hat c_{\lambda_w}^2}}_{2t}(e)$ is determined by $\widetilde{X}\,$ 
%and the extension of the bundle to the compactification of $\widetilde{X}\,.$
%\endremark

%\vfill\eject

\n {\bf \S 7~~The Elliptic Factor}
\medskip
In Proposition 6.6 we evaluated the Pf Laplace transform of $\text{dim F vol X 
}e^{-4t\langle\lambda,\rho_n\rangle}f^{\hat c_{\lambda_w}^2}_{2t}(e)$ and we showed that it is a meromorphic function on $\Bbb C$ with simple poles and integral residues. One should regard the term $\text{dim F vol X }e^{-4t\langle\lambda,\rho_n\rangle}f^{\hat c_{\lambda_w}^2}_{2t}(e)$ as a 
trace, a $\Gamma$-trace. Then it is natural to apply to it the construction 
of the $\theta$-regularized determinant from [M-S;II]. The details of these steps are the content of this section. The discrete series term will be treated here as a contribution to the $\Gamma$-trace. As it depends on the {\it Ansatz}, in \S 9 it will be related with that for $\tilde k^{\lambda}_t(e)$.

To begin we recall the symbolic calculus from [M-S; II] already used in the proof of Lemma 6.4. Suppose that
$$\quad Pf\int^\infty_0 e^{-tp} \, f(t) \,dt\,=\,{d \over dp}\,F(p),$$
and suppose that
$$F(p) \,\sim\, A\,\log (p) \,+\,B,\,\quad p\to +\infty.$$
Then
$$\quad Pf\int^\infty_0 e^{-tp} \, f(t) \,{dt \over t}\, = - F(p) \,-\gamma\,A 
\,+B,\quad 
\text{ $\gamma$, again, Euler's constant}.$$

By analogy with the discussion in [M-S;II], we denote by
$\log\text{det}_\theta (I+pG^{\hat c_{\lambda_w}^2}_e)$ the unique meromorphic function 
satisfying:

\list{(i)}{$ Pf\int^\infty_0 e^{-tp}\dim F\vol X e^{-4t\langle\lambda,\rho_n\rangle}f^{\hat c_{\lambda_w}^2}_{2t}(e)\,dt\,={d \over dp}\log\text{det}_\theta (I+pG^{\hat c_{\lambda_w}^2}_e),$}

\list{(ii)}{$\log\text{det}_\theta (I+pG^{\hat c_{\lambda_w}^2}_e)\mid_{p=0}=0$.}

Consequently,
$$Pf\int^\infty_0 e^{-tp}\dim F\vol X e^{-4t\langle\lambda,\rho_n\rangle}f^{\hat c_{\lambda_w}^2}_{2t}(e)\,{dt\over t} = -\log\text{det}_\theta (I+pG^{\hat c_{\lambda_w}^2}_e) +C\,. $$
To determine the constant $C$ either one uses the behaviour of $\log\text{det}_\theta 
(I+pG^{\hat c_{\lambda_w}^2}_e)$ for $p\to\infty$ as described above, or the behavour of $\dim F\vol X 
e^{-4t\langle\lambda,\rho_n\rangle}f^{\hat c_{\lambda_w}^2}_{2t}(e)$ for $t \to 0^+$ in the following way. Set 
$$\zeta_e^{\hat c_{\lambda_w}^2}(s) = {1\over\Gamma(s)}\int^\infty_0 t^{s-1}\dim F\vol X 
e^{-4t\langle\lambda,\rho_n\rangle}f^{\hat c_{\lambda_w}^2}_{2t}(e)\,dt.$$
Then $\zeta_e^{\hat c_{\lambda_w}^2}(s)$ is regular at $s=0$, and in [M-S;II] p.190 we show that 
$$ C = {\zeta_e^{\hat c_{\lambda_w}^2}}'(0) -\gamma c^{\hat c_{\lambda_w}^2}_{e,0},$$
where $c^{\hat c_{\lambda_w}^2}_{e,0}$ is the coefficient of $t^0$ in the small time asymptotics 
of 
$\dim F\vol X e^{-4t\langle\lambda,\rho_n\rangle}f^{\hat c_{\lambda_w}^2}_{2t}(e).$ With obvious analogy with zeta regularized 
determinant we denote ${\zeta_e^{\hat c_{\lambda_w}^2}}'(0)$ 
by $-\log\text{det}_{\zeta}\Delta^{\hat c_{\lambda_w}^2}_e $.  

A small technical correction is necessary in the preceding; namely, the procedure in [M-S;II] assumes that zero modes have been removed from the trace. To handle this complication, we separate the contributions from the discrete series and principal series; remove possible zero modes from each; then follow the above procedure.

To simplify the expressions in (6.6) we introduce some notation. Recall that 
$$\align \qquad\qquad c_{\lambda_w}\alpha_{\Bbb R} &= ((\lambda_w\circ 
C_\kappa) + \rho_{Q_{\kappa}})( 
X_{\kappa})\alpha_{\Bbb 
R},\\
\qquad\qquad e(\xi,\lambda_w) &= e(\xi^{(1)}_w,W_{\lambda^{(1)}_w},\frak p^{(1)}_-) 
e(\xi^{(2)}_w,W_{\lambda^{(2)}_w},\frak p^{(2)}_{\Bbb C});\qquad\qquad\qquad\quad
\endalign$$ 
then we set
$$\align  p^{\lambda_w}_c({x\alpha_{\Bbb R}\over 2}) &:= 2(-1)^{q-q_M}{\varpi(\Lambda_{\xi(w)}-{{x\alpha_{\Bbb R}}\over 2}) \over \varpi(\rho )},\quad (cf. (7.6))\\
m_{\omega,\hat{ c}_{\lambda_w}^2} &:= e^{\hat c_{\lambda_w}^2}_{M_Q}(\omega,V_\lambda), \quad (cf. (5.9))\tag7.1\\
M^{\hat{c}_{\lambda_w}^2}_2 &:=\sum_{\omega\in\cl 
E_2(G)}d(\Lambda_{\omega})m_{\omega,\hat{ c}_{\lambda_w}^2}.
\endalign$$

We note that $m_{\omega,\hat{ c}_{\lambda_w}^2}$, $e^{\hat c_{\lambda_w}^2}_{M_Q}(\omega,V_\lambda)$ and $M^{\hat{c}_{\lambda_w}^2}_2$ are dependent on the {\it Ansatz}.

The discrete series contribution, (6.7), is
$$\quad \chi_a(X,\Bbb F)\sum_{\omega\in\cl E_2 (G)}d(\Lambda_{\omega})m_{\omega, \hat{c}_{\lambda_w}^2} 
e^{-4t\langle\lambda,\rho_n\rangle}e^{t(\omega (\Omega _G) - \langle\lambda^\ast,\lambda^\ast + 2\rho\rangle)}.\,$$
Evaluating an elementary integral and using the functional calculus it is straightforward to 
calculate 
$$\align &\quad -\, Pf \int^\infty_0 e^{-t(\Vert {z\over 2}\alpha_{\Bbb R}\Vert^2-\Vert 
c_{\lambda_w}\alpha_{\Bbb R}\Vert^2)}
\sum_{\omega\in\cl E_2 (G)}d(\Lambda_{\omega})m_{\omega, \hat{c}_{\lambda_w}^2} 
e^{-4t\langle\lambda,\rho_n\rangle}e^{t(\omega (\Omega _G) - \langle\lambda^\ast,\lambda^\ast + 2\rho\rangle)}{dt \over t}\,=\\
&\log\prod_{\omega\in \cl E_2(G)}(\Vert {z\over 2}\alpha_{\Bbb R}\Vert^2-\Vert 
c_{\lambda_w}\alpha_{\Bbb R}\Vert^2-\Vert \Lambda_{\omega}\Vert^2+\Vert\lambda
 + \rho\Vert^2)^{d(\Lambda_{\omega})m_{\omega, \hat{c}_{\lambda_w}^2} } +\gamma M^{\hat{c}_{\lambda_w}^2}_2 .
\endalign$$
Notice that the second 
summand 
above contains $\gamma$ times the coefficient of $t^0$ in the small time asymptotics of the 
discrete series contribution to $e^{-4t\langle\lambda,\rho_n\rangle}f^{\hat c_{\lambda_w}^2}_{2t}(e).$ 

Taking the exponential gives
$$e^{\gamma M^{\hat{c}_{\lambda_w}^2}_2 }\,\prod_{\omega\in \cl 
E_2(G)}(\Vert {z\over 2}\alpha_{\Bbb R}\Vert^2-\Vert c_{\lambda_w}
\alpha_{\Bbb R}\Vert^2-\Vert \Lambda_{\omega}\Vert^2+\Vert\lambda +\rho\Vert^2)^{d(\Lambda_{\omega})m_{\omega, \hat{c}_{\lambda_w}^2} }.$$
We rewrite this in a way suggestive of the spectral interpretations of various terms. Let 
 the dimension of the ``zero modes'', i.e. those $\omega$ with $-\Vert 
\Lambda_{\omega}\Vert^2+\Vert\lambda + \rho\Vert^2 = 0$, be denoted by 
$$n^{\hat{c}_{\lambda_w}^2}_{e,\cl E_2} = \,{\underset{\smallmatrix 
\omega\in \cl E_2(G)\\
-\Vert \Lambda_{\omega}\Vert^2+\Vert\lambda + \rho\Vert^2 =
0\endsmallmatrix}\to\sum} d(\Lambda_{\omega})m_{\omega, \hat{c}_{\lambda_w}^2}.$$

 Notationally, we set
$$\align 
\text{det}_\theta \left(I+(\Vert {z\over 2}\alpha_{\Bbb R}\Vert^2-\Vert 
c_{\lambda_w}
\alpha_{\Bbb R}\Vert^2)G_{\cl E_2}^{\hat{c}_{\lambda_w}^2} \right)& :=\,{\underset{\smallmatrix 
\omega\in \cl E_2(G)\\
-\Vert \Lambda_{\omega}\Vert^2+\Vert\lambda + \rho\Vert^2\neq 
0\endsmallmatrix}\to\prod}\left(1\,+\,{\Vert {z\over 2}\alpha_{\Bbb R}\Vert^2-\Vert 
c_{\lambda_w}
\alpha_{\Bbb R}\Vert^2\over-\Vert \Lambda_{\omega}\Vert^2+\Vert\lambda + 
\rho\Vert^2}\right)^{d(\Lambda_{\omega})m_{\omega,\hat{ c}_{\lambda_w}^2} },\\
det_{\zeta}{\Delta}^{\hat{c}_{\lambda_w}^2}_{e,\cl E_2}\, &:= \,{\underset{\smallmatrix \omega\in 
\cl E_2(G)\\
-\Vert \Lambda_{\omega}\Vert^2+\Vert\lambda + \rho\Vert^2\neq 0\endsmallmatrix} 
\to\prod}
\bigl(-\Vert \Lambda_{\omega}\Vert^2+\Vert\lambda + 
\rho\Vert^2\bigr)^{d(\Lambda_{\omega})m_{\omega, {\hat{c}_{\lambda_w}^2}} }.\endalign$$
Both these are defined to be $1$ if there is no discrete series contribution. Then the 
exponential 
which was
$$\align
&e^{\gamma M^{\hat{c}_{\lambda_w}^2}_2}\,{\underset{\smallmatrix \omega\in \cl E_2(G)\\
-\Vert \Lambda_{\omega}\Vert^2+\Vert\lambda + \rho\Vert^2\neq 0\endsmallmatrix} 
\to\prod}
(-\Vert \Lambda_{\omega}\Vert^2+\Vert\lambda + 
\rho\Vert^2)^{d(\Lambda_{\omega})m_{\omega, \hat{c}_{\lambda_w}^2} }\\
&\qquad\cdot (\Vert {z\over 2}\alpha_{\Bbb R}\Vert^2-\Vert c_{\lambda_w}
\alpha_{\Bbb R}\Vert^2)^{n^{\hat{c}_{\lambda_w}^2}_{e,\cl E_2}}{\underset{\smallmatrix \omega\in \cl 
E_2(G)\\
-\Vert \Lambda_{\omega}\Vert^2+\Vert\lambda + \rho\Vert^2\neq 
0\endsmallmatrix}\to\prod}\left(I\,+\,{\Vert {z\over 2}\alpha_{\Bbb R}\Vert^2-\Vert 
c_{\lambda_w}
\alpha_{\Bbb R}\Vert^2\over -\Vert \Lambda_{\omega}\Vert^2+\Vert\lambda + 
\rho\Vert^2}\right)^{d(\Lambda_{\omega})m_{\omega, \hat{c}_{\lambda_w}^2} },
\endalign$$
is written as 
$$\lno {&\qquad\qquad\qquad \widetilde{\Cal B}_{\cl E_2}({z\over 2}\alpha_{\Bbb R})\,:=&(7.2)\cr
&e^{\gamma M^{\hat{c}_{\lambda_w}^2}_2 }\,\,det_{\zeta}{\Delta}^{\hat{c}_{\lambda_w}^2}_{e,\cl E_2}\quad(\Vert {z\over 2}\alpha_{\Bbb 
R}\Vert^2-\Vert c_{\lambda_w}
\alpha_{\Bbb R}\Vert^2)^{n^{\hat{c}_{\lambda_w}^2}_{e,\cl E_2}}\,\,\text{det}_\theta \left(I+(\Vert 
{z\over 2}\alpha_{\Bbb R}\Vert^2-\Vert c_{\lambda_w}
\alpha_{\Bbb R}\Vert^2)G_{\cl E_2}^{\hat{c}_{\lambda_w}^2}\right).\cr
&\qquad\qquad\qquad \Cal B_{\cl E_2}({z\over 2}\alpha_{\Bbb R})\,:=e^{-\gamma M^{\hat{c}_{\lambda_w}^2}_2}\widetilde{\Cal B}_{\cl E_2}({z\over 2}\alpha_{\Bbb R})\cr}$$
Notice that 
$$\lno{\widetilde{\Cal B}_{\cl E_2}(-z\alpha_{\Bbb R})\,&=\,\widetilde{\Cal B}_{\cl E_2}(z\alpha_{\Bbb 
R})\,,\cr
\qquad\lim_{z\to 0}\,{\widetilde{\Cal B}_{\cl E_2}((z+2c_{\lambda_w}){\alpha_{\Bbb 
R}\over 2})\over\langle\alpha^{\vee}_{\Bbb R},{z\over 2}\alpha_{\Bbb 
R}\rangle^{n^{\hat{c}_{\lambda_w}^2}_{e,\cl E_2}}}\,
\,&=\,e^{\gamma M^{\hat{c}_{\lambda_w}^2}_2}det_{\zeta}{\Delta}^{\hat{c}_{\lambda_w}^2}_{e,\cl E_2}.&(7.3)\cr}$$

We repeat this formalism for the principal series contribution to (6.6). As usual, there is the complication resulting from the use of the {\it Ansatz} in that we have a sum of terms $\sum\limits_{{w\in W_\kappa,\hat{c}_{\lambda_w}^2=c^2}}(-1)^{\ell(w)+1}$.  First we construct the 
functions directly, then interpret them in the determinant formalism.  

Consider a constituent of the third term in (6.6). As $\widetilde p_{\lambda_w}'$ is a polynomial in $z^2$ (cf. Remark after Lemma 6.5) we take 
$\widetilde{p}_{\lambda_w}$ 
to 
be an integral with respect to $z^2$ of $4\langle\alpha_{\Bbb R}^\vee,z\alpha_{\Bbb R}\rangle \widetilde 
p_{\lambda_w}'\left({z\over 2}\alpha_{\Bbb R}\right)$, with constant so that 
$\widetilde{p}_{\lambda_w}(c_{\lambda_w}\alpha_{\Bbb R}) = 0$. As $\widetilde{p}_{\lambda_w}$ is a polynomial in $z^2$ we also have $\widetilde{p}_{\lambda_w}(-c_{\lambda_w}\alpha_{\Bbb R}) = 0$.

Recall from (6.8) the contribution of the principal series to  $\text{dim F vol X }e^{-4t\langle\lambda,\rho_n\rangle}f^{\hat c_{\lambda_w}^2}_{2t}(e)$, viz. $\sum\limits_{{w\in W_\kappa,\hat{c}_{\lambda_w}^2=c^2}}(-1)^{\ell(w)+1}e(\xi,\lambda_w) $
$$\align\chi_a (X,\Bbb F) 
 {(-1)^{q-q_M}}\left({i \over 2}\right){2\over \Vert\alpha_{\Bbb 
R}\Vert}\int_{\frak 
a^\ast_{\frak q}} e^{-t(\Vert c_{\lambda_w}\alpha_{\Bbb 
R}\Vert^2+\Vert\nu\Vert^2)}{\varpi 
(\Lambda_{\xi(w)}+i\nu ) 
\over 
\varpi 
(\rho )}\left\{\matrix \tanh{\pi \over 2}\langle \nu, \alpha^{\vee}_{\Bbb R} \rangle 
\\ \coth{\pi \over 2}\langle \nu, \alpha^{\vee}_{\Bbb R} \rangle \endmatrix
\right\}\,d\nu\,.\endalign$$

 We have seen in Proposition (6.6) that each of the terms
 $$2\langle\alpha_{\Bbb R}^\vee,z\alpha_{\Bbb R}\rangle Pf\int^\infty_0 e^{-t(\Vert {z\over 2}\alpha_{\Bbb 
R}\Vert^2)}\int_{\frak 
a^\ast_{\frak q}} e^{-t(\Vert\nu\Vert^2)}{\varpi 
(\Lambda_{\xi(w)}+i\nu ) 
\over 
\varpi 
(\rho )}\left\{\matrix \tanh{\pi \over 2}\langle \nu, \alpha^{\vee}_{\Bbb R} \rangle 
\\ \coth{\pi \over 2}\langle \nu, \alpha^{\vee}_{\Bbb R} \rangle \endmatrix
\right\}\,d\nu {dt}$$
 is 
meromorphic. Then 
$$2\langle\alpha_{\Bbb R}^\vee,z\alpha_{\Bbb R}\rangle Pf\int^\infty_0 e^{-t(\Vert {z\over 2}\alpha_{\Bbb 
R}\Vert^2)}\int_{\frak 
a^\ast_{\frak q}} e^{-t(\Vert\nu\Vert^2)}{\varpi 
(\Lambda_{\xi(w)}+i\nu ) 
\over 
\varpi 
(\rho )}\left\{\matrix \tanh{\pi \over 2}\langle \nu, \alpha^{\vee}_{\Bbb R} \rangle 
\\ \coth{\pi \over 2}\langle \nu, \alpha^{\vee}_{\Bbb R} \rangle \endmatrix
\right\}\,d\nu {dt\over t}$$

is meromorphic.  This, together with those obtained from the other two terms in (6.6), gives for each $w$ a meromorphic function 
$M_w({z\over 2}\alpha_{\Bbb R})$ satisfying 
$$\align \quad &-Pf\int^\infty_0 e^{-t(\Vert {z\over 2}\alpha_{\Bbb 
R}\Vert^2-\Vert 
c_{\lambda_w}\alpha_{\Bbb R}\Vert^2)}\text{dim F vol X }e^{-4t\langle\lambda,\rho_n\rangle}f^{\hat c_{\lambda_w}^2}_{2t}(e){dt \over 
t}\, 
=  \\
& \chi_a (X,\Bbb F)\log \widetilde{\Cal B}_{\cl E_2}({z\over 2}\alpha_{\Bbb R})  + \chi_a (X,\Bbb F)\sum_{{w\in W_\kappa,\hat{c}_{\lambda_w}^2=c^2}}(-1)^{\ell(w)+1}e(\xi,\lambda_w)\log 
M_w({z\over 2}\alpha_{\Bbb R})\\
 +& \chi_a (X,\Bbb F)\sum_{{w\in W_\kappa,\hat{c}_{\lambda_w}^2=c^2}}(-1)^{\ell(w)+1}
e(\xi,\lambda_w)\widetilde{p}_{\lambda_w}({z\over 2}\alpha_{\Bbb R}) +  \chi_a (X,\Bbb F)\sum_{{w\in W_\kappa,\hat{c}_{\lambda_w}^2=c^2}}(-1)^{\ell(w)+1}C_w' 
.\endalign$$

Now each constituent of the second term in (6.6) (without the factors $e(\xi,\lambda_w)$ and $\chi_a (X,\Bbb F)$),
as a function on $\Bbb C \,\alpha_{\Bbb R}$ has simple 
poles with integral residues of constant sign say $\epsilon$. The Weierstrass theorem gives an 
entire 
function $F_w$ on $\Bbb C\,\alpha_{\Bbb R}$ as an 
infinite product, with $\epsilon\log$-derivative with respect to $z$ having precisely these 
poles and residues. By construction $M_w$ has these same poles and residues, thus there is then an entire 
function $h^w$, unique up 
to a constant, with $M_w=F_w^{\epsilon}e^{h^w}$.  We will choose the normalizing constant for 
$M_w$ 
so that the leading coefficient of $ M_w({z\over 2}\alpha_{\Bbb R})$ at $z=2c_{\lambda_w}$ 
is 1. 

For notational convenience we denote the set $\{w\in W_\kappa:\hat{c}_{\lambda_w}^2=c^2\}$ by $W_ {{c}^2}$. 

For each $w\in W_{c^2}$ we have the constant $C_w'$ which arises from the $\log\text{det}_\theta $ procedure.  Because of the common factor of $\chi_a(X,\Bbb F)$ in Proposition 6.6, we include it with $C_w'$. This was also done for the $\cl E_2$ term. To determine $C_w'$ we repeat 
the 
procedure used for the discrete series. Namely, we factor out the zero at 
$c_{\lambda_w}\alpha_{\Bbb R}$; we construct a 
$\theta $-regularized determinant from the principal series contribution; we group it with a 
$\zeta$-regularized determinant; and then the coefficient of $t^0$ of the small time heat trace 
of the principal series contribution contributes a constant factor. 

In Proposition 6.3 (ii) we find for each $w\in W_{c^2}$ the coefficient of $t^0$, denoted $c^w_{0}$. From the proof of Proposition 6.6 we see that $\chi_a(X,\Bbb 
F)/\text{dim F vol X} = 
c^{-1}_G\varpi (\rho )\vert W(G,T)\vert,$ which is independent of the {\it Ansatz}. So we 
set 
$$\tilde c^{w}_{0} = c^w_{0} \text{dim F vol X}/\chi_a(X,\Bbb F).$$ 
Then because there is a common factor of $\chi_a(X,\Bbb F)$ we get
$$\align  c^{\hat{c}_{\lambda_w}^2}_{e,0}\,=&\,\chi_a(X,\Bbb F)(\sum_{\omega\in\cl E_2(G)}d(\Lambda_{\omega})m_{\omega,  \hat{c}_{\lambda_w}^2} )+ \chi_a(X,\Bbb F)\sum_{W_ {{c}^2}}(-1)^{\ell(w)+1}
e(\xi,\lambda_w)\tilde 
c^w_{0}\\
=&\,\chi_a(X,\Bbb F)\,(M^{\hat{c}_{\lambda_w}^2}_2  +\sum_{W_ {{c}^2}}(-1)^{\ell(w)+1}
e(\xi,\lambda_w)\tilde c^w_{0}).\tag7.4\endalign$$

We denote by $n^{\lambda_w}_{e,P}$ the order of the 
zero of $M_w({z\over 2}\alpha_{\Bbb R})^{e(\xi,\lambda_w)}$ at $2c_{\lambda_w}\alpha_{\Bbb R}$ (cf Lemma 7.2).  
Recall that $M_w({z\over 2}\alpha_{\Bbb R})$ is normalized to have leading coefficient $1$ at $2c_{\lambda_w}\alpha_{\Bbb R}$, and $\tilde p_{\lambda_w}({z\over 2}\alpha_{\Bbb R})$ is normalized to be $0$ at $2c_{\lambda_w}\alpha_{\Bbb R}$. Define 
$det_{\zeta}{\Delta}^{\lambda_w}_{e,P}$ by
$$\log det_{\zeta}{\Delta}^{\lambda_w}_{e,P} + \gamma \chi_a(X,\Bbb F) e(\xi,\lambda_w)\tilde c^{w}_{0} = C_w' .$$

Set
$$\align  \widetilde{\Cal B}_{P}({z\over 2}\alpha_{\Bbb R})\,:=&
\prod_{W_ {\hat{c}^2} }\, \{e^{\gamma e(\xi,\lambda_w) \tilde c^w_{0}} e^{e(\xi,\lambda_w)\widetilde{p}_{\lambda_w}
({z\over 2}\alpha_{\Bbb R})}det_{\zeta}{\Delta}^{\lambda_w}_{e,P}\, M_w({z\over 2}\alpha_{\Bbb R})^{e(\xi,\lambda_w)}\}^{(-1)^{\ell (w)+1}}\,; \\
\Cal B_{P}\, =&\prod_{W_ {\hat{c}^2} }\{e^{-\gamma e(\xi,\lambda_w)\tilde c^w_{0}}\}^{(-1)^{\ell(w)+1}}\widetilde{\Cal B}_{P}.\endalign$$

\remark{Remark} $ \widetilde{\Cal B}_{P}$ is not dependent on the {\it Ansatz} as follows from earlier observations, and in \S 9 will be related to $\tilde k^{\lambda}_t(e)$.
\endremark

With these normalizations we see that we might call
$$\text{det}_\theta (I+(\Vert 
{z\over 2}\alpha_{\Bbb R}\Vert^2-\Vert c_{\lambda_w}
\alpha_{\Bbb R}\Vert^2)G_{P}^{\lambda_w})\,:=\,{e^{e(\xi,\lambda_w)\widetilde{p}_{\lambda_w}
({z\over 2}\alpha_{\Bbb R})}M_w({z\over 2}\alpha_{\Bbb R})^{e(\xi,\lambda_w)}\over\langle\alpha^{\vee}_{\Bbb 
R},(z - 2c_{\lambda_w}){\alpha_{\Bbb 
R}\over 2}\rangle^{n^{\lambda_w}_{e,P}}},$$
 so that analogous to (7.2) for $\widetilde{\Cal B}_{\cl E_2}$ we would have
$$\align &\qquad\qquad  \widetilde{\Cal B}_{P}(z\alpha_{\Bbb R})\,=\\
 &\prod_{W_ {\hat{c}^2} }\{det_{\zeta}{\Delta}^{\lambda_w}_{e,P}\, e^{\gamma 
e(\xi,\lambda_w)\tilde c^w_{0}}\langle\alpha^{\vee}_{\Bbb R},(z - 2c_{\lambda_w}){\alpha_{\Bbb 
R}\over 2}\rangle^{n^{\lambda_w}_{e,P}}\text{det}_\theta (I+(\Vert 
{z\over 2}\alpha_{\Bbb R}\Vert^2-\Vert c_{\lambda_w}
\alpha_{\Bbb R}\Vert^2)G_{P}^{\lambda_w})\}^{(-1)^{\ell (w)+1}}\,.
\endalign$$ 
However, this would be misleading because the notation for the $\theta$-determinant is an even function of $z$. Rather, we must take into account the next result. The other possible zeros will be handled subsequently.

\proclaim{Lemma 7.1  }
$$M_w(-s\alpha_{\Bbb R}) = M_w(s\alpha_{\Bbb 
R})e^{\pm 2\pi\int^s_0p^{\lambda_w}_c({x\over 2}\alpha_{\Bbb R})\left\{
\smallmatrix \tan{\pi \over 2}\langle\alpha^{\vee}_{\Bbb R},{x\over 2}\alpha_{\Bbb 
R}\rangle\\ 
\cot{\pi \over 2}\langle\alpha^{\vee}_{\Bbb R},{x\over 2}\alpha_{\Bbb 
R}\rangle\endsmallmatrix\right\}\,dx}\tag7.6$$
\endproclaim

\demo{Proof}  Since $M_w$ has 
integral residues, we may use
$$\log M_w(s\alpha_{\Bbb R})-\log 
M_w(-s\alpha_{\Bbb R}) = 
\int^s_0\left[{M_{w}'(x\alpha_{\Bbb R}) \over 
M_w(x\alpha_{\Bbb R})} + {M_{w}'
(-x\alpha_{\Bbb R}) \over M_w(- x\alpha_{\Bbb R})}\right]\,dx\,.$$
As
$${M_{w}'(x\alpha_{\Bbb R})\over 
M_w( x\alpha_{\Bbb R})} = 2p^{\lambda_w}_c({x\over 2}\alpha_{\Bbb 
R})\left\{
\matrix {\Gamma '\left({1+\langle\alpha^{\vee}_{\Bbb R},{x\over 2}\alpha_{\Bbb 
R}\rangle 
\over 
2}\right) \over \Gamma \left({1+\langle\alpha^{\vee}_{\Bbb R},{x\over 2}\alpha_{\Bbb 
R}\rangle 
\over 
2}
\right)}\\ {\Gamma '\left( 1+{\langle\alpha^{\vee}_{\Bbb R},{x\over 2}\alpha_{\Bbb 
R}\rangle \over 
2}\right) \over \Gamma \left( 1+{\langle\alpha^{\vee}_{\Bbb R},{x\over 2}\alpha_{\Bbb 
R}\rangle 
\over 
2}\right)}-{1 \over \langle\alpha^{\vee}_{\Bbb R},{x\over 2}\alpha_{\Bbb 
R}\rangle}\endmatrix\right.\tag7.7$$
and $p^{\lambda_w}_c$ is an odd function, it is enough to compute ${\Gamma '
\left({1+\langle\alpha^{\vee}_{\Bbb R},{x\over 2}\alpha_{\Bbb R}\rangle \over 
2}\right) 
\over 
\Gamma\left({1+\langle\alpha^{\vee}_{\Bbb R},{x\over 2}\alpha_{\Bbb R}\rangle \over 
2}\right)} - 
{\Gamma '
\left({1-\langle\alpha^{\vee}_{\Bbb R},{x\over 2}\alpha_{\Bbb R}\rangle \over 
2}\right) 
\over 
\Gamma\left({1-\langle\alpha^{\vee}_{\Bbb R},{x\over 2}\alpha_{\Bbb R}\rangle \over 
2}\right)}$ and 
\newline 
${\Gamma '\left( 1+{\langle\alpha^{\vee}_{\Bbb R},{x\over 2}\alpha_{\Bbb R}\rangle 
\over 
2}\right) 
\over \Gamma\left( 1+{\langle\alpha^{\vee}_{\Bbb R},{x\over 2}\alpha_{\Bbb R}\rangle 
\over 2}
\right)} -{2 \over \langle\alpha^{\vee}_{\Bbb R},{x\over 2}\alpha_{\Bbb R}\rangle} - 
{\Gamma 
'\left(1-{\langle\alpha^{\vee}_{\Bbb R},{x\over 2}\alpha_{\Bbb R}\rangle \over 
2}\right) 
\over 
\Gamma\left( 
1-{\langle\alpha^{\vee}_{\Bbb R},{x\over 2}\alpha_{\Bbb R}\rangle \over 2}\right)}$.  
These are 
easily seen using [W-W] p. 247 and the 
partial fraction expansions of $\tanh z$ and $\coth z$ or alternatively [W-W] p.240 Cor.2 
to give $\pi\tan{\pi \over 2}\langle\alpha^{\vee}_{\Bbb R},{x\over 2}\alpha_{\Bbb 
R}\rangle$ and 
$-\pi\cot{\pi \over 2}\langle\alpha^{\vee}_{\Bbb R},{x\over 2}\alpha_{\Bbb R}\rangle$ 
respectively. The $2$ in (7.7) comes from (6.6) and (7.1).\hfill\bls
\enddemo
\remark{Remark} In [Ba p. 279] one can find an analytic treatment of $M_w$ for $SL(2,\Bbb R)$  and both cases of $\chi(f_\alpha)$ through the explicit relationship of $M_w$ to his G-function. Of course more properties of this $M_w$ can also be found there. As this might interest the reader we sketch the relationship. It follows readily from [Ba p.281]
$${d \over dz}\log G(z+1) = {1\over 2} \log 2\pi - (z + {1\over 2}) + z {d \over dz}\log \Gamma(z + 1).$$
For $SL(2,\Bbb R)$ $2p^{\lambda_w}_c ({z\over 2} \alpha_{\Bbb R})= z.$ Then from (7.7) and the above, up to an explicit entire function, $M_w$ is easily expressible in terms of $G$. As will be clearer in \S 9 this explains the appearance in [Sk] and others of the Barnes $G$-function in formulae for $\zeta$-regularized determinants of Laplacians associated to the upper half plane.
\endremark

\proclaim{Lemma 7.2} 
(1) For each ${z_0\over 2}\alpha_{\Bbb R}$ define the integer $n^{\lambda_w}_{e,P}({z_0\over 2}\alpha_{\Bbb 
R})$  by
$$n^{\lambda_w}_{e,P}({z_0\over 2}\alpha_{\Bbb 
R})=\cases 0 &if z_0\not= -(2n+1)~(\text{resp.}~-2n)\\
2p^{\lambda_w}_c({z_0\over 2}\alpha_{\Bbb R}) &if z_0=-(2n+1)~(\text{resp.}~-2n
)\endcases$$
according as $\chi(f_\alpha )=\pm 1$. Then the order of $M_w({z\over 2}\alpha_{\Bbb R})$ at ${z\over 2}\alpha_{\Bbb R}$ is
$n^{\lambda_w}_{e,P}({z\over 2}\alpha_{\Bbb R}).$

(2) A similar statement holds for 
$$\quad e^{2\pi\int^{z+ z_0}_0 p^{\lambda_w}_c({x\over 2}\alpha_{\Bbb R})\left\{\smallmatrix 
\tan{\pi \over 2}\langle\alpha^{\vee}_{\Bbb R},{x\over 2}\alpha_{\Bbb R}\rangle\\ 
\cot{\pi \over 
2}\langle\alpha^{\vee}_{\Bbb R},{x\over 2}\alpha_{\Bbb 
R}\rangle\endsmallmatrix\right\}\,dx}\,.$$
\endproclaim

\demo{Proof}  (1) Since the principal series contribution to $e^{-4t\langle\lambda,\rho_n\rangle}f^{{\hat{c}_{\lambda_w}}^2}_{2t}(e)$ is independent of $\Gamma$ and $M_w$ is 
meromorphic on $\Bbb C\,\alpha_{\Bbb R}$ a result of this type certainly holds.  To 
determine 
the precise order at $z=z_0$, notice that for $z_0\not= 0$ and $\Cal C$ a small contour around $0$
$$\eqalign{\text{order}\,M_w\mid_{z=z_0} &= {1\over 2 \pi i}\int_{\Cal C}{d \over dz}\log 
M_w({z+z_0 \over 2}\alpha_{\Bbb R})\cr
&={1\over 2 \pi i}\int_{\Cal C}2p^{\lambda_w}_c({{z+z_0}\over 2}\alpha_{\Bbb R})\left\{\matrix {\Gamma 
' \over \Gamma}\left( {1+\langle\alpha^{\vee}_{\Bbb R},{{z+z_0}\over 2}\alpha_{\Bbb 
R}\rangle \over 
2}\right)\\
{\Gamma ' \over \Gamma}\left(1+{\langle\alpha^{\vee}_{\Bbb R},{z+z_0\over 
2}\alpha_{\Bbb 
R}\rangle 
\over 2}\right)-{1 \over \langle\alpha^{\vee}_{\Bbb R},{z+z_0\over 2}\alpha_{\Bbb 
R}\rangle}
\endmatrix\right.\cr}$$
From the partial fraction expansion of $\Gamma '/\Gamma$ one can evaluate 
this limit, obtaining the result for $z_0\not= 0$.  If $z_0=0$, then 
$p^{\lambda_w}_c(0)=0$ to order 1 so that the order of $M_w$ at zero is zero.

(2) is a consequence of (1) and the functional equation in 
Lemma 7.1.
\hfill\bls
\enddemo

\remark{Remarks 7.3} (i) A different method to determine the order was used in Fried p. 
48. 

(ii) Since the poles and residues of  the $\epsilon$-logderivative $F_w$ are easily determined, one can form 
the Weierstrass product representation of $F_w$ explicitly.  This interesting 
expression 
shall be used elsewhere.  However, we point out that for 
$SL(2,\Bbb R)$ and the $\cot$ case it appears already in [W-W] p. 148, \#20.

(iii) Following the convention introduced earlier, we set   $n^{\hat{c}_{\lambda_w}^2}_{e,P} =  \sum\limits_{W_ {{c}^2}} (-1)^{\ell (w)+1}
n^{\lambda_w}_{e,P}.$
\endremark

Thus, similar to (7.3), we have
$$\lno{\qquad\qquad  \widetilde{\Cal B}_{P}(-{z\over 2}\alpha_{\Bbb R})\,= \widetilde{\Cal B}_{P}({z\over 2}\alpha_{\Bbb 
R})&e^{\pm 2\pi \sum\limits_{W_ {{c}^2}}(-1)^{\ell (w)+1}e(\xi,\lambda_w)\int^{z\over 2}_0p^{\lambda_w}_c({x\over 2}\alpha_{\Bbb R})\left\{
\smallmatrix \tan{\pi \over 2}\langle\alpha^{\vee}_{\Bbb R},{x\over 2}\alpha_{\Bbb 
R}\rangle\\ 
\cot{\pi \over 2}\langle\alpha^{\vee}_{\Bbb R},{x\over 2}\alpha_{\Bbb 
R}\rangle\endsmallmatrix\right\}\,dx}\cr
\qquad \lim_{z\to 0}\,{\widetilde{\Cal B}_{P}((z  + 2c_{\lambda_w}){\alpha_{\Bbb 
R}\over 2})\over\langle\alpha^{\vee}_{\Bbb R},{z\over 2}\alpha_{\Bbb 
R}\rangle^{n^{\hat{c}_{\lambda_w}^2}_{e,P}}}\,&=\prod_{W_ {\hat{c}^2} }\{det_{\zeta}{\Delta}^{\lambda_w}_{e,P}\, e^{\gamma 
e(\xi,\lambda_w)\tilde c^w_{0}}\}^{(-1)^{\ell (w)+1}}.
&(7.8)\cr}$$

Putting together the various contributions to $\dim F\vol X e^{-4t\langle\lambda,\rho_n\rangle}f^{\hat {c}_{\lambda_w}^2}_{2t}(e)$ we make for the bundle $\Bbb V_\lambda$ and for 
$\hat{ c}_{\lambda_w}^2$

\definition{Definition 7.4} 
Define the $\widetilde{\Cal B}$-factor associated to the elliptic elements by
$$\align \widetilde{\Cal B} (\widetilde{X},\Bbb W_{\hat{c}_{\lambda_w}^2},{z\over 2}\alpha_{\Bbb R})=& \,
 \widetilde{\Cal B}_{\cl E_2}({z\over 2}\alpha_{\Bbb R})\widetilde{\Cal B}_{P}({z\over 2}\alpha_{\Bbb R})\,\\
{\Cal B} (\widetilde{X},\Bbb W_{\hat{c}_{\lambda_w}^2},{z\over 2}\alpha_{\Bbb R})= &\,
 {\Cal B}_{\cl E_2}({z\over 2}\alpha_{\Bbb R}){\Cal B}_{P}({z\over 2}\alpha_{\Bbb R})\,.\endalign$$
\enddefinition

Summarizing the preceding discussion we have the basic properties of $\widetilde{\Cal B} 
(\widetilde{X},\Bbb 
W_{\hat{c}_{\lambda_w}^2},{z\over 2}\alpha_{\Bbb R}).$ 

\proclaim{Proposition 7.5}  
$$\align (i)&\quad \widetilde{\Cal B} (\widetilde{X},\Bbb W_{\hat{c}_{\lambda_w}^2},{z\over 2}\alpha_{\Bbb R})^{\chi_a (X,\Bbb F)}\, 
= e^{ - Pf\int^\infty_0 e^{-t(\Vert {z\over 2}\alpha_{\Bbb R}\Vert^2-\Vert 
c_{\lambda_w}\alpha_{\Bbb R}\Vert^2)}\dim F\vol Xe^{-4t\langle\lambda,\rho_n\rangle} f^{\hat c_{\lambda_w}^2}_{2t}(e){dt \over t}}\\
(ii)&\quad  \widetilde{\Cal B} (\widetilde{X},\Bbb W_{\hat{c}_{\lambda_w}^2}, \cdot) \text{ is 
meromorphic on $\Bbb C\,\alpha_{\Bbb R}$;}\\
(iii)&\quad \widetilde{\Cal B} (\widetilde{X},\Bbb W_{\hat{c}_{\lambda_w}^2},\cdot)\text{ satisfies 
the functional equation} \\
 &\widetilde{\Cal B} (\widetilde{X},\Bbb W_{\hat{c}_{\lambda_w}^2},- {z\over 2}\alpha_{\Bbb R})\,=\,\widetilde{\Cal B} 
(\widetilde{X},\Bbb 
W_{\hat{c}_{\lambda_w}^2}, {z\over 2}\alpha_{\Bbb R})\,e^{\pm 2\pi \sum\limits_{W_ {{c}^2}}(-1)^{\ell (w)+1}e(\xi,\lambda_w)\int^{z\over 2}_0p^{\lambda_w}_c({x\over 2}\alpha_{\Bbb R})\left\{
\smallmatrix \tan{\pi \over 2}\langle\alpha^{\vee}_{\Bbb R},{x\over 2}\alpha_{\Bbb 
R}\rangle\\
\cot{\pi \over 2}\langle\alpha^{\vee}_{\Bbb R},{x\over 2}\alpha_{\Bbb 
R}\rangle\endsmallmatrix\right\}\,dx}\\
(iv)&\text{ For the leading term near $c_{\lambda_w}\alpha_{\Bbb R}$ one has }\\
&\qquad\lim_{z\to 0}\,{\widetilde{\Cal B} (\widetilde{X},\Bbb W_{\hat{c}_{\lambda_w}^2},((z + 2c_{\lambda_w}){\alpha_{\Bbb R}\over 2})\over\langle\alpha^{\vee}_{\Bbb R}{z\over 2}\alpha_{\Bbb 
R}\rangle^{(n^{\hat{c}_{\lambda_w}^2}_{e,P}+n^{\hat{c}_{\lambda_w}^2}_{e,\cl E_2})}}\,= 
 \quad e^{\gamma M^{\hat{c}_{\lambda_w}^2}_2} det_{\zeta}{\Delta}^{\hat{c}_{\lambda_w}^2}_{e,\cl 
E_2}\,\prod_{W_ {\hat{c}^2} }\{det_{\zeta}{\Delta}^{\lambda_w}_{e,P}  e^{\gamma 
e(\xi,\lambda_w)\tilde c^w_{0}}\}^{(-1)^{\ell(w)+1}}\\
(v)&\text{ For the leading term near $c_{\lambda_w}\alpha_{\Bbb R}$ one has } \\
&\qquad \lim_{z\to 0}\,{\Cal B (\widetilde{X},\Bbb W_{\hat{c}_{\lambda_w}^2},((z + 2c_{\lambda_w}){\alpha_{\Bbb R}\over 2})\over\langle\alpha^{\vee}_{\Bbb R}{z\over 2}\alpha_{\Bbb 
R}\rangle^{(n^{\hat{c}_{\lambda_w}^2}_{e,P}+n^{\hat{c}_{\lambda_w}^2}_{e,\cl E_2})}}\,= 
 \quad   det_{\zeta}{\Delta}^{\hat{c}_{\lambda_w}^2}_{e,\cl E_2}\,\prod_{W_ {\hat{c}^2}} \{ det_{\zeta}{\Delta}^{\lambda_w}_{e,P} \}^{(-1)^{\ell(w)+1}}
\endalign$$
\endproclaim

\remark{Remark} Since $\widetilde{\Cal B}$ and ${\Cal B}$ differ only by a constant, $(iii)$ and $(v)$ are satisfied by ${\Cal B}$.
\endremark

%\vfill\eject
\medskip
\n {\bf \S 8~~The Semisimple Factor}
\medskip

We proceed to the analytic construction of geometric zeta functions specified on the (nontrivial) 
semisimple conjugacy classes in $\Gamma$, more specifically, owing to our ongoing assumption, supported on $\cl E_1(\Gamma)$. In Proposition 4.4 (for regular $\gamma$) and Lemma 4.5 ( for singular $\gamma$) we   computed the orbital integrals of  $\tilde k^{\lambda}_t$ as a sum over $W_\kappa$. We group the terms in this sum according to $\{w\in W_\kappa\vert \hat{c}_{\lambda_w}^2=c^2\}$. 
Then by means of the {\it Ansatz} we have heat kernels on homogeneous bundles corresponding to each of these. Besides the $\cl E_1(\Gamma)$ assumption we added the assumption that there is only one class of cuspidal maximal parabolic subgroup. Lemma 5.2 then gives a \lq\lq matching\rq\rq of orbital integrals associated to this cuspidal parabolic, namely those of the {\it Ansatz}  to the corresponding ones for $\tilde k^{\lambda}_t$. 
The Pf Laplace transform of the orbital integral expansion of $\tilde k^{\lambda}_t$ essentially will be the geometric zeta function. The Pf Laplace transform of the orbital integral expansion for terms coming from $\{w\in W_\kappa\vert \hat{c}_{\lambda_w}^2=c^2\}$ will define essentially the intermediate geometric zeta functions. Because of the \lq\lq matching\rq\rq of orbital integrals this will be independent of the {\it Ansatz}; however, the {\it Ansatz} is the tool used to obtain the convergence and meromorphic continuation as was done in [M-S;II]. The different $\hat{c}_{\lambda_w}^2$ which geometrically represent the different actions of the geodesic flow on the coefficient bundles is the ultimate cause of these complications. In the final section we assemble these intermediate geometric zeta functions into the geometric zeta function and identify leading behavior. 

A trace formula has two sides. The geometric side has just been described. The spectral side will have to be modified slightly to account for the discrete series contribution - this too is  fallout from the {\it Ansatz}.

We begin with a brief R\'esum\'e. Recall from \S 5 the constructions obtained with the use of the {\it Ansatz}. For each value of 
${\hat c_{\lambda_w}^2}$ we have defined a ``truncation'' $H^{{\hat c_{\lambda_w}^2}}_{\varphi ,t}$ acting on 
$$\sum_{\mu}\oplus \vert m^{\hat c_{\lambda_w}^2}_{\mu}\vert [L^2(\Gamma\backslash G;\varphi)\otimes 
V_{\mu}]^K$$
as the operator with Schwartz kernel 
$$h^{\hat c_{\lambda_w}^2}_{\varphi ,t}(\dot x,\dot y ) = \sum_{\gamma\in \Gamma}\varphi (
\gamma)\otimes \tilde h^{\hat c_{\lambda_w}^2}_t(y^{-1}\gamma x),\qquad \dot x = \Gamma x,\dot 
y =\Gamma y\,,$$
where
$$\align
\tilde h^{\hat c_{\lambda_w}^2}_t &=\sum_{\mu}\oplus \vert m^{\hat c_{\lambda_w}^2}_{\mu}\vert \tilde h_{\mu,t}\\
&=\,\sum_{\mu}\oplus \vert m^{\hat c_{\lambda_w}^2}_{\mu}\vert e^{{t\over 2}(\Omega_G 
-\langle\lambda^\ast,\lambda^\ast 
+2\rho\rangle I)}\\
&=\,\sum_{\mu}\oplus \vert m^{\hat c_{\lambda_w}^2}_{\mu}\vert e^{-{t\over 2}(\Delta_\mu + 
\langle\lambda^\ast 
,\lambda^\ast + 
2\rho\rangle I - \langle\mu^\ast,\mu^\ast + 2\rho\rangle I)}.
\endalign $$
Equivalently,   
$$H^{\hat c_{\lambda_w}^2}_{\varphi ,t}  = R_{\Gamma ,\varphi}(\tilde h^{\hat c_{\lambda_w}^2}_t)$$  
acting on 
$$\sum_{\mu}\oplus \vert m^{\hat c_{\lambda_w}^2}_{\mu}\vert [L^2(\Gamma\backslash G;\varphi)\otimes 
V_{\mu}]^K.$$
Also we recall the notation (5.7) for the super-trace 
$$\theta^{\hat c_{\lambda_w}^2}(t)\,=\,\Tr_s (H^{\hat c_{\lambda_w}^2}_{\varphi ,t}).$$

\remark  {Scalar Remark } The everpresent scalar  $e^{2t\langle\lambda,\rho_n\rangle}$, recall that  $4\langle\lambda,\rho_n\rangle$ originates in (1.15), is a nuisance. It is not present in the geometric \S 3. As the harmonic analysis is to be used to justify the geometric zeta function this term should not be needed. There are several ways one could proceed. As it is scalar one could multiply all the heat operators by a form of it; or using the identity $\Vert\lambda^\ast +\rho\Vert^2 + 4\langle\lambda,\rho_n\rangle = \Vert\lambda +\rho\Vert^2$ it can be absorbed with a small change in the original coefficient bundle; or as $4\langle\lambda,\rho_n\rangle = -i\lambda(H_0)$ one could take coefficient bundles on which the $K$ central element acts trivially. We will use the first option. At a suitable point we will multiply by a scalar operator and then take the trace. In particular, we will use $e^{-2t\langle\lambda,\rho_n\rangle}\theta^{\hat c_{\lambda_w}^2}(t)$.

\endremark

The notation to represent the above direct sum of operators becomes cumbersome. So we choose to use the suggestive, albeit misleading, notation 
$$e^{-{t\over 2}\square^{\hat c_{\lambda_w}^2}} :=\,\sum_{\mu}\oplus 
\vert m^{\hat c_{\lambda_w}^2}_{\mu}\vert e^{-{t\over 2}(\Delta_\mu + 
\langle\lambda^\ast ,\lambda^\ast + 2\rho\rangle I - \langle\mu^\ast,\mu^\ast + 2\rho\rangle I)},$$
and thus 
$$\theta^{\hat{ c}_{\lambda_w}^2}(t)\,=\,\Tr_s \,e^{-{t\over 2}\square^{\hat c_{\lambda_w}^2}}\, :=\, 
 \sum_{\mu}m^{\hat c_{\lambda_w}^2}_\mu 
Tr\, 
e^{-{t\over 2}(\Delta_\mu +\langle\lambda^\ast ,\lambda^\ast + 2\rho\rangle I - \langle\mu^\ast,\mu^\ast + 2\rho\rangle I)}.$$
In a similar way we denote the direct sum of Green's operators by
$$(zI + \square^{\hat c_{\lambda_w}^2})^{-1} :=\, \sum_{\mu}\oplus \vert m^{\hat c_{\lambda_w}^2}_{\mu}\vert (zI + 
\langle\lambda^\ast ,\lambda^\ast + 
2\rho\rangle I -
\langle\mu^\ast,\mu^\ast + 2\rho\rangle I + \Delta_\mu)^{-1}.$$
A minor complication that arises is that 
the operators in
$\square^{\hat c_{\lambda_w}^2}$ need not be non-negative operators. However each of the operators 
$\langle\lambda^\ast 
,\lambda^\ast 
+ 2\rho\rangle I -
\langle\mu^\ast,\mu^\ast + 2\rho\rangle I + \Delta_\mu$ has spectrum on $[L^2(\Gamma\backslash 
G;\varphi)\otimes 
V_{\mu}]^K$ bounded below. 
Using operators for which $m^{\hat c_{\lambda_w}^2}_{\mu}\ne 0$ we define 
$\mu^{\hat c_{\lambda_w}^2}_{b}$ by
$$\mu^{\hat c_{\lambda_w}^2}_{b}\,=\,\min_\mu \min\{ Spec(\Delta_\mu + \langle\lambda^\ast 
,\lambda^\ast 
+ 2\rho\rangle I - \langle\mu^\ast,\mu^\ast + 2\rho\rangle I)\}.$$ 

Since the construction and analytic properties of $\theta$-regularized determinants 
in [M-S;II] applies to each of the operators $e^{-{t\over 2}\Delta_\mu}$, we may apply the 
procedure to the finitely many involved in $e^{-{t\over 2}\square^{\hat c_{\lambda_w}^2}}$, and hence to ${\theta}^{\hat c_{\lambda_w}^2}(t)$. For $t\to 
0^+$, notice that
$\theta^{\hat c_{\lambda_w}^2}(t)=\Tr_s H^{\hat c_{\lambda_w}^2}_{\varphi ,t}$ has an asymptotic expansion. 
Indeed, from Corollary 
5.7 
$$\theta^{\hat c_{\lambda_w}^2}(t)\,\sim \text{dim F vol X }f^{\hat c_{\lambda_w}^2}_{t}(e),~~~t\to 
0^+.$$
From the fact that $e^{-2t\langle\lambda,\rho_n\rangle}\sim 1, ~~~t\to 0^+$, 
and from Proposition 6.3 we see that $e^{-2t\langle\lambda,\rho_n\rangle}f^{\hat c_{\lambda_w}^2}_{t}(e)$ is a pseudofunction and thus so is $e^{-2t\langle\lambda,\rho_n\rangle}\theta^{\hat c_{\lambda_w}^2}(t)$.

Since 
$$\theta^{\hat c_{\lambda_w}^2} (t) =  \sum_{\mu}m^{\hat c_{\lambda_w}^2}_\mu Tr\, 
e^{-{t\over 2}(\Delta_\mu + 
\langle\lambda^\ast ,\lambda^\ast + 2\rho\rangle I - \langle\mu^\ast,\mu^\ast + 2\rho\rangle I)},$$
writing 
$$L^2(\Gamma\bs G;\varphi)\, =\,\sum m_iH_{\pi_i}$$
and recalling (5.9) one has 
$$\align
\theta^{\hat c_{\lambda_w}^2} (t) &= \sum m_i\sum_{\mu}m^{\hat c_{\lambda_w}^2}_{\mu}\dim 
[H_{\pi_i} \otimes 
V_\mu]^K e^{-t\lambda_i}\\
&=\sum_ie^{-t\lambda_i}m_i e^{\hat c_{\lambda_w}^2}_{M_Q}(\pi_i,V_{\lambda}).
\endalign$$
Consequently the behavior at infinity is suitable to evaluate a Laplace transform, 
certainly for 
$\Re z+\mu^{\hat c_{\lambda_w}^2}_{b}>0$. 

To handle zero modes notice that on $[L^2(\Gamma\backslash 
G;\varphi)\otimes 
V_{\mu}]^K$ the operator $\Delta_\mu + 
\langle\lambda^\ast,\lambda^\ast+2\rho\rangle I - \langle \mu^\ast,\mu^\ast + 2\rho\rangle I$ 
has 
finite 
dimensional kernel, with associated projection operator say $P_\mu$. Set 
$$\square^{\hat{c}_{\lambda_w}^2}_+\,:=\,\sum\oplus \vert m^{\hat c_{\lambda_w}}_\mu \vert (\Delta_\mu + 
\langle\lambda^\ast,\lambda^\ast+2\rho\rangle I - \langle \mu^\ast,\mu^\ast + 
2\rho\rangle I)(I-P_\mu),$$
and let $G^{\hat{c}_{\lambda_w}^2}_+$ denote the inverse of $\square^{\hat{c}_{\lambda_w}^2}_+$ on $\sum\oplus 
\vert m^{\hat{c}_{\lambda_w}^2}_\mu \vert (I-P_\mu)[L^2(\Gamma\backslash G;\varphi)\otimes V_{\mu}]^K$, 
i.e. 
the sum of operators $\sum\oplus \vert m^{\hat c_{\lambda_w}}_\mu \vert (\Delta_\mu + 
\langle\lambda^\ast,\lambda^\ast+2\rho\rangle I - \langle \mu^\ast,\mu^\ast + 
2\rho\rangle I)^{-1}$.  To subtract the zero 
modes we use
$$\align
\theta^{\hat{c}_{\lambda_w}^2}_+(t)\,&:=\, Tr_se^{-{t\over 2}\square^{\hat{c}_{\lambda_w}^2}_+}\\
&=\theta^{\hat{c}_{\lambda_w}^2}(t) - \sum_{\mu}m^{\hat{c}_{\lambda_w}^2}_\mu Tr P_\mu\\
&:=\theta^{\hat{c}_{\lambda_w}^2}(t) - N^{\hat{c}_{\lambda_w}^2}\tag8.1
\endalign$$
noting that each $P_\mu$ has finite rank and only finitely many 
$m^{\hat{c}_{\lambda_w}^2}_\mu\ne 
0$. We note that $ N^{\hat{c}_{\lambda_w}^2}$ is dependent on the {\it Ansatz}.  (The presence of zero modes is addressed and handled in [D-P].)

Hence 
$$Pf\int^\infty_0 e^{-tp}\theta^{\hat{c}_{\lambda_w}^2}(t)\,dt\,= \, Pf\int^1_0 
e^{-tp}\theta^{\hat{c}_{\lambda_w}^2}(t)\,dt + 
\int^\infty_1  e^{-tp}\theta^{\hat{c}_{\lambda_w}^2} (t)\,dt\,$$
is defined and meromorphic on $\Bbb C$. Indeed, notice that from the $Pf$ Laplace 
transform theory the first integral on the right side gives an entire function in $p$ and the second 
is
$$\Tr_s ((pI + \square^{\hat{c}_{\lambda_w}^2}_+)^{-1} e^{-(\square^{\hat{c}_{\lambda_w}^2}_+ + pI)}) + {N^{\hat{c}_{\lambda_w}^2} e^{-p}\over 
p}$$
which is meromorphic in $p$. Following [M-S;II] we can then introduce the heuristic 
notation
$$\Tr_\theta (pI+\square^{\hat{c}_{\lambda_w}^2})^{-1} := Pf\int^\infty_0 e^{-tp}\Tr_s e^{-t\square^{\hat{c}_{\lambda_w}^2}}
\,dt\,= \,Pf\int^\infty_0 e^{-tp}\theta^{\hat{c}_{\lambda_w}^2}(2t)\,dt,$$

here, as in \S 6, we replace the  `t' parameter with `2t'. When we multiply by $e^{-4t\langle\lambda,\rho_n\rangle}$ the result is just a linear shift in $p$ (again the symbolic calculus). Thus $\Tr_\theta (pI+\square^{\hat{c}_{\lambda_w}^2})^ {-1}\,$
is analytic for $\Re p + \mu^{\hat{c}_{\lambda_w}^2}_{b}>0$ and has meromorphic continuation to $\Bbb C$.  Similarly for $\square^{\hat{c}_{\lambda_w}^2}_+$.

\proclaim{Lemma 8.1} 
Let $p = \Vert (s + c_{\lambda_w})\alpha_{\Bbb R} \Vert^2 - \Vert 
c_{\lambda_w}\alpha_{\Bbb R}\Vert^2$. Then we have the expansion uniformly convergent for $\Re 
(s^2+2sc_{\lambda_w})\Vert \alpha_{\Bbb R} \Vert^2 +\mu^{\hat{c}_{\lambda_w}^2}_{b}>0$
$$\lno{\qquad\qquad\qquad &\qquad {1\over\Vert \alpha_{\Bbb R}\Vert }\sum\limits_{W_ {{c}^2}}(-1)^{\ell (w) +1}\sum_{[\gamma]\in \cl E_1(\Gamma)}
\Tr\varphi (\gamma ){(-1)^{q_{\gamma, \perp}} \widehat{L}(\gamma ;\lambda_w )  \over  \xi_{\rho_{Q}}(\gamma)\prod\limits_{\alpha\in\Delta_{\Bbb R}^+\cup \Delta_{\Bbb C}^+}[1-\xi_{-\alpha}(\gamma)]} \lambda_{\gamma^*} e^{-(s+c_{\lambda_w})\alpha_{\Bbb R} (\ell_\gamma \widehat{X}_\kappa)}\,&(8.2)\cr
 &=\langle\alpha_{\Bbb R}^\vee,(s+c_{\lambda_w})\alpha_{\Bbb 
R}\rangle Pf\int^\infty_0 
e^{-t( \Vert(s+c_{\lambda_w})\alpha_{\Bbb R}\Vert^2 -\Vert 
c_{\lambda_w} \alpha_{\Bbb R}\Vert^2)}e^{-4t\langle\lambda,\rho_n\rangle}[\theta^{\hat{c}_{\lambda_w}^2}(2t)-\text{dim F vol X } f^{\hat{c}_{\lambda_w}^2}_{2t}(e)]dt. \cr}$$
\endproclaim

\demo{Proof}  We begin with a number of remarks. Basic to the proof is the familiar 
integral used by Bochner in the principle of subordination
$${e^{-wx} \over 2w} = \int^\infty_0 e^{-tw^2} {e^{-x^2
/4t} \over (4\pi t)^{1/2}}dt\,,~~\qquad~~\Re w\not\in (-\infty ,0]\,.$$
Notice that with $w^2 \,=\, \Vert(s+c_{\lambda_w})\alpha_{\Bbb R}\Vert^2$ 
and $x^2\, 
=\,\Vert\ell_\gamma \widehat{X}_\kappa\Vert^2$ the integral has value 
$${e^{-(s+c_{\lambda_w})\ell_\gamma\Vert\alpha_{\Bbb R}\Vert} \over  2(s+c_{\lambda_w})\Vert\alpha_{\Bbb 
R}\Vert}.$$ As $c_{\lambda_w}$ in general can be positive or negative but takes on only finitely many real values, it is clear that for some slit plane $\Re (s+c_{\lambda_w})\not\in (-\infty ,0]$ uniformly in $c_{\lambda_w}$. 
From the discussion circa (4.12) we see that $\Vert\alpha_{\Bbb R}\Vert\Vert X_\kappa\Vert\,=\,1\,=\,\alpha_{\Bbb R}(X_\kappa).$  Also, from (6.5)  $c_{\lambda_w}\alpha_{\Bbb R} =((\lambda_w\circ C_\kappa)+\rho_{Q_{\kappa}})(X_{\kappa})\alpha_{\Bbb R} = ((\lambda_w\circ C_\kappa)+\rho_{Q_{\kappa}})
 (\widehat{X}_{\kappa}){\alpha_{\Bbb R}\over \Vert\alpha_{\Bbb R}\Vert} = \widehat{c}_{\lambda_w}{\alpha_{\Bbb R}\over \Vert\alpha_{\Bbb R}\Vert} .$
Then the integral can be written as 
$${e^{-(s+c_{\lambda_w})\alpha_{\Bbb 
R}(\ell_\gamma \widehat{X}_\kappa)}\over  2(s+c_{\lambda_w})\Vert\alpha_{\Bbb 
R}\Vert}= {e^{-(s+c_{\lambda_w})\alpha_{\Bbb 
R}(\ell_\gamma \widehat{X}_\kappa)}\over  2(s\Vert\alpha_{\Bbb R}\Vert+\widehat{c}_{\lambda_w})}.$$

We need to compare the geometric expression with the orbital integral expression. So from (5.11) we have
$$\eqalign{& \qquad\qquad \theta^{\hat c_{\lambda_w}^2}(t)-\text{dim F vol X 
}f^{\hat c_{\lambda_w}^2}_t(e)\,=e^{2t\langle\lambda,\rho_n\rangle}e^{-{t\over 2}\hat c^2_{\lambda_w}}\cr
& \cdot \sum_{[\gamma]\in 
{\Cal E_1(\Gamma)}}\Tr\varphi(\gamma)\vol (\Gamma_\gamma \backslash G_\gamma ) {e^{-\Vert\log 
\gamma_{\Bbb R}\Vert^2/2t}\over (2\pi
t)^{1/2}}{N_{\gamma}^{-1}c^{-1}_{\gamma} \over  \xi_{\rho_{Q}}(\gamma)\prod\limits_{\alpha\in
P^c_{\gamma}}[1-\xi_{-\alpha}(\gamma)]}\vert W(G^0_h,A^0_I)\vert \cr
&\cdot\sum_{{w\in W_\kappa,\hat{c}_{\lambda_w}^2= \,c^2}}(-1)^{\ell(w)+1}\sum_{W(K\cap M^+, A_I)/W(G^0_h,A^0_I)}\varepsilon(\sigma)\wt{\omega}_{\gamma}(\sigma(\lambda_w + 
\rho_M))\xi_{\sigma(\lambda_w + \rho_M) - \rho_M}(\gamma_I),\cr}$$

so 
$$\eqalign{&  e^{t( \Vert c_{\lambda_w} \alpha_{\Bbb R}\Vert^2)}e^{-4t\langle\lambda,\rho_n\rangle}\theta^{\hat{c}_{\lambda_w}^2}(2t)-\text{dim F vol X }
e^{t( \Vert c_{\lambda_w} \alpha_{\Bbb R}\Vert^2)} e^{-4t\langle\lambda,\rho_n\rangle}f^{\hat c_{\lambda_w}^2}_{2t}(e)\,= e^{t( \Vert c_{\lambda_w} \alpha_{\Bbb R}\Vert^2)}e^{-t\hat c^2_{\lambda_w}}\cr
& \cdot \sum_{[\gamma]\in 
{\Cal E_1(\Gamma)}}\Tr\varphi(\gamma)\vol (\Gamma_\gamma \backslash G_\gamma )  {e^{-\Vert\log 
\gamma_{\Bbb R}\Vert^2/4t}\over (4\pi
t)^{1/2}}{N_{\gamma}^{-1}c^{-1}_{\gamma} \over  \xi_{\rho_{Q}}(\gamma)\prod\limits_{\alpha\in
P^c_{\gamma}}[1-\xi_{-\alpha}(\gamma)]}\vert W(G^0_h,A^0_I)\vert \cr
&\cdot\sum_{{w\in W_\kappa,\hat{c}_{\lambda_w}^2= \,c^2}}(-1)^{\ell(w)+1}\sum_{W(K\cap M^+, A_I)/W(G^0_h,A^0_I)}\varepsilon(\sigma)\wt{\omega}_{\gamma}(\sigma(\lambda_w + 
\rho_M))\xi_{\sigma(\lambda_w + \rho_M) - \rho_M}(\gamma_I),\cr}$$

then, using the subordination integral as described above and that $\langle\alpha_{\Bbb R}^\vee,(s+c_{\lambda_w})\alpha_{\Bbb 
R}\rangle  = 2(s + c_{\lambda_w})$, the left side of (8.2) becomes 
$$\eqalign{& {1\over\Vert \alpha_{\Bbb R}\Vert }\sum_{[\gamma]\in 
{\Cal E_1(\Gamma)}}\Tr\varphi(\gamma)\vol (\Gamma_\gamma \backslash G_\gamma ) e^{-(s+c_{\lambda_w})\alpha_{\Bbb 
R}(\ell_\gamma \widehat{X}_\kappa)}  {N_{\gamma}^{-1}c^{-1}_{\gamma} \over \xi_{\rho_{Q}}(\gamma) \prod\limits_{\alpha\in
P^c_{\gamma}}[1-\xi_{-\alpha}(\gamma)]}\vert W(G^0_h,A^0_I)\vert \cr
&\cdot\sum_{{w\in W_\kappa,\hat{c}_{\lambda_w}^2= \,c^2}}(-1)^{\ell(w)+1}\sum_{W(K\cap M^+, A_I)/W(G^0_h,A^0_I)}\varepsilon(\sigma)\wt{\omega}_{\gamma}(\sigma(\lambda_w + 
\rho_M))\xi_{\sigma(\lambda_w + \rho_M) - \rho_M}(\gamma_I).\cr}$$

Next from (6.5) 
$$c_{\lambda_w}\alpha_{\Bbb R} =(\lambda_w\circ C_\kappa+\rho_{Q_{\kappa}})(X_{\kappa})\alpha_{\Bbb R}$$
and, from  (1.17), $-\lambda_w\circ C_\kappa (\widehat{X}_\kappa ) =  \lambda_w\circ C_\kappa^{-1}(\widehat{X}_\kappa )$ then
$$\eqalign{-c_{\lambda_w}\alpha_{\Bbb R} &= -((\lambda_w\circ C_\kappa)+\rho_{Q_{\kappa}})( X_{\kappa})\alpha_{\Bbb R}=((\lambda_w\circ C_\kappa^{-1})-\rho_{Q_{\kappa}})(X_{\kappa})\alpha_{\Bbb R}\cr
-c_{\lambda_w}\alpha_{\Bbb R}(\ell_\gamma \widehat {X}_\kappa)   &= (\lambda_w\circ C_\kappa^{-1})(X_{\kappa})\alpha_{\Bbb 
R}(\ell_\gamma \widehat{X}_\kappa)-\rho_{Q_{\kappa}}( X_{\kappa})\alpha_{\Bbb R}(\ell_\gamma \widehat{X}_\kappa)\cr
&= (\lambda_w\circ C_\kappa^{-1})(X_{\kappa})\ell_\gamma\Vert\alpha_{\Bbb 
R}\Vert -\rho_{Q_{\kappa}}( X_{\kappa})\ell_\gamma\Vert\alpha_{\Bbb R}\Vert \cr
&= (\lambda_w\circ C_\kappa^{-1})(\ell_\gamma\widehat{X}_{\kappa}) -\rho_{Q_{\kappa}}(\ell_\gamma\widehat{X}_{\kappa})\cr} $$

here $e^{(\lambda_w\circ C_\kappa)(\ell_\gamma \widehat{X}_{\kappa})}$ is the 
action on bundles computed for (3.8). Then substituting into the previous line we have for the left side of (8.2)
$$\eqalign{& e^ {\lambda_w\circ C_\kappa^{-1}(\ell_\gamma \widehat{X}_\kappa)}  {1\over\Vert \alpha_{\Bbb R}\Vert }\sum_{[\gamma]\in 
{\Cal E_1(\Gamma)}}\Tr\varphi(\gamma)\vol (\Gamma_\gamma \backslash G_\gamma ) e^{-(s \alpha_{\Bbb R}+\rho_{Q_{\kappa}} )(\ell_\gamma \widehat{X}_\kappa)}  {N_{\gamma}^{-1}c^{-1}_{\gamma} \over  \xi_{\rho_{Q}}(\gamma)\prod\limits_{\alpha\in
P^c_{\gamma}}[1-\xi_{-\alpha}(\gamma)]}\cr
&\cdot\vert W(G^0_h,A^0_I)\vert \sum_{{w\in W_\kappa,\hat{c}_{\lambda_w}^2= \,c^2}}(-1)^{\ell(w)+1}\sum_{W(K\cap M^+, A_I)/W(G^0_h,A^0_I)}\varepsilon(\sigma)\wt{\omega}_{\gamma}(\sigma(\lambda_w + 
\rho_M))\xi_{\sigma(\lambda_w + \rho_M) - \rho_M}(\gamma_I).\cr}$$
As described before Lemma 4.5, $P_\gamma \subset P(\frak{m}_{Q_\kappa}) \text{ so } P_\gamma^c = \Delta_I^c \cup P(\frak{m}_{Q_\kappa})\backslash 
P_\gamma$,  and under our assumption $\Delta_I^c = \Delta(\frak u) \cup \Delta(\frak v).$  Also therein is that $\rho_{Q_\kappa}(\gamma) = \rho_{Q_\kappa}(\gamma_{\Bbb R})$. Then continuing 
$$\lno{  \qquad   &\qquad\qquad{1\over\Vert \alpha_{\Bbb R}\Vert }\sum_{{w\in W_\kappa,\hat{c}_{\lambda_w}^2= \,c^2}}(-1)^{\ell(w)+1}\sum_{[\gamma]\in 
{\Cal E_1(\Gamma)}}{\Tr\varphi(\gamma) \over \xi_{\rho_{Q}}(\gamma)  \prod\limits_{\alpha\in\Delta_{\Bbb R}^+\cup \Delta_{\Bbb C}^+}[1-\xi_{-\alpha}(\gamma)]} {\lambda_{\gamma^*}^{-1}\vol (\Gamma_\gamma \backslash G_\gamma ) N_{\gamma}^{-1}c^{-1}_{\gamma} \over \prod\limits_{\alpha\in P(\frak{m}_{Q_\kappa})\backslash 
P_\gamma}[1-\xi_{-\alpha}(\gamma_I)]}&(8.3) \cr
&\quad\cdot\vert W(G^0_h,A^0_I)\vert \sum_{W(K\cap M^+, A_I)/W(G^0_h,A^0_I)}\varepsilon(\sigma)\wt{\omega}_{\gamma}(\sigma(\lambda_w + 
\rho_M))\xi_{\sigma(\lambda_w + \rho_M) - \rho_M}(\gamma_I)  \cr
&\qquad\qquad\qquad \lambda_{\gamma^*} e^ {\lambda_w\circ C_\kappa^{-1} (\ell_\gamma \widehat{X}_\kappa)} e^{-(s \alpha_{\Bbb R}+\rho_{Q_{\kappa}} )(\ell_\gamma \widehat{X}_\kappa)} 
.\cr}$$

For the geometric computation, from (3.10), (3.11), Proposition 3.1, and the notation after (4.11)

$$\eqalign{   \chi_{\gamma,w} ({\Bbb V}) &= (-1)^{q_{\gamma, \perp}}e^{\lambda_w\left(C_{\kappa}(\ell_\gamma  \widehat{X}_\kappa )\right)}
\widehat{L}(\gamma ;\lambda_w ) \cr
e^ {\lambda_w\circ C_\kappa^{-1} (\ell_\gamma \widehat{X}_\kappa)}  \chi_{\gamma,w} ({\Bbb V}) & =(-1)^{q_{\gamma, \perp}} \widehat{L}(\gamma ;\lambda_w ) \cr
&=  (-1)^{q_{\gamma, \perp}} \,^0L'(\gamma ;\lambda_w ){\vol (\Gamma_\gamma
\backslash G_\gamma ) \over \vol (^0\Gamma '_\gamma\backslash G'_\gamma )
\lambda_{\gamma^*}}\cr
 &=  \,{\vol (\Gamma_\gamma
\backslash G_\gamma ) \over \vol (^0\Gamma '_\gamma\backslash G'_\gamma )
\lambda_{\gamma^*}}  [G_\gamma:G_\gamma^0]^{-1}v(^0\Gamma 
'_\gamma\backslash 
G'_\gamma ) \vert 
{W(\frak g_\gamma,\frak h_{\psi_1})}\vert^{-1}\cr
&\cdot\prod_{\alpha\in P_{\gamma_I}}
\langle\rho_{\gamma_I},\alpha\rangle^{-1} {\sum\limits_{\sigma\in W(K\cap M^+_{Q_\kappa}) /W_\gamma}\varepsilon (\sigma)e^{\sigma
(\lambda_w +\rho_M) -\rho_M}(\gamma_I)\prod\limits_{\alpha\in P_\gamma}
\langle \sigma(\lambda_w +\rho_M),\alpha\rangle \over \prod\limits_{\alpha\in 
P(\frak{m}_{Q_\kappa})\backslash P_\gamma}(1-e^{-\alpha}(\gamma_I))}\,,\cr
&=   \,\vol (\Gamma_\gamma\backslash G_\gamma )\cdot\lambda_{\gamma^*}^{-1}[G_\gamma:G_\gamma^0]^{-1}(2\pi )^{-\# P_{\gamma}}
2^{-\nu_{\gamma}/2}v(K'_\gamma )v(A_{I})^{-1} \cr
& \cdot \sum_{\sigma\in W(K\cap M^+_{Q_\kappa})/W_\gamma}\varepsilon (\sigma)
e^{\sigma(\lambda_w+\rho_{M})- \rho_M}(\gamma_I){\prod\limits_{\alpha\in P_\gamma}\langle \sigma(\lambda_w
+\rho_{M}), \alpha\rangle\over \prod\limits_{\alpha\in P(\frak{m}_{Q_\kappa})\backslash 
P_\gamma}(1-e^{-\alpha}(\gamma_I))}\,\cr
&=  \,\vol (\Gamma_\gamma\backslash G_\gamma )\cdot\lambda_{\gamma^*}^{-1}{N_{\gamma}^{-1}c_{\gamma}^{-1}\over \prod\limits_{\alpha\in P(\frak{m}_{Q_\kappa})\backslash 
P_\gamma}(1-e^{-\alpha}(\gamma_I))}\vert W(G^0_\gamma,A^0_I)\vert\cr
&\cdot 
 \sum_{\sigma\in W(K\cap M^+_{Q_\kappa})/W_\gamma}\varepsilon (\sigma)
\wt{\omega}_{\gamma}(\sigma(\lambda_w + 
\rho_M))e^{\sigma(\lambda_w+\rho_{M})- \rho_M}(\gamma_I)\,,\cr}$$
where we use the notation following (4.11).

A comparison of this with (8.3) gives the formula in (8.2). We still need to justify an interchange of integrals and the convergence.
The technique to justify the interchange was introduced in [M-S;I, 
Proposition 6.1], and can be used here.  So we just sketch the argument.

As in (8.1) for fixed $T>0$, we set
$$I_T= \langle\alpha_{\Bbb R}^\vee,(s+c_{\lambda_w})\alpha_{\Bbb 
R}\rangle\,Pf\int^T_0 
e^{-t(\Vert(s+c_{\lambda_w})\alpha_{\Bbb R}\Vert^2 -\Vert 
c_{\lambda_w}\alpha_{\Bbb 
R}\Vert^2)}e^{-4t\langle\lambda,\rho_n\rangle}\theta^{\hat{c}_{\lambda_w}^2}(2t)\,dt$$
and
$$I_\infty = \langle\alpha_{\Bbb R}^\vee,(s+c_{\lambda_w})\alpha_{\Bbb 
R}\rangle\int^\infty_T   
e^{-t(\Vert(s+c_{\lambda_w})\alpha_{\Bbb R}\Vert^2 -\Vert 
c_{\lambda_w}\alpha_{\Bbb 
R}\Vert^2)}e^{-4t\langle\lambda,\rho_n\rangle}\theta^{\hat{c}_{\lambda_w}^2}(2t)\,dt\,.$$
Then 
$$\eqalign{I_\infty &=  \langle\alpha_{\Bbb R}^\vee,(s+c_{\lambda_w})\alpha_{\Bbb R}\rangle\Tr_s  ((s^2+2sc_{\lambda_w})\Vert\alpha_{\Bbb 
R}\Vert^2 I + 
\square^{\hat{c}_{\lambda_w}^2})^{-1}e^{-T(\square^{\hat{c}_{\lambda_w}^2} + (\Vert(s+c_{\lambda_w})\alpha_{\Bbb 
R}\Vert^2 
-\Vert c_{\lambda_w}\alpha_{\Bbb R}\Vert^2)I)}\,.\cr}$$
Since $\Re (s^2+2sc_{\lambda_w})\Vert\alpha_{\Bbb R}\Vert^2+\mu^{\hat{c}_{\lambda_w}^2}_{b}>0$, the Green's 
operator 
$G^{\hat{c}_{\lambda_w}^2}_+$ is smoothing.  
Also, as explained in \S 5,  the operator $H^{\hat{c}_{\lambda_w}^2}_{\varphi ,t}$ 
is a 
virtual direct sum of heat kernels $\tilde h_{\mu,t}$ which at time $T$, 
are certainly in $[{\Cal S}(G)\otimes \END (V_\mu)]^{K^M
\times K^M}$ times a smoothing operator.  Thus the local trace of the 
Schwartz kernel is in ${\Cal S}(G)$, hence admissible.  Finally, since trace 
class operators are an ideal and the Green's operator is bounded, 
the operator is trace class.  Since the Green's operator acts by a scalar in 
any irreducible representation, the invariant Fourier transform of the operator will 
factor in such a way that we may use Lemma 5.4 to obtain that $I_\infty$ 
will have 
the absolutely convergent orbital integral expansion
$$\lno{&\qquad\qquad I_\infty= \sum\limits_{W_ {{c}^2}}(-1)^{\ell (w)+1}\sum_{[\gamma ]\in \cl E_1(\Gamma)}\Tr \varphi (
\gamma )N^{-1}_\gamma c^{-1}_\gamma ~{\vol (
\Gamma_\gamma\backslash G_\gamma ) \over\xi_{\rho_{Q}}(\gamma) \prod\limits_{\alpha\in P^c_\gamma} 
[1-\xi_{-\alpha}(\gamma )]} &(8.4)\cr
&\quad \vert W(G^0_h,A^0_I)\vert 
 \sum_{W(K\cap M^+, A_I)/W(G^0_h,A^0_I)}\varepsilon(\sigma)\wt{\omega}_{\gamma}(\sigma(\lambda_w + 
\rho_M))\xi_{\sigma(\lambda_w + \rho_M) - \rho_M}(\gamma_I) \cr
 &\qquad\qquad\int_{\frak a^{\ast}_{\frak q}}{ \langle\alpha_{\Bbb R}^\vee,(s+c_{\lambda_w})\alpha_{\Bbb 
R}\rangle 
\over \Vert(s+c_{\lambda_w})\alpha_{\Bbb R}\Vert^2+\Vert\nu\Vert^2}e^{-T(\Vert(s+c_{\lambda_w})\alpha_{\Bbb R}\Vert^2+\Vert\nu\Vert^2)}
e^{i\nu(\ell_\gamma X_\kappa)}\,d\nu\cr
&+\text{identity~term}\,.\cr}$$

For the $I_T$ term, Lemma 5.7 shows that if the contribution of the identity term is 
omitted from 
$e^{-4t\langle\lambda,\rho_n\rangle}\theta^{\hat{c}_{\lambda_w}^2}(2t)$, then the resulting series is uniformly, absolutely 
convergent.  
Hence the finite integral in $t$ may be interchanged with the summation giving
$$\lno{\qquad\qquad I_T & = \sum\limits_{W_ {{c}^2}}(-1)^{\ell (w)+1}\sum_{[\gamma ]\in \cl E_1(\Gamma)}\Tr \varphi (
\gamma )N^{-1}_\gamma c^{-1}_\gamma ~{\vol (
\Gamma_\gamma\backslash G_\gamma ) \over\xi_{\rho_{Q}}(\gamma) \prod\limits_{\alpha\in P^c_\gamma} 
[1-\xi_{-\alpha}(\gamma )]} &(8.5)\cr
& \vert W(G^0_h,A^0_I)\vert 
 \sum_{W(K\cap M^+, A_I)/W(G^0_h,A^0_I)}\varepsilon(\sigma)\wt{\omega}_{\gamma}(\sigma(\lambda_w + 
\rho_M))\xi_{\sigma(\lambda_w + \rho_M) - \rho_M}(\gamma_I) \cr
&\quad \int^T_0  \langle\alpha_{\Bbb R}^\vee,(s+c_{\lambda_w})\alpha_{\Bbb 
R}){e^{-\Vert\ell_\gamma \widehat{
X}_\kappa\Vert^2 
/4t} \over (4\pi 
t)^{1/2}}e^{-t\Vert(s+c_{\lambda_w})\alpha_{\Bbb 
R}\Vert^2}\,dt\cr
&\quad +\text{identity~term}\,.\cr}$$
The expansion in (8.2) results from interchanging the $t$ and $\nu$ integrations in 
(8.4) and then adding the absolutely convergent expansions (8.4) and (8.5).  
Linearity of the finite part gives the identity term. \hfill\bls
\enddemo
\remark{Remark - {\it Ansatz dependence}}We note that the left side of (8.2) is geometric and not dependent on the {\it Ansatz}. The proof shows that the series converges to give a holomorphic function in an appropriate domain of $\Bbb C$. The Pf calculus applied to the right side produces a meromorphic continuation to $\Bbb C$. While the {\it Ansatz} is used to produce the right side, by uniqueness of meromorphic continuation it will be independent of this as it extends the holomorphic left side.
\endremark
It is time to nail down the various constants that have not yet been specified. This 
will be  done by matching asymptotic expansions. Recall that 
$$\square^{\hat{c}_{\lambda_w}^2}_+\,:=\,\sum_{\mu}\oplus \vert m^{\hat{c}_{\lambda_w}^2}_\mu\vert (\Delta_\mu + 
 \langle\lambda^\ast,\lambda^\ast+2\rho\rangle -  \langle \mu^\ast,\mu^\ast + 
2\rho\rangle)(I-P_\mu),$$
and set $$\log\text{det}_\zeta \square^{\hat{c}_{\lambda_w}^2}_+\,:=\, \sum_{\mu}\oplus m^{\hat{c}_{\lambda_w}^2}_\mu 
\log\text{det}_\zeta [(\Delta_\mu +  \langle\lambda^\ast,\lambda^\ast+2\rho\rangle -  \langle 
\mu^\ast,\mu^\ast + 2\rho\rangle)(I-P_\mu)].$$

We note the $\log\text{det}_\zeta \square^{\hat{c}_{\lambda_w}^2}_+$ is dependent on the {\it Ansatz}.

Recall that $G^{\hat{c}_{\lambda_w}^2}_+$ is the inverse of $\square^{\hat{c}_{\lambda_w}^2}_+$ on $\sum\oplus 
(I-P_\mu)[L^2(\Gamma\backslash G;\varphi)\otimes V_{\mu}]^K.$  We subtract the zero 
modes as before
$$\align
\theta^{\hat{c}_{\lambda_w}^2}_+(t)\,&=\theta^{\hat{c}_{\lambda_w}^2}(t) - 
\sum_{\mu}m^{\hat{c}_{\lambda_w}^2}_\mu 
Tr P_\mu\\
&=\theta^{\hat{c}_{\lambda_w}^2}(t) - N^{\hat{c}_{\lambda_w}^2}.
\endalign$$

The symbolic calculus of $Pf$ Laplace transforms gives the relationship
$$\log\text{det}_\theta (I+(z+4\langle\lambda,\rho_n\rangle)G^{\hat{c}_{\lambda_w}^2}_+) = -Pf\int^\infty_0 
e^{-tz}e^{-4t\langle\lambda,\rho_n\rangle}\theta^{\hat{c}_{\lambda_w}^2}_+(2t) 
{ dt \over t}+C\,. $$
To determine the constant $C$, notice that
$$\eqalign{e^{-4t\langle\lambda,\rho_n\rangle}\theta^{\hat{c}_{\lambda_w}^2}_+(2t) &= e^{-4t\langle\lambda,\rho_n\rangle}\theta^{\hat{c}_{\lambda_w}^2}(2t)- e^{-4t\langle\lambda,\rho_n\rangle} N^{\hat{c}_{\lambda_w}^2}\cr
&= \text{dim F vol X }e^{-4t\langle\lambda,\rho_n\rangle}f^{\hat{c}_{\lambda_w}^2}_{2t}(e)- e^{-4t\langle\lambda,\rho_n\rangle} N^{\hat{c}_{\lambda_w}^2}\cr
&\quad +[e^{-4t\langle\lambda,\rho_n\rangle}\theta^{\hat{c}_{\lambda_w}^2}(2t)-\text{dim F vol X }e^{-4t\langle\lambda,\rho_n\rangle}f^{\hat{c}_{\lambda_w}^2}_{2t}(e)],\cr}$$
and from Lemma 5.6 one finds the term in brackets is $O(e^{-c/t})$, $c>0$, for 
$t\to 0^+$. As $ e^{-4t\langle\lambda,\rho_n\rangle}\sim 1,t\to 0^+$ and the nonpositive powers in the asymptotic expansions, $t\to 0^+$, of $e^{-4t\langle\lambda,\rho_n\rangle}\theta^{\hat{c}_{\lambda_w}^2}_+(2t)$ and 
$\text{dim 
F}\vol 
X \,e^{-4t\langle\lambda,\rho_n\rangle}f^{\hat{c}_{\lambda_w}^2}_{2t}(e) - N^{\hat{c}_{\lambda_w}^2}$ agree, then the asymptotic expansions, for large 
$z$, of 
$Pf\int^\infty_0 
e^{-tz}[\text{dim F vol X }e^{-4t\langle\lambda,\rho_n\rangle}f^{\hat{c}_{\lambda_w}^2}_{2t}(e)- e^{-4t\langle\lambda,\rho_n\rangle} N^{\hat{c}_{\lambda_w}^2}] { dt \over t}$ and $Pf\int^\infty_0 e^{-tz}e^{-4t\langle\lambda,\rho_n\rangle}\theta^{\hat{c}_{\lambda_w}^2}_+(2t) { dt \over t}$  must be 
equal, and also agree 
up to the constant with that of $- \log\text{det}_\theta (I+(z+4\langle\lambda,\rho_n\rangle)G^{\hat{c}_{\lambda_w}^2}_+)$.  Suppose that $t \to 0^+$
$$\text{dim F vol X }e^{-4t\langle\lambda,\rho_n\rangle}f^{\hat{c}_{\lambda_w}^2}_{2t}(e)- e^{-4t\langle\lambda,\rho_n\rangle} N^{\hat{c}_{\lambda_w}^2} \sim \quad\sum\limits_{-p}^{\infty}b_k t^{k}.
$$
As only integral powers of $t$ occur in the asymptotics 
as $t\to 0^+$, from (1.12) in [M-S;II] it follows that $z \to \infty$
$$\eqalign{\log\text{det}_\theta (I+(z+4\langle\lambda,\rho_n\rangle)G^{\hat{c}_{\lambda_w}^2}_+) & \sim \quad b_0 \log z \,-\log\text{det}_\zeta \square^{\hat{c}_{\lambda_w}^2}_+\,\cr
& + \sum\limits _{k=1}^{p} b_{-k} \left( 
\gamma+ \log z-\sum^k_{j=1} {1 \over j}\right) {(-z)^k \over k!}.\cr}$$
In particular, the only coefficient of a term of the form $\log^n z$ is $b_0$.  Now $b_0$ is the sum of contributions to the identity term from the discrete series ({\it Ansatz} dependent) which were denoted in (7.1) as $ \chi_a(X,\Bbb F)\, M^{\hat{c}_{\lambda_w}^2}_2\,$; contributions from the principal series (not {\it Ansatz} dependent) which follow from Proposition 6.3 (ii) and (7.4) to be $ \chi_a(X,\Bbb F)\, \sum\limits_{W_ {{c}^2}}(-1)^{\ell (w)+1}e(\xi,\lambda_w)\tilde c_0^w $; and the zero modes denoted $ - N^{\hat{c}_{\lambda_w}^2}$. 
Hence 
from 
the symbolic calculus for antiderivatives we get 
$$\eqalign{Pf\int^\infty_0 &e^{-tz}e^{-4t\langle\lambda,\rho_n\rangle}\theta^{\hat{c}_{\lambda_w}^2}_+(2t){dt \over 
t}\,\,=\,-\log\text{det}_\theta 
(I+(z+4\langle\lambda,\rho_n\rangle)G^{\hat{c}_{\lambda_w}^2}_+)\,\cr
&-\gamma  [ \chi_a(X,\Bbb F)\, M^{\hat{c}_{\lambda_w}^2}_2+ \chi_a(X,\Bbb F)\, \sum\limits_{W_ {{c}^2}}(-1)^{\ell (w)+1}e(\xi,\lambda_w)\tilde c_0^w  \,- N^{\hat{c}_{\lambda_w}^2}\,]-\log\text{det}_\zeta\square^{\hat{c}_{\lambda_w}^2}_+.\cr}\tag8.6$$  

\definition{Definition 8.2}  For $z\in\Bbb C$ with $\Re (\Vert (z +c_{\lambda_w})\alpha_{\Bbb R}\Vert^2\, -\, \Vert c_{\lambda_w}\alpha_{\Bbb R}\Vert^2 + \mu^{\hat{c}_{\lambda_w}^2}_{b}>0$ 
and $\Re z\text { in an appropriate slit plane}$,  
define $\log Z^{\hat{c}_{\lambda_w}^2}_\varphi $ by the absolutely convergent 
series
$$\quad\log Z^{\hat{c}_{\lambda_w}^2}_\varphi (z+c_{\lambda_w}) =-\sum\limits_{W_ {{c}^2}}(-1)^{\ell (w)+1}\sum_{[\gamma ]\in\cl E_1(\Gamma)} \Tr \varphi (\gamma )
{(-1)^{q_{\gamma, \perp}} \widehat{L}(\gamma ;\lambda_w ) \over \xi_{\rho_{Q}}(\gamma)\prod\limits_{\alpha\in\Delta_{\Bbb R}^+\cup \Delta_{\Bbb C}^+} [1-\xi_{-\alpha}(\gamma)]} {e^{-(z + c_{\lambda_w} )\alpha_{\Bbb R}(\ell_\gamma \widehat{X}_\kappa)} \over \mu_\gamma}\,.$$
\enddefinition
\remark{Remark}  Since $\xi_{\rho_Q}(\gamma) = \xi_{\rho_Q}(\gamma_{\Bbb R})$, the convergence in 
this domain follows from Lemma 8.1. Also, as the conjugacy classes are bounded away from zero, $\lim\limits_{z\to +\infty}\log Z^{\hat{c}_{\lambda_w}^2}_\varphi (z)=0$. 
\endremark

Using from \S4 
$$\vert \text{Det}(I - Ad(h^{-1}) |_{\frak g/\frak m^1} \vert  = \xi_{2\rho_{Q}}(h)\prod\limits_{\alpha\in \Delta_{\Bbb R}^+\cup \Delta_{\Bbb C}^+} [1 - \xi_{\alpha} (h^{-1})]^2,$$
we have an alternative Definition 8.3 that is in complete analogy with that in [M-S;I; (6.7)].

\definition{Definition 8.3}  For $s\in\Bbb C$ with $\Re (\Vert (s\alpha_{\Bbb R}\Vert^2\, -\, \Vert c_{\lambda_w}\alpha_{\Bbb R}\Vert^2 + \mu^{\hat{c}_{\lambda_w}^2}_{b}>0$ 
and 
$\Re s\text { in an appropriate slit plane}$,  
define $\log Z^{\hat{c}_{\lambda_w}^2}_\varphi  (s) $ by the absolutely convergent series
$$\quad\log Z^{\hat{c}_{\lambda_w}^2}_\varphi (s) =-\sum\limits_{\{w\in W_\kappa:\hat{c}_{\lambda_w}^2=c^2\}}(-1)^{\ell (w)+1}\sum_{[\gamma ]\in\cl E_1(\Gamma)} \Tr \varphi (\gamma )
{(-1)^{q_{\gamma, \perp}}\widehat{L}(\gamma ;\lambda_w )\over \vert \text{Det}(I - Ad(h^{-1}) |_{\frak g/\frak m^1} \vert^{1\over 2}} {e^{-s\alpha_{\Bbb R}(\ell_\gamma \widehat{X}_\kappa)} \over \mu_\gamma}\,.$$
\enddefinition

\proclaim{Theorem 8.4}  \hfill

(i) $Z^{\hat{c}_{\lambda_w}^2}_\varphi$ has a meromorphic 
continuation to $\Bbb C\,\alpha_{\Bbb R}$ given by
$$\align\qquad Z^{\hat{c}_{\lambda_w}^2}_\varphi (z)\quad =&\quad\text{det}_\zeta \square^{\hat{c}_{\lambda_w}^2}_+\,\,(\Vert z\alpha_{\Bbb R}\Vert^2- \Vert c_{\lambda_w}\alpha_{\Bbb R}\Vert^2 +4\langle\lambda,\rho_n\rangle)^{ N^{\hat{c}_{\lambda_w}^2}}\\
&\qquad \Cal B (\widetilde X,\Bbb W_{\hat{c}_{\lambda_w}^2},  z\alpha_{\Bbb R})^{-\chi_a(X,\Bbb F )}\text{det}_\theta (I+ (\Vert z \alpha_{\Bbb R}\Vert^2 
-\Vert c_{\lambda_w}\alpha_{\Bbb R}\Vert^2 +4\langle\lambda,\rho_n\rangle)G^{\hat{c}_{\lambda_w}^2}_+)\tag8.7\endalign 
$$

(ii) $Z^{\hat{c}_{\lambda_w}^2}_\varphi$ satisfies the functional equation
$$\eqalign{\qquad Z^{\hat{c}_{\lambda_w}^2}_\varphi (-z) \, = 
 \, Z^{\hat{c}_{\lambda_w}^2}_\varphi (z) e^{\mp 2\pi\chi_a
(X,\Bbb F )\sum\limits_{W_ {{c}^2}}(-1)^{\ell (w)+1}e(\xi,\lambda_w)\tilde c_0^w \int^{z}_0 p^{\lambda_w}_c(x{\alpha_{\Bbb R}\over 2})\left\{\smallmatrix \tan{\pi 
\over 2} x\\ 
\cot{\pi \over 2} x\endsmallmatrix\right\}\,dx}\cr}\tag8.8$$
\endproclaim

\demo{Proof} (i)  For $\Re (\Vert(s+c_{\lambda_w})\alpha_{\Bbb R}\Vert^2)- \Vert 
c_{\lambda_w}\alpha_{\Bbb R}\Vert^2 + \mu^{\hat{c}_{\lambda_w}^2}_{b}>0$ from the definition of $\log Z^{\hat{c}_{\lambda_w}^2}_\varphi$ and Lemma 
8.1 we get
$$\align&\log Z^{\hat{c}_{\lambda_w}^2}_\varphi (s + c_{\lambda_w}) =-\sum\limits_{W_ {{c}^2}}(-1)^{\ell (w)+1}\sum_{[\gamma ]\in\cl E_1(\Gamma)} {\Tr \varphi (\gamma )\over \xi_{\rho_{Q}}(\gamma)}
{(-1)^{q_{\gamma, \perp}} \widehat{L}(\gamma ;\lambda_w )  \over \prod\limits_{\alpha\in\Delta_{\Bbb R}^+\cup \Delta_{\Bbb C}^+}[1-\xi_{-\alpha}(\gamma)]} {e^{-(s + c_{\lambda_w})\alpha_{\Bbb R}(\ell_\gamma \widehat{X}_\kappa)} \over \mu_\gamma}, \text{so}\\
&{d \over ds}\log Z^{\hat{c}_{\lambda_w}^2}_\varphi(s + c_{\lambda_w}) = \sum\limits_{W_ {{c}^2}}(-1)^{\ell (w)+1}\sum_{[\gamma ]\in\cl E_1(\Gamma)} {\Tr \varphi (\gamma )\over \xi_{\rho_{Q}}(\gamma)}{(-1)^{q_{\gamma, \perp}} \widehat{L}(\gamma ;\lambda_w )  \over \prod\limits_{\alpha\in\Delta_{\Bbb R}^+\cup \Delta_{\Bbb C}^+}[1-\xi_{-\alpha}(\gamma)]} {\ell_\gamma \Vert\alpha_{\Bbb R}\Vert \over \mu_\gamma}e^{-(s + c_{\lambda_w})\alpha_{\Bbb R} (\ell_\gamma\widehat{X}_{\kappa})} \\
 &\qquad\qquad\qquad\qquad =\Vert\alpha_{\Bbb R}\Vert\sum\limits_{W_ {{c}^2}}(-1)^{\ell (w)+1}\sum_{[\gamma ]\in\cl E_1(\Gamma)} {\Tr \varphi (\gamma )\over \xi_{\rho_{Q}}(\gamma)}
{(-1)^{q_{\gamma, \perp}} \widehat{L}(\gamma ;\lambda_w )  \over \prod\limits_{\alpha\in\Delta_{\Bbb R}^+\cup \Delta_{\Bbb C}^+}[1-\xi_{-\alpha}(\gamma)]} \lambda_{\gamma^*} e^{-(s + c_{\lambda_w})\alpha_{\Bbb R} (\ell_\gamma\widehat{X}_{\kappa})} \\
 = &\Vert \alpha_{\Bbb R}\Vert^2\langle\alpha_{\Bbb R}^\vee,(s+c_{\lambda_w})\alpha_{\Bbb R}\rangle Pf\int^\infty_0 e^{-t( \Vert(s+c_{\lambda_w})\alpha_{\Bbb R}\Vert^2 -\Vert 
c_{\lambda_w} \alpha_{\Bbb R}\Vert^2)}e^{-4t\langle\lambda,\rho_n\rangle}[\theta^{\hat{c}_{\lambda_w}^2}(2t)-\text{dim F vol X } f^{\hat{c}_{\lambda_w}^2}_{2t}(e)]dt\\
= &\quad 2(s+c_{\lambda_w})\Vert \alpha_{\Bbb R}\Vert^2 Pf\int^\infty_0 e^{-t( \Vert(s+c_{\lambda_w})\alpha_{\Bbb R}\Vert^2 -\Vert 
c_{\lambda_w} \alpha_{\Bbb R}\Vert^2)}e^{-4t\langle\lambda,\rho_n\rangle}[\theta^{\hat{c}_{\lambda_w}^2}(2t)-\text{dim F vol X } f^{\hat{c}_{\lambda_w}^2}_{2t}(e)]dt.
\endalign$$
As $\lim\limits_{z\to +\infty}\log Z^{\hat{c}_{\lambda_w}^2}_\varphi (z)=0$ the symbolic calculus (cf. \S 7) gives 
$$\align&\log Z^{\hat{c}_{\lambda_w}^2}_\varphi (s + c_{\lambda_w}) =-Pf\int^\infty_0 e^{-t( \Vert(s+c_{\lambda_w})\alpha_{\Bbb R}\Vert^2 -\Vert 
c_{\lambda_w} \alpha_{\Bbb R}\Vert^2)}e^{-4t\langle\lambda,\rho_n\rangle}[\theta^{\hat{c}_{\lambda_w}^2}(2t)-\text{dim F vol X } f^{\hat{c}_{\lambda_w}^2}_{2t}(e)]{dt\over t}\\
&\qquad\qquad\text { and using (8.1)}\\
&\qquad\qquad=-Pf\int^\infty_0 e^{-t( \Vert(s+c_{\lambda_w})\alpha_{\Bbb R}\Vert^2 -\Vert 
c_{\lambda_w} \alpha_{\Bbb R}\Vert^2)}e^{-4t\langle\lambda,\rho_n\rangle}\theta_+^{\hat{c}_{\lambda_w}^2}(2t){dt\over t}\\
&\qquad\qquad\qquad+Pf\int^\infty_0 e^{-t( \Vert(s+c_{\lambda_w})\alpha_{\Bbb R}\Vert^2 -\Vert 
c_{\lambda_w} \alpha_{\Bbb R}\Vert^2)}e^{-4t\langle\lambda,\rho_n\rangle}[\text{dim F vol X } f^{\hat{c}_{\lambda_w}^2}_{2t}(e)-N^{\hat{c}_{\lambda_w}^2}] { dt \over t}.
\endalign$$

Then from (8.6)
$$\align
&\log Z^{\hat{c}_{\lambda_w}^2}_\varphi ( s + c_{\lambda_w}) =  \log\text{det}_\theta (I+ (\Vert(s+c_{\lambda_w})\alpha_{\Bbb R}\Vert^2 -\Vert 
c_{\lambda_w} \alpha_{\Bbb R}\Vert^2+4\langle\lambda,\rho_n\rangle)G^{\hat{c}_{\lambda_w}^2}_+) + C\\
 +&\ Pf \int^\infty_0 e^{-t ( \Vert(s+c_{\lambda_w})\alpha_{\Bbb R}\Vert^2 -\Vert 
c_{\lambda_w} \alpha_{\Bbb R}\Vert^2)}[{\dim F\vol X e^{-4t\langle\lambda,\rho_n\rangle}f^{\hat{c}_{\lambda_w}^2}_{2t}(e) - e^{-4t\langle\lambda,\rho_n\rangle}N^{\hat{c}_{\lambda_w}^2}}] 
{ dt \over t}\\
=  & \log\text{det}_\theta (I+ (\Vert(s+c_{\lambda_w})\alpha_{\Bbb R}\Vert^2 -\Vert 
c_{\lambda_w} \alpha_{\Bbb R}\Vert^2 +4\langle\lambda,\rho_n\rangle)G^{\hat{c}_{\lambda_w}^2}_+)   -\chi_a (X,\Bbb F)  \log\widetilde{\Cal B} (\widetilde{X},\Bbb W_{\hat{c}_{\lambda_w}^2},(s +c_{\lambda_w})\alpha_{\Bbb R})\\
&\qquad\qquad-\ Pf \int^\infty_0 e^{-t( \Vert(s+c_{\lambda_w})\alpha_{\Bbb R}\Vert^2 -\Vert c_{\lambda_w} \alpha_{\Bbb R}\Vert^2)}e^{-4t\langle\lambda,\rho_n\rangle}N^{\hat{c}_{\lambda_w}^2} { dt \over t} + C\\
&= \log\text{det}_\theta (I+ (\Vert(s+c_{\lambda_w})\alpha_{\Bbb R}\Vert^2 -\Vert 
c_{\lambda_w} \alpha_{\Bbb R}\Vert^2+4\langle\lambda,\rho_n\rangle)G^{\hat{c}_{\lambda_w}^2}_+)   -\chi_a (X,\Bbb F)  \log\widetilde{\Cal B} (\widetilde{X},\Bbb W_{\hat{c}_{\lambda_w}^2},(s +c_{\lambda_w})\alpha_{\Bbb R})\\
&+\log ( \Vert(s+c_{\lambda_w})\alpha_{\Bbb R}\Vert^2- \Vert c_{\lambda_w}\alpha_{\Bbb R}\Vert^2 +4\langle\lambda,\rho_n\rangle)^{N^{\hat{c}_{\lambda_w}^2}}+\gamma N^{\hat{c}_{\lambda_w}^2} + C,
\endalign$$
 $C= \gamma \chi_a(X,\Bbb F) [ \, M^{\hat{c}_{\lambda_w}^2}_2+ \, \sum\limits_{W_ {{c}^2}}(-1)^{\ell (w)+1}e(\xi,\lambda_w)\tilde c_0^w]  -\gamma N^{\hat{c}_{\lambda_w}^2}+\log\text{det}_\zeta\square^{\hat{c}_{\lambda_w}^2}_+$ (cf. (8.6)).

Hence we obtain
$$ 
\log Z^{\hat{c}_{\lambda_w}^2}_\varphi( s + c_{\lambda_w})  = \log\text{det}_\theta (I+ (\Vert (s +  c_{\lambda_w})\alpha_{\Bbb R} \Vert^2- 
\Vert c_{\lambda_w}\alpha_{\Bbb R}\Vert^2 +4 \langle\lambda,\rho_n\rangle)G^{\hat{c}_{\lambda_w}^2}_+)\, +\,\log\text{det}_\zeta \square^{\hat{c}_{\lambda_w}^2}_+  $$
$$\qquad +\gamma \chi_a(X,\Bbb F)[M^{\hat{c}_{\lambda_w}^2}_2\,+\,\sum\limits_{W_ {{c}^2}}(-1)^{\ell (w)+1}e(\xi,\lambda_w)\tilde {c}^w_{0}] - \log \widetilde{\Cal B} (\widetilde X,\Bbb W_{\hat{c}_{\lambda_w}^2},(s + c_{\lambda_w})\alpha_{\Bbb R} )^{\chi_a(X,\Bbb F )}$$
$$\qquad\qquad +\log ( \Vert(s+c_{\lambda_w})\alpha_{\Bbb R}\Vert^2- \Vert c_{\lambda_w}\alpha_{\Bbb R}\Vert^2 +4\langle\lambda,\rho_n\rangle)^{N^{\hat{c}_{\lambda_w}^2}} \,   .$$

The result follows then from the relationship of $\widetilde{\Cal B}$ with $\Cal B$, (6.5) and the substitution $z\to (s + c_{\lambda_w})$.

(ii) This functional equation is an obvious consequence of the one for $\widetilde{\Cal B}$ given in Proposition 7.5.\hfill\bls
\enddemo

\remark{Remark} As suggested by the referee, we clarify the relationship of $Z^{\hat{c}_{\lambda_w}^2}_\varphi (z)$ with Selberg's when $G$ is $SU(n,1)$. In this case $K^M = K\cap M^+_Q = M^+_Q $ and from the Appendix the strong form of the {\it Ansatz} is available. This means that we may handle one at a time the $W_{\lambda_w}), w \in W_\kappa$. With trivial coefficients or holomorphic p-form coefficients this can be compared to the Selberg function on [F2; p.41]. But for choice of Haar measure they agree. To compare functional equations, i.e. (8.8) with (FE) in [F2; p.41], one uses Definition 8.2 and puts $\xi_{\rho_{Q}}(\gamma)$ in the numerator then observes that they agree, once one identifies that $a_1$ in (FE) as defined [F2; p.37] is, up to Haar measure, the arithmetic genus $\chi_a(X,\Bbb F)$, with our $p^{\lambda_w}_c$ containing $deg(\sigma)$  in $a_1$. An important distinction is that the product form in (8.7) which is obtained by use of the Pf symbolic calculus contains more explicit information which one must extract from Selberg's. The elliptic term $\widetilde{\Cal B} (\widetilde X,\Bbb W_{\hat{c}_{\lambda_w}^2},(s + c_{\lambda_w})\alpha_{\Bbb R} )$ is significant - analogous to the Archimedean place contribution to a completed zeta function - and, as has been pointed out earlier, is in $SL(2,\Bbb R)$ very much a Barnes G-function whose type is known and makes the type of $Z^{\hat{c}_{\lambda_w}^2}_\varphi (z)$ easily computed. For $SL(2,\Bbb R)$, because of the leading behavior, the function $Z^{\hat{c}_{\lambda_w}^2}_\varphi (z)$ can be used to compute $\zeta$-regularized determinants of Laplacians on bundles (cf. [D-P] and [Sk]). Because the \lq\lq zero modes \rq\rq have been factored out, the leading behaviour contains the determinant value, whereas when just using Selberg the \lq\lq zero modes\rq\rq  are not omitted resulting in the need of derivatives of the Selberg function to remove the effect of the term $(\Vert z\alpha_{\Bbb R}\Vert^2- \Vert c_{\lambda_w}\alpha_{\Bbb R}\Vert^2 +4\langle\lambda,\rho_n\rangle)^{ N^{\hat{c}_{\lambda_w}^2}}$ to obtain the determinant (cf. [D-P] and [Sk]).
\endremark

%\vfill\eject

%Finally, notice that with $\gamma = \gamma_I e^{\ell_{\gamma} \widehat{X}_{\kappa}}$ and $Y=  \widehat{X}_{\kappa}$
%$${1 \over \prod\limits_{\alpha\not\in
%\Delta_{\Bbb I}}[1-\xi_{-\alpha}(\gamma)]} = {(-1)^{q_\gamma}\xi_{\rho_{Q_{\kappa}}}(\gamma)\xi_{-\rho_{M,c}}(\gamma_I)\over\det %
%(I-\gamma\mid {\frak v}_{\Bbb C}) \det (I-\gamma\mid {\frak u}_{\Bbb C})}
%= {(-1)^{q_\gamma} \xi_{-\rho_{M,c}}(\gamma_I)e^{\rho_{Q_{\kappa}}( X_{\kappa})\alpha_{\Bbb R}(\ell_\gamma 
%\widehat{X}_{\kappa})}\over\det 
%(I-\gamma\mid {\frak v}_{\Bbb C}) \det (I-\gamma\mid {\frak u}_{\Bbb C})}. $$

\n{\bf \S 9~~Torsion and the Geometric Zeta Function}

We may now consider the Weil-type zeta function, $Z_{\Bbb V ,\Bbb F}
(z)$ from \S2, and examine its behavior for small $z$.  We recall that besides the support condition of $\cl E_1(\Gamma)$ we 
have the ongoing assumption that $G$ has only one class of cuspidal maximal 
parabolic subgroups and relevant to this class an appropriate form of the {\it Ansatz}.  The material from the earlier sections was developed to allow us to express the Weil-type zeta function in terms of the intermediate geometric zeta functions - from this the analytic properties follow. The detailed asymptotic information of the various heat kernels $t\to 0^+$ will give information for z near zero
of $Z_{\Bbb V ,\Bbb F} (z)$. So from \S2 
$$Z_{\Bbb V ,\Bbb F}(z) =\exp -\sum_{[\gamma ]\in\cl E_1(\Gamma)}\Tr \varphi (\gamma )
\chi_\gamma (\Bbb V){e^{-z \alpha_{\Bbb R}(\ell_\gamma \widehat{X}_\kappa)} \over \mu_\gamma}.$$

For $\Re z^2$ suitably large, using  (3.12)  and Def. 8.3 we have
$$\eqalign{ \log &Z_{\Bbb V ,\Bbb F}(z) =-\sum_{[\gamma ]\in\cl E_1(\Gamma)}\Tr \varphi 
(\gamma )\chi_\gamma (\Bbb V){e^{-z \alpha_{\Bbb R}(\ell_\gamma \widehat{X}_\kappa)} \over \mu_\gamma}\cr
&= -\sum_{[\gamma ]\in\cl E_1(\Gamma)}\Tr\varphi (\gamma )\sum_{w\in  W_\kappa}(-1 )^{\ell
(w)+1}{\chi_{\gamma,w} ({\Bbb V})\over \det (I-\Ad\gamma^{-1}\mid{\frak v}) \det
(I-\Ad\gamma^{-1}\mid{\frak u})}{e^{-z \alpha_{\Bbb R}(\ell_\gamma \widehat{X}_\kappa)}\over \mu_\gamma} \cr
&=    -\sum_{[\gamma ]\in\cl E_1(\Gamma)}\Tr\varphi (\gamma )\sum_{c^2} \sum_{{w\in W_\kappa,\hat{c}_{\lambda_w}^2=c^2}}(-1)^{\ell(w)+1} {\chi_{\gamma,w} ({\Bbb V})\over \det (I-\Ad\gamma^{-1}\mid{\frak v}) \det
(I-\Ad\gamma^{-1}\mid{\frak u})}{e^{-z \alpha_{\Bbb R}(\ell_\gamma \widehat{X}_\kappa)}\over \mu_\gamma} \cr
&=   -\sum_{c^2}\sum_{{w\in W_\kappa,\hat{c}_{\lambda_w}^2=c^2}}(-1)^{\ell(w)+1} \sum_{[\gamma ]\in\cl E_1(\Gamma)}\Tr\varphi (\gamma ) {\chi_{\gamma,w} ({\Bbb V})\over \det (I-\Ad\gamma^{-1}\mid{\frak v}) \det
(I-\Ad\gamma^{-1}\mid{\frak u})}{e^{-z \alpha_{\Bbb R}(\ell_\gamma \widehat{X}_\kappa)}\over \mu_\gamma} \cr
&= -\sum_{c^2} \sum_{{w\in W_\kappa,\hat{c}_{\lambda_w}^2=c^2}}(-1)^{\ell(w)+1}\sum_{[\gamma ]\in\cl E_1(\Gamma)}\Tr\varphi (\gamma ) {\chi_{\gamma,w} ({\Bbb V})e^{\lambda_w\left(C^{-1}_{\kappa}(\ell_\gamma Y)\right)}
\over  \vert \text{Det}(I - Ad(h^{-1}) |_{\frak g/\frak m^1} \vert^{1\over 2}}{e^{\lambda_w\left(C_{\kappa}(\ell_\gamma Y)\right)}\xi_{\rho_Q}(\gamma_{\Bbb R})
e^{-z \alpha_{\Bbb R}(\ell_\gamma \widehat{X}_\kappa)}\over \mu_\gamma} \cr
&= -\sum_{c^2} \sum_{{w\in W_\kappa,\hat{c}_{\lambda_w}^2=c^2}}(-1)^{\ell(w)+1}\sum_{[\gamma ]\in\cl E_1(\Gamma)}\Tr\varphi (\gamma ) {(-1)^{q_{\gamma, \perp}}\widehat{L}(\gamma ;\lambda_w )\over  \vert \text{Det}(I - Ad(h^{-1}) |_{\frak g/\frak m^1} \vert^{1\over 2}}{ e^{-(z - c_{\lambda_w})\alpha_{\Bbb R}(\ell_\gamma \widehat{X}_\kappa)}\over \mu_\gamma} \cr
&= \sum_{c^2}\log Z^{\hat{c}_{\lambda_w}^2}_\varphi (z - c_{\lambda_w}).\cr}$$
Hence
$$Z_{\Bbb V ,\Bbb F} (z) = \prod_{\hat{c}_{\lambda_w}^2=c^2}Z^{\hat{c}_{\lambda_w}^2}_\varphi (z- c_{\lambda_w}) .$$
From Lemma 8.1 we have that  $log$ of each factor is absolutely convergent for $\Re z^2$ suitably large, so the geometric zeta function is holomorphic in a suitable domain in $\Bbb C$.  By Theorem 8.4 each factor is meromorphic on $\Bbb C$, thus we 
obtain that $Z_{\Bbb V ,\Bbb F}(z\alpha_{\Bbb R})$ is meromorphic on $\Bbb C$, and the meromorphic continuation is independent of the {\it Ansatz}.

We are interested in the behavior for $z$ near $0$. We could use the functional equation (8.8) and Proposition 7.5 but to isolate individual contributions from discrete series, principal series and the {\it Ansatz} it should be clearer to return to the defining expressions. So 

$$Z_{\Bbb V ,\Bbb F}(z)\, =\, \prod_{\hat{c}_{\lambda_w}^2} Z^{\hat{c}_{\lambda_w}^2}_\varphi(z- c_{\lambda_w}) .$$
Then for each $W_{c^2}:=\{w\in W_\kappa,\hat{c}_{\lambda_w}^2=c^2\}$ we have from Theorem 8.4 (i)
$$\align &Z^{\hat{c}_{\lambda_w}^2}_\varphi (z-c_{\lambda_w}) = \text{det}_\zeta \square^{\hat{c}_{\lambda_w}^2}_+(\Vert (z-c_{\lambda_w})\alpha_{\Bbb R}\Vert^2- \Vert c_{\lambda_w}\alpha_{\Bbb R}\Vert^2+4\langle\lambda,\rho_n\rangle)^{ N^{\hat{c}_{\lambda_w}^2}}\,\\
&\qquad\qquad 
\Cal B (\widetilde X,\Bbb W_{\hat{c}_{\lambda_w}^2},(2z-2c_{\lambda_w}){\alpha_{\Bbb R}\over 2}))^{-\chi_a(X,\Bbb F )}\text{det}_\theta (I+ (\Vert (z-c_{\lambda_w})\alpha_{\Bbb R}\Vert^2- \Vert 
c_{\lambda_w}\alpha_{\Bbb R }\Vert^2 +4\langle\lambda,\rho_n\rangle)G^{\hat{c}_{\lambda_w}^2}_+). \endalign$$

For $z\sim 0$, the{ \it Ansatz} \lq\lq zero modes\rq\rq give  $(\Vert (z-c_{\lambda_w})\alpha_{\Bbb R}\Vert^2- \Vert c_{\lambda_w}\alpha_{\Bbb R}\Vert^2+4\langle\lambda,\rho_n\rangle)^{ N^{\hat{c}_{\lambda_w}^2}} \sim (4\langle\lambda,\rho_n\rangle)^{ N^{\hat{c}_{\lambda_w}^2}}, $ and  from the symbolic calculus normalization the semisimple factor gives $\text{det}_\theta (I+ (\Vert (z-c_{\lambda_w})\alpha_{\Bbb R}\Vert^2- \Vert 
c_{\lambda_w}\alpha_{\Bbb R}\Vert^2)G^{\hat{c}_{\lambda_w}^2}_+)\sim 1,$ then $\text{det}_\theta (I+ (\Vert (z-c_{\lambda_w})\alpha_{\Bbb R}\Vert^2- \Vert 
c_{\lambda_w}\alpha_{\Bbb R }\Vert^2 +4\langle\lambda,\rho_n\rangle)G^{\hat{c}_{\lambda_w}^2}_+) \sim \text{det}_\theta (I+ (4\langle\lambda,\rho_n\rangle)G^{\hat{c}_{\lambda_w}^2}_+).$  Now the elliptic factor $${\Cal B} (\widetilde{X},\Bbb W_{\hat{c}_{\lambda_w}^2},{s\over 2}\alpha_{\Bbb R})= \,
 {\Cal B}_{\cl E_2}({s\over 2}\alpha_{\Bbb R}){\Cal B}_{P}({s\over 2}\alpha_{\Bbb R})\,.$$
 The discrete series term is given in (7.2) as $${\Cal B}_{\cl E_2}({s\over 2}\alpha_{\Bbb R})\,=det_{\zeta}{\Delta}^{\hat{c}_{\lambda_w}^2}_{e,\cl E_2}\quad(\Vert {s\over 2}\alpha_{\Bbb 
R}\Vert^2-\Vert c_{\lambda_w}
\alpha_{\Bbb R}\Vert^2)^{n^{\hat{c}_{\lambda_w}^2}_{e,\cl E_2}}\,\,\text{det}_\theta \left(I+(\Vert 
{s\over 2}\alpha_{\Bbb R}\Vert^2-\Vert c_{\lambda_w}
\alpha_{\Bbb R}\Vert^2)G_{\cl E_2}^{\hat{c}_{\lambda_w}^2}\right).$$
For $z$ near $0$, $s = 2(z-c_{\lambda_w})\sim -2c_{\lambda_w}$ the $\text{det}_\theta\sim 1,$ and the ${\cl E_2}$ term has a zero of order ${n^{\hat{c}_{\lambda_w}^2}_{e,\cl E_2}}$ and leading coefficient $det_{\zeta}{\Delta}^{\hat{c}_{\lambda_w}^2}_{e,\cl E_2}$.

The principal series term 
$${\Cal B}_{P}({s\over 2}\alpha_{\Bbb R})\,=
\prod_{W_ {{c}^2} }\, \{ e^{e(\xi,\lambda_w)\widetilde{p}_{\lambda_w}
({s\over 2}\alpha_{\Bbb R})}det_{\zeta}{\Delta}^{\lambda_w}_{e,P}\, M_w({s\over 2}\alpha_{\Bbb R})^{e(\xi,\lambda_w)}\}^{(-1)^{\ell (w)+1}}$$
has for $s = 2(z-c_{\lambda_w})\sim -2c_{\lambda_w},\,  \widetilde{p}_{\lambda_w}({s\over 2}\alpha_{\Bbb R})\sim 0$ from its normalization. So it remains to identify the leading coefficient of $M_w({s\over 2}\alpha_{\Bbb R})$ and its order. These can be obtained from the functional equation (7.6) and Lemma 7.2. Now  $ M_w({s\over 2}\alpha_{\Bbb R})$ was normalized to have leading coefficient $1$ at $c_{\lambda_w}\alpha_{\Bbb R}$, so the leading coefficient of  $ M_w({s\over 2}\alpha_{\Bbb R})$ at $-c_{\lambda_w}\alpha_{\Bbb R}$ comes from the exponential in the functional equation, namely (cf.(7.6)),
$$e^{\pm 2\pi\int^{c_{\lambda_w}}_0p^{\lambda_w}_c({x\over 2}\alpha_{\Bbb R})\left\{
\smallmatrix \tan{\pi \over 2}\langle\alpha^{\vee}_{\Bbb R},{x\over 2}\alpha_{\Bbb 
R}\rangle\\ 
\cot{\pi \over 2}\langle\alpha^{\vee}_{\Bbb R},{x\over 2}\alpha_{\Bbb 
R}\rangle\endsmallmatrix\right\}\,dx}.$$
The polynomial in the integrand is given in (7.1) as $ p^{\lambda_w}_c({x\alpha_{\Bbb R}\over 2}) = 2(-1)^{q-q_M}{\varpi(\Lambda_{\xi(w)}-{{x\alpha_{\Bbb R}}\over 2}) \over \varpi(\rho )}$
and recall that $ c_{\lambda_w} = ((\lambda_w\circ C_\kappa) + \rho_{Q_{\kappa}})( X_{\kappa}).$

From Lemma 7.2 we have the order of  $ M_w({s\over 2}\alpha_{\Bbb R})$ at $-c_{\lambda_w}\alpha_{\Bbb R}$ given by
 $$n^{\lambda_w}_{e,P}(-c_{\lambda_w}\alpha_{\Bbb R})=\cases 0 & if -c_{\lambda_w} \not= -(2n+1)~(\text{resp.}~-2n)\\
2p^{\lambda_w}_c({-2c_{\lambda_w}\over 2}\alpha_{\Bbb R} )&if -c_{\lambda_w}=-(2n+1)~(\text{resp.}~-2n)\endcases$$
according as $\chi(f_\alpha )=\pm 1$. Then $2p^{\lambda_w}_c({-2c_{\lambda_w}\over 2}\alpha_{\Bbb R} )= 4(-1)^{q-q_M}{\varpi(\Lambda_{\xi(w)}-(-c_{\lambda_w}\alpha_{\Bbb R} )) \over \varpi(\rho )}.$

If $\langle\lambda,\rho_n\rangle\not = 0$, putting together the various leading coefficients for $W_{c^2}$ we obtain
$$\,\text{det}_\zeta \square^{\hat{c}_{\lambda_w}^2}_+\,\text{det}_\theta (I+ (4\langle\lambda,\rho_n\rangle)G^{\hat{c}_{\lambda_w}^2}_+)\,(4\langle\lambda,\rho_n\rangle)^{ N^{\hat{c}_{\lambda_w}^2}} \,\,\biggl\{\,det_{\zeta}{\Delta}^{\hat{c}_{\lambda_w}^2}_{e,\cl 
E_2}\,det_{\zeta}{\Delta}^{\hat{c}_{\lambda_w}^2}_{e,P}\biggr\}^{-\chi_a(X,\Bbb F )}\,\, ,$$
where  $det_{\zeta}{\Delta}^{\hat{c}_{\lambda_w}^2}_{e,P}:=\prod\limits_{W_ {\hat{c}^2} }\biggl\{det_{\zeta}{\Delta}^{\lambda_w}_{e,P}\,\,  e^{\pm e(\xi,\lambda_w) 2\pi\int^{c_{\lambda_w}}_0p^{\lambda_w}_c({x\over 2}\alpha_{\Bbb R})\left\{
\smallmatrix \tan{\pi \over 2}\langle\alpha^{\vee}_{\Bbb R},{x\over 2}\alpha_{\Bbb 
R}\rangle\\ 
\cot{\pi \over 2}\langle\alpha^{\vee}_{\Bbb R},{x\over 2}\alpha_{\Bbb 
R}\rangle\endsmallmatrix\right\}\,dx}\biggr\}^{(-1)^{\ell(w)+1}}$. Then taking the product over the various $\hat{c}_{\lambda_w}^2=c^2$ we obtain for the leading coefficient 
$$\,\prod_{\hat{c}_{\lambda_w}^2}\{\text{det}_\zeta \square^{\hat{c}_{\lambda_w}^2}_+\}\,\text{det}_\theta (I+ (4\langle\lambda,\rho_n\rangle)G^{\hat{c}_{\lambda_w}^2}_+)\,(4\langle\lambda,\rho_n\rangle)^{ N^{\hat{c}_{\lambda_w}^2}}\,\biggl\{det_{\zeta}{\Delta}^{\hat{c}_{\lambda_w}^2}_{e,\cl 
E_2}\,det_{\zeta}{\Delta}^{\hat{c}_{\lambda_w}^2}_{e,P}\biggr\}^{-\chi_a(X,\Bbb F )}.\,\,\tag9.1$$
Of course if $\langle\lambda,\rho_n\rangle = 0$ the term is omitted as is $\text{det}_\theta (I+ (4\langle\lambda,\rho_n\rangle)G^{\hat{c}_{\lambda_w}^2}_+)$.

The order of the geometric zeta function follows in the same way. If $\langle\lambda,\rho_n\rangle = 0$, then there is the term $(\Vert (z-c_{\lambda_w})\alpha_{\Bbb R}\Vert^2- \Vert c_{\lambda_w}\alpha_{\Bbb R}\Vert^2 )^{ N^{\hat{c}_{\lambda_w}^2}}$. From the elliptic factor the discrete series term has $\Vert (z-c_{\lambda_w})\alpha_{\Bbb 
R}\Vert^2-\Vert c_{\lambda_w}
\alpha_{\Bbb R}\Vert^2)^{n^{\hat{c}_{\lambda_w}^2}_{e,\cl E_2}}$ with $n^{\hat{c}_{\lambda_w}^2}_{e,\cl E_2} = \,{\underset{\smallmatrix 
\omega\in \cl E_2(G)\\
-\Vert \Lambda_{\omega}\Vert^2+\Vert\lambda + \rho\Vert^2 =
0\endsmallmatrix}\to\sum} d(\Lambda_{\omega})m_{\omega, \hat{c}_{\lambda_w}^2}$, while the principal series factor contributes an order $ n^{\hat{c}_{\lambda_w}^2}_{e,P}:= \sum_{W_ {{c}^2} }
{(-1)^{\ell(w)+1}}n^{\lambda_w}_{e,P}(-c_{\lambda_w}\alpha_{\Bbb R})$. Taking into account the arithmetic genus and summing over the possible $\hat{c}_{\lambda_w}^2$ we obtain

 $$\nu = \sum\limits_{\hat{c}_{\lambda_w}^2}\{(n^{\hat{c}_{\lambda_w}^2}_{e,\cl E_2}+n^{\hat{c}_{\lambda_w}^2}_{e,P})\{-\chi_a(X,\Bbb F )\}\, +\, N^{\hat{c}_{\lambda_w}^2}\}.\tag9.2$$
 
 For clarity we recall that $ n^{\hat{c}_{\lambda_w}^2}_{e,\cl E_2}$ is the dimension of the ``zero modes'' in the discrete series contribution to $f^{\hat{c}_{\lambda_w}^2}_t(e)$, i.e. those $\omega$ with $\Vert \Lambda_{\omega}\Vert^2-\Vert\lambda + \rho\Vert^2 = 0$; and $n^{\hat{c}_{\lambda_w}^2}_{e,P}$ for the principal series given by the polynomial factor in the Plancherel density in Lemma 7.2, and $N^{\hat{c}_{\lambda_w}^2}$ are the ``zero modes'' coming from $\square^{\hat{c}_{\lambda_w}^2}_+$. This gives some understanding of  $Z_{\Bbb V ,\Bbb F} $ from the geometric side though more will follow when we sum over the Weyl group. We still need to connect it to the spectral side.

Recall from \S 4 (4.2) the holomorphic torsion theta function given by 
$$\lno{\theta_{\lambda,\phi}(t)=\sum^n_{q=0}(-1)^qq\Tr
e^{-t\square^{0,q}},&\cr}$$
and the associated combination of heat kernels (4.3)
$$\lno{\tilde k^{\lambda}_t=\sum^n_{q=0}(-1)^qq\tr \tilde h^q_t\,.
&\cr}$$

For each $q$ let $\zeta_{q}(s)$ be the zeta function for $\Tr e^{-t\square^{0,q}}$ 
which, of course, is computed from $\tr \tilde h^q_t$. The holomorphic torsion is defined by \footnote {In [R-S;I] the power is $q+1$, whereas in [R-S;II] the power is $q$. We shall use the former.}
$$\log \tau(X,\Bbb V,\Bbb F) \, = \, {1\over 2}\sum_{q\ge 0}(-1)^{q+1} q\zeta_q'(0) = - {1\over 2}\sum_{q\ge 0}(-1)^{q} q\zeta_q'(0) ,$$where $\zeta_q (s) := Tr \,\{\square^{0,q}_+ \}^{-s}$.

First we handle the scalar factor $e^{2t\langle\lambda,\rho_n\rangle}$. From Lemma 4.6 we have 
$$\eqalign{\tilde k^{\lambda}_t(e)&=\sum_{\omega\in\cl
E_2(G)}e^{{t\over 2}(\chi_{\omega}(\Omega_G)\,-\,\langle 
\lambda^\ast,\lambda^\ast+2\rho\rangle}d_{\pi_{\omega}} 
e'(\omega,V_{\lambda},\frak p_-)  \cr
&+ e^{2t\langle\lambda,\rho_n\rangle}\sum_{W_{\kappa}} 
(-1)^{\ell(w)+1}
\int_{\frak
a^{\ast}_{\frak q}}e(\xi,\lambda_w)e^{-{t\over 2}\hat c^2_{\lambda_w}}
e^{-{t\over 2}{\Vert \nu \Vert^2}} {d\mu \over d\nu}\left( w^{\ast}(\lambda^{\ast}
+\rho - 2\rho_n)\vert_{\frak t_{\kappa}}, \nu \right) d\nu. \cr} $$ 
Using the relationship $\Vert\lambda^\ast +\rho\Vert^2 + 4\langle\lambda,\rho_n\rangle = \Vert\lambda +\rho\Vert^2$ in the discrete series contribution above we obtain in each term for $\theta_{\lambda,\varphi}(t)$ in Theorem 4.7 the factor $e^{2t\langle\lambda,\rho_n\rangle}$. Thus we have 
$$ \eqalign{e^{-2t\langle\lambda,\rho_n\rangle}&\theta_{\lambda,\phi}(t)\, =\, \sum^n_{q=0}(-1)^qq\Tr
e^{-{t\over 2}\{\square^{0,q}+4\langle\lambda,  \rho_n\rangle I\}}\cr
&=\sum_{[\gamma]\in\cl
E_1(\Gamma)}\Tr\varphi(\gamma)\vol(\Gamma_{\gamma}\bs
G_{\gamma}){e^{-\Vert\log\gamma_{\Bbb R}\Vert^2/2t}\over (2\pi
t)^{1/2}}~ {N^{-1}_{\gamma}c^{-1}_{\gamma} \over 
 \xi_{\rho_{Q}}(\gamma)\prod\limits_{\alpha\in P^c_{\gamma}}[1-\xi_{-\alpha}(\gamma)]} \vert W(G^0_\gamma,A^0_I)\vert \cr
&\times\sum_{W_{\kappa}}(-1)^{\ell(w)+1}e^{-{t\over 
2}\hat c^2_{\lambda_w}}\sum_{W(K\cap M^+, A_I)/W(G^0_h,A^0_I)}\epsilon(s)\wt{\omega}_{\gamma}(s(\lambda_w+\rho_M)){\xi}_{s(\lambda_w+\rho_M) - \rho_M}(\gamma_I)\cr  
 &+\text{ dim F vol X }\sum_{\omega\in\cl
E_2(G)}d_{\pi_{\omega}}e'(\omega,V_{\lambda},\frak p_-)e^{{t\over 
2}(\chi_{\omega}(\Omega_G)\,-\,\langle \lambda,\lambda+2\rho\rangle)}\cr
&+\text{ dim F vol X }
\sum_{W_{\kappa}} (-1)^{\ell(w)+1}
\int_{\frak
a^{\ast}_{\frak q}}e(\xi,\lambda_w)e^{-{t\over 2}\hat c^2_{\lambda_w}}
e^{-{t\over 2}{\Vert \nu \Vert^2}} {d\mu\over d\nu}\left( w^{\ast}(\lambda^{\ast}
+\rho - 2\rho_n)\vert_{\frak t_{\kappa}}, \nu \right) d\nu.\cr}.\tag 9.3$$

 $$ \eqalign{e^{-2t\langle\lambda,\rho_n\rangle}&\theta_{\lambda,\phi}(t)\, =\, \sum^n_{q=0}(-1)^qq\Tr
e^{-{t\over 2}\{\square^{0,q}+4\langle\lambda, \rho_n\rangle I\}}\cr
&=\sum_{[\gamma]\in\cl
E_1(\Gamma)}\Tr\varphi(\gamma)\vol(\Gamma_{\gamma}\bs
G_{\gamma}){e^{-\Vert\log\gamma_{\Bbb R}\Vert^2/2t}\over (2\pi
t)^{1/2}}~{N^{-1}_{\gamma}c^{-1}_{\gamma} \over \xi_{\rho}(\gamma)
\prod\limits_{\alpha\in P^c_{\gamma}}[1-\xi_{-\alpha}(\gamma)]}\vert W(G^0_\gamma,A^0_I)\vert \cr 
&\qquad 
\sum_{\hat{c}_{\lambda_w}^2}\sum\limits_{W_ {{c}^2}}(-1)^{\ell(w)+1}e^{-{t\over 
2}\hat c^2_{\lambda_w}}\sum_{W(K\cap M^+, A_I)/W(G^0_h,A^0_I)}\epsilon(s)\wt{\omega}_{\gamma}(s(\lambda_w+\rho_M)){\xi}_{s(\lambda_w+\rho_M) - \rho_M}(\gamma_I)\cr  
&+\text{ dim F vol X }\sum_{\omega\in\cl
E_2(G)}d_{\pi_{\omega}}e'(\omega,V_{\lambda},\frak p_-)e^{{t\over 
2}(\chi_{\omega}(\Omega_G)\,-\,\langle \lambda,\lambda+2\rho\rangle)}\cr
&+\text{ dim F vol X }
\sum_{\hat{c}_{\lambda_w}^2}\sum\limits_{W_ {{c}^2}}(-1)^{\ell(w)+1}
\int_{\frak
a^{\ast}_{\frak q}}e(\xi,\lambda_w)e^{-{t\over 2}\hat c^2_{\lambda_w}}
e^{-{t\over 2}{\Vert \nu \Vert^2}} {d\mu\over d\nu}\left( w^{\ast}(\lambda^{\ast}
+\rho - 2\rho_n)\vert_{\frak t_{\kappa}}, \nu \right) d\nu.\cr}.\tag 9.4$$

We remind the reader of the Remark circa (5.5) that \lq\lq Due to the use of the {\it Ansatz}, we shall systematically
attach  $\hat c_{\lambda_w}^2$ to the notation (often as superscript) 
 to objects for which we take the $\sum_{{w\in W_\kappa,\hat{c}_{\lambda_w}^2=c^2}}(-1)^{\ell(w)+1}$.\rq\rq  Consequently an expression $\sum\limits_{ {\hat{c}_{\lambda_w}^2}}\dots^{\hat c_{\lambda_w}^2}  :=  \sum\limits_{ {\hat{c}_{\lambda_w}^2}} \sum_{{w\in W_\kappa,\hat{c}_{\lambda_w}^2=c^2}}(-1)^{\ell(w)+1}\dots$, in fact unwinds to be  $\sum_{W_\kappa} (-1)^{\ell(w)+1}\dots$. 
 
A comparison of $e^{-2t\langle\lambda,\rho_n\rangle}\theta_{\lambda,\phi}(t)$ in (9.4) with $\sum\limits_{ {\hat{c}_{\lambda_w}^2}}e^{-2t\langle\lambda,\rho_n\rangle} \theta^{\hat{c}_{\lambda_w}^2}(t)$ as given in (5.11) and (5.10) shows that the two agree except for the discrete series contributions to the respective identity terms, viz. $\sum_{\hat{c}_{\lambda_w}^2}\sum_{\pi_{\omega}\in\cl E_2(G)}
e^{{t\over 2}(\pi_{\omega} (\Omega_G) -\langle\lambda^\ast ,\lambda^\ast +2\rho\rangle)}
d_{\pi_{\omega}} e^{\hat c_{\lambda_w}^2}_{M_Q}(\pi_\omega ,V_{\lambda})\,\, $ for $\sum\limits_{ {\hat{c}_{\lambda_w}^2}}f^{\hat c_{\lambda_w}^2}_t(e)$, versus $\sum_{\omega\in\cl
E_2(G)} e^{{t\over 2}(\omega(\Omega_G)\,-\,\langle 
\lambda^\ast,\lambda^\ast+2\rho\rangle)}d_{\pi_{\omega}} 
e'(\omega,V_{\lambda},\frak p_-)  $ for $\tilde k^{\lambda}_t(e)$. This in turn comes down to the equality of $\sum_{\hat{c}_{\lambda_w}^2} e^{\hat c_{\lambda_w}^2}_{M_Q}(\pi_\omega ,V_{\lambda})$ and $e'(\omega,V_{\lambda},\frak p_-)  $.

Now 
$$\eqalign{ e'(\pi,V,\frak p_-) &=\sum(-1)^qq\dim
[H^{\infty}_{\pi}\otimes\Lambda^q\frak p_+\otimes V]^K =\dim  [H^{\infty}_{\pi}\otimes\sum(-1)^qq\Lambda^q\frak p_+\otimes V]^K\cr
&\text { and $\sum(-1)^qq\Lambda^q\frak p_+\otimes V$ restricted to the equirank subgroup $L_\kappa$ gives}\cr
&\simeq 
\sum_{W_\kappa}(-1)^{\ell(w)+1}(\Lambda^{\ev}\frak p^{(1)}_+-
\Lambda^{\odd}\frak p^{(1)}_+)\otimes
(\Lambda^{\ev}\frak p^{(2)}_{\Bbb C}-\Lambda^{\odd}\frak p^{(2)}_{\Bbb C})\otimes 
W_{\lambda_w}.\cr}$$
On the other hand,
$$\eqalign{\sum_{\hat{c}_{\lambda_w}^2} e^{\hat{c}_{\lambda_w}^2}_{M_Q}(\pi,V_{\lambda}) &= \sum_{\hat{c}_{\lambda_w}^2}\sum_{\mu}
 m^{\hat c_{\lambda_w}^2}_{\mu}\dim [H^{\infty}_\pi \otimes V_\mu]^K\cr
 & = \sum_{\mu}\{\sum_{\hat{c}_{\lambda_w}^2} m^{\hat c_{\lambda_w}^2}_{\mu}\}\dim [H_\pi \otimes V_\mu]^K \cr
&=\dim [H^{\infty}_\pi \otimes \sum_{\mu}
\{\sum_{\hat{c}_{\lambda_w}^2} m^{\hat c_{\lambda_w}^2}_{\mu}\}V_\mu]^K \cr}$$
and the choice of {\it Ansatz} is to have $\sum_{\mu}
\{\sum_{\hat{c}_{\lambda_w}^2} m^{\hat c_{\lambda_w}^2}_{\mu}\}V_\mu$  restricted to $L_\kappa$ give 
$$\eqalign{\sum_{\hat{c}_{\lambda_w}^2}(\sum_j(-1)^{j+1}\Lambda^j\tilde{\frak p}_+^{[\ev]})\otimes(\sum_{{w\in W_\kappa,\hat{c}_{\lambda_w}^2=c^2}}(-1)^{\ell(w)+1}W_{\lambda_w})\, &\cr
\simeq 
\sum_{W_\kappa}(-1)^{\ell(w)+1}(\Lambda^{\ev}\frak p^{(1)}_+-
\Lambda^{\odd}\frak p^{(1)}_+)\otimes
(\Lambda^{\ev}\frak p^{(2)}_{\Bbb C}-\Lambda^{\odd}\frak p^{(2)}_{\Bbb C})\otimes 
W_{\lambda_w}&.\cr}$$
 
Since $L_\kappa$ is equirank and both sums of virtual $K$-types restrict to it with the same result, it follows that the virtual $K$-types are equivalent. Thus the discrete series contributions are equal. In \S 7 the discrete series contribution was further divided into those that were \lq\lq zero modes\rq\rq, i.e. those with $\Vert \Lambda_{\omega}\Vert^2-\Vert\lambda + \rho\Vert^2 = 0$ or not. Since the above computation is independent of the condition, the identity holds for the \lq\lq zero modes\rq\rq of the discrete series contribution to $\tilde k^{\lambda}_t(e)$. Thus these are given by  
$$\eqalign{\sum_{\hat{c}_{\lambda_w}^2} n^{\hat{c}_{\lambda_w}^2}_{e,\cl E_2} &= \sum_{\hat{c}_{\lambda_w}^2} \,{\underset{\smallmatrix 
\omega\in \cl E_2(G)\\
-\Vert \Lambda_{\omega}\Vert^2+\Vert\lambda + \rho\Vert^2 =
0\endsmallmatrix}\to\sum} d(\Lambda_{\omega})m_{\omega, \hat{c}_{\lambda_w}^2}\cr
&= {\underset{\smallmatrix 
\omega\in \cl E_2(G)\\
-\Vert \Lambda_{\omega}\Vert^2+\Vert\lambda + \rho\Vert^2 =
0\endsmallmatrix}\to\sum} d(\Lambda_{\omega}) \sum_{\hat{c}_{\lambda_w}^2} m_{\omega, \hat{c}_{\lambda_w}^2} \quad \text {(cf. (7.1))}\cr
&= {\underset{\smallmatrix 
\omega\in \cl E_2(G)\\
-\Vert \Lambda_{\omega}\Vert^2+\Vert\lambda + \rho\Vert^2 =
0\endsmallmatrix}\to\sum} d(\Lambda_{\omega}) e'(\omega,V_{\lambda},\frak p_-)\cr}$$ and will be denoted $N_{e,\cl E_2}.$ 

Similarly we need to consider the principal series contributions to the identity term $\sum\limits_{\hat{c}_{\lambda_w}^2}n^{\hat{c}_{\lambda_w}^2}_{e,P} =  \sum\limits_{ {\hat{c}_{\lambda_w}^2}} \sum_{{w\in W_\kappa,\hat{c}_{\lambda_w}^2=c^2}}(-1)^{\ell(w)+1}n^{\lambda_w}_{e,P}(-c_{\lambda_w}\alpha_{\Bbb R}) = \sum_{W_\kappa} (-1)^{\ell(w)+1}n^{\lambda_w}_{e,P}(-c_{\lambda_w}\alpha_{\Bbb R}).$ We recall that, when non-zero, $n^{\lambda_w}_{e,P}(-c_{\lambda_w}\alpha_{\Bbb R})= 2p^{\lambda_w}_c({-2c_{\lambda_w}\over 2}\alpha_{\Bbb R} )\text { if} -c_{\lambda_w}=-(2n+1)~(\text{resp.}~-2n)$ and $c_{\lambda_w}\alpha_{\Bbb R} =(\lambda_w\circ C_\kappa+\rho_{Q_{\kappa}})(X_{\kappa})\alpha_{\Bbb R}$ while $2p^{\lambda_w}_c({-2c_{\lambda_w}\over 2}\alpha_{\Bbb R} )= 4(-1)^{q-q_M}{\varpi(\Lambda_{\xi(w)}-(-c_{\lambda_w}\alpha_{\Bbb R} )) \over \varpi(\rho )}.$ In the proof of Proposition 6.6 one finds that $\Lambda_{\xi(w)}+c_{\lambda_w}\alpha_{\Bbb R} $ is $\Delta(\frak g_{\Bbb C},\frak h_{\Bbb C})$ integral, so by Weyl, $2p^{\lambda_w}_c({-2c_{\lambda_w}\over 2}\alpha_{\Bbb R} )$ is a nonnegative integer.
From (1.15) we have 
$$\Vert\Lambda_\xi\Vert^2 -\Vert\lambda^\ast +\rho\Vert^2 = -
\{( \lambda_w\circ C_\kappa + \rho_{Q_\kappa})(\widehat{X}_\kappa )\}^2\, + 
4\langle\lambda,\rho_n\rangle\, $$
 in other words 
 $$\Vert\Lambda_\xi\Vert^2 +\{( \lambda_w\circ C_\kappa + \rho_{Q_\kappa})(\widehat{X}_\kappa )\}^2 -\Vert\lambda +\rho\Vert^2 = 0,$$  and we have seen before $c_{\lambda_w}\alpha_{\Bbb R} =(\lambda_w\circ C_\kappa+\rho_{Q_{\kappa}})(X_{\kappa})\alpha_{\Bbb R} = (\lambda_w\circ C_\kappa+\rho_{Q_{\kappa}})
 (\widehat{X}_{\kappa}){\alpha_{\Bbb R}\over \Vert\alpha_{\Bbb R}\Vert} ,$ so 
 $$\Vert\Lambda_\xi\Vert^2 +\Vert(\lambda_w\circ C_\kappa+\rho_{Q_{\kappa}})(X_{\kappa})\alpha_{\Bbb R}\Vert^2  -\Vert\lambda +\rho\Vert^2 = 0.$$
In this way $n^{\lambda_w}_{e,P}(-c_{\lambda_w}\alpha_{\Bbb R})$ is the number of  this principal series \lq\lq zero modes\rq\rq  and $N_{e,P} := \sum\limits_{\hat{c}_{\lambda_w}^2}n^{\hat{c}_{\lambda_w}^2}_{e,P} =\sum_{W_\kappa} (-1)^{\ell(w)+1}n^{\lambda_w}_{e,P}(-c_{\lambda_w}\alpha_{\Bbb R})$ is the principal series contribution to the \lq\lq zero modes\rq\rq of $\tilde k^{\lambda}_t(e)$. Summarizing,  $N_{e,\cl E_2} + N_{e,P} $ represent the elliptic  \lq\lq zero modes\rq\rq.

There remains another task, viz., the possible \lq\lq zero modes \rq\rq of $\square^{0,q}$ need to be removed, as in \S8, before the zeta function is computed. Denote by $N_q$ the dimension of the kernel of $\square^{0,q}$ and set $N_\lambda = \sum (-1)^q q N_q.$ Since on $[H^{\infty}_{\pi}\otimes \Lambda^q\frak p_+\otimes V]^K$, from Lemma 1.1, we have
$\square^{0,q}_{\pi}=-{1\over 2}\pi(\Omega_G)\otimes I\otimes I+{1\over
2}\langle \lambda^{\ast}+2\rho,\lambda^{\ast}\rangle I\otimes I\otimes I,$ (in particular one can have $H_\pi = L^2(\Gamma\backslash G,\varphi$), $\pi = R_{\Gamma,\varphi}$) and the operator from the {\it Ansatz} on vector bundles is  $ e^{{t\over 2}(\Omega_G -\langle\lambda^\ast,\lambda^\ast +2\rho\rangle I)}$ - i.e. the same -  the computation above also will show that 
$$N_\lambda  = \sum (-1)^q q N_q = \sum_{\hat{c}_{\lambda_w}^2} N^{\hat{c}_{\lambda_w}^2}.$$
Since $X:=\Gamma\backslash G/K$ is compact and smooth, we know from Hodge theory that the cohomology is represented by harmonic forms, i.e. the kernels of the various  $\square^{0,q}$. Then 
with $h^{0,q}$ the Hodge number we have 
$$N_\lambda  = \sum (-1)^q q h^{0,q}$$ represents the semisimple \lq\lq zero modes\rq\rq.

With the kernels removed, the equality on the geometric side of heat kernels $ \theta_{\lambda,\phi,+}(t)$  and $\sum\limits_{ {\hat{c}_{\lambda_w}^2}}  \theta^{\hat{c}_{\lambda_w}^2}_+(t)$ shows that the computation on the spectral side gives

$$\log \tau(X,\Bbb V,\Bbb F)^2 =\sum_{q\ge 0}(-1)^{q+1} q\zeta_q'(0)\,=\, \sum_{q\ge 0}(-1)^q q \log\text{det}_\zeta \square^{0,q}_+= \sum_{\hat{c}_{\lambda_w}^2} \log\text{det}_\zeta \square^{\hat{c}_{\lambda_w}^2}_+  \, .$$
Indeed, from (8.6) and the symbolic calculus as summarized circa \S 7
$$\eqalign{\quad &-Pf\int^\infty_0 e^{-tz}e^{-4t\langle\lambda,\rho_n\rangle} \theta_{\lambda,\phi,+}(2t){dt \over 
t}\,\,=-Pf\int^\infty_0 e^{-tz}e^{-4t\langle\lambda,\rho_n\rangle}\sum\limits_{ {\hat{c}_{\lambda_w}^2}} \theta^{\hat{c}_{\lambda_w}^2}_+(2t){dt \over 
t}\,\,\cr
&=\,\sum\limits_{ {\hat{c}_{\lambda_w}^2}} \log\text{det}_\theta 
(I+(z+4\langle\lambda,\rho_n\rangle)G^{\hat{c}_{\lambda_w}^2}_+)\,\cr
&+\gamma \sum\limits_{ {\hat{c}_{\lambda_w}^2}} \{  \chi_a(X,\Bbb F)\, M^{\hat{c}_{\lambda_w}^2}_2+ \chi_a(X,\Bbb F)\, \sum\limits_{W_ {{c}^2}}(-1)^{\ell (w)+1}e(\xi,\lambda_w)\tilde c_0^w  \,- N^{\hat{c}_{\lambda_w}^2}\,\}+\sum\limits_{ {\hat{c}_{\lambda_w}^2}} \log\text{det}_\zeta\square^{\hat{c}_{\lambda_w}^2}_+.\cr}$$  
 The first two terms of the third line come from the coefficient of $t^0$ in the short time asymptotics and by the above equality agree with those from $e^{-4t\langle\lambda,\rho_n\rangle}\theta_{\lambda,\phi,+}(2t)$ as per Cor. 5.7. Similarly for the dimension of the kernels. 
 
 Then (9.1) can be written 
 $$\,e^{\sum_{q\ge 0}(-1)^{q+1} q\zeta_q'(0)\,}(4\langle\lambda,\rho_n\rangle)^{ N_{\lambda}} \prod_{\hat{c}_{\lambda_w}^2}\,\text{det}_\theta (I+ (4\langle\lambda,\rho_n\rangle)G^{\hat{c}_{\lambda_w}^2}_+)\,\,\biggl\{det_{\zeta}{\Delta}^{\hat{c}_{\lambda_w}^2}_{e,\cl 
E_2}\,det_{\zeta}{\Delta}^{\hat{c}_{\lambda_w}^2}_{e,P}\biggr\}^{-\chi_a(X,\Bbb F )}\,\,.\tag9.1\'$$ 

Similarly, in (9.2) the above relates the first term to $\tilde k^{\lambda}_t(e)$.  

\proclaim{Theorem 9.1} The zeta function
$$Z_{\Bbb V ,\Bbb F}(z) =\exp -\sum_{[\gamma ]\in\cl E_1(\Gamma)}\Tr \varphi (\gamma )
\chi_\gamma (\Bbb V){e^{-z \alpha_{\Bbb R}(\ell_\gamma \widehat{X}_\kappa)} \over \mu_\gamma}$$

{\bf (i)}
is analytic for $\max\limits_{w\in W_\kappa}\Vert  (z +c_{\lambda_w})\alpha_{\Bbb R}\Vert^2  -\, \Vert c_{\lambda_w}\alpha_{\Bbb R}\Vert^2 + \mu^{\hat{c}_{\lambda_w}^2}_{b}>0$ and \,
 $\Re z$ \,  in an appropriate slit plane;
  
{\bf (ii)} has meromorphic continuation to $\Bbb C$ by
$$Z_{\Bbb V ,\Bbb F}(z) =\prod_{\hat{c}_{\lambda_w}^2}
Z^{\hat{c}_{\lambda_w}^2}_\varphi ((z-c_{\lambda_w})\alpha_{\Bbb R})\,;$$

{\bf (iii)} the leading term of its Laurent  expansion for z near zero
$$Z_{\Bbb V ,\Bbb F} (z) \sim \tau(X,\Bbb V,\Bbb F)^{2}\, k  \, z^\nu$$
where  \qquad
$\tau(X,\Bbb V,\Bbb F)\text{ is the holomorphic torsion of }(\Bbb V,\Bbb F),$ 
$$k =(4\langle\lambda,\rho_n\rangle)^{ N_{\lambda}}\prod_{\hat{c}_{\lambda_w}^2}   \text{det}_\theta (I+ (4\langle\lambda,\rho_n\rangle)G^{\hat{c}_{\lambda_w}^2}_+)\,\,\prod_{\hat{c}_{\lambda_w}^2}\biggl\{\,det_{\zeta}{\Delta}^{\hat{c}_{\lambda_w}^2}_{e,\cl 
E_2}\,det_{\zeta}{\Delta}^{\hat{c}_{\lambda_w}^2}_{e,P}\biggr\}^{-\chi_a(X,\Bbb F )}\,
\text{if} \,\langle\lambda,\rho_n\rangle \not = 0,$$  
%where  $det_{\zeta}{\Delta}^{\hat{c}_{\lambda_w}^2}_{e,P}:=\prod\limits_{W_ {\hat{c}^2} }\biggl\{det_{\zeta}{\Delta}^{\lambda_w}_{e,P}\,\,  e^{\pm e(\xi,\lambda_w)2\pi\int^{c_{\lambda_w}}_0p^{\lambda_w}_c({x\over 2}\alpha_{\Bbb R})\left\{
%\smallmatrix \tan{\pi \over 2}\langle\alpha^{\vee}_{\Bbb R},{x\over 2}\alpha_{\Bbb 
%R}\rangle\\ 
%\cot{\pi \over 2}\langle\alpha^{\vee}_{\Bbb R},{x\over 2}\alpha_{\Bbb 
%R}\rangle\endsmallmatrix\right\}\,dx}\biggr\}^{(-1)^{\ell(w)+1}}$, 
and \quad $ \nu = -(N_{e,\cl E_2} + N_{e,P})\, \chi_a(X,\Bbb F )\, +\, \sum (-1)^q q h^{0,q}$ ;
 
{\bf (iv)} in particular if $\chi_a(X)=0$ then \quad
$ \nu = \sum (-1)^q q h^{0,q}$.

\endproclaim

Motivated by the result in [R-S;II] we specialize to the case of two flat bundles of the same dimension. 

\proclaim{Corollary 9.2} Let $\varphi_i:\Gamma\to U (F_i), i=1,2$ be non-trivial unitary representations on finite dimensional complex vector spaces $F_i$ of the same dimension. Assume that $\langle\lambda,\rho_n\rangle = 0$. Then the ratio zeta function
$Z_{\Bbb V ,\Bbb F_1;\Bbb F_2}:=  {Z_{\Bbb V ,\Bbb F_1}\over Z_{\Bbb V ,\Bbb F_2}} $ 
satisfies
$$ Z_{\Bbb V ,\Bbb F_1;\Bbb F_2}(z) \sim
\biggl({\tau(X,\Bbb V,\Bbb F_1)\over \tau(X,\Bbb V,\Bbb F_2)}\biggr)^2 z^{\nu_1 - \nu_2}, \quad
z\, \sim \, 0,$$
where \quad $\nu_1 - \nu_2 =  \sum (-1)^q q \left(h^{0,q}_1 - h^{0,q}_2\right)$, \,  with 
$h^{0,q}_i := h^{0,q}_i (X,\Bbb V,\Bbb F_i), \, i=1,2$.
In particular, if $H^{0,*}(X, \Bbb V ,\Bbb F_1) \cong H^{0,*}(X, \Bbb V ,\Bbb F_2)$ then
$Z_{\Bbb V ,\Bbb F_1;\Bbb F_2}$ is regular at $z=0$ and 
 $$Z_{\Bbb V ,\Bbb F_1;\Bbb F_2}(0) = {\tau(X,\Bbb V,\Bbb F_1)^2\over \tau(X,\Bbb V,\Bbb F_2)^2}.
$$ 

\endproclaim
\demo{Proof} From Lemma 4.6 
$$\eqalign{\tilde k^{\lambda}_t(e)&=\sum_{\omega\in\cl
E_2(G)}e^{{t\over 2}(\omega(\Omega_G)\,-\,\langle 
\lambda^\ast,\lambda^\ast+2\rho\rangle)}d_{\pi_{\omega}} 
e'(\omega,V_{\lambda},\frak p_-)  \cr
&+ e^{2t\langle\lambda,\rho_n\rangle}\sum_{W_{\kappa}} 
(-1)^{\ell(w)+1}
\int_{\frak
a^{\ast}}e(\xi,\lambda_w)e^{-{t\over 2}\hat c^2_{\lambda_w}}
e^{-{t\over 2}{\Vert \nu \Vert^2}} {d\mu\over d\nu}\left( w^{\ast}(\lambda^{\ast}
+\rho - 2\rho_n)\vert_{\frak t_{\kappa}}, \nu \right) d\nu, \cr} $$
with each of the terms independent of the representation $(\varphi_i,F_i)$. Since $\text{dim} F_1 = \text {dim} F_2$ the contributions to the respective torsion theta functions $\theta_{\lambda,\phi,+}(t)$ are equal, and their contribution to the constant $k$ and $\nu$ are equal. As the {\it Ansatz} is independent of the flat bundle (but dependent on the bundle $\Bbb V$) the other terms in $k$ similarly agree. What remains is the power $\nu_i = \sum (-1)^q q h^{0,q}_i.$ So for $ z \to 0$ the assertion holds.

\enddemo

\n{\bf \S 10~~Closing remarks}

 {\bf (1)} Using the results from \S 1, e.g. Proposition 1.16, it appears reasonable to suggest that, up to a universal constant, $\prod_{\hat{c}_{\lambda_w}^2}\biggl\{\{det_{\zeta}{\Delta}^{\hat{c}_{\lambda_w}^2}_{e,\cl E_2}\,det_{\zeta}{\Delta}^{\hat{c}_{\lambda_w}^2}_{e,P}\}\biggr\}$ is $\tau(X^d,\Bbb V)^2$, the holomorphic torsion of the compact dual symmetric space. The ingredients of a possible proof are as follows. The starting point is that the centralizer of $H_\kappa$ in $\frak g$ is $\frak m^{(1)}_\kappa \oplus \frak k^*_\kappa$. After complexification (and under the hypotheses here), one has in the compact dual, $U$, to $G$ a compact real form of  $\frak m^{(1)}_\kappa \oplus \frak k^*_\kappa$ to which the Bott-Borel-Weil construction can be used to the data given in Proposition 1.16 to define the $U$ representations contributing to the derived Euler number.  On the other hand, to the integral used to compute the principal series contribution to the identity term of the heat kernels of the {\it Ansatz} one applies Cauchy's theorem. The known pole structure of ${\Gamma'\over \Gamma}$ provides the poles for a residue computation with the other factor $p^{\lambda_w}_c({x\alpha_{\Bbb R}\over 2})$ the dimension, yielding the series of representations needed. Continuing with Pf Laplace transforms of the corresponding heat kernels, as herein, instead of the zeta functions in [K] should provide the identification.
 
{\bf (2)} As a ``primary'' companion for the above remark, we shall exhibit a proportionality
relation (\`a la Hirzebruch)  between the ``anti-derivatives'' of holomorphic torsion for  $X$ 
and for its compact dual $X^d$. 
The proof is based on an easy consequence of the Bismut-Gillet-Soul\'{e} anomaly 
formula for the Quillen metric [B-G-S, III], which we first describe.

So let $X$ be an arbitrary K\"ahler manifold $X$ of complex dimension $n$ and let $V$ be a
Hermitian holomorphic bundle on $X$. Under the metric scaling $g_\epsilon=  \epsilon^{-2} g$,
where $g$ is the metric of $X$, the corresponding Hodge-Laplace operators scale as
 \, $\square_{\epsilon, +}^{0, q} = \epsilon^2 \square_+^{0, q}$. The variation of their 
$\zeta$-regularized determinants is given by the formula
 $$
\log \det \square_{\epsilon, +}^{0, q} \, = \,  \log \det \square_{+}^{0, q} +  2\zeta_q(0) \log \epsilon ,
$$
and therefore the holomorphic torsion varies according to
the rule 
 $$ \lno { \log \tau^2_{g_\epsilon}  - \log \tau^2_{g}  \, = \, 
2  \sum^n_{q= 0}(-1)^{q} q \zeta_q(0)\, \log \epsilon . 
  && (10.1)\cr}
 $$
Individually, the values $\zeta_q(0)$ have no direct cohomological meaning in dimensions
$n > 1$. However, their ``derived'' alternating sum $ \sum^n_{q= 0}(-1)^{q} q\,\zeta_q(0)$ (called
above ``anti-derivative'' of torsion) will be shown to be computable 
in terms of Chern and Hodge numbers.  

Consider the determinant line \, $  \Lambda=\prod_{0 \leq q \leq n} \left(\wedge^{h_{0,q}} 
H^{0,q}(X,\Bbb V)\right)^{(-1)^{q}} $.
If $\{\omega_1, \ldots, \omega_r \}$ is a basis of $H^{0, \ev}(X,\Bbb V)$ and 
$\{\varpi_1, \ldots, \varpi_s\}$ a basis of $H^{0, \odd}(X,\Bbb V)$, then
$\lambda = \omega_1\wedge \ldots \wedge\omega_r \otimes (\varpi_1\wedge \ldots \wedge \varpi_s)^{-1}$
is a nonzero section of $\Lambda$.  Its  $L^2$-norm relative to the metric $g$ is 
$$    \vert \lambda \vert^2_{\Lambda, g}= 
\frac{\det \left(\langle \omega_i,\omega_j \rangle_g \right)_{1\leq i, j \leq r}}
{\det \left(\langle \varpi_i,\varpi_j \rangle_g \right)_{1\leq i, j \leq s}} \, ,
 $$
whereas its Quillen norm is \,
$  \Vert \lambda \Vert_{\Lambda, g} \, = \, \tau_g \, \vert \lambda \vert_{\Lambda, g} $. 
The scaling  $g_\epsilon =  \epsilon^{-2} g$, $\epsilon>0$, has the 
effect of multiplying the Hodge $\ast_q$-operator, and hence
the inner product on $H^{0,q}$, by the factor $ \epsilon^{-(2n-2q)}$.
Therefore, the $L^2$-metric of $\Lambda$ scales according to the rule
$$   \lno { \log\vert \cdot \vert_{\Lambda, g_\epsilon }^2 \, - \,
\log \vert \cdot\vert_{\Lambda, g}^2 \, = \,  \left(-2 n \chi_a(X, \Bbb V, \Bbb F) +
2 \sum^n_{q= 0} (-1)^{q} q \, h^{0,q}\right)  \,\log\epsilon  \, .
 && (10.2)  \cr}  $$ 
We now apply the Bismut-Gillet-Soul\'{e} anomaly formula [B-G-S III; Theorem 1.22] 
to the $1$-parameter family of metrics $g_\epsilon$.  Since our definitions for 
the logarithms of the torsion and of both metrics differ by those in [B-G-S III]  by a sign, 
the application reads as follows:
 $$\eqalign {-\frac{\partial}{\partial \epsilon} \log \Vert \phi_\epsilon\Vert^2 &= 
\left(\frac{1}{2\pi i}\right)^n 
\int_X \left(\frac{\partial}{\partial b} \bigg\vert_{b=0}Td \left(-R^+-b (g_\epsilon)^{-1}
\frac{\partial g_\epsilon }{\partial \epsilon} \right)\, \Tr (e^{-L_{\Bbb V}})   \right) \cr
&=  \left(\frac{1}{2\pi i}\right)^n 
\int_X \left(\frac{\partial}{\partial b} \bigg\vert_{b=0}Td \left(-R^++2b \epsilon^{-1} I \right) \, 
\Tr (e^{-L_{\Bbb V}})   \right) ,
 \cr} \tag 10.3 $$
where  $\phi_\epsilon : \Lambda_g \to \Lambda_{g_\epsilon}$ is the canonical isomorphism, 
$R^+$ is the curvature of $T^{1,0}$ and $L_{\Bbb V}$ is the curvature of the canonical connection of the
holomorphic bundle $\Bbb V$. Recall that the Todd function is
$$ Td(x) = \frac{x}{1-e^{-x}} =\sum_{m\geq 0} \frac{B_m^+ \, x^m}{m!} = 1 + \frac{x}{2}
+\sum_{k\geq 1} \frac{B_{2k} \, x^{2k}}{{2k}!} .
$$
and in terms of Chern classes,
$$
Td = I + \frac{1}{2}c_1 +  \frac{1}{12}(c_1^2 + c_2) +  \frac{1}{24}c_1c_2 + \ldots \, .
$$
To compute 
$\frac{\partial}{\partial b} \bigg\vert_{b=0}Td \left(-R^+ +2b \epsilon^{-1} I \right)$ we use
Jacobi's formula for the logarithmic derivative of the determinant, which
is valid for matrices over any commutative ring. In our case, substituting $c= 2b \epsilon^{-1}$ one has
$$\eqalign { {\partial \over \partial c}\bigg\vert_{c=0} & \det {-R^++cI \over I-e^{-(-R^++cI)}} 
 = \det {-R^+ \over I-e^{R^+}} \, \Tr \left(\frac{I-e^{R^+}}{-R^+} \cdot {\partial \over \partial c}\bigg\vert_{c=0} 
 {-R^++cI \over I-e^{-(-R^++cI)}} \right) \cr
 &= Td (-R^+) \Tr \frac{e^{-R^+} - I +R^+}{-R^+(e^{-R^+} - I)} = Td (-R^+) \Tr \left(
  \frac{1}{2}I + \sum_{k\geq 1} \frac{B_{2k}}{(2k)!} (-R^+)^{2k-1}\right) \cr
  &=  {n \over 2} Td (-R^+) + \sum_{k\geq 1} \frac{B_{2k}}{(2k)!} Td (-R^+) \Tr \left((-R^+)^{2k-1}\right) \cr
  &=  {n \over 2} Td (-R^+) + \frac{1}{12} \left(n + \frac{1}{2}c_1(-R^+)  +  
  \frac{1}{12}(c_1(-R^+)^2 + c_2(-R^+) ) +  \ldots\right)  c_1(-R^+) \cr
  &\qquad \qquad \qquad \quad + \sum_{k\geq 2} \frac{B_{2k}}{(2k)!} Td (-R^+) \Tr \left((-R^+)^{2k-1}\right) \cr
 & = {n \over 2} Td (-R^+) + {n \over 12} c_1(-R^+) + \frac{1}{24}c_1(-R^+)^2 
  + \text{higher Chern classes}
\cr} $$
Inserting this answer in (10.3), with the change 
of derivatives  ${\partial \over \partial b} = 
 {2 \over \epsilon} {\partial \over \partial c}$, one obtains
$$\eqalign { \frac{\partial}{\partial \epsilon} \log \Vert \phi_\epsilon\Vert^2 & = -{ 2 \over \epsilon}
\left(\frac{1}{2\pi i}\right)^n  \int_X \left({n \over 2} Td (-R^+) + {n \over 12} c_1(-R^+) 
   + \text{higher Chern classes} \right) \Tr (e^{-L_{\Bbb V}})  \cr 
&= - { 1 \over \epsilon} \left( n \chi_a(X, \Bbb V) 
+ {n \over 6} \int_X c_1(-R^+) 
 \Tr (e^{-L_{\Bbb V}}) + \text{higher Chern numbers} \right)
\cr} $$
 Therefore, ignoring the higher Chern numbers, one has
 $$   \log\Vert \cdot \Vert_{g_\epsilon}^2 - \log \Vert \cdot \Vert_g^2 \,
 \sim \, - \left(n \chi_a(X, \Bbb V) + {n \over 6} \int_X c_1(-R^+) 
 \Tr (e^{-L_{\Bbb V}}) \right)   .
  $$
Together with (10.1), (10.2) and the definition of the Quillen norm this leads to  
$$ \eqalign {  
  2  \sum^n_{q= 0}(-1)^q q \, \zeta_q (0) & 
-  2 n \chi_a(X, \Bbb V) + 2 \sum^n_{q= 0} (-1)^{q} q \, h^{0,q} \cr 
& \sim \,  - n \chi_a(X, \Bbb V) -  {n \over 3} \int_X c_1(-R^+)  \Tr (e^{-L_{\Bbb V}}) ,
\cr} $$ 
i.e. 
$$ \lno {\sum^n_{q= 0}(-1)^q q \, \zeta_q (0) + \sum^n_{q= 0}(-1)^q q h^{0,q} \sim \, 
 {n \over 2} \chi_a(X, \Bbb V) 
-  {n \over 6} \int_X c_1(-R^+)  \Tr (e^{-L_{\Bbb V}}) .
  && (10.4)\cr} $$
 For $n=1$ there are no higher Chern numbers involved and the above relation 
becomes a known identity for Riemann surfaces (comp.  [Bo, (1.5.6)]).

\proclaim{Proposition 10.1} Let $X=\Gamma \backslash G/K$ cum $\Bbb V$ be as in the
previous sections, and let $X^d=G^d/K$ cum ${\Bbb V}^d$ denote the dual pair. Then
there exists a constant $C(\Gamma)$ such that
 $$\eqalign { \sum^n_{q= 0}(-1)^q q \zeta^{(X, \Bbb V)}_q (0) &+ 
 \sum^n_{q= 0} (-1)^q q h^{0,q} (X, \Bbb V) \cr
 &= \, C(\Gamma)
 \left(\sum^n_{q= 0}(-1)^q q \, \zeta^{(X^d, {\Bbb V}^d)}_q (0) + 
 \sum^n_{q= 0}(-1)^q q h^{0,q} (X^d, {\Bbb V}^d)\right).
   \cr} $$
\endproclaim

\demo{Proof} We apply the quasi-identity (10.4) to both $(X, \Bbb V)$ 
and $(X^d, {\Bbb V}^d)$. Since all Chern classes in the right hand side 
(whether visible or not) are subject to the same Hirzebruch proportionality
relation, the claimed equality follows.
\enddemo

 {\bf (3)} The intermediate geometric zeta functions
$$\quad\log Z^{\hat{c}_{\lambda_w}^2}_\varphi (s) =-\sum\limits_{\{w\in W_\kappa:\hat{c}_{\lambda_w}^2=c^2\}}(-1)^{\ell (w)+1}\sum_{[\gamma ]\in\cl E_1(\Gamma)} \Tr \varphi (\gamma )
{(-1)^{q_{\gamma, \perp}}\widehat{L}(\gamma ;\lambda_w )\over \vert \text{Det}(I - Ad(h^{-1}) |_{\frak g/\frak m^1} \vert^{1\over 2}} {e^{-s\alpha_{\Bbb R}(\ell_\gamma \widehat{X}_\kappa)} \over \mu_\gamma}\,$$
and the geometric zeta function
$$Z_{\Bbb V ,\Bbb F}(z) =\exp -\sum_{[\gamma ]\in\cl E_1(\Gamma)}\Tr \varphi (\gamma )
\chi_\gamma (\Bbb V){e^{-z \alpha_{\Bbb R}(\ell_\gamma \widehat{X}_\kappa)} \over \mu_\gamma}$$
can be defined independent of the {\it Ansatz}. For their formulation one needs the critical support condition of $\cl E_1(\Gamma)$; however, the assumption that $G$ has only one class of cuspidal maximal parabolic subgroups is due to technical difficulties in the harmonic analysis on $G$. Therefore, we think it quite possible that with techniques from geometric analysis (e.g. the approach of Bismut) 
the following statement could be shown to hold true:

For  $G$ a group of Hermitian type and $\Gamma$ a discrete co-compact subgroup,
the geometric zeta function
$$Z_{\Bbb V ,\Bbb F}(z) =\exp -\sum_{[\gamma ]\in\cl E_1(\Gamma)}\Tr \varphi (\gamma )
\chi_\gamma (\Bbb V){e^{-z \alpha_{\Bbb R}(\ell_\gamma \widehat{X}_\kappa)} \over \mu_\gamma}$$
is holomorphic in a region of $\Bbb C$ and has a meromorphic continuation to $\Bbb C$.
\remark{Remark} With the strong form of the {\it Ansatz} available in the Appendix for the $\Bbb R$-rank one groups, $SU(n,1)$, one can construct the intermediate zeta functions directly for $w\in W_\kappa$ bypassing several complications such as those involving $\hat{c}_{\lambda_w}^2=c^2$. Then this statement is valid in $\Bbb R$-rank one providing evidence for its validity more generally.
\endremark

%\vfill\eject 
 
%%%%%%%%%%%%%%%%%%%%%%%%%%%%%%%%%%%%%%%%%
 %%%%%%%%%%%%%%%%%%%%%%%%%%%%%%%%%%%%%%%%%
\

\

\centerline{References}

\examb{[A-S;III]}{Atiyah, M.,  Singer, I.M.:  {\it The index of elliptic operators: III.} 
Ann. Math. 87, 546-604 (1968).}

\examb{[B-M]}{Barbasch, D.,  Moscovici, H.: {\it $L^2$-index and the Selberg trace 
formula.} J. Func. Anal. 53, 151-201 (1983).}

\examb{[Ba]}{Barnes, E.W.: {\it The Theory of the G-Function.} Quat. J. of Pure and Applied Math. XXXI, 264-314 (1900).}

\examb{[B-G-S;I]}{Bismut, J. M., Gillet, H, Soul\'e, C.: {\it Analytic torsion and 
holomorphic determinant bundles.I.} Commun. Math. Phys. 115, 49-78 (1988).}

\examb{[B-G-S;II]}{Bismut, J. M., Gillet, H, Soul\'e, C.: {\it Analytic torsion and 
holomorphic determinant bundles.II.} Commun.Math.Phys. 115, 79-126 (1988).}

\examb{[B-G-S;III]}{Bismut, J. M., Gillet, H., Soul\'e, C.: {\it Analytic torsion and 
holomorphic determinant bundles.III.} Commun. Math. Phys. 115, 301-351 (1988).}

\examb{[Bi]}{Bismut,  J. M.: {\it Hypoelliptic Laplacian and orbital integrals.} 
Annals of Mathematics Studies, vol. 177, Princeton University Press, Princeton, NJ, (2011).}

\examb{[Bi;eta]}{Bismut,  J. M.: {\it Eta invariants and the hypoelliptic Laplacian.} J. Eur. Math. Soc. 21, 2355-2515 (2019).}

\examb{[B]}{Borel, A.: {\it Cohomologie de sous-groupes discrets et repr\'esentations 
de groupes semi-simples.} Ast\'erisque  32-33, 73-12 (1976).}
 
\examb{[Bo]}{Bost, J. B.: {\it Fibr\'es d\'eterminants, d\'eterminants 
r\'egularis\'eset 
mesures sur les espaces de modules des courbes complexes.}
S\'eminaire Bourbaki, F\'evrier 1987, Expos\'e No. 676. }

\examb{[De-W]}{DeGeorge, D., Wallach, N.: {\it Limit formulas for multiplicities in 
$L^2(\Gamma\backslash  G)$.} Ann. Math. 107, 133-150 (1978).}

\examb{[De]} {Delorme, P.: {\it Formules limites et formules asymptotiques pour les 
multiplicites dans $L^2(G / \Gamma)$.} Duke Math. J. 53, 691-731 (1986).}

 \examb{[Den]}{Deninger, C.: {\it Analogies between analysis on foliated spaces and arithmetic
geometry.} Groups and analysis, 174--190, London Math. Soc. Lecture Note Ser., 354, Cambridge
Univ. Press, 2008.}

\examb{[D-P]}{D'Hoker, E., Phong, D.H.: {\it On Determinants of Laplacians on Riemann Surfaces.} Commun. Math. Phys. 104, 537-545 (1986).}

\examb{[D]}{Donnelly, H.: {\it Asymptotic expansions for the compact quotients of 
properly discontinuous group actions.} Ill. J. Math.23, 485-496 (1979).}

\examb{[D-K-V]}{Duistermaat, J. J., Kolk, J. A. C., Varadarajan, V.S.: {\it Spectra
of compact locally symmetric manifolds of negative curvature.}  Invent. Math.
{52}, 27-93 (1979).}

\examb{[Fr]}{Frahm, J.: {\it Appendix to: Holomorphic torsion with coefficients and geometric zeta functions for certain Hermitian locally symmetric manifolds.} 
Submitted (2018).}

\examb{[F1]}{Fried, D.: {\it Analytic torsion and closed geodesics on hyperbolic manifolds.} 
Invent. 
math. 84, 523-540 (1986).}

\examb{[F2]}{Fried, D.: {\it Torsion and closed geodesics on complex hyperbolic 
manifolds.} Invent. math. 91, 31-51 (1988).}

\examb{[G-R]}{Gradstein, I. S., Ryshik, I. M.: {\it Tables of Series,Products and 
Integrals.} Deutscher Verlag der Wissenschaften Berlin. 1957.}

\examb{[H-C;InvEig]}{Harish-Chandra: {\it Invariant eigendistributions on a semi-simple Lie group.} 
Trans. AMS. 119, 457-508 (1965).}

\examb{[H-C;InvInt]}{Harish-Chandra: {\it Some results on an Invariant Integral on a Semi-Simple Lie Algebra.} 
Annals of Math. 80, 551-593 (1964).}

\examb{[H-C;DSII]}{Harish-Chandra: {\it Discrete Series for semisimple Lie 
groups, II.} 
Acta Math. 116, 1-111 (1966).}

\examb{[H-C;I]}{Harish-Chandra: {\it Harmonic analysis on real reductive groups, I.}   
J. Func. Anal.19, 104-204 (1975).}

\examb{[H-C;S]}{Harish-Chandra: {\it Supertempered distributions on real reductive 
groups.} Stud. Appl. Math. Adv. Math.(Suppl.) 8, 139-153 (1983).}

\examb{[Hb]}{Herb, R.: {\it Fourier inversion of invariant integrals on semisimple 
real 
Lie groups.} Trans. Am. Soc. 249, 281-302 (1979).}

\examb{[Hb-S]}{Herb, R., Sally, P.: {\it Singular invariant eigendistributions as 
characters in the Fourier transform of invariant distributions.} J. Func. Anal. 33, 
195-210 (1979).}

%\examb{[Hi]}{Hirai,T.: {\it A note on automorphisms of real semisimple Lie algebras.} J. Math. Soc. Japan Vol 28 No. 2 (1976).}

\examb{[H-P]}{Hotta, R., Parthasarathy, R.: {\it A geometric meaning of the multiplicity 
of integrable discrete classes in $L^2(\Gamma\backslash  G)$.} Osaka J. Math 10, 
211-234 (1973).}

\examb{[Ka-K]}{Kaiser, C., K\"ohler, K.: {\it A fixed point formula of Lefschetz type in Arakelov geometry III: representations of Chevalley schemes and heights of flag varieties.} Inventiones mathematicae volume 147, 633-669 (2002).}

\examb{[K]}{K\"ohler, K.: {\it Holomorphic torsion on Hermitian symmetric spaces.} J. 
Reine Angew. Math. 460, 93-116 (1995).}

\examb{[La]}{Lavoine,  J.: {\it Calcul symbolique: distributions et pseudo-fonctions, avec une table de nouvelles transform�es de Laplace. }
Monographies du Centre d'\'Etudes Math�matiques en vue des Applications: B.-M\'ethodes de Calcul, II Centre National de la Recherche Scientifique, Paris, 1959. }

\examb{[Licht]}{Lichtenbaum, S.: {\it The Weil-\'etale for number rings.} Ann. 
of Math. 170 , 657-683 (2009) .}

\examb{[Mor]}{Morin, B.: {\it Zeta functions of regular arithmetic schemes at $s=0$.} Duke
Math. J. 163 , 1263-1336 (2014) .}

\examb{[Lang]}{Langlands, R.P.: {\it The dimension of spaces of automorphic forms.} Amer. Jour. Math. 85, 99-125 (1963).}

\examb{[M-S;I]}{Moscovici, H., Stanton, R. J.: {\it Eta invariants of Dirac operators on 
locally symmetric manifolds.} Invent. math. 95, 629-666 (1989).}

\examb{[M-S;II]}{Moscovici, H., Stanton, R. J.: {\it R-torsion and zeta functions for 
locally symmetric manifolds.} Invent.math. 105, 185-216 (1991).}

\examb{[O-O]}{Okamoto, K., Ozeki, H.: {\it On square-integrable 
$\ov{\partial}$-cohomology 
spaces attached to Hermitian symmetric spaces.} Osaka J. Math. 4, 95-110 (1967).}

\examb{[\'O-S]}{\'Olafsson, G., Stanton, R. J.: {\it Extensions of real bounded symmetric domains.} J. Func. Anal. 279, Issue 8 (2020).}

\examb{[R-S;I]}{Ray, D. B., Singer ,I. M.: {\it R-torsion and the Laplacian on Riemannian 
manifolds.} Adv. Math. 7, 145-210 (1971).}

\examb{[R-S;II]}{Ray, D. B., Singer, I. M.: {\it Analytic torsion for complex manifolds.} 
Ann.Math. 98, 154-177 (1973).}

\examb{[Sa]}{Satake, I.:{ \it Algebraic Structures of Symmetric Domains.} Iwanami 
Shoten, Publishers and Princeton Univ. Press. 1980.}

\examb{[Sc]}{Schmid, W.: {\it On the Characters of the Discrete Series: The Hermitian 
Symmetric Case.} Invent. math. 30, 47-144 (1975).}

\examb{[Sh]}{Shen, S.: {\it Analytic torsion, dynamical zeta functions and the Fried
conjecture.}  Analysis and PDE,  11 (2018).}

\examb{[Sk]}{Sarnak, P.: {\it Determinants of Laplacians.}  Commun. Math. Phys. 110, 113-120 (1987).}

\examb{[Va]}{Varadarajan, V. S. {\it Harmonic Analysis on Real Reductive Groups.} Lecture Notes in Mathematics, Vol. 576, Springer-Verlag, Berlin, 1977.}

\examb{[Wa]}{Warner, G.:{\it  Harmonic Analysis on Semisimple Lie Groups I.,II. } Die Grundlehren der mathematischen Wissenschaften in Einzeldarstellungen, Band 188, 189, Springer-
Verlag Berlin, 1972.}

\examb{[Wo]}{Wolf, J. A.: {\it Fine structure of Hermitian symmetric spaces. } Symmetric spaces, pp. 271-357. Pure and App. Math., Vol. 8, Dekker, New York, 1972. }

\examb{[W-W]}{Whittaker, E.T., Watson, G. N.:{ \it A Course of Modern Analysis.} 
Cambridge 
Univ. Press. Fourth edition 1973.}
\end

%\subjclass[2010]{Primary ; Secondary .}

%\keywords{}

\maketitle

%\begin{abstract}
%abstract
%\end{abstract}

Let $G$ be a connected simple Hermitian Lie group with finite center. Motivated by \cite{MS} we assume that $G$ has only one conjugacy class of cuspidal maximal parabolic subgroups and let $Q=MAN$ be the Langlands decomposition of such a parabolic. Choose a maximal compact subgroup $K\subseteq G$ such that $K_M=M\cap K$ is maximal compact in $M$. We write $\frakg$, $\frakk$, $\frakm$ and $\frakk_\frakm$ for the Lie algebras of $G$, $K$, $M$ and $K_M$. A classification of all possible $\frakg$ is given in Table~\ref{tb:MaxCuspParabolics}.

\begin{table}[h]
	\begin{tabular}{llll}
		\hline
		$\frakg$ & $\frakk$ & $\frakm$ & $\frakk_\frakm$\\
		\hline\hline
		$\su(m,n)$ & $\fraks(\fraku(m)\oplus\fraku(n))$ & $\fraku(m-1,n-1)$ & $\fraku(m-1)\oplus\fraku(n-1)$\\
		$\so^*(2n)$ & $\fraku(n)$ & $\so^*(2n-4)\oplus\su(2)$ & $\fraku(n-2)\oplus\su(2)$\\
		$\so(2,2n)$ & $\so(2)\oplus\so(2n)$ & $\sl(2,\RR)\oplus\so(2n-2)$ & $\so(2)\oplus\so(2n-2)$\\
		$\frake_{6(-14)}$ & $\so(10)\oplus\fraku(1)$ & $\su(1,5)$ & $\fraku(5)$\\
		$\frake_{7(-25)}$ & $\frake_6\oplus\fraku(1)$ & $\so(2,10)$ & $\so(2)\oplus\so(10)$\\
		\hline\vspace{-3mm}
  \end{tabular}
  \caption{Simple Hermitian Lie algebras with a unique conjugacy class of maximal cuspidal parabolic subalgebras}
  \label{tb:MaxCuspParabolics}
\end{table}

Write $K^0(K)$ resp. $K^0(K_M)$ for the Grothendieck group of representations of $K$ resp. $K_M$, and let
$$ \Res:K^0(K)\to K^0(K_M) $$
be the map restricting representations of $K$ to $K_M$. The image of the restriction map is denoted by $\Res(K^0(K))$.

In what follows we use the notation introduced in \cite{MS}. In particular, we consider a holomorphic, Hermitian, homogeneous vector bundle $\VV_\lambda$ over a compact Hermitian locally symmetric manifold $X=\Gamma\backslash G/K$ associated with an irreducible, unitary representation $V_\lambda$ of $K$ of highest weight $\lambda$. The regularity assumptions on $\lambda$ stated in \cite{MS} at the beginning of \S1 are irrelevant for our considerations, so we omit them.

In \cite[page 46]{MS} the Ansatz is used that for every $c\in\RR$ the virtual representation
\begin{equation}
 \Big(\sum_j(-1)^j\Lambda^j\tilde{\frakp}_+^{[\ev]}\Big)\otimes\Big(\sum_{\substack{w\in W_\kappa\\\hat{c}_{\lambda_w}^2=c^2}}(-1)^{\ell(w)}W_{\lambda_w}\Big)\label{eq:AnsatzSum}\tag{$\star$}
\end{equation}
of $K_M$ is contained in $\Res(K^0(K))$. We say that $\lambda$ is a \textit{good} highest weight if this is true.

\begin{thmalph}\label{thm:AnsatzCaseByCase}
\begin{enumerate}
\item\label{thm:AnsatzCaseByCase1} For $G=\SU(1,n)$, $n\geq1$, the restriction map $\Res:K^0(K)\to K^0(K_M)$ is surjective. In particular, every highest weight is good.
\item\label{thm:AnsatzCaseByCase2} For $G=\SO_0(2,2n)$, $n\geq2$, a highest weight $\lambda$ is good if and only if either $\lambda(H_0)=n\sqrt{-1}$ or $V_\lambda|_{\SO(2n)}$ is self-dual.
\item\label{thm:AnsatzCaseByCase3} For $G=\SU(2,3)$ no highest weight is good.
\end{enumerate}
\end{thmalph}

In the formulation of Theorem~\ref{thm:AnsatzCaseByCase}~\eqref{thm:AnsatzCaseByCase2} we have used the (unique) element $H_0\in\frakz(\frakk)$ with $\ad(H_0)|_{\frakp_\pm}=\pm i\cdot\id_{\frakp_\pm}$ (see \cite[page 3]{MS}).

\begin{remalph}
\begin{enumerate}
\item Part \eqref{thm:AnsatzCaseByCase1} of Theorem~\ref{thm:AnsatzCaseByCase} shows that the Ansatz works without additional assumptions on $\VV_\lambda$ if $G/K$ is of rank one. This result can also be found in \cite {M-V83} and \cite{B-O95}. We thank the referee for these references.
\item Part \eqref{thm:AnsatzCaseByCase2} of Theorem~\ref{thm:AnsatzCaseByCase} shows that for $G=\SO_0(2,2n)$ the Ansatz only works if $\VV_\lambda$ is a vector bundle associated to a representation of $K=\SO(2)\times\SO(2n)$ for which either the $\SO(2)$-factor acts by a fixed character depending on $n$, or the $\SO(2n)$-factor acts in a self-dual way. In particular, the Ansatz works for all line bundles since in this case $\SO(2n)$ acts by the trivial representation which is self-dual.
\item The Hermitian symmetric spaces $G/K$ of rank two in the above list belong to the following cases:
\begin{itemize}
\item $\frakg=\su(2,n)$ ($n\geq2$),
\item $\frakg=\so^*(8),\so^*(10)$,
\item $\frakg=\so(2,2n)$ ($n\geq2$),
\item $\frakg=\frake_{6(-14)}$,
\end{itemize}
with the low-dimensional isomorphisms
$$ \su(2,2)\simeq\so(2,4), \qquad \mbox{and} \qquad \so^*(8)\simeq\so(2,6). $$
By part \eqref{thm:AnsatzCaseByCase2} of the previous theorem, the Ansatz works with additional restrictions for $\frakg=\so(2,2n)$, so the same is true for $\frakg=\su(2,2)$ and $\frakg=\so^*(8)$. Already for $\frakg=\su(2,3)$ the Ansatz fails as stated in part \eqref{thm:AnsatzCaseByCase3} of the theorem, and it is likely that the same is true for $\frakg=\su(2,n)$, $n\geq3$. For a complete treatment of all rank two cases the algebras $\frakg=\so^*(10)$ and $\frakg=\frake_{6(-14)}$ are missing.
\item In Section~\ref{sec:SO*(2n)} we provide for the case $G=\SO^*(2n)$ an explicit description of both the restriction map $\Res:K^0(K)\to K^0(K_M)$ and all quantities occurring in \eqref{eq:AnsatzSum}. However, we did not succeed in obtaining a characterization of the good highest weights for $n\geq5$. We conjecture that for $n\geq5$ no highest weight is good, and some computations for $n=5$ indicate that this is indeed the case.
\end{enumerate}
\end{remalph}

\subsection*{Notation}

$\NN=\{0,1,2,\ldots\}$, $\ZZ^n_+=\{\lambda\in\ZZ^n:\lambda_1\geq\ldots\geq\lambda_n\}$, $\ZZ^n_{++}=\{\lambda\in\ZZ^n:\lambda_1\geq\ldots\geq\lambda_{n-1}\geq|\lambda_n|\}$.

\section{$G=\SU(m,n)$}

In this section we verify Theorem~\ref{thm:AnsatzCaseByCase}~\eqref{thm:AnsatzCaseByCase1} and \eqref{thm:AnsatzCaseByCase3}.

\subsection{Some subgroups of $\SU(m,n)$}

Let $G=\SU(m,n)$ and choose the maximal compact subgroup
\begin{equation*}
 K = \left\{\left(\begin{array}{cc}g&\\&h\end{array}\right):g\in\upU(m),h\in\upU(n),\det(g)\det(h)=1\right\} \simeq \upS(\upU(m)\times\upU(n)).
\end{equation*}
Put
$$ X_\kappa = \begin{pmatrix}0&&1\\&\0_{m+n-2}&\\1&&0\end{pmatrix}, $$
then $\ad(X_\kappa)$ acts on $\frakg$ with eigenvalues $0$, $\pm1$ and $\pm2$. Write $\frakm^1$ for the $0$-eigenspace and $\frakn$ for the direct sum of the positive eigenspaces, then $\frakq=\frakm^1\oplus\frakn$ is a cuspidal maximal parabolic subalgebra of $\frakg$. We further decompose $\frakm^1=\frakm\oplus\fraka$ where $\fraka=\RR X_\kappa$ and $\frakm$ is a direct sum of semisimple and compact abelian ideals. On the group level, $Q=N_G(\frakq)$ is a cuspidal maximal parabolic subgroup of $G$ with Langlands decomposition $Q=MAN$, where $MA=Z_G(\fraka)$, $A=\exp(\fraka)$ and $N=\exp(\frakn)$. The intersection $K_M=K\cap M$ is maximal compact in $M$ and given by
\begin{equation*}
 K_M = \left\{\left(\begin{array}{cccc}z&&&\\&g&&\\&&h&\\&&&z\end{array}\right):z\in\upU(1),g\in\upU(m-1),h\in\upU(n-1),z^2\det(g)\det(h)=1\right\}.
\end{equation*}

\subsection{The branching law}

Both $K$ and $K_M$ are connected, so that we can describe irreducible representations in terms of their highest weights. Let
$$ \frakt = \{\sqrt{-1}\diag(t_1,\ldots,t_{m+n}):t_i\in\RR,t_1+\cdots+t_{m+n}=0\} \subseteq \frakk, $$
then $\frakt$ is a maximal torus of $\frakk$ and $\frakg$. The root system $\Delta(\frakk_\CC,\frakt_\CC)$ is given by $\{\pm(\varepsilon_i-\varepsilon_j):1\leq i<j\leq m\mbox{ or }m+1\leq i<j\leq m+n\}$, where
$$ \varepsilon_i(\sqrt{-1}\diag(t_1,\ldots,t_{m+n})) = \sqrt{-1}t_i. $$
We choose the positive system $\Delta^+(\frakk_\CC,\frakt_\CC)=\{\varepsilon_i-\varepsilon_j:1\leq i<j\leq m\mbox{ or }m+1\leq i<j\leq m+n\}$. With this notation, irreducible representations of $K$ are parametrized by their highest weights $\lambda=\lambda_1'\varepsilon_1+\cdots+\lambda_m'\varepsilon_m+\lambda_1''\varepsilon_{m+1}+\cdots+\lambda_n''\varepsilon_{m+n}$, where $\lambda'=(\lambda_1',\ldots,\lambda'_m)\in\ZZ^m_+$ and $\lambda''=(\lambda_1'',\ldots,\lambda_n'')\in\ZZ^n_+$. We write $\lambda=(\lambda',\lambda'')$ for convenience and denote by $\pi_\lambda=\pi_{\lambda',\lambda''}$ the corresponding equivalence class of representations. Note that $\pi_{\lambda'+k,\lambda''+k}\simeq\pi_{\lambda',\lambda''}$ for $k\in\ZZ$, where $\lambda'+k=(\lambda_1'+k,\ldots,\lambda_m'+k)$ and similar for $\lambda''$.

The intersection $\frakt_M=\frakt\cap\frakk_M$ is a maximal torus of $\frakk_M$ and we write $\overline{\varepsilon}_i=\varepsilon_i|_{\frakt_M}$. Then $\overline{\varepsilon}_1=\overline{\varepsilon}_{m+n}$ and $\Delta^+(\frakk_{M,\CC},\frakt_{M,\CC})=\{\overline{\varepsilon}_i-\overline{\varepsilon}_j:2\leq i<j\leq m\mbox{ or }m+1\leq i<j\leq m+n-1\}$ is a positive system in $\Delta(\frakk_{M,\CC},\frakt_{M,\CC})$. With this notation, irreducible representations of $K_M$ are parametrized by highest weights $p\overline{\varepsilon}_1+\mu_1'\overline{\varepsilon}_2+\cdots+\mu_{m-1}'\overline{\varepsilon}_m+\mu_1''\overline{\varepsilon}_{m+1}+\cdots+\mu_{n-1}''\overline{\varepsilon}_{m+n-1}$, where $\mu'=(\mu_1',\ldots,\mu_{m-1}')\in\ZZ^{m-1}_+$, $\mu''=(\mu_1'',\ldots,\mu_{n-1}'')\in\ZZ^{n-1}_+$, $p\in\ZZ$ and we write $\tau_{\mu',\mu'',p}$ for the corresponding equivalence class of representations. Note that $\tau_{\mu'+k,\mu''+k,p+2k}\simeq\tau_{\mu',\mu'',p}$.

For tuples $\lambda\in\ZZ^k_+$ and $\mu\in\ZZ^{k-1}_+$ we write
$$ \mu\subseteq\lambda \quad :\Leftrightarrow \quad \lambda_1\geq \mu_1\geq \lambda_2\geq\ldots\geq \lambda_{k-1}\geq \mu_{k-1}\geq \lambda_k $$
and denote $|\lambda|=\lambda_1+\cdots+\lambda_k$.

\begin{lemma}\label{lem:U(m,n)branching}
For $(\lambda',\lambda'')\in\ZZ^m_+\times\ZZ^n_+$ the following branching law holds:
$$ \pi_{\lambda',\lambda''}|_{K_M} \simeq \bigoplus_{\substack{\mu'\subseteq\lambda'\\\mu''\subseteq\lambda''}}\tau_{\mu',\mu'',|\lambda'|+|\lambda''|-|\mu'|-|\mu''|}. $$
\end{lemma}

\begin{proof}
The classical branching law for the restriction $\upU(k)\searrow\upU(k-1)$ states that the irreducible representation of $\upU(k-1)$ of parameter $\mu\in\ZZ^{k-1}_+$ occurs in the restriction of the irreducible representation of $\upU(k)$ of parameter $\lambda\in\ZZ^k_+$ if and only if $\mu\subseteq\lambda$, and in this case it occurs with multiplicity one (see e.g. \cite[Theorem 9.14]{Kna02}). Keeping track of the action of the $U(1)$-factor in $K_M$ gives the claimed branching formula.
\end{proof}

\begin{proof}[Proof of Theorem~\ref{thm:AnsatzCaseByCase}~\eqref{thm:AnsatzCaseByCase1}]
Assume $n=1$, then $\pi_{\lambda',\lambda''}\simeq\pi_{\lambda'-\lambda'',0}$, $\lambda'\in\ZZ_+^m$, $\lambda''\in\ZZ_+^1=\ZZ$, and we abbreviate $\pi_{\lambda}=\pi_{\lambda,0}$, $\lambda'\in\ZZ^m_+$. We further write $\tau_{\mu,p}=\tau_{\mu,0,p}$ since $\ZZ^{n-1}_+=\ZZ^0$. In this simplified notation the branching law reads
$$ \pi_\lambda|_{K_M} = \bigoplus_{\mu\subseteq\lambda} \tau_{\mu,|\lambda|-|\mu|}. $$
The representations $\pi_\lambda$ are pairwise inequivalent and $\tau_{\mu,p}\simeq\tau_{\mu+k,p+2k}$, $k\in\ZZ$. We now show by induction on
$$ \ell(\mu) = (\mu_1-\mu_{m-1})+\cdots+(\mu_{m-2}-\mu_{m-1}) $$
that every $\tau_{\mu,p}$ is in the image of the restriction map. For $\ell(\mu)=0$ we have $\mu=(q,\ldots,q)$, $q\in\ZZ$, and
$$ \tau_{\mu,p}\simeq\tau_{(2q-p,\ldots,2q-p),2q-p}\simeq\pi_{(2q-p,\ldots,2q-p)}|_{K_M}. $$
Hence $\tau_{\mu,p}$ is in the image of the restriction map. Next assume that $\tau_{\nu,p}$ is in the image of $\Res$ for all $\nu\in\ZZ^{m-1}_+$ with $\ell(\nu)\leq k$. Let $\mu\in\ZZ^{m-1}_+$ with $\ell(\mu)=k+1$ and $p\in\ZZ$. Since
$$ \tau_{\mu,p} \simeq \tau_{(\mu_1+\mu_{m-1}-p,\ldots,\mu_{m-2}+\mu_{m-1}-p,2\mu_{m-1}-p),2\mu_{m-1}-p} $$
we may assume that $\mu_{m-1}=p$. Put $\lambda=(\mu_1,\ldots,\mu_{m-1},\mu_{m-1})$, then $\nu\subseteq\lambda$ implies $\nu_i\leq\mu_i$ for $i=1,\ldots,m-2$ and $\nu_{m-1}=\mu_{m-1}$ so that
$$ \ell(\nu) = (\nu_1-\nu_{m-1})+\cdots+(\nu_{m-2}-\nu_{m-1}) \leq (\mu_1-\mu_{m-1})+\cdots+(\mu_{m-2}-\mu_{m-1}) = \ell(\mu). $$
Moreover, the only $\nu\subseteq\lambda$ with $\ell(\nu)=\ell(\mu)=k+1$ is $\nu=\mu$. Applying the induction hypothesis to every $\tau_{\nu,|\lambda|-|\nu|}$ with $\nu\subseteq\lambda$, $\nu\neq\mu$, we find that 
$$ \pi_\lambda|_{K_M} \simeq \tau_{\mu,|\lambda|-|\mu|} \oplus \mbox{terms in $\Res(K^0(K))$.} $$
Since $|\lambda|-|\mu|=\mu_{m-1}=p$ this shows that $\tau_{\mu,p}$ is contained in $\Res(K^0(K))$, and the proof is complete.
\end{proof}

\subsection{The space $\tilde{\frakp}_+^{[\ev]}$}

Define a homomorphism $\kappa:\su(1,1)\to\su(m,n)$ as the composition of the homomorphism
$$ \widetilde{\kappa}:\su(1,1)\to\su(m,n), \quad \left(\begin{array}{cc}x&y\\z&w\end{array}\right)\mapsto\left(\begin{array}{ccc}x&&y\\&\0_{m+n-2}\\z&&w\end{array}\right) $$
and the Lie algebra isomorphism
$$ \sl(2,\RR)\to\su(1,1), \quad \left(\begin{array}{cc}a&b\\c&-a\end{array}\right)\mapsto\left(\begin{array}{cc}\frac{b-c}{2}i&a-\frac{b+c}{2}i\\a+\frac{b+c}{2}i&-\frac{b-c}{2}i\end{array}\right). $$
Then the decomposition $\frakg=\frakg^{[0]}\oplus\frakg^{[1]}\oplus\frakg^{[2]}$ into $\kappa(\sl(2,\RR))$-isotypic components is given by
\begin{align*}
 \frakg^{[0]} &= \left\{\begin{pmatrix}a&&\\&X&\\&&a\end{pmatrix}a\in i\RR,X\in\fraku(m-1,n-1),2a+\tr(X)=0\right\},\\
 \frakg^{[1]} &= \begin{pmatrix}0&\star&0\\\star&\0_{m+n-2}&\star\\0&\star&0\end{pmatrix},\\
 \frakg^{[2]} &= \kappa(\sl(2,\RR)).
\end{align*}
We identify $\frakg_\CC$ with $\sl(m+n,\CC)$, then the choice of $\kappa$ determines
$$ \frakp_+ = \left\{\begin{pmatrix}\0_m&X\\&\0_n\end{pmatrix}:X\in M(m\times n,\CC)\right\}, $$
so that
$$ \tilde{\frakp}_+^{[\ev]} = \frakp_+^{[0]} = \begin{pmatrix}0&&&\\&\0_{m-1}&\star&\\&&\0_{n-1}&\\&&&0\end{pmatrix}. $$
As representation of $K_M$ we have $\tilde{\frakp}_+^{[\ev]}\simeq\tau_{(1,0,\ldots,0),(0,\ldots,0,-1),0}$. Using \cite[Exercise 6.11]{FH91} we obtain the following decomposition for its exterior powers of:
$$ \Lambda^j\tilde{\frakp}_+^{[\ev]} \simeq \bigoplus\tau_{\mu',\mu'',0}, $$
where the summation is over all $(\mu',\mu'')\in\ZZ_+^{m-1}\times\ZZ_+^{n-1}$ with $\mu_1'\geq\ldots\geq\mu_{m-1}'\geq0\geq \mu_1''\geq\ldots\geq\mu_{n-1}''$, $|\mu'|=j=-|\mu''|$ and such that the partition $(-\mu_{n-1}'',\ldots,-\mu_1'')$ of $j$ is conjugate to $\mu'$, i.e.
$$ -\mu_j'' = \#\{i:\mu_i'\geq n-j\}. $$

\begin{lemma}\label{lem:SU(m,n)exteriorpplus}
For $n=2$ we have
$$ \tau_{0,-1,1}\otimes\Big(\sum_j(-1)^j\Lambda^j\tilde{\frakp}_+^{[\ev]}\Big)\in\Res(K^0(K)). $$
\end{lemma}

\begin{proof}
We write
$$ (1)_j = (\underbrace{1,\ldots,1}_{\text{$j$ times}},0,\ldots,0), $$
the length of the vector being clear from the context. In the Grothendieck group $K^0(K_M)$ we form the alternating sum
\begin{align*}
 \Big(\sum_{j=0}^m(-1)^{j-1}\pi_{(1)_j,(0,-j)}\Big)\Big|_{K_M} ={}& \sum_{j=1}^{m-1}(-1)^{j-1}\sum_{k=0}^j\big(\tau_{(1)_j,-k,k-j}+\tau_{(1)_{j-1},-k,k-j+1}\big)\\
 & -\tau_{0,0,0}+(-1)^{m-1}\sum_{k=0}^m\tau_{(1)_{m-1},-k,k-m+1}\\
 ={}& \sum_{j=1}^m(-1)^{j-1}\tau_{(1)_{j-1},-j,1}\\
 ={}& \Big(\sum_{j=0}^{m-1}(-1)^j\tau_{(1)_j,-j,0}\Big)\otimes\tau_{0,-1,1}\\
 ={}& \Big(\sum_{j=0}^{m-1}(-1)^j\Lambda^j\tilde{\frakp}^{[\ev]}_+\Big)\otimes\tau_{0,-1,1},
\end{align*}
where we have used Lemma~\ref{lem:U(m,n)branching} in the first step and a telescoping argument in the second step.
\end{proof}

\begin{remark}
It is likely that a similar statement holds in general, but since the Ansatz already fails in the case $n=2$, we did not attempt to prove such a generalization.
\end{remark}

\subsection{The numbers $\hat{c}_{\lambda_w}$}

We now compute the numbers $\hat{c}_{\lambda_w}$. The element $H_\kappa\in\su(m,n)$ is given by
$$ H_\kappa = \kappa\left(\begin{array}{cc}0&1\\-1&0\end{array}\right) = \diag(i,0,\ldots,0,-i). $$
This implies that
$$ L_\kappa = Z_K(H_\kappa) = \upS(\upU(1)\times\upU(m-1)\times\upU(n-1)\times\upU(1)). $$
The torus $\frakt$ is contained in $\frakl_\kappa$ and $\Delta(\frakl_{\kappa,\CC},\frakt_\CC)=\{\pm(\varepsilon_i-\varepsilon_j):2\leq i<j\leq m\mbox{ or }m+1\leq i<j\leq m+n-1\}$. We choose the positive roots $\Delta^+(\frakl_{\kappa,\CC},\frakt_\CC)=\{\varepsilon_i-\varepsilon_j:2\leq i<j\leq m\mbox{ or }m+1\leq i<j\leq m+n-1\}\subseteq\Delta^+(\frakk_\CC,\frakt_\CC)$.

The Weyl group $W(\frakk_\CC)$ of $\Delta(\frakk_\CC,\frakt_\CC)$ is naturally isomorphic to $S_m\times S_n$, where $S_k$ denotes the symmetric group in $k$ letters. Then
\begin{align*}
 W_\kappa &= \{w\in W(\frakk_\CC):w^{-1}\alpha>0\,\forall\,\alpha\in\Delta^+(\frakl_{\kappa,\CC},\frakt_\CC)\}\\
 &= \{(w_1,w_2)\in S_m\times S_n:w_1^{-1}(2)<\ldots<w_1^{-1}(m)\mbox{ and }w_2^{-1}(1)<\ldots<w_2^{-1}(n-1)\}\\
 &= \{(w_1^{(i)},w_2^{(j)}):i=1,\ldots,m,\,j=1,\ldots,n\},
\end{align*}
where
$$ w_1^{(i)}(k) = \begin{cases}k+1 & \mbox{for $1\leq k<i$,}\\1 & \mbox{for $k=i$,}\\k & \mbox{for $i<k\leq m$,}\end{cases} \qquad w_2^{(j)}(k) = \begin{cases}k & \mbox{for $1\leq k<j$,}\\n & \mbox{for $k=j$,}\\k-1 & \mbox{for $j<k\leq n$.}\end{cases} $$
Note that $\ell(w_1^{(i)})=i-1$ and $\ell(w_2^{(j)})=j-1$, and since $\ell(w_1^{(i)},w_2^{(j)})=\ell(w_1^{(i)})+\ell(w_2^{(j)})$ we obtain
$$ (-1)^{\ell(w_1^{(i)},w_2^{(j)})} = (-1)^{i+j}. $$
We further have
$$ \rho_c = (\rho_{c,1},\rho_{c,2}) $$
with
$$ \rho_1 = (\tfrac{m-1}{2},\tfrac{m-3}{2},\ldots,\tfrac{1-m}{2}) = (\tfrac{m-2k+1}{2})_{k=1,\ldots,m}, \qquad \rho_{c,2} = (\tfrac{n-2k+1}{2})_{k=1,\ldots,n}.  $$

Now let $\lambda=(\lambda',\lambda'')\in\ZZ^m_+\times\ZZ^n_+$ be a highest weight of an irreducible $K$-representation. Then
\begin{align*}
 w_1^{(i)}(\lambda'+\rho_{c,1})-\rho_{c,1} ={}& w_1^{(i)}(\lambda_1'+\tfrac{m-1}{2},\ldots,\lambda_m'+\tfrac{1-m}{2})-(\tfrac{m-1}{2},\ldots,\tfrac{1-m}{2})\\
 ={}& (\lambda_i'+\tfrac{m-2i+1}{2},\lambda_1'+\tfrac{m-1}{2},\ldots,\widehat{\lambda_i'+\tfrac{m-2i+1}{2}},\ldots,\lambda_m'+\tfrac{1-m}{2})\\
 & \hspace{6.5cm}-(\tfrac{m-1}{2},\tfrac{m-3}{2},\ldots,\tfrac{1-m}{2})\\
 ={}& (\lambda_i'+1-i,\lambda_1'+1,\ldots,\lambda_{i-1}'+1,\lambda_{i+1}',\ldots,\lambda_m')
\end{align*}
and similarly
$$ w_2^{(j)}(\lambda''+\rho_{c,2})-\rho_{c,2} = (\lambda_1'',\ldots,\lambda_{j-1}'',\lambda_{j+1}''-1,\ldots,\lambda_n''-1,\lambda_j''+n-j). $$
Together this gives for $w=(w_1^{(i)},w_2^{(j)})$:
\begin{align*}
 \lambda_w = (w_1^{(i)},w_2^{(j)})((\lambda',\lambda'')+\rho_c)-\rho_c ={}& \Big((\lambda_i'+1-i,\lambda_1'+1,\ldots,\lambda_{i-1}'+1,\lambda_{i+1}',\ldots,\lambda_m'),\\
 & \quad (\lambda_1'',\ldots,\lambda_{j-1}'',\lambda_{j+1}''-1,\ldots,\lambda_n''-1,\lambda_j''+n-j)\Big).
\end{align*}
Restricting this weight to $\frakt_M$ we obtain a dominant integral weight for $\frakk_M$ which belongs to the representation
$$ \tau_{(\lambda_1'+1,\ldots,\lambda_{i-1}'+1,\lambda_{i+1}',\ldots,\lambda_m'),(\lambda_1'',\ldots,\lambda_{j-1}'',\lambda_{j+1}''-1,\ldots,\lambda_n''-1),\lambda_i'+\lambda_j''+n+1-i-j}. $$

We now compute $\hat{c}_{\lambda_w}=((\lambda_w\circ C_\kappa)+\rho_Q)(\widehat{X}_\kappa)$:
\begin{align*}
 ((\lambda_w\circ C_\kappa)+\rho_Q)(X_\kappa) &= -\sqrt{-1}\lambda_w(H_\kappa) + \rho_Q(X_\kappa)\\
 &= (\lambda_i'-\lambda_j''-n-i+j+1) + (m+n-1)\\
 &= \lambda_i'-\lambda_j''+(m-i)+j.
\end{align*}

\subsection{Invariants for $(m,n)=(3,2)$}

For $(\mu',\mu'',p)\in\ZZ_+^2\times\ZZ\times\ZZ$ with $\mu'=(\mu_1',\mu_2')$ we put
\begin{equation}
 I(\tau_{\mu',\mu'',p}) := (\mu_1'+\mu''-p)+(\mu_2'+\mu''-p)+(\mu_1'+\mu''-p)^2-(\mu_2'+\mu''-p)^2\label{eq:U(m,n)DefInvariant}
\end{equation}
and extend this $\ZZ$-linearly to the Grothendieck group $K^0(K_M)$ which is the free $\ZZ$-module generated by all $\tau_{\mu',\mu'',p}$. Then $I$ is an invariant for the image $\Res(K^0(K))$ of the restriction map:

\begin{lemma}\label{lem:Invariant}
For all $(\lambda',\lambda'')\in\ZZ_+^3\times\ZZ_+^2$ we have $I(\pi_{\lambda',\lambda''}|_{K_M})=0$. In particular, $I$ vanishes on the image $\Res(K^0(K))$ of the restriction map.
\end{lemma}

\begin{proof}
Write $\lambda'=(a+K+L,a+L,a)$ and $\lambda''=(b+M,b)$ with $a,b\in\ZZ$ and $K,L,M\geq0$, then
$$ \pi_{\lambda',\lambda''}|_{K_M} \simeq \bigoplus_{k=0}^K\bigoplus_{\ell=0}^L\bigoplus_{m=0}^M \tau_{(a+L+k,a+\ell),b+m,a+b+(K-k)+(L-\ell)+(M-m)}. $$
According to the four summands in \eqref{eq:U(m,n)DefInvariant} we split $I(\pi_{\lambda',\lambda''}|_{K_M})$ into four parts which are straightforward to compute:
\begin{align*}
 \sum_{k,\ell,m}(\mu_1'+\mu''-p) &= \sum_{k,\ell,m}\Big(2k+\ell+2m-K-M\Big) = \frac{L}{2}(K+1)(L+1)(M+1),\\
 \sum_{k,\ell,m}(\mu_2'+\mu''-p) &= \sum_{k,\ell,m}\Big(k+2\ell+2m-K-L-M\Big) = -\frac{K}{2}(K+1)(L+1)(M+1),\\
 \sum_{k,\ell,m}(\mu_1'+\mu''-p)^2 &= \sum_{k,\ell,m}\Big(2k+\ell+2m-K-M\Big)^2\\
 &= (K+1)(L+1)(M+1)\Big[\frac{1}{3}K(K+2)+\frac{1}{6}L(2L+1)+\frac{1}{3}M(M+2)\Big],\\
 \sum_{k,\ell,m}(\mu_2'+\mu''-p)^2 &= \sum_{k,\ell,m}\Big(k+2\ell+2m-K-L-M\Big)^2\\
 &= (K+1)(L+1)(M+1)\Big[\frac{1}{6}K(2K+1)+\frac{1}{3}L(L+2)+\frac{1}{3}M(M+2)\Big].
\end{align*}
Subtracting the last expression from the sum of the first three expressions gives $0$.
\end{proof}

We use the invariant $I$ to prove that for $(m,n)=(3,2)$ no highest weight is good.

\begin{proof}[Proof of Theorem~\ref{thm:AnsatzCaseByCase}~\eqref{thm:AnsatzCaseByCase3}]
In view of \eqref{eq:AnsatzSum} and Lemma~\ref{lem:SU(m,n)exteriorpplus} it suffices to show that for any $\lambda$ there exists $c\in\RR$ such that the sum
\begin{equation}
 \tau_{0,1,-1}\otimes\Big(\sum_{\substack{w\in W_\kappa\\\hat{c}_{\lambda_w}^2=c^2}}(-1)^{\ell(w)}W_{\lambda_w}\Big) = \sum_{\substack{w\in W_\kappa\\\hat{c}_{\lambda_w}^2=c^2}}(-1)^{\ell(w)}\tau_{0,1,-1}\otimes W_{\lambda_w}\label{eq:U(3,2)relevantsummation}
\end{equation}
is not contained in $\Res(K^0(K))$. For $\lambda=(\lambda',\lambda'')$ we list all summands with corresponding $\hat{c}_{\lambda_w}$ and invariant $I$ in Table~\ref{tb:U(3,2)summands}.
\begin{table}[h]
	\begin{tabular}{cccc}
		$w$ & $(-1)^{\ell(w)}\tau_{0,1,-1}\otimes W_{\lambda_w}$ & $\hat{c}_{\lambda_w}$ & $I(\tau_{0,1,-1}\otimes W_{\lambda_w})=0$\\
		\hline
		$a=(w_1^{(1)},w_2^{(1)})$ & $\tau_{(\lambda_2',\lambda_3'),\lambda_2'',\lambda_1'+\lambda_1''}$ & $A=\lambda_1'-\lambda_1''+3$ & $2\lambda_1'-\lambda_2'-\lambda_3'+2\lambda_1''-2\lambda_2''=0$\\
		$b=(w_1^{(2)},w_2^{(1)})$ & $-\tau_{(\lambda_1'+1,\lambda_3'),\lambda_2'',\lambda_2'+\lambda_1''-1}$ & $B=\lambda_2'-\lambda_1''+2$ & $\lambda_1'-2\lambda_2'+\lambda_3'-2\lambda_1''+2\lambda_2''=-3$\\
		$c=(w_1^{(3)},w_2^{(1)})$ & $\tau_{(\lambda_1'+1,\lambda_2'+1),\lambda_2'',\lambda_3'+\lambda_1''-2}$ & $C=\lambda_3'-\lambda_1''+1$ & $\lambda_1'+\lambda_2'-2\lambda_3'-2\lambda_1''+2\lambda_2''=-6$\\
		$d=(w_1^{(1)},w_2^{(2)})$ & $-\tau_{(\lambda_2',\lambda_3'),\lambda_1''+1,\lambda_1'+\lambda_2''-1}$ & $D=\lambda_1'-\lambda_2''+4$ & $2\lambda_1'-\lambda_2'-\lambda_3'-2\lambda_1''+2\lambda_2''=4$\\
		$e=(w_1^{(2)},w_2^{(2)})$ & $\tau_{(\lambda_1'+1,\lambda_3'),\lambda_1''+1,\lambda_2'+\lambda_2''-2}$ & $E=\lambda_2'-\lambda_2''+3$ & $\lambda_1'-2\lambda_2'+\lambda_3'+2\lambda_1''-2\lambda_2''=-7$\\
		$f=(w_1^{(3)},w_2^{(2)})$ & $-\tau_{(\lambda_1'+1,\lambda_2'+1),\lambda_1''+1,\lambda_3'+\lambda_2''-3}$ & $F=\lambda_3'-\lambda_2''+2$ & $\lambda_1'+\lambda_2'-2\lambda_3'+2\lambda_1''-2\lambda_2''=-10$
  \end{tabular}
  \caption{}\label{tb:U(3,2)summands}
\end{table}
The inequalities $\lambda_1'\geq\lambda_2'\geq\lambda_3'$ and $\lambda_1''\geq\lambda_2''$ imply the following relations between $A,B,C,D,E,F$:
$$ \begin{array}{ccccc}D&>&E&>&F\\\begin{sideways}$<$\end{sideways}&&\begin{sideways}$<$\end{sideways}&&\begin{sideways}$<$\end{sideways}\\A&>&B&>&C\end{array} $$
Using these inequalities one can systematically consider all possible groupings of representations with the same $\hat{c}_{\lambda_w}^2$ to show that there is always $c\in\RR$ such that the sum \eqref{eq:U(3,2)relevantsummation} does not belong to $\Res(K^0(K))$. We illustrate this for the case $C^2=D^2$ and $E^2=F^2$. Since $D>C$ this implies $D=-C$ and we get a first identity
\begin{equation}
 \lambda_1'+\lambda_3'-\lambda_1''-\lambda_2'' = -5.\label{eq:NegativeEq1}
\end{equation}
Further, since $E>F$ we have $E=-F$ which implies the additional identity
\begin{equation}
 \lambda_2'+\lambda_3'-2\lambda_2''=-5.\label{eq:NegativeEq2}
\end{equation}
We now have the following possibilities for $B^2$:
\begin{itemize}
\item $B^2\notin\{A^2,E^2=F^2\}$. Since $D>B>C$ and $D^2=C^2$ we further have $B^2\notin\{C^2=D^2\}$, so that for $c=B$ the sum \eqref{eq:U(3,2)relevantsummation} consists of only one representation. By Lemma~\ref{lem:Invariant} this representations needs to have vanishing invariant $I$ which implies
\begin{equation}
 \lambda_1'-2\lambda_2'+\lambda_3'-2\lambda_1''+2\lambda_2''=-3.\label{eq:NegativeEq3a}
\end{equation}
Solving \eqref{eq:NegativeEq1}, \eqref{eq:NegativeEq2} and \eqref{eq:NegativeEq3a} yields
$$ (\lambda_1',\lambda_2',\lambda_3') = (-1+\tfrac{1}{2}\lambda_1''+\tfrac{1}{2}\lambda_2'',-1-\tfrac{1}{2}\lambda_1''+\tfrac{3}{2}\lambda_2'',-4+\tfrac{1}{2}\lambda_1''+\tfrac{1}{2}\lambda_2''). $$
Moreover, since $D>A,B,E,F>C$ we have $C^2=D^2\notin\{A^2,B^2,E^2,F^2\}$, so that for $c=D=-C$ the sum \eqref{eq:U(3,2)relevantsummation} consists of only two representations, and the invariant $I$ of the sum evaluates to
$$ -2(\lambda_1''-\lambda_2''+1)(\lambda_1''-\lambda_2''-4). $$
Again by Lemma~\ref{lem:Invariant} this invariant has to vanish, which implies $\lambda_1''=\lambda_2''+4$ and hence
$$ (\lambda_1',\lambda_2',\lambda_3') = (1+\lambda_2'',-3+\lambda_2'',-2+\lambda_2''), $$
a contradiction to $\lambda_2'\geq\lambda_3'$.
\item $B^2=A^2$. Since $A>B$ we have $A=-B$ which implies the additional identity
\begin{equation}
 \lambda_1'+\lambda_2'-2\lambda_1'' = -5.\label{eq:NegativeEq3b}
\end{equation}
Solving \eqref{eq:NegativeEq1}, \eqref{eq:NegativeEq2} and \eqref{eq:NegativeEq3b} yields
\begin{equation}
 (\lambda_1',\lambda_2',\lambda_3') = (-\tfrac{5}{2}+\tfrac{3}{2}\lambda_1''-\tfrac{1}{2}\lambda_2'',-\tfrac{5}{2}+\tfrac{1}{2}\lambda_1''+\tfrac{1}{2}\lambda_2'',-\tfrac{5}{2}-\tfrac{1}{2}\lambda_1''+\tfrac{3}{2}\lambda_2'').\label{eq:NegativSolb}
\end{equation}
Then for $c=D-C$ the sum \eqref{eq:U(3,2)relevantsummation} consists of only two representations, and the invariant $I$ of the sum evaluates to
$$ 2(\lambda_1''-\lambda_2''+1)^2. $$
This is $\neq0$ since $\lambda_1''\geq\lambda_2''$, a contradiction to Lemma~\ref{lem:Invariant}.
\item $B^2=E^2=F^2$. Since $E>B$ we have $E=-B$ which implies the additional identity
\begin{equation}
 2\lambda_2'-\lambda_1''-\lambda_2'' = -5.\label{eq:NegativeEq3c}
\end{equation}
Solving \eqref{eq:NegativeEq1}, \eqref{eq:NegativeEq2} and \eqref{eq:NegativeEq3c} yields \eqref{eq:NegativSolb}, and by the same argument as above this shows that for $c=D$ the sum \eqref{eq:U(3,2)relevantsummation} is not contained in $\Res(K^0(K))$.\qedhere
\end{itemize}
\end{proof}

\section{$G=\SO_0(2,2n)$}

In this section we verify Theorem~\ref{thm:AnsatzCaseByCase}~\eqref{thm:AnsatzCaseByCase2}.

\subsection{Some subgroups of $\SO_0(2,2n)$}

Let $G=\SO_0(2,2n)$, $n\geq2$, and choose the maximal compact subgroup $K=\SO(2)\times\SO(2n)\subseteq G$. Put
$$ X_\kappa = \left(\begin{array}{ccc}\0_2&\1_2&\\\1_2&\0_2&\\&&\0_{2n-2}\end{array}\right), $$
then $\ad(X_\kappa)$ acts on $\frakg$ with eigenvalues $0$, $\pm1$ and $\pm2$. Write $\frakm^1$ for the $0$-eigenspace and $\frakn$ for the direct sum of the positive eigenspaces, then $\frakq=\frakm^1\oplus\frakn$ is a cuspidal maximal parabolic subalgebra of $\frakg$. We further decompose $\frakm^1=\frakm\oplus\fraka$ where $\fraka=\RR X_\kappa$ and $\frakm$ is a direct sum of semisimple and compact abelian ideals. On the group level, $Q=N_G(\frakq)$ is a cuspidal maximal parabolic subgroup of $G$ with Langlands decomposition $Q=MAN$, where $MA=Z_G(\fraka)$, $A=\exp(\fraka)$ and $N=\exp(\frakn)$. The intersection $K_M=K\cap M$ is maximal compact in $M$ and given by
$$ K_M = \left\{\left(\begin{array}{ccc}g&&\\&g&\\&&h\end{array}\right):g\in\SO(2),h\in\SO(2n-2)\right\} \simeq \SO(2)\times\SO(2n-2). $$

\subsection{The branching law}

Both $K$ and $K_M$ are connected, so that we can describe irreducible representations in terms of their highest weights. Let
$$ J = \left(\begin{array}{cc}0&1\\-1&0\end{array}\right) $$
and
$$ \frakt = \{\diag(t_0J,t_1J,\ldots,t_nJ):t_i\in\RR\} \subseteq \frakk, $$
then $\frakt$ is a maximal torus of $\frakk$ and $\frakg$. The root system $\Delta(\frakk_\CC,\frakt_\CC)$ is given by $\{\pm\varepsilon_i\pm\varepsilon_j:1\leq i<j\leq n$, where
$$ \varepsilon_i(\diag(t_0J,\ldots,t_nJ)) = \sqrt{-1}t_i. $$
We choose the positive system $\Delta^+(\frakk_\CC,\frakt_\CC)=\{\varepsilon_i\pm\varepsilon_j:1\leq i<j\leq n\}$. With this notation, irreducible representations of $K$ are parametrized by their highest weights $\lambda=\lambda_0\varepsilon_0+\lambda_1\varepsilon_1+\cdots+\lambda_n\varepsilon_n$, where $p=\lambda_0\in\ZZ$ and $\lambda'=(\lambda_1,\ldots,\lambda_n)\in\ZZ_{++}^n$. We denote by $\pi_\lambda=\pi_{\lambda',p}$ the corresponding equivalence class of representations.

The intersection $\frakt_M=\frakt\cap\frakk_M$ is a maximal torus of $\frakk_M$ and we write $\overline{\varepsilon}_i=\varepsilon_i|_{\frakt_M}$. Then $\overline{\varepsilon}_0=\overline{\varepsilon}_1$ and $\Delta^+(\frakk_{M,\CC},\frakt_{M,\CC})=\{\overline{\varepsilon}_i-\overline{\varepsilon}_j:2\leq i<j\leq n\}$ is a positive system in $\Delta(\frakk_{M,\CC},\frakt_{M,\CC})$. With this notation, irreducible representations of $K_M$ are parametrized by highest weights $\mu=q\overline{\varepsilon}_1+\mu_1\overline{\varepsilon}_2+\cdots+\mu_{n-1}\overline{\varepsilon}_n$, where $q\in\ZZ$ and $\mu=(\mu_1,\ldots,\mu_{n-1})\in\ZZ_{++}^{n-1}$, and we write $\tau_{\mu,q}$ for the corresponding equivalence class of representations.

\begin{lemma}
For $(\lambda',p)\in\ZZ_{++}^n\times\ZZ$ the following branching law holds:
$$ \pi_{\lambda',p}|_{K_M} \simeq \bigoplus_{\substack{\mu\in\ZZ^{n-1}_{++}\\\lambda_1\geq\mu_1\geq\lambda_3\\\cdots\\\lambda_{n-2}\geq\mu_{n-2}\geq|\lambda_n|\\\lambda_{n-1}\geq|\mu_{n-1}|}}\bigoplus_{k=0}^\ell \tau_{\mu,p+\ell_n+2k-\ell}^{\oplus m(k)}, $$
where
\begin{align*}
 &\ell_1 = \lambda_1-\max(\lambda_2,\mu_1), \qquad \ell_i = \min(\lambda_i,\mu_{i-1})-\max(\lambda_{i+1},\mu_i) \quad (2\leq i\leq n-2),\\
 &\ell_{n-1} = \min(\lambda_{n-1},\mu_{n-2})-\max(|\lambda_n|,|\mu_{n-1}|), \qquad \ell_n = \sgn(\lambda_n)\sgn(\mu_{n-1})\min(|\lambda_n|,|\mu_{n-1}|),
\end{align*}
and $\ell=\sum_{i=1}^{n-1}\ell_i$, and the multiplicities $m(k)$ are given by
$$ m(k)=\#\left\{(k_1,\ldots,k_{n-1})\in\ZZ^{n-1}:0\leq k_i\leq\ell_i,\sum_{i=1}^{n-1}k_i=k\right\}. $$
\end{lemma}

Note that
$$ m(0)=m(\ell)=1. $$

\begin{proof}
By \cite[Theorem 1.1]{Tsu81} the restriction $\pi_{\lambda',p}|_{K_M}$ is the claimed direct sum with multiplicities $m(k)$ given by the coefficient of $1$ in the Laurent series expansion of
$$ \prod_{i=1}^{n-1}\frac{X^{\ell_i+1}-X^{-\ell_i-1}}{X-X^{-1}}. $$
But since
$$ \frac{X^{\ell_i+1}-X^{-\ell_i-1}}{X-X^{-1}} = \sum_{k_i=0}^{\ell_i} X^{2k_i-\ell_i} $$
the claimed formula for the multiplicities follows.
\end{proof}

\begin{proposition}\label{prop:RestrictionSO(2,n)}
The image $\Res(K^0(K))$ of the restriction map is spanned over $\ZZ$ by
\begin{align*}
 &\tau_{\mu,q} && \mbox{for $\mu_{n-1}=0$, and}\\
 &\tau_{\mu,q+s}+\tau_{\mu^*,q-s} && \mbox{for $\mu_{n-1}\neq0$, $s\in\ZZ$ and $\mu^*=(\mu_1,\ldots,\mu_{n-2},-\mu_{n-1})$,}
\end{align*}
with $(\mu,q)\in\ZZ^{n-1}_{++}\times\ZZ$.
\end{proposition}

\begin{proof}
That the image is actually contained in this span follows immediately from the fact that in the branching law the representations appear in pairs. More precisely, for each $\tau_{\mu,p+\ell_n+2k-\ell}$ that occurs in the restriction of $\pi_{\lambda',p}$ with multiplicity $m(k)$, the representation $\tau_{\mu,p-\ell_n+2k-\ell}$, $\mu=(\mu_1,\ldots,\mu_{n-2},-\mu_{n-1})$ occurs with the same multiplicity. (For this note that if $\mu_{n-1}$ is changed to $-\mu_{n-1}$ only the number $\ell_n$ switches sign, everything else stays the same.) For $\mu_{n-1}=0$ these are the same representations, and for $\mu_{n-1}\neq0$ they form a pair
$$ \tau_{\mu,q+s}\oplus\tau_{\mu,q-s}, $$
where $q=p+2k-\ell$ and $s=\ell_n=\sgn(\lambda_n)\sgn(\mu_{n-1})\min(|\lambda_n|,|\mu_{n-1}|)$.\\
To show that the image is actually equal to the span, we use induction on $\ell(\mu)=\mu_1+\cdots+\mu_{n-2}+|\mu_{n-1}|$. For $\ell(\mu)=0$ we have $\mu=(0,\ldots,0)$ and for $\lambda'=(0,\ldots,0)$:
$$ \pi_{(0,\ldots,0),q}|_{K_M} \simeq \tau_{(0,\ldots,0),q} = \tau_{\mu,q} \qquad \forall\,q\in\ZZ. $$
For $\ell(\mu)>0$ and $|s|\leq|\mu_{n-1}|$ let $\lambda'=(\mu_1,\ldots,\mu_{n-2},|\mu_{n-1}|,\sgn(\mu_{n-1})s)$, then
$$ \pi_{\lambda',q}|_{K_M} = \tau_{\mu,q+s}\oplus\tau_{\mu^*,q-s}\oplus\mbox{lower order terms} \qquad \forall\,q\in\ZZ. $$
(In fact, for $\mu$ and $\mu^*$ we have $\ell_1=\ldots=\ell_{n-1}=0$ in the branching law and hence the multiplicity is $1$.) By applying the induction hypothesis it follows that $\tau_{\mu,q+s}\oplus\tau_{\mu^*,q-s}$ is contained in the image of the restriction map. Now, let $s\in\ZZ$ be arbitrary and write $s=s_1+\cdots+s_{2k+1}$ with $|s_i|\leq|\mu_{n-1}|$. Then
\begin{align*}
 & \tau_{\mu,q+s}+\tau_{\mu^*,q-s}\\
 ={}& (\tau_{\mu,q+s}+\tau_{\mu^*,q+s-2s_1}) - (\tau_{\mu^*,q+s-2s_1}+\tau_{\mu,q+s-2(s_1+s_2)})\\
 & + (\tau_{\mu,q+s-2(s_1+s_2)}+\tau_{\mu^*,q+s-2(s_1+s_2+s_3)}) - (\tau_{\mu^*,q+s-2(s_1+s_2+s_3)}+\tau_{\mu,q+s-2(s_1+s_2+s_3+s_4)})\\
 & + \cdots - \cdots\\
 & + (\tau_{\mu,q+s-(s_1+\cdots+s_{2k})}+\tau_{\mu^*,q+s-2(s_1+\cdots+s_{2k+1})})
\end{align*}
where each sum in parenthesis is contained in the image of the restriction map. Hence also $\tau_{\mu,q+s}+\tau_{\mu^*,q-s}$ is contained in the image and the proof is complete.
\end{proof}

\subsection{The space $\tilde{\frakp}_+^{[\ev]}$}

Define a homomorphism $\kappa:\sl(2,\RR)\to\so(2n,2)$ on the standard basis $E,F,H\in\sl(2,\RR)$ by
$$ \kappa(E) = \frac{1}{2}\left(\begin{array}{ccc}-J&J&\\-J&J&\\&&\0_{2n-2}\end{array}\right), \kappa(F) = \frac{1}{2}\left(\begin{array}{ccc}J&J&\\-J&-J&\\&&\0_{2n-2}\end{array}\right), \kappa(H) = X_\kappa. $$
Then the decomposition $\frakg=\frakg^{[0]}\oplus\frakg^{[1]}\oplus\frakg^{[2]}$ into $\kappa(\sl(2,\RR))$-isotypic components is given by
\begin{align*}
 \frakg^{[0]} &= \left\{\left(\begin{array}{ccc}A&B&\\B&A&\\&&Z\end{array}\right):A\in\so(2),B\in\Sym(2,\RR),\tr(B)=0,Z\in\so(2n-2)\right\},\\
 \frakg^{[1]} &= \left\{\left(\begin{array}{ccc}\0_2&\0_2&X^\top\\\0_2&\0_2&-Y^\top\\X&Y&\0_{2n-2}\end{array}\right):X,Y\in M((2n-2)\times2,\RR)\right\},\\
 \frakg^{[2]} &= \kappa(\sl(2,\RR)).
\end{align*}
The choice of $\kappa$ determines
$$ \frakp_+ = \left\{\left(\begin{array}{ccc}0&0&z^\top\\0&0&-iz^\top\\z&-iz&\0_{2n}\end{array}\right):z\in\CC^{2n}\right\} \subseteq \so(2,2n)_\CC, $$
so that
$$ \tilde{\frakp}_+^{[\ev]} = \frakp_+^{[0]} = \left\{\left(\begin{array}{ccc}\0_2&Z&\\Z&\0_2&\\&&\0_{2n-2}\end{array}\right):Z=\left(\begin{array}{cc}z&-iz\\-iz&-z\end{array}\right),z\in\CC\right\}. $$
As representation of $K_M$ we have $\tilde{\frakp}_+^{[\ev]}\simeq\tau_{(0,\ldots,0),-2}\simeq\pi_{(0,\ldots,0),-2}|_{K_M}$. Hence
\begin{equation}
 \sum_j(-1)^j\Lambda^j\tilde{\frakp}_+^{[\ev]} \in \Res(K^0(K)).\label{eq:SO(2,2n)exteriorpplus}
\end{equation}

\subsection{The numbers $\hat{c}_{\lambda_w}$}

We now compute the numbers $\hat{c}_{\lambda_w}$. The element $H_\kappa\in\frakg$ is given by
$$ H_\kappa = \kappa(E-F) = \left(\begin{array}{ccc}-J&&\\&J&\\&&\0_{2n-2}\end{array}\right). $$
which implies that
$$ L_\kappa = Z_K(H_\kappa) = \SO(2)\times\SO(2)\times\SO(2n-2). $$
Note that $\frakt\subseteq\frakl_\kappa$ and that the choice of positive roots $\Delta^+(\frakk_\CC,\frakt_\CC)$ is compatible with $\frakl_\kappa$ in the sense that $\Delta^+(\frakl_{\kappa,\CC},\frakt_\CC)=\Delta(\frakl_{\kappa,\CC},\frakt_\CC)\cap\Delta^+(\frakk_\CC,\frakt_\CC)$ is a positive system of roots for $\frakl_\kappa$. The Weyl group $W(\frakk_\CC)$ of $\Delta(\frakk_\CC)$ is naturally isomorphic to $\{\pm1\}^n_\even\rtimes S_n$, where $S_n$ denotes the symmetric group in $n$ letters and $\{\pm1\}^n_\even$ is the kernel of the homomorphism $\{\pm1\}^n\to\{\pm1\},\,(\delta_1,\ldots,\delta_n)\mapsto\delta_1\cdots\delta_n$. Then
\begin{align*}
 W_\kappa &= \{w\in W(\frakk_\CC):w^{-1}\alpha>0\,\forall\,\alpha\in\Delta^+(\frakl_{\kappa,\CC})\} = \{(\delta_\pm,w^{(i)}):i=1,\ldots,n\},
\end{align*}
where
$$ w^{(i)}(k) = \begin{cases}k+1 & \mbox{for $1\leq k<i$,}\\1 & \mbox{for $k=i$,}\\k & \mbox{for $i<k\leq n$,}\end{cases} \qquad \delta_\pm=(\pm1,1,\ldots,1,\pm1). $$
Note that $\ell(w^{(i)})=i-1$ and $\ell(\delta_\pm)=1$, so that
$$ (-1)^{\ell(\delta_\pm,w^{(i)})} = (-1)^{i-1}. $$
We further have
$$ \rho_c = (n-1,n-2,\ldots,1,0) = (n-j)_{j=1,\ldots,n}.  $$

Now let $\lambda=(\lambda',p)\in\ZZ_{++}^n\times\ZZ$ be a highest weight of an irreducible $K$-representation. Then for $1\leq i<n$ the Weyl group element $w=(\delta_\pm,w^{(i)})$ only acts on $\lambda'$ and we have
\begin{align*}
 \lambda'_w ={}& (\delta_\pm,w^{(i)})(\lambda'+\rho_c)-\rho_c\\
 ={}& (\delta_\pm,w^{(i)})(\lambda_1+n-1,\lambda_2+n-2,\ldots,\lambda_n)-(n-1,n-2,\ldots,0)\\
 ={}& \delta_\pm(\lambda_i+n-i,\lambda_1+n-1,\ldots,\widehat{\lambda_i+n-i},\ldots,\lambda_n)\\
 & \hspace{7cm}-(n-1,n-2,\ldots,0)\\
 ={}& (\pm(\lambda_i+n-i),\lambda_1+n-1,\ldots,\widehat{\lambda_i+n-i},\ldots,\lambda_{n-1}+1,\pm\lambda_n)\\
 & \hspace{7cm}-(n-1,n-2,\ldots,0)\\
 ={}& \begin{cases}(\lambda_i-i+1,\lambda_1+1,\ldots,\lambda_{i-1}+1,\lambda_{i+1},\ldots,\lambda_{n-1}+1,\lambda_n) & \mbox{for $+$,}\\(-\lambda_i-2n+i+1,\lambda_1+1,\ldots,\lambda_{i-1}+1,\lambda_{i+1},\ldots,\lambda_{n-1}+1,-\lambda_n) & \mbox{for $-$.}\end{cases}
\end{align*}
Restricting the weight $\lambda_w=(\lambda'_w,p)$ to $\frakt_M$ we obtain a dominant integral weight for $\frakk_M$ which belongs to the representation
$$ W_{\lambda_w} = \begin{cases}\tau_{(\lambda_1+1,\ldots,\lambda_{i-1}+1,\lambda_{i+1},\ldots,\lambda_{n-1},\lambda_n),p+\lambda_i-i+1} & \mbox{for $w=(\delta_+,w^{(i)})$,}\\\tau_{(\lambda_1+1,\ldots,\lambda_{i-1}+1,\lambda_{i+1},\ldots,\lambda_{n-1},-\lambda_n),p-\lambda_i-2n+i+1} & \mbox{for $w=(\delta_-,w^{(i)})$.}\end{cases} $$
For $i=n$ a similar computation yields
$$ \lambda'_w = \begin{cases}(\lambda_n-n+1,\lambda_1+1,\ldots,\lambda_{n-2}+1,\lambda_{n-1}+1)&\mbox{for $w=(\delta_+,w^{(n)})$,}\\(-\lambda_n-n+1,\lambda_1+1,\ldots,\lambda_{n-2}+1,-\lambda_{n-1}-1)&\mbox{for $w=(\delta_+,w^{(n)})$,}\end{cases} $$
and
$$ W_{\lambda_w} = \begin{cases}\tau_{(\lambda_1+1,\ldots,\lambda_{n-2}+1,\lambda_{n-1}+1),p+\lambda_n-n+1} & \mbox{for $w=(\delta_+,w^{(n)})$,}\\\tau_{(\lambda_1+1,\ldots,\lambda_{n-2}+1,-\lambda_{n-1}-1),p-\lambda_n-n+1} & \mbox{for $w=(\delta_-,w^{(i)})$.}\end{cases} $$

%Note that by Proposition~\ref{prop:RestrictionSO(2,n)} every summand is contained in the image of the restriction map since the difference $(\lambda_i-i+1)-(-\lambda_i-2n+i+1)=2\lambda_i+2n-2i$ is even. However, we must check whether the two terms in each summand have matching $\hat{c}_{\lambda_w}^2$.

We now compute $\hat{c}_{\lambda_w}=((\lambda_w\circ C_\kappa)+\rho_Q)(\widehat{X}_\kappa)$. Since $\widehat{X}_\kappa=(-\|H_\kappa\|^2)^{-1/2}X_\kappa$ and $C_\kappa(X_\kappa)=-\sqrt{-1}H_\kappa$ we have
$$ \hat{c}_{\lambda_w} = (-\|H_\kappa\|^2)^{-1/2}\big(\rho_Q(X_\kappa)-\sqrt{-1}\lambda_w(H_\kappa)\big). $$
Here $\rho_Q(X_\kappa)=\frac{1}{2}(2(2n-2)+2)=2n-1$ and
$$ \lambda_w(H_\kappa) = \begin{cases}\sqrt{-1}(-p+\lambda_i-i+1)&\mbox{for $w=(\delta_+,w^{(i)})$,}\\\sqrt{-1}(-p-\lambda_i-2n+i+1)&\mbox{for $w=(\delta_-,w^{(i)})$,}\end{cases} $$
so that
$$ \hat{c}_{\lambda_w} = (-\|H_\kappa\|^2)^{-1/2}\times\begin{cases}-p+\lambda_i-i+2n&\mbox{for $w=(\delta_+,w^{(i)})$,}\\-p-\lambda_i+i&\mbox{for $w=(\delta_-,w^{(i)})$.}\end{cases} $$

\begin{proof}[Proof of Theorem~\ref{thm:AnsatzCaseByCase}~\eqref{thm:AnsatzCaseByCase2}]
In view of \eqref{eq:AnsatzSum} and \eqref{eq:SO(2,2n)exteriorpplus} a highest weight $\lambda=(\lambda',p)\in\ZZ^n_{++}\times\ZZ$ is good if and only if for all $c\in\RR$ the sum
\begin{equation}
 \sum_{\substack{w\in W_\kappa\\\hat{c}_{\lambda_w}^2=c^2}}(-1)^{\ell(w)}W_{\lambda_w}\label{eq:SO(2,2n)relevantsummation}
\end{equation}
is contained in $\Res(K^0(K))$. Since $(-1)^{\ell(\delta_\pm)}=1$ and $(-1)^{\ell(w^{(i)})}=(-1)^{i-1}$ the sum \eqref{eq:SO(2,2n)relevantsummation} without the restriction $\hat{c}_{\lambda_w}^2=c^2$ takes the form
\begin{multline}
 \sum_{i=1}^{n-1} (-1)^{i-1}\Big(\tau_{(\lambda_1+1,\ldots,\lambda_{i-1}+1,\lambda_{i+1},\ldots,\lambda_{n-1},\lambda_n),p+\lambda_i-i+1}+\tau_{(\lambda_1+1,\ldots,\lambda_{i-1}+1,\lambda_{i+1},\ldots,\lambda_{n-1},-\lambda_n),p-\lambda_i-2n+i+1}\Big)\\
 +(-1)^{n-1}\Big(\tau_{(\lambda_1+1,\ldots,\lambda_{n-2}+1,\lambda_{n-1}+1),p+\lambda_n-n+1}+\tau_{(\lambda_1+1,\ldots,\lambda_{n-2}+1,-\lambda_{n-1}-1),p-\lambda_n-n+1}\Big).\label{eq:BigSumSO(2,2n)}
\end{multline}
Each expression in parentheses in \eqref{eq:BigSumSO(2,2n)} is of the form $\tau_{\mu,q+s}+\tau_{\mu^*,q-s}$ with $q=p-n+1$ and $s=\lambda_i+n-i$ and therefore contained in $\Res(K^0(K))$, thanks to Proposition~\ref{prop:RestrictionSO(2,n)}. Moreover, this is the only possible way of combining two representations in \eqref{eq:BigSumSO(2,2n)} to a sum of the form $\tau_{\mu,q+s}+\tau_{\mu^*,q-s}$. Therefore, a highest weight $\lambda=(\lambda',p)$ is good if and only if for each expression in parentheses in \eqref{eq:BigSumSO(2,2n)} either both representations are separately contained in $\Res(K^0(K))$ or both representations have the same value of $\hat{c}_{\lambda_w}^2$. Note that in the $i$-th expression in parentheses ($1\leq i\leq n$) the two representations correspond to the Weyl group elements $w=(\delta_\pm,w^{(i)})$ and have therefore the same value of $\hat{c}_{\lambda_w}^2$ if and only if either $\lambda_i=-(n-i)$ (in which case the $\hat{c}_{\lambda_w}$'s agree) or $p=n$ (in which case the $\hat{c}_{\lambda_w}$'s are opposite numbers).\\
Assume first that $\lambda_n=0$, i.e. $V_\lambda|_{\SO(2n)}$ is self-dual. Then in the first $n-1$ expressions in parentheses in \eqref{eq:BigSumSO(2,2n)} every single representation is contained in $\Res(K^0(K))$ by Proposition~\ref{prop:RestrictionSO(2,n)}. Moreover, for $i=n$ we have $\lambda_i=0=-(n-i)$ and hence the values of $\hat{c}_{\lambda_w}^2$ of the two representations in the last expression in parentheses agree. This shows that $\lambda$ is a good highest weight.\\
Now assume that $\lambda_n\neq0$. Then none of the representations in \eqref{eq:BigSumSO(2,2n)} is separately contained in $\Res(K^0(K))$, whence the highest weight $\lambda$ is good if and only if for each $i\in\{1,\ldots,n\}$ the values of $\hat{c}_{\lambda_w}^2$ agree for the two Weyl group elements $w=(\delta_\pm,w^{(i)})$. As remarked above, this is only the case if for every $i\in\{1,\ldots,n\}$ either $\lambda_i=-(n-i)$ or $p=n$. For $i=1$ we have $\lambda_1\geq0$, but $-(n-i)=-(n-1)<0$ since $n\geq2$ and therefore the highest weight $\lambda$ is good if and only if $p=n$. This finishes the proof, since $p=-\sqrt{-1}\lambda(H_0)$ for
\begin{equation*}
 H_0 = \begin{pmatrix}0&1&\\-1&0&\\&&\0_{2n}\end{pmatrix}.\qedhere
\end{equation*}
\end{proof}

\section{$G=\SO^*(2n)$}\label{sec:SO*(2n)}

In this section we obtain for $G=\SO^*(2n)$ the full branching law from $K$ to $K_M$, i.e. the explicit description of the restriction map $\Res:K^0(K)\to K^0(K_M)$. Moreover, we compute $\tilde{\frakp}_+^{[\ev]}$, $W_{\lambda_w}$ and $\hat{c}_{\lambda_w}^2$.

\subsection{Some subgroups of $\SO^*(2n)$}

Let $G=\SO^*(2n)$, realized as
$$ \SO^*(2n) = \left\{g\in\GL(2n,\CC):g^\top\begin{pmatrix}&\1_n\\\1_n&\end{pmatrix}g=\begin{pmatrix}&\1_n\\\1_n&\end{pmatrix},g^*\begin{pmatrix}\1_n&\\&-\1_n\end{pmatrix}g=\begin{pmatrix}\1_n&\\&-\1_n\end{pmatrix}\right\}, $$
and choose the maximal compact subgroup
$$ K = \SO^*(2n)\cap\upU(2n) = \left\{\diag(k,\overline{k}):k\in\upU(n)\right\} \simeq \upU(n). $$
The Lie algebra $\frakg$ of $G$ is given by 
$$ \frakg = \left\{\begin{pmatrix}A&B\\B^*&-A^\top\end{pmatrix}:A\in\fraku(n),B\in\Skew(n,\CC)\right\} $$
Put
$$ X_\kappa = \begin{pmatrix}\0_2&&J&\\&\0_{n-2}&&\\-J&&\0_2&\\&&&\0_{n-2}\end{pmatrix}, \qquad \mbox{where }J=\begin{pmatrix}0&1\\-1&0\end{pmatrix}, $$
then $\ad(X_\kappa)$ acts on $\frakg$ with eigenvalues $0$, $\pm1$ and $\pm2$. Write $\frakm^1$ for the $0$-eigenspace and $\frakn$ for the direct sum of the positive eigenspaces, then $\frakq=\frakm^1\oplus\frakn$ is a cuspidal maximal parabolic subalgebra of $\frakg$. We further decompose $\frakm^1=\frakm\oplus\fraka$ where $\fraka=\RR X_\kappa$ and $\frakm$ is a direct sum of semisimple and compact abelian ideals. On the group level, $Q=N_G(\frakq)$ is a cuspidal maximal parabolic subgroup of $G$ with Langlands decomposition $Q=MAN$, where $MA=Z_G(\fraka)$, $A=\exp(\fraka)$ and $N=\exp(\frakn)$. The intersection $K_M=K\cap M$ is maximal compact in $M$ and given by
$$ K_M = \left\{\diag(k,\overline{k}):k=\diag(k_1,k_2),k_1\in\SU(2),k_2\in\upU(n-2)\right\} \simeq \SU(2)\times\upU(n-2). $$

\subsection{The branching law}

Both $K$ and $K_M$ are connected, so we can describe irreducible representations in terms of their highest weights. Let
$$ \frakt = \{\diag(k,\overline{k}):k=\sqrt{-1}\diag(t_1,\ldots,t_n),t_i\in\RR\}, $$
then $\frakt$ is a maximal torus in $\frakk$ and $\frakg$. The root system $\Delta(\frakk_\CC,\frakt_\CC)$ is given by $\{\pm(\varepsilon_i-\varepsilon_j):1\leq i<j\leq n\}$, where
$$ \varepsilon_i(\diag(k,\overline{k}) = \sqrt{-1}t_i $$
if $k$ is of the above form. We choose the positive system $\Delta^+(\frakk_\CC,\frakt_\CC)=\{\varepsilon_i-\varepsilon_j:1\leq i<j\leq n\}$. With this notation, irreducible representations of $K$ are parametrized by their highest weights $\lambda=\lambda_1\varepsilon_1+\cdots+\lambda_n\varepsilon_n$, where $\lambda=(\lambda_1,\ldots,\lambda_n)\in\ZZ^n_+$. Denote by $\pi_\lambda$ the corresponding equivalence class of representations of $K$.

The intersection $\frakt_M=\frakt\cap\frakk_M$ is a maximal torus in $\frakk_M$ and we write $\overline{\varepsilon}_i=\varepsilon_i|_{\frakt_M}$. Then $\overline{\varepsilon}_1=-\overline{\varepsilon}_2$ and $\Delta^+(\frakk_{M,\CC},\frakt_{M,\CC})=\{2\overline{\varepsilon}_1\}\cup\{\overline{\varepsilon}_i-\overline{\varepsilon}_j:3\leq i<j\leq n\}$ is a positive system in $\Delta(\frakk_{M,\CC},\frakt_{M,\CC})$. With this notation, irreducible representations of $K_M$ are parametrized by highest weights $p\overline{\varepsilon}_1+\nu_1\overline{\varepsilon}_3+\cdots+\nu_{n-2}\overline{\varepsilon}_n$, where $p\geq0$ and $\nu=(\nu_1,\ldots,\nu_{n-2})\in\ZZ^{n-2}_+$. Write $\tau_{\nu,p}$ for the corresponding equivalence class of representations.

\begin{lemma}
$$ \pi_\lambda|_{K_M} \simeq \bigoplus_{\substack{\nu\in\ZZ^{n-2}_+\\\lambda_1\geq\nu_1\geq\lambda_3\\\cdots\\\lambda_{n-2}\geq\nu_{n-2}\geq\lambda_n}} \tau_{\nu,p(\lambda,\nu)}, $$
where
$$ p(\lambda,\nu) = \lambda_1-\sum_{i=2}^{n-1}|\lambda_i-\nu_{i-1}|-\lambda_n \geq 0. $$
\end{lemma}

\begin{proof}
We use the Gelfand--Tsetlin basis for the representation $\pi_\lambda$ of $\upU(n)\subseteq\GL(n,\CC)$ (see e.g. \cite{Mol06}). Let $\Lambda=(\Lambda_{i,j})_{1\leq j\leq i\leq n}$ be a Gelfand--Tsetlin pattern with $\Lambda_{n,j}=\lambda_j$ and $\xi_\Lambda$ be the corresponding weight vector in $\pi_\lambda$. Then $\xi_\Lambda$ is a highest weight vector in an irreducible $\GL(n-2,\CC)$-representation if and only if $\Lambda_{i,j}=\nu_j$ for all $1\leq j\leq i\leq n-2$ with some $\nu\in\ZZ^{n-2}_+$ with $\lambda_j\geq\nu_j\geq\lambda_{j+2}$ for $1\leq j\leq n-2$. This explains the direct sum in the decomposition. Further, such a vector $\xi_\Lambda$ is also a highest weight vector for $\GL(2,\CC)$ (i.e. $\pi_\lambda(E_{n-1,n})\xi_\Lambda=0$) if and only if for the second row $\mu_j=\Lambda_{n-1,j}$, $1\leq j\leq n-1$, we have
$$ \mu_j=\lambda_j \quad \mbox{or} \quad \mu_j=\nu_{j-1} \qquad \forall\,j=1,\ldots,n-1. $$
It is easy to see that $\nu\subseteq\mu\subseteq\lambda$ (i.e. $\Lambda$ is in fact a Gelfand--Tsetlin pattern) only if $\mu_1=\lambda_1$ and $\mu_i=\min(\lambda_i,\nu_{i-1})$, $2\leq i\leq n-1$. In this case the vector $\xi_\Lambda$ satisfies
$$ \pi_\lambda(E_{n-1,n-1})\xi_\Lambda = (|\mu|-|\nu|)\xi_\Lambda \qquad \mbox{and} \qquad \pi_\lambda(E_{n,n})\xi_\Lambda = (|\lambda|-|\mu|)\xi_\Lambda, $$
and hence is the highest weight vector in an irreducible $\GL(2,\CC)$-representation of dimension $(|\mu|-|\nu|)-(|\lambda|-|\mu|)+1=2|\mu|-|\lambda|-|\nu|+1$. Inserting the explicit form of $\mu$ then shows that
\begin{equation*}
 2|\mu|-|\lambda|-|\nu| = \lambda_1-\sum_{i=2}^{n-1}|\lambda_i-\nu_{i-1}|-\lambda_n.\qedhere
\end{equation*}
\end{proof}

%\begin{remark}
%It immediately follows that for $\lambda\in\ZZ$:
%$$ \pi_{(\lambda,\ldots,\lambda)}|_{K_M} \simeq \tau_{(\lambda,\ldots,\lambda),0}, $$
%so that all characters of $K_M$ are contained in the image of the restriction map $\Res$ (in contrast to the case $\SU(m,n)$ with $\min(m,n)>1$). However, it does not seem that the restriction map is surjective, the first two non-trivial restrictions are
%\begin{align*}
% \pi_{(\lambda+1,\lambda,\ldots,\lambda)}|_{K_M} &\simeq \tau_{(\lambda+1,\lambda,\ldots,\lambda),0}\oplus\tau_{(\lambda,\ldots,\lambda),1},\\
% \pi_{(\lambda,\ldots,\lambda,\lambda-1)}|_{K_M} &\simeq \tau_{(\lambda,\ldots,\lambda,\lambda-1),0}\oplus\tau_{(\lambda,\ldots,\lambda),1}.
%\end{align*}
%Looking at these restrictions, a possible statement that could hold would be that the formal differences
%$$ \tau_{\mu+\nu,p}-\tau_{\mu+\nu^*,p}\in\Rep K_M, $$
%where $\nu\in\ZZ^{n-1}$ and $\nu^*=(-\nu_{n-1},\ldots,-\nu_1)$, are contained in the image of the restriction map. But I would have to look into this further to find out whether such a statement is true, right now it's just a wild guess. Let me know if such a statement could be relevant for you.
%\end{remark}

\subsection{The space $\tilde{\frakp}_+^{[\ev]}$}

Define a homomorphism $\kappa:\sl(2,\RR)\to\so^*(2n)$ on the standard basis $E,F,H\in\sl(2,\RR)$ by
$$ \kappa(E) = \frac{\sqrt{-1}}{2}\begin{pmatrix}\1_2&&-J&\\&\0_{n-2}&&\\-J&&-\1_2&\\&&&\0_{n-2}\end{pmatrix}, \quad \kappa(F) = \frac{\sqrt{-1}}{2}\begin{pmatrix}-\1_2&&-J&\\&\0_{n-2}&&\\-J&&\1_2&\\&&&\0_{n-2}\end{pmatrix} $$
and $\kappa(H)=X_\kappa$. Then the decomposition $\frakg=\frakg^{[0]}\oplus\frakg^{[1]}\oplus\frakg^{[2]}$ into $\kappa(\sl(2,\RR))$-isotypic components is given by
\begin{align*}
 \frakg^{[0]} &= \left\{\begin{pmatrix}a&&\0_2&\\&A&&B\\\0_2&&-a^\top&\\&B^*&&-A^\top\end{pmatrix}:\begin{array}{l}a\in\fraku(2),\tr(a)=0\\A\in\fraku(n-2)\\B\in\Skew(n-2,\CC)\end{array}\right\},\\
 \frakg^{[1]} &= \left\{\begin{pmatrix}\0_2&z&\0_2&w\\-z^*&\0_{n-2}&-w^\top&\0_{n-2}\\\0_2&-\overline{w}&\0_2&\overline{z}\\w^*&\0_{n-2}&-z^\top&\0_{n-2}\end{pmatrix}:\begin{array}{l}z\in M(2\times(n-2),\CC)\end{array}\right\},\\
 \frakg^{[2]} &= \kappa(\sl(2,\RR)).
\end{align*}
The choice of $\kappa$ determines
$$ \frakp_+ = \left\{\begin{pmatrix}\0_n&B\\\0_n&\0_n\end{pmatrix}:B\in\Skew(n,\CC)\right\} \subseteq \left\{\begin{pmatrix}A&B\\C&-A^\top\end{pmatrix}:\begin{array}{l}A\in\gl(n,\CC)\\B,C\in\Skew(n,\CC)\end{array}\right\} = \so^*(2n)_\CC, $$
so that
$$ \tilde{\frakp}_+^{[\ev]} = \frakp_+^{[0]} = \left\{\begin{pmatrix}\0_n&B\\\0_n&\0_n\end{pmatrix}:B\in\begin{pmatrix}\0_2&\\&\Skew(n-2,\CC)\end{pmatrix}\right\}. $$
As representation of $K_M$ we have $\tilde{\frakp}_+^{[\ev]}\simeq\tau_{(1,1,0,\ldots,0),0}$.

\subsection{The numbers $\hat{c}_{\lambda_w}$}

We now compute the numbers $\hat{c}_{\lambda_w}$. The element $H_\kappa\in\frakg$ is given by
$$ H_\kappa = \kappa(E-F) = \sqrt{-1}\diag(1,1,0,\ldots,0,-1,-1,0,\ldots,0) $$
which implies that
$$ L_\kappa = \{\diag(k,\overline{k}):k=\diag(k_1,k_2),k_1\in\upU(2),k_2\in\upU(n-2)\} \simeq \upU(2)\times\upU(n-2). $$
Note that $\frakt\subseteq\frakl_\kappa$ and that the choice of positive roots $\Delta^+(\frakk_\CC,\frakt_\CC)$ is compatible with $\frakl_\kappa$ in the sense that $\Delta^+(\frakl_{\kappa,\CC},\frakt_\CC)=\Delta^(\frakl_{\kappa,\CC},\frakt_\CC)\cap\Delta^+(\frakk_\CC,\frakt_\CC)$ is a positive system of roots for $\frakl_\kappa$. The Weyl group $W(\frakk_\CC)$ of $\Delta(\frakk_\CC)$ is naturally isomorphic to the symmetric group $S_n$ in $n$ letters. Then
$$ W_\kappa = \{w\in W(\frakk_\CC):w^{-1}\alpha>0\,\forall\,\alpha\in\Delta^+(\frakl_{\kappa,\CC},\frakt_\CC)\} = \{w_{ij}:1\leq i<j\leq n\}, $$
where
$$ w_{ij}(k) = \begin{cases}k+2&\mbox{for $1\leq k<i$,}\\1&\mbox{for $k=i$,}\\k+1&\mbox{for $i<k<j$,}\\2&\mbox{for $k=j$,}\\k&\mbox{for $j<k\leq n$.}\end{cases} $$
Note that
$$ (-1)^{\ell(w_{ij})} = (-1)^{i+j+1}. $$
We further have
$$ \rho_c = (\tfrac{n-1}{2},\tfrac{n-3}{2},\ldots,\tfrac{1-n}{2}) = (\tfrac{n-2i+1}{2})_{i=1,\ldots,n}. $$
Now let $\lambda=(\lambda_1,\ldots,\lambda_n)\in\ZZ^n_+$ be a highest weight of an irreducible $K$-representation. Then for $w=w_{ij}$ we have
\begin{align*}
 \lambda_w ={}& w_{ij}(\lambda+\rho_c)-\rho_c\\
 ={}& w_{ij}(\lambda_1+\tfrac{n-1}{2},\ldots,\lambda_n+\tfrac{1-n}{2})-(\tfrac{n-1}{2},\ldots,\tfrac{1-n}{2})\\
 ={}& (\lambda_i+\tfrac{n-2i+1}{2},\lambda_j+\tfrac{n-2j+1}{2},\lambda_1+\tfrac{n-1}{2},\ldots,\widehat{\lambda_i+\tfrac{n-2i+1}{2}},\ldots,\widehat{\lambda_j+\tfrac{n-2j+1}{2}},\ldots,\lambda_n+\tfrac{1-n}{2})\\
 & \hspace{11.3cm}-(\tfrac{n-1}{2},\ldots,\tfrac{1-n}{2})\\
 ={}& (\lambda_i-i+1,\lambda_j-j+2,\lambda_1+2,\ldots,\lambda_{i-1}+2,\lambda_{i+1}+1,\ldots,\lambda_{j-1}+1,\lambda_{j+1},\ldots,\lambda_n).
\end{align*}
Restricting the weight $\lambda_w$ to $\frakt_M$ we obtain a dominant integral weight for $K_M$ which belongs to the representation
$$ W_{\lambda_w} = \tau_{(\lambda_1+2,\ldots,\lambda_{i-1}+2,\lambda_{i+1}+1,\ldots,\lambda_{j-1}+1,\lambda_{j+1},\ldots,\lambda_n),\lambda_i-\lambda_j-i+j-1}. $$

We now compute $\hat{c}_{\lambda_w}=((\lambda_w\circ C_\kappa+\rho_Q)(\widehat{X}_\kappa)$. Since $\widehat{X}_\kappa=(-\|H_\kappa\|^2)^{-1/2}X_\kappa$ and $C_\kappa(X_\kappa)=-\sqrt{-1}H_\kappa$ we have
$$ \hat{c}_{\lambda_w} = ((\lambda_w\circ C_\kappa)+\rho_Q)(\widehat{X}_\kappa) $$
$$ \hat{c}_{\lambda_w} = (\lambda_i+\lambda_j-i-j+3) + (2n-3) = \lambda_i+\lambda_j+2n-i-j. $$

%\bibliographystyle{amsplain}
%\bibliography{bibdb}